\documentclass[reqno,10pt]{amsart}
\usepackage{amsfonts}
\usepackage[colorlinks,linkcolor=blue,anchorcolor=red,citecolor=blue]{hyperref}
\usepackage[margin=1in]{geometry} 
\usepackage{marginnote}
\usepackage{bm}

\numberwithin{equation}{section}

\usepackage{amsmath,amssymb,amsthm,amsfonts}
\usepackage{mathrsfs}
\usepackage{bbm}
\usepackage{hyperref}
\usepackage{xcolor}

\allowdisplaybreaks[4]%
\everymath{\displaystyle}

\newtheorem{lemma}{Lemma}[section]
\newtheorem{theorem}{Theorem}[section]

\newtheorem{proposition}{Proposition}[section]
\newtheorem{remark}{Remark}[section]

\newcommand{\dis}{\displaystyle}

\newcommand{\R}{\mathbb{R}}

\renewcommand{\S}{\mathbb{S}}

\newcommand{\FP}{\mathbf{P}}

\newcommand{\FI}{\mathbf{I}}

\newcommand{\Fb}{\mathbf{b}}
\newcommand{\Fk}{\mathbf{k}}

\newcommand{\rd}{\mathrm{d}}

\newcommand{\CA}{\mathcal{A}}

\newcommand{\CE}{\mathcal{E}}

\newcommand{\CM}{\mathcal{M}}

\newcommand{\CT}{\mathcal{T}}

\newcommand{\al}{\alpha}

\newcommand{\ga}{\gamma}

\newcommand{\Om}{\Omega}
\newcommand{\la}{\lambda}
\newcommand{\de}{\delta}

\newcommand{\pa}{\partial}

\newcommand{\eps}{\epsilon}

\newcommand{\vho}{\varrho}
\newcommand{\vps}{\varepsilon}

\newcommand{\Ga}{\Gamma}

\newcommand{\eqdef}{\overset{\mbox{\tiny{def}}}{=}}

\begin{document}

\title[Free boundary problem of the Boltzmann equation]{On a free boundary problem of the Boltzmann equation}

\author[S.-C. Chang]{Shengchuang Chang}
\address[SCC]{School of Mathematics and Statistics and Hubei Key Laboratory of Mathematical Sciences, Central China Normal University, Wuhan 430079, P.R.~China}
\email{csc981020@mails.ccnu.edu.cn}

\author[R.-J. Duan]{Renjun Duan}
\address[RJD]{Department of Mathematics, The Chinese University of Hong Kong,
	Shatin, Hong Kong, P.R.~China}
\email{rjduan@math.cuhk.edu.hk}

\author[S.-Q. Liu]{Shuangqian Liu}
\address[SQL]{School of Mathematics and Statistics and Key Lab NAA--MOE, Central China Normal University, Wuhan 430079, P.R.~China}
\email{sqliu@ccnu.edu.cn}

\begin{abstract}
This paper studies a free boundary problem for the Boltzmann equation that models the interaction of rarefied gas with a moving wall---a classical piston problem in kinetic theory. The piston motion is governed by Newton's law under the drag force exerted by the gas, with or without an additional Hookean restoring force. A central challenge for such a problem is the strong coupling between the kinetic equation and the free moving boundary, for which the classical Lagrangian formulation is unavailable due to low-regularity of solutions. Our approach relies on two new ingredients: a conformal transformation is introduced to reduce the moving domain to a fixed domain, and a coupled energy structure linking the kinetic distribution and free boundary variables is uncovered to reveal intrinsic dissipation and cancellation mechanisms at the interface. These structural observations lead to a global nonlinear theory for the fully coupled system. As a result, we establish the global existence, uniqueness, nonlinear stability, and exponential convergence to equilibrium of solutions near a global Maxwellian. This provides the first rigorous global well-posedness theory for a fully coupled Boltzmann free boundary problem of piston type.\end{abstract}
	



	
\maketitle
\tableofcontents	
\thispagestyle{empty}

\section{Introduction}
The Boltzmann equation is a fundamental equation in kinetic theory that describes the statistical behavior of a dilute gas out of equilibrium, cf.~\cite{Ce1988}. In various physical contexts, the domain occupied by a rarefied gas is not fixed but evolves dynamically over time, leading to the formulation of free boundary problems for the Boltzmann equation, cf.~\cite{BCM-15Book,Sone07}. Such problems naturally occur in gas-piston systems where the motion of a piston is influenced by the kinetic pressure of the gas, cf.~\cite{ATC,CaMa-12PMP,GrPi-99PA,TAK-15PRE,TSK-18PA}. The situation becomes more complicated with quite few mathematical results when we consider the rarefied gas dynamics confined by a movable surface (for instance, an elastic plate) where gas-structure interaction plays a crucial role and requires kinetic modeling to capture non-equilibrium effects. However, we may refer to \cite{BGN,Ca21BAMS,KKLTT} and references therein for an extensive literature of the fluid-structure interaction; their extensions to some similar results under the gas-structure interaction seem very interesting to be explored, since the fluid equations are hydrodynamic approximations of the kinetic Boltzmann equation for collisional rarefied gas. Moreover, while free boundary problems for fluid dynamic equations have been extensively investigated (cf.~\cite{B-83,CS-07,CS-12,tani-96}, for instance), their counterparts for kinetic models remain largely open due to the intrinsic difficulties caused by the low regularity of solutions (cf.~\cite{Guo-2010,guo-17-inv}), the strong coupling between particle transport and boundary motion, and the nonlinear collision effects of particles.

In this article, we mainly focus on a free boundary value problem for the Boltzmann equation that models the interactions of rarefied gas with a piston (i.e., a point mass) in the spatially one-dimensional setting. Precisely, we consider the motion of rarefied gas confined in the time-dependent finite interval
\begin{equation*}
-1<X_1<x_w(t)
\end{equation*}
with slab symmetry in space, where the left boundary $X_1=-1$ is a fixed wall, while the right boundary $X_1=x_w(t)$ at which the piston is located with the moving velocity $v_w(t)$ is a free boundary that fluctuates near the fixed point $X_1=0$. The problem is formulated to solve the Boltzmann equation
\begin{eqnarray}\label{BE}
		&\dis \pa_tF+v_1 \pa_{X_1} F=Q(F,F), 
		\quad t>0,\ X_1\in (-1,x_w(t)),\  v\in \R^3,
\end{eqnarray}
coupled with the Newton's system
\begin{align}\label{ODE}
\left\{\begin{array}{rll}
		\frac{{\rm d}}{{\rm d}t}x_w(t)&=v_w(t),\\[2mm]
		\frac{{\rm d}}{{\rm d}t}v_w(t)&=-\kappa x_w(t)-\mathcal M\left(P_r[F]-P_l[F] \right).
\end{array}\right.
	\end{align}
Here, the unknown $F(t,X_1,v)\geq0$ stands for the density distribution function of gas particles with velocity $v=(v_1,v_2,v_3)\in \R^3$ at time $t>0$ and spatial position $X_1\in (-1,x_w(t))$, and the other unknowns $[x_w(t),v_w(t)]$ denote the displacement and velocity of the piston. The bilinear Boltzmann collision operator $Q(\cdot,\cdot)$ acting only on velocity variable is given by 
	\begin{equation*}
		Q(G,F)(v)=\int_{\R^3}\int_{\S^2}B(v-u,\omega)\left[ G(u')F(v')-G(u)F(v)\right]{\rm d}\omega {\rm d} u,
	\end{equation*}
where for $\omega\in \mathbb{S}^{2}$, the velocity pairs $(v,u)$ before collision and $(v',u')$ after collision satisfy
$$
v'=v-[(v-u)\cdot\omega]\omega, \quad  u'=u+[(v-u)\cdot\omega]\omega,
$$
in terms of the conservation laws of molecular momentum and energy for elastic collisions:
\begin{equation*}
	v+u=v'+u', \quad |v|^{2}+|u|^{2}=|v'|^{2}+|u'|^{2}.
\end{equation*}
Through the paper, for simplicity we consider only the hard-sphere model 
\begin{equation*}
B(v-u,\omega)=|(v-u)\cdot \omega|,
\end{equation*}
and an extension of main results to more general angular cutoff collision kernels with hard or Maxwell molecule potentials can be made in a straightforward way.  

The ODE system \eqref{ODE} is used to determine the motion of piston driven by the drag force exerted by the surrounding gas particles, with or without the inclusion of Hooke's law. Here, the constant $\kappa\geq0$ is Hooke's law coefficient, and $\mathcal M>0$ is a positive constant reciprocal to the piston mass regarded as a point mass.
The terms $P_r[F]$ and $P_l[F]$ represent the average drag forces exerted by the gas particles on the boundary from the right and left sides, respectively; they are defined by
\begin{equation}\label{def-PF}
\left\{\begin{aligned}
	P_r[F]=\int_{\mathbb R^3}[v_1-v_w(t)]^2F(t,x_w(t)+0,v){\rm d}v,\\ P_l[F]=\int_{\mathbb R^3}[v_1-v_w(t)]^2F(t,x_w(t)-0,v){\rm d}v.
\end{aligned}\right.
\end{equation}
In \eqref{def-PF} above, we further define $F(t,x_w(t)\pm 0,v)$ at the moving interface $X_1=x_w(t)\pm 0$ as 
		\begin{align}\label{+0-0 data->0}
\left\{\begin{array}{ll}
		&F(t,x_w(t)+0,v)=\left\{\begin{array}{rll}
			&\mu(v),\ \textrm{if}\ v_1-v_w(t)<0,\\[2mm]
			&\sqrt{2\pi}\mu_w(v)\int_{u_1-v_w(t)<0}\mu(u)|u_1-v_w(t)|{\rm d}u,\ \textrm{if}\ v_1-v_w(t)>0,
		\end{array}\right.\\[8mm]
		&F(t,x_w(t)-0,v)=F(t,x_w(t),v).
\end{array}\right.
	\end{align}
Here, the distribution function of rarefied gas on the right side of the interface $X_1=x_w(t)+ 0$ is given by the global Maxwellian $\mu(v)$ for incoming particles with relative velocity $v_1-v_w(t)<0$ and by the local Maxwellian $\sqrt{2\pi}\mu_w(v)\int_{u_1-v_w(t)<0}\mu(u)|u_1-v_w(t)|{\rm d}u$ for reflected outgoing particles with relative velocity $v_1-v_w(t)>0$, where we have denoted
\begin{equation*}
\left\{\begin{aligned}
	&\mu(v)=\mu(v_1,v_2,v_3)=\frac{1}{(2\pi)^{3/2}}e^{-\frac{|v|^2}{2}},\\ 
	&\mu_w(v)=\mu(v_1-v_w,v_2,v_3)=\frac{1}{(2\pi)^{3/2}}e^{-\frac{|v_1-v_w|^2+|v_2|^2+|v_3|^2}{2}}.
\end{aligned}\right.
\end{equation*}
On the left side of the interface $X_1=x_w(t)- 0$, the distribution of rarefied gas is given by $F(t,x_w(t),v)$ taken as the boundary data of the unknown function $F(t,X_1,v)$ for the Boltzmann dynamics \eqref{BE}. This models a situation in which the right hand side acts as a thermal reservoir in equilibrium prior to collision, while the left-hand side evolves dynamically according to the Boltzmann equation. For the later use, we notice that under the assumption in \eqref{+0-0 data->0}, we can rewrite $P_r[F]-P_l[F]$ as
 \begin{align}\label{P_rF-P_lF}
 	P_r[F]-P_l[F]&=\int_{v_1-v_w(t)<0}\mu(v)|v_1-v_w(t)|^2{\rm d}v+\frac{\sqrt{2\pi}}{2}\int_{v_1-v_w(t)<0}\mu(v)|v_1-v_w(t)|{\rm d}v\notag\\
 	&\quad-\int_{\mathbb R^3}|v_1-v_w(t)|^2F(t,x_w(t),v){\rm d}v,
 \end{align}
see \eqref{der.ppm} for the brief derivation in the Appendix \ref{Appendix} later.

We further supplement the Boltzmann-Newton system \eqref{BE} and \eqref{ODE} with the following initial conditions:
	\begin{align}\label{initial F}
		F(0,X_1,v)=F_0(X_1,v),\quad \left(x_w(0),v_w(0)\right)=\left(x_{w0},v_{w0}\right).
	\end{align}
The boundary condition at the fixed wall located at $X_1=-1$ is given by the diffuse  reflection boundary condition
\begin{equation}\label{drbcl}
F(t,-1,v)|_{v_1>0}=\sqrt{2\pi}\mu(v)\int_{u_1<0}F(t,-1,u)|u_1|{\rm d}u.	
	\end{equation}
	At the moving boundary, we impose a diffuse reflection condition in the moving frame
	\begin{align}\label{drbcr}
		F(t,x_w(t),v)|_{v_1-v_w(t)<0}=\sqrt{2\pi}\mu_{w}(v)\int_{u_1-v_w(t)>0}F(t,x_w(t),u)|u_1-v_w(t)|{\rm d}u.
	\end{align}
In terms of \eqref{BE} under the boundary conditions \eqref{drbcl} and \eqref{drbcr}, the total mass of rarefied gas in $(-1,x_w(t))$ is conservative, namely 
\begin{equation*}
\int_{-1}^{x_w(t)}{\rm d}X_1 \int_{\R^3}{\rm d}v\,F(t,X_1,v)={\rm constant},\quad \forall\,t\geq 0,
\end{equation*} 
see Lemma \ref{mass-lem} for the proof later. Therefore, we impose the following compatibility condition
\begin{equation}
\label{cond.comid}
\int_{-1}^{x_{w0}}{\rm d}X_1 \int_{\R^3}{\rm d}v\,F_0(X_1,v)=1=\int_{-1}^{0}{\rm d}X_1 \int_{\R^3}{\rm d}v\,\mu(v),
\end{equation} 
so that we may expect $F(t,X_1,v)\to \mu(v)$ and $[x_w(t),v_w(t)]\to [0,0]$ in large time, provided that their initial data satisfying \eqref{cond.comid} are a small perturbation of $\mu(v)$ and $[0,0]$. The main goal of this paper is to provide a rigorous mathematical proof for the free boundary problem on the Boltzmann equation under the above formulation.

In what follows we review the mathematical and numerical literature on the motion of a rigid body (piston, plate, or sphere) in a rarefied gas, where the gas is described by the Boltzmann equation or related kinetic models. Indeed, the interaction between a moving rigid body and a surrounding rarefied gas is a fundamental problem in kinetic theory with applications ranging from aerospace engineering to nanoscale devices. When the gas is so dilute that intermolecular collisions are negligible, the problem exhibits a long-time memory effect: molecules reflected from the body at early times can return to collide with it later, leading to algebraically slow decay of the body's motion. This behavior contrasts sharply with the exponential decay predicted by simple Stokes drag models. The introduction of intermolecular collisions destroys this memory effect, but the resulting decay rate is not simply exponential; instead, it can be algebraic with a different exponent, as shown by numerical studies. 

\begin{itemize}
  \item  {\it Collisionless gas: long-time memory and algebraic decay}. 

The study of a body moving in a collisionless gas began with the work of Caprino, Marchioro, and Pulvirenti \cite{CMP-06CMP} as well as \cite{CCM-07M3AS}. They considered a circular disk immersed in an infinite rarefied gas subject to a Hookean restoring force and showed that, under specular reflection, the displacement decays algebraically as 
$$
|x_w(t)| \sim \frac{C}{t^{d+2}}
$$ 
in dimension $d$. This result was extended by Aoki, Cavallaro, Marchioro, and Pulvirenti \cite{ACMP-08M2AN}, who analyzed the approach to steady motion under a constant external force. The algebraic decay is a consequence of the long-time memory effect: molecules carry information about the body's past motion and return to influence its future dynamics. We also refer to Koike \cite{Ko18krm} for an extension to the motion of a rigid body immersed in a semi-infinite expanse of gas.

Tsuji and Aoki \cite{TT-AK-12} performed numerical simulations for a linear pendulum in a free-molecular gas and confirmed the algebraic decay even in cases with many oscillations. They also studied a special Lorentz gas, where a background interaction destroys the memory effect, leading to faster, nearly exponential decay. This highlighted the crucial role of collisions in altering the decay rate. See also Koike \cite{Ko18jsp}.

Aoki and Golse \cite{AoGo-11KRM} investigated the speed of approach to equilibrium for a collisionless gas, providing a rigorous framework for understanding the slow decay. Their work complements the numerical studies and offers insights into the underlying functional-analytic mechanisms.

Chen and Strauss \cite{CS-14ARMA, CS-15CMP, CS-15SIMA} examined the motion of a body in a sea of free-transported particles with both specular and diffuse reflection. They derived criteria for velocity reversal and established convergence to equilibrium, showing that the diffuse reflection condition leads to a slower decay rate 
$
1/{t^{d+1}}
$ 
compared to the specular case. 

Cavallaro \cite{Ca-07RMA} studied the motion of a convex body in the mean-field approximation, while Sisti and Ricciuti \cite{SR-14SIMA} analyzed the effects of concavity on the body's motion in a Vlasov gas. These works extend the collisionless analysis to more general geometries and interaction potentials.
  
  \medskip
  
  \item {\it Collisional gas: the Boltzmann equation and BGK model.}

When intermolecular collisions are included, the gas is described by the Boltzmann equation, which is a nonlinear integro-differential equation. The moving-boundary problem becomes significantly more challenging due to the coupling between the body's motion and the gas's velocity distribution function.

Tsuji and Aoki \cite{TA-13JCP} developed a numerical method based on characteristics to capture the singularities (discontinuities and weaker singularities) produced by an oscillating plate in a rarefied gas. In their subsequent paper \cite{TA-14PRE}, they investigated the decay of a linear pendulum in a collisional gas using the BGK model. They found numerical evidence that the displacement decays as $|x_w(t)| \sim C t^{-3/2}$ for large time, which is slower than the collisionless case ($t^{-2}$) and much slower than exponential. This surprising result indicates that collisions do not simply restore exponential decay; instead, they produce a new algebraic exponent.

Dechristé and Mieussens \cite{dm-jcp} developed a Cartesian cut cell method for rarefied flow simulations around moving obstacles, providing a robust numerical framework for more complex geometries. Their work is complementary to the one-dimensional studies and paves the way for higher-dimensional simulations.

Khalil \cite{nk-jsp} studied the adiabatic piston problem using kinetic equations and established an H-theorem, demonstrating the thermodynamic consistency of the mesoscopic description. This connects the moving-boundary problem to statistical mechanics and the approach to equilibrium.

\end{itemize}	

Our goal of this paper is to analyze the free boundary problem on the Boltzmann-Newton system \eqref{BE} and  \eqref{ODE} with initial-boundary conditions \eqref{initial F}, \eqref{drbcl} and \eqref{drbcr} satisfying \eqref{cond.comid}, where the difference of pressure $P_r[F]-P_l[F] $ is given by \eqref{P_rF-P_lF} for the gas-piston interaction. A fundamental obstacle is that the classical Lagrangian transformation (generally used for fluid equations) is not available; the low regularity of solutions to the Boltzmann equation with boundaries prevents the straightforward use of a flow map to fix the domain. To overcome this difficulty, we introduce a  nonlinear (conformal) transformation which maps the time-dependent physical domain $-1<X_1<x_w(t)$ onto a fixed reference domain, which we may take to be $\Omega=(0,1)$. In the new variables, the gas distribution is re-expressed as
\[
\bar F(t,x_1,v)=F(t,X_1(t,x_1),v),
\]
and the Boltzmann equation is reformulated with coefficients that now depend on the free boundary variables.

A second key insight in our work is that once the problem is reformulated on a fixed spatial domain, the kinetic distribution and the free boundary variables can be viewed together as a coupled system that supports a natural quadratic energy structure. In other words, when one writes the perturbed solution in the form
\[
\bar F(t,x_1,v)=\mu(v)+\sqrt{\mu(v)}\,f(t,x_1,v), 
\]
one finds that the kinetic perturbation $f(t,x_1,v)$ and the free boundary variables $[x_w(t),v_w(t)]$ appear jointly in energy estimates. This structure reveals intrinsic damping both in the gas and in the piston dynamics. In particular, upon a careful analysis one can show that the momentum exchange term, which arises from the difference between the right- and left-hand side pressure, not only activates a damping mechanism in the piston velocity but also provides a cancellation mechanism in the energy estimates.

Our main result given by Theorem \ref{main result} in the next section asserts that, for initial data sufficiently close to a global Maxwellian and $[x_w=0,v_w=0]$, the full free boundary problem for the Boltzmann-Newton equations is globally well-posed. More precisely, we prove the following:
\begin{itemize}
\item Global-in-time existence and uniqueness of the solutions $f(t,x_1,v)$ and $[x_w(t),v_w(t)]$ in appropriate function spaces;
\item Exponential decay of the kinetic perturbation $f(t,x_1,v)$ and the free boundary deviations $|x_w(t)|$ and $|v_w(t)|$, as $t\to\infty$;
\item  Weighted $W^{1,p}$ estimates (for some $p>2$) that imply the stability in the $L^{1+\delta}$ norm and ultimately yield the local-in-time existence and uniqueness.
\end{itemize}
The analysis relies on a delicate combination of energy estimates and trace theory on the boundary. In particular, we are able to derive the quadratic energy structure that couples the kinetic perturbation with the free boundary variables and highlight the cancellation mechanisms that yield intrinsic damping. We further work within a framework that couples an $L^\infty-L^2$ approach (cf.~\cite{Guo-2010} and \cite{EGKM-13}) to control the collision operator and its nonlinear terms, together with localized weighted $W^{1,p}$ estimates for spatial derivatives (cf.~\cite{cao kim,Ck-22ARMA}). In addition, an intricate study of the boundary integrals, via a careful reformulation of the diffuse reflection conditions, is required to capture the dissipation of the free boundary motion.

Our results represent, to the best of our knowledge, the first global well-posedness theory for a fully coupled Boltzmann free boundary problem of piston type. The additional insight into the quadratic energy structure and the cancellation mechanism at the free boundary may also provide a new pathway for tackling related kinetic models with moving interfaces. In particular, extensions of the current results could be made along two main directions, such as the motion of the piston in infinite rarefied gas and the gas-structure interaction in the multi-dimensions. For the former, we refer to \cite{VZ} and \cite{Ko21jde, Ko23ARMA} for studies of some fluid equations, and for the latter, we refer to \cite{BGN,Ca21BAMS,KKLTT} for the fluid-structure interaction as mentioned before.     

The rest of this paper is organized as follows. In Section~\ref{sec2}, we introduce the conformal transformation that fixes the spatial domain and describe in detail how the free boundary conditions transform under this mapping. With an introduction of notations to be used throughout the paper, we state the main result in Theorem \ref{main result} followed by a summary of key points and strategies for the proof. In Section~\ref{inter-bd-sec}, we establish a dissipation mechanism for the piston motion $[x_w(t),v_w(t)]$ through a newly constructed quadratic form. In Section~\ref{mac-l2-sec}, we establish the global $L^2$ framework. In particular, we derive energy estimates for the linearized problem, prove the macroscopic dissipation, and deduce the exponential time-decay structure. Section~\ref{lif-sec} focuses on the $L^\infty$ estimates with appropriate weight functions. Using the method of characteristics and tracking the solution along the backward and forward trajectories, we obtain uniform weighted estimates that are crucial to close the nonlinear iteration. In Section~\ref{sec-loc-reg}, we prove any-finite-time weighted $W^{1,p}$ estimates, which serve as the key step to establish the $L^{1+\delta}$ stability result as shown in Section~\ref{sec-sta}. Finally, Section~\ref{glex-sec} combines the various a priori estimates together with a continuation argument to prove global existence and uniqueness of solutions near equilibrium. The Appendix \ref{Appendix} collects some auxiliary estimates and technical lemmas concerning properties of the collision operator and the change-of-variable maps associated with the kinetic transport.


\section{Main results}\label{sec2}	
	In order to reformulate the problem on a fixed spatial domain, we introduce the following nonlinear transformation from the original variable $X_1$ to a rescaled spatial variable $x_1$
	\begin{align}
		x_1=\frac{1+X_1}{1+x_w(t)}.\notag
	\end{align}
This transformation is clearly one-to-one, provided that
$|x_w(t)|\ll1,$  with the derivatives satisfying
\begin{align}\label{Newderi}
		\frac{\pa X_1}{\pa x_1}=1+x_w(t),\quad \frac{\pa X_1}{\pa t}=v_w(t)x_1.
	\end{align}
Moreover, one sees that, under the assumption $|x_w(t)|\ll1$, this transformation maps the time-dependent physical domain $-1<X_1<x_w(t)$ onto the fixed interval
$0<x_1<1$.

	In terms of \eqref{Newderi}, denoting $\bar{F}(t,x_1,v)=F(t,X_1(t,x_1),v)$, we can rewrite the free boundary problem as
	\begin{eqnarray}\label{RBE}
		&\dis \pa_t\bar{F}+\frac{v_1-v_w(t)x_1}{1+x_w(t)}\pa_{x_1} \bar{F}=Q(\bar{F},\bar{F}),
	\end{eqnarray}
subject to the initial and boundary condition
\begin{align}\label{tr-ibvp}
\left\{\begin{array}{rll}
		&\bar{F}(0,x_1,v)=\bar{F}_0(x_1,v)=F(0,X_1(0,x_1),v),\\[2mm]
		&\bar{F}(t,0,v)|_{v_1>0}=\sqrt{2\pi}\mu(v)\int_{u_1<0}\bar F(t,0,u)|u_1|{\rm d}u,\\
		&\bar{F}(t,1,v)|_{v_1-v_w(t)<0}=\sqrt{2\pi}\mu_{w}(v)\int_{u_1-v_w(t)>0}\bar{F}(t,1,u)|u_1-v_w(t)|{\rm d}u,
        \end{array}\right.
	\end{align}
where the evolution of the moving boundary $x_w(t)$ and its velocity  $v_w(t)$ is governed by the coupled system
	\begin{align*}
    \left\{\begin{array}{rll}
        &\frac{{\rm d}}{{\rm d}t}x_w(t)=v_w(t),\\[2mm]
		&\frac{{\rm d}}{{\rm d}t}v_w(t)=-\kappa x_w(t)-\mathcal M\left(P_r[\bar F]-P_l[\bar F] \right),
    \end{array}\right.
		\end{align*}
with initial data
\begin{align}
\left(x_w(0),v_w(0)\right)=\left(x_{w0},v_{w0}\right).\notag
\end{align}
Here,
\begin{align}
	P_r[\bar F]-P_l[\bar F]
	&=\int_{v_1-v_w(t)<0}\mu(v)|v_1-v_w(t)|^2{\rm d}v+\frac{\sqrt{2\pi}}{2}\int_{v_1-v_w(t)<0}\mu(v)|v_1-v_w(t)|{\rm d}v\notag\\
	&\quad-\int_{v_1-v_w(t)>0}\bar F(t,1,v)|v_1-v_w(t)|^2{\rm d}v-\frac{\sqrt{2\pi}}{2}\int_{v_1-v_w(t)>0}\bar F(t,1,v)|v_1-v_w(t)|{\rm d}v.\notag
\end{align}
Next, we define the phase boundary of the phase space $\Omega\times\mathbb R^3:=(0,1)\times\mathbb R^3$ as
\begin{equation*}
\gamma=\{x_1=0, v\in\mathbb R^3\}\cup\{x_1=1, v\in\mathbb R^3\}.
\end{equation*}
This boundary can be decomposed into the outgoing, incoming, and grazing sets:
\begin{align*}
	\gamma^i_{+}&=\{x_1=i,  v\in\mathbb R^3\big| [v_1-v_w(t)x_1]n(x_1)>0\},\\
	\gamma_-^{i}&=\{x_1=i,  v\in\mathbb R^3\big|  [v_1-v_w(t)x_1]n(x_1)<0\},\\
	\gamma^i_{0}&=\{x_1=i,  v\in\mathbb R^3\big| [v_1-v_w(t)x_1]n(x_1)=0\},
\end{align*}
where $i=0,1$, and $ n(0)=-1, n(1)=1.$
We also denote $\gamma_{0}=\gamma^0_{0}\cup\gamma^1_{0}$.

We now consider the following ODE describing the evolution of the spatial trajectory in the $x_1$ direction
\begin{equation}\label{ODE X}
	\frac{{\rm d}}{{\rm d}s}[X_1^f(s;t,x_1,v_1)]=\frac{v_1-v_w(s)X_1^f(s;t,x_1,v_1)}{1+x_w(s)},
\end{equation}
with initial condition $X_1^f(t;t,x_1,v_1)=x_1$. Direct computation gives
\begin{equation}\label{X}
	X_1^f(s)=X_1^f(s;t,x_1,v_1)=\frac{1+x_w(t)}{1+x_w(s)}x_1+\frac{v_1(s-t)}{1+x_w(s)}.
\end{equation}
We then define the backward exit time $t_b(t,x_1,v_1)$ as
\begin{equation*}
	t_{\textbf b}(t,x_1,v_1)=\sup\{s\geq0 : X_1^f(\tau;t,x_1,v_1)\in(0,1)\ for\ all\ \tau\in(t-s,t)\}.
\end{equation*}
Accordingly, the backward exit position is denoted by $x_{1\textbf b}=X_1^f(t-t_{\textbf b};t,x_1,v_1)$.
Clearly, $x_{1\textbf b}=0$ or $x_{1\textbf b}=1$. Since $x_w(t)$  varies near $X_1=0$, we observe that $X_1^f(s)\sim x_1+v_1(s-t)$,  which guarantees that the backward exit time $t_{\textbf b}$ is well-defined.

To construct solutions near the global Maxwellian, we introduce the linearized collision operator $L$ and nonlinear operator $\Ga$, cf.~\cite{Glassey}. Here, $L$ is defined by
	$$L=\nu-K,$$
	where the collision frequency $\nu(v)$ and the operator $K$ are given respectively by
	$$
	\nu(v)=\int_{\R^3}\int_{\S^2}B(v-u,\omega)\mu(u){\rm d}\omega {\rm d}u\sim (1+|v|),
	$$
    and
\begin{align}\label{k-def}
	(Kf)(v)=\int_{\R^3}\int_{\S^2}B(v-u,\omega)\sqrt{\mu(u)}\left( \sqrt{\mu(u')}f(v')+\sqrt{\mu(v')}f(u')-\sqrt{\mu(v)}f(u) \right){\rm d}\omega {\rm d}u.
\end{align}
Moreover, it is well-known that there exists $\nu_0>0$ such that
$
\nu\geq\nu_0,$
and there exist constants $C_{\textbf k_1}>0$ and $C_{\textbf k_2}$ such that
\begin{align*}
Kf:=K_2f-K_1f=\int_{\mathbb R^3}\textbf k_2(v,u)f(u){\rm d}u-\int_{\mathbb R^3}\textbf k_1(v,u)f(u){\rm d}u,
\end{align*}
with
\begin{align*}
	\textbf k_1(v,u)=C_{\textbf k_1}|v-u| e^{-\frac{|v|^2+|u|^2}{4}},\
	\textbf k_2(v,u)=C_{\textbf k_2}|v-u|^{-1}e^{-\frac{|v-u|^2}{8}}e^{-\frac{||v|^2-|u|^2|^2}{8|v-u|^2}}.
\end{align*}
The kernel of $L$, denoted as $\ker L$, is a five-dimensional space spanned by
	$$
	\{1,v,|v|^2-3\}\sqrt{\mu}:= \{\phi_i\}_{i=1}^5.
	$$
	we define the orthogonal projection from $L^2$ onto $\ker L$ by
	\begin{align}\label{mac-p}
		\FP g=\left\{a+\Fb\cdot v+\frac{1}{2}(|v|^2-3)c\right\}\sqrt{\mu},\ g\in L^2,
	\end{align}
where $a={ a(t,x_1)}\in \R$, $\Fb=\Fb(t,x_1)\in\R^3$ and $c=c(t,x_1)\in\R$. Correspondingly, we denote by
${\FI-\FP}$ the orthogonal complement of
$\FP$ in $L^2_v$, with
$\FI$ being the identity operator.

	The nonlinear collision operator $\Ga$ is defined as
	$$
	\Ga(f,f)=\Ga_+(f,f)-\Ga_-(f,f), \quad \Ga_\pm(f,f)=\frac{1}{\sqrt{\mu}}Q_\pm({\sqrt{\mu}f},{\sqrt{\mu}f}),
	$$
	where the gain and loss parts of the Boltzmann bilinear operator are
	$$
	Q_+(f,g)=\int_{\R^3}\int_{\S^2}B(v-u,\omega)f(v')g(u') {\rm d}\omega {\rm d}u,\quad
	Q_-(f,g)=\int_{\R^3}\int_{\S^2}B(v-u,\omega)f(v)g(u){\rm d}\omega {\rm d}u.
	$$
	Substituting the ansatz
\[
\bar{F}(t,x_1,v)=\mu(v)+\sqrt{\mu(v)}\,f(t,x_1,v)
\]
into the initial boundary value problem \eqref{RBE} and \eqref{tr-ibvp},
we obtain the reformulated equation for the fluctuation \(f\):
	\begin{eqnarray}\label{PBE}
		&\dis \pa_t f+\frac{v_1-v_w(t)x_1}{1+x_w(t)}\pa_{x_1} f+Lf=\Ga(f,f),
	\end{eqnarray}
	posed on the fixed domain $\Om=(0,1)$, which arises from rescaling the original time-dependent physical domain $(-1,x_w(t))$ to a time-independent interval. In this formulation, the effect of the moving boundary is encoded in the coefficients of the equation and in the time evolution of the wall position $x_w(t)$ and velocity $v_w(t)$. The dynamics of the moving boundary point are governed by the following system
\begin{align}\label{trb-ode}
\left\{\begin{array}{rll}
		\frac{{\rm d}}{{\rm d}t}x_w(t)&=v_w(t),\\[2mm]
		\frac{{\rm d}}{{\rm d}t}v_w(t)&=-\kappa x_w(t)-\mathcal M\left(P_r[\mu+\sqrt{\mu}f]-P_l[\mu+\sqrt{\mu}f] \right),
\end{array}\right.
	\end{align}
with
\begin{align}
	P_r[\mu&+\sqrt{\mu}f]-P_l[\mu+\sqrt{\mu}f]\notag\\
	&=\int_{v_1-v_w(t)<0}\mu(v)|v_1-v_w(t)|^2{\rm d}v+\frac{\sqrt{2\pi}}{2}\int_{v_1-v_w(t)<0}\mu(v)|v_1-v_w(t)|{\rm d}v\notag\\
	&\quad-\int_{v_1-v_w(t)>0} \mu(v)|v_1-v_w(t)|^2{\rm d}v-\frac{\sqrt{2\pi}}{2}\int_{v_1-v_w(t)>0} \mu(v)|v_1-v_w(t)|{\rm d}v\notag\\
	&\quad-\int_{v_1-v_w(t)>0} \sqrt{\mu(v)}f(t,1,v)|v_1-v_w(t)|^2{\rm d}v\notag\\
	&\quad-\frac{\sqrt{2\pi}}{2}\int_{v_1-v_w(t)>0}\sqrt{\mu(v)}f(t,1,v)|v_1-v_w(t)|{\rm d}v.\notag
\end{align}
This system is subject to the initial condition
	\begin{align}\label{ic f}
		f(0,x_1,v)=f_0(x_1,v),\ (x_w,v_w)(0)=(x_{w0},v_{w0}),
	\end{align}
and the boundary condition
	\begin{align}
		&f(t,0,v)|_{v_1>0}=\sqrt{2\pi\mu(v)}\int_{u_1<0}f(t,0,u)\sqrt{\mu(u)}|u_1|{\rm d}u=P_\gamma f(t,0,v),\label{drbl-f}\\
		&f(t,1,v)|_{v_1-v_w(t)<0}=\sqrt{2\pi}\frac{\mu_{w}(v)}{\sqrt{\mu (v)}}\int_{u_1-v_w(t)>0}f(t,1,u)\sqrt{\mu(u)}|u_1-v_w(t)|{\rm d}u\notag\\
		&\qquad\qquad\qquad\qquad\qquad+\sqrt{2\pi}\frac{\mu_{w}(v)}{\sqrt{\mu (v)}}\int_{u_1-v_w(t)>0}\mu(u)|u_1-v_w(t)|{\rm d}u-\sqrt{\mu(v)}\notag\\
		&\qquad\qquad\qquad\qquad=P_\gamma f(t,1,v)+r(t,v),\label{drb-f}
	\end{align}
where we define the projection
\begin{equation*}
	P_\gamma f(t,x_1,v):=\sqrt{2\pi\mu(v)}\int_{(u_1-v_w(t)x_1)n(x_1)>0}f(t,x_1,u)\sqrt{\mu(u)}|u_1-v_w(t)x_1|{\rm d}u,\ x_1=0,1,
\end{equation*}
and the error term
\begin{align}
	r(t,v):=&\sqrt{2\pi}\frac{\mu_{w}(v)-\mu(v)}{\sqrt{\mu (v)}}\int_{u_1-v_w(t)>0}f(t,1,u)\sqrt{\mu(u)}|u_1-v_w(t)|{\rm d}u\notag\\
	&+\sqrt{2\pi}\frac{\mu_{w}(v)-\mu(v)}{\sqrt{\mu (v)}}\int_{u_1-v_w(t)>0}\mu(u)|u_1-v_w(t)|{\rm d}u\notag\\
	&+\sqrt{2\pi\mu(v)}\int_{u_1-v_w(t)>0}[\mu(u)-\mu_{w}(u)]|u_1-v_w(t)|{\rm d}u.\notag
\end{align}

We introduce notations before stating the main result. We denote by $\left\|\ \cdot\ \right\|_{p}$  the $L^p(\Omega\times\mathbb R^3)$ norm or $L^p(\Omega)$ norm for $1\leq p\leq+\infty$. The $L^p$ norm on the boundary is denoted by $L^p(\partial\Omega\times\mathbb R^3;{\rm d}\tilde{\gamma})$, and can be decomposed as follows:
\begin{align*}
	|f|_{p}^p=|f|_{p,+}^p-|f|_{p,-}^p,
\end{align*}
where
\begin{align*}
	|f|_{p,+}^p&=\int_{\ga_+}|f(x_1,v)|^p{\rm d}\tilde{\ga}=\int_{\ga_+}|f(x_1,v)|^p|v_1-v_w(t)x_1|{\rm d}v\\[2mm]
&=\int_{v_1-v_w(t)>0}[v_1-v_w(t)]	|f(t,1,v)|^p{\rm d}v-\int_{v_1<0}v_1	|f(t,0,v)|^p{\rm d}v\\[2mm]
&=:|f|_{p,+,1}^p+|f|_{p,+,0}^p,
\end{align*}
and similarly,
\begin{align*}
	|f|_{p,-}^p&=\int_{\ga_-}|f(x_1,v)|^p{\rm d}\tilde{\ga}=\int_{\ga_-}|f(x_1,v)|^p|v_1-v_w(t)x_1|{\rm d}v\\[2mm]
&=-\int_{v_1-v_w(t)<0}[v_1-v_w(t)]	|f(t,1,v)|^p{\rm d}v+\int_{v_1>0}v_1	|f(t,0,v)|^p{\rm d}v\\[2mm]
&=:|f|_{p,-,1}^p+|f|_{p,-,0}^p,
\end{align*}
with $\ga_{\pm}=\ga_{\pm}^0\cup\ga_{\pm}^1.$ In particular, $|\cdot|_{\infty}$ denotes the $L^\infty$ norm on the boundary $\{0\}\times\R^3\cup\{1\}\times\R^3$. Moreover, we define the (signed) boundary flux by ${\rm d}\tilde\ga$ as $(v_1-v_w(t)x_1)n(x_1){\rm d}v$. Denote $\langle v\rangle=\sqrt{1+|v|^2}$. For $\zeta\geq0$,
we define a weight function
\begin{equation*}
w(v):=e^{\zeta |v|^2}.
\end{equation*}
As introduced in \cite{cao kim}, we define a kinetic distance weight by
\begin{align}\label{def alpha}
	\alpha_\vps(t,x_1,v_1):&=\chi\left(\frac{t-t_{\textbf b}(t,x_1,v_1)+\vps}{\vps}\right)|v_1-v_w(t-t_{\textbf b}(t,x_1,v_1))x_{1\textbf b}|\notag\\
    &\quad+\left[1-\chi\left(\frac{t-t_{\textbf b}(t,x_1,v_1)+\vps}{\vps}\right)\right],
\end{align}
where the smooth cutoff function $\chi$ satisfies
\begin{align}\label{def chi}
0\leq\chi\leq 1;\	\chi(\tau)=0,\ \tau\leq0;\  \chi(\tau)=1,\ \tau\geq1;\ \textrm{and}\
	\frac{{\rm d}}{{\rm d}\tau}\chi(\tau)\in[0,4]\ \textrm{for all}\ \tau\in\mathbb R.
\end{align}
Note that on the incoming boundary $\gamma_-$, we have
\begin{equation}\label{al on gamma}
	\alpha_\vps(t,x_1,v_1)=|v_1-v_w(t)x_1|.
\end{equation}
Define the weighted function
\begin{equation}\label{wc}
{W_c}(t,x_1,v)=e^{\zeta'|v|^2}\alpha^\vho_\vps(t,x_1,v_1),\ 0<\zeta'<1,\ \vho>0.
\end{equation}
Following the argument in the appendix of \cite{cao kim}, we obtain
\begin{equation}\label{par alpha}
	\left(\partial_t+\frac{v_1-v_w(t)x_1}{1+x_w(t)}\partial_{x_1}\right){W_c}(t,x_1,v)=0.
\end{equation}
Finally, we use $O(\cdot)$ and $o(\cdot)$ to denote terms of the same order and of higher order, respectively.

The main result of this paper is now stated as follows.

\begin{theorem}\label{main result}
	Let $0<\zeta'<\zeta\ll1$, $2<p<+\infty$, $1-\frac{2}{p}<\vho<1-\frac{1}{p}$, and $\mathcal M>0$. For any $\kappa\geq0$,
    there exists a constant $\vps_0>0$ such that if the initial data satisfies
	\begin{equation}\label{to-id-drbcl}
		\left\|wf_0\right\|_\infty+|x_{w0}|+|v_{w0}|\leq \vps_0,\ \left\|{W_c}\partial_{x_1}f_0\right\|_{p}<\infty,
\ \int_{\Om\times\R^3}\sqrt{\mu}f_0{\rm d}x_1{\rm d}v=-\frac{x_{w0}}{1+x_{w0}},
	\end{equation}
then there exists a unique non-negative solution $F(t,x_1,v)=\mu+\sqrt{\mu}f(t,x_1,v)\geq0$ and $[x_w(t),v_w(t)]$ to the free boundary problem on the coupled Boltzmann-Newton system \eqref{BE}, \eqref{ODE}, $\eqref{initial F}$, \eqref{drbcl} and \eqref{drbcr}. Moreover, there exist constants $\la>0$ and $C>0$ such that the solution $[f,x_w,v_w]$ enjoys the exponential decay estimate
	\begin{align}\label{es decay-drbcl}
		\left\|	w f(t)\right\|_\infty+|x_w(t)|+|v_{w}(t)|
		\leq Ce^{-\lambda t}\Big\{\left\|w f_0\right\|_\infty+|x_{w0}|+|v_{w0}|\Big\},
	\end{align}
	and the spatial derivative satisfies
	\begin{equation}\label{es par f-drbcl}
		\left\|{W_c}\partial_{x_1}f(t)\right\|_{p}\lesssim e^{Ct^{p}},
	\end{equation}
	for any $t\geq0.$ Moreover, if $f_0$ is continuous except on $\gamma_0$ with
	\begin{equation*}
		f_0(0,v)|_{v_1>0}=\sqrt{2\pi\mu(v)}\int_{u_1<0}f_0(0,u)\sqrt{\mu(u)}|u_1|{\rm d}u,
	\end{equation*} and
	\begin{align*}
		f_0(1,v)|_{v_1-v_{w0}<0}&=\sqrt{2\pi}\frac{\mu_{w0}(v)}{\sqrt{\mu (v)}}\int_{u_1-v_{w0}>0}f_0(1,u)\sqrt{\mu(u)}|u_1-v_{w0}|{\rm d}u\\
		&\quad	+\sqrt{2\pi}\frac{\mu_{w0}(v)}{\sqrt{\mu (v)}}\int_{u_1-v_{w0}>0}\mu(u)|u_1-v_{w0}|{\rm d}u-\sqrt{\mu(v)},
	\end{align*}
	then $f(t,x_1,v)$ is continuous in $[0,\infty)\times\{\bar\Omega\times\mathbb R^3\setminus\gamma_0\}$.
\end{theorem}

\begin{remark}
Theorem \ref{main result} provides the first global well-posedness theory for a fully coupled Boltzmann free boundary problem of piston type. Its proof is based on three main ingredients: a conformal transformation reducing the moving domain to a fixed one, a coupled energy structure revealing hidden dissipation and cancellation mechanisms at the interface, and a combined $L^\infty-L^2$ framework yielding global nonlinear control. In particular, the damping of the boundary dynamics arises from the interaction between the kinetic distribution, the mass flux, and the conservation laws of the system. These structural ideas are not specific to the present piston model. Indeed, they provide a framework for studying nonlinear kinetic free boundary problems and are expected to be applicable to a broader class of related problems.
\end{remark}


\subsection{Key points and strategies}

Over the past several decades, substantial progress has been made in the local and global well-posedness theory of free boundary value problems (FBVPs) arising in fluid dynamics. In contrast, for FBVPs associated with kinetic equations, much less is known. This paper is devoted to the one-dimensional classical piston-type problem described by the Boltzmann equation, which constitutes a genuine free boundary value problem for a kinetic model.

Several fundamental difficulties arise in the analysis of this problem. We outline the main ones, together with the strategies employed to overcome them, as follows.

\medskip
\noindent$\bullet$ \textit{Boundary linear coupling and cancellation mechanisms at the free boundary.}

A distinctive feature of the present free boundary kinetic problem is the linear coupling between the particle distribution and the boundary variables $[x_w(t),v_w(t)]$, which encodes the microscopic momentum exchange between particles and the moving boundary. This coupling induces a damping mechanism for the boundary velocity, reflecting the fact that particles impinging on the boundary exert a restoring force that tends to relax the boundary motion.
More precisely, from \eqref{trb-ode}, one derives the following energy inequality:
\begin{align}
\frac{{\rm d}}{{\rm d}t}&\left\{\kappa|x_w(t)|^2+|v_w(t)|^2\right\}+\left(\frac{8}{\sqrt{2\pi}}+\sqrt{2\pi}\right)\mathcal M|v_w(t)|^2\notag\\
&-2\mathcal Mv_w\int_{v_1-v_w(t)>0}[v_1-v_w(t)]^2\sqrt{\mu(v)}\left(\FI-\frac{\sqrt{\mu_{w}}}{\sqrt{\mu}} P_\gamma \right)f{\rm d}v\notag\\
	&-2\sqrt{2\pi}\mathcal Mv_w\int_{u_1-v_w(t)>0}f(1)\sqrt{\mu(u)}|u_1-v_w(t)|{\rm d}u\leq  h.o.t.
    \notag
\end{align}
However, the interaction term $2\sqrt{2\pi}v_w\int_{u_1-v_w(t)>0}f(1)\sqrt{\mu(u)}|u_1-v_w(t)|{\rm d}u$ cannot be controlled directly at this level. Fortunately, this term can be canceled by the later $L^2$ estimate on the kinetic density function $f$, which gives as
\begin{align}
\frac{{\rm d}}{{\rm d}t}&\left\| (1+x_w)^{\frac{1}{2}}f(t)\right\|_{2}^2+\delta\left\|(\nu(1+x_w))^{\frac{1}{2}}(\textbf I-\textbf P)f\right\|_{2}^2+\left|\left(\FI-\frac{\sqrt{\mu_w}}{\sqrt{\mu}} P_\gamma \right)f\right|_{2,+,1}^2+|(\mathbf I-P_\gamma)f|_{2,+,0}^2\notag\\
&-(2\pi)^{-1/2}(2+\frac{3\pi}{2})|v_w(t)|^2+2\sqrt{2\pi}v_w\int_{v_1-v_w(t)>0}f(1)\sqrt{\mu}|v_1-v_w(t)|{\rm d}v\leq h.o.t. \notag
\end{align}
The resulting energy inequality retains a strictly positive dissipation in $v_w$, since it holds
$$
\left(\frac{8}{\sqrt{2\pi}}+\sqrt{2\pi}\right)-\frac{1}{\sqrt{2\pi}}(2+\frac{3\pi}{2})=\frac{6}{\sqrt{2\pi}}+\frac{\sqrt{2\pi}}{4}>0.
$$

\medskip
\noindent$\bullet$ \textit{Mass flux and mass conservation stabilization.}

In addition to the dissipation of the free boundary velocity $v_w(t)$, one also expects to recover the dissipation of the free boundary position $x_w(t)$ from a basic $L^2$-type estimate. More precisely, one can derive
\begin{align}\label{damx}
\frac{{\rm d}}{{\rm d}t}(x_wv_w)&+\kappa|x_w(t)|^2-\mathcal Mx_w\int_{v_1-v_w(t)>0}[v_1-v_w(t)]^2\sqrt{\mu(v)}\left(\FI-\frac{\sqrt{\mu_{w}}}{\sqrt{\mu}} P_\gamma \right)f{\rm d}v\notag\\
	&-\sqrt{2\pi}\mathcal Mx_w\int_{u_1-v_w(t)>0}f(t,1,u)\sqrt{\mu(u)}|u_1-v_w(t)|{\rm d}u\notag\\
= &|v_w(t)|^2-\left(\frac{4}{\sqrt{2\pi}}+\frac{\sqrt{2\pi}}{2}\right)\mathcal Mx_wv_w+h.o.t.
\end{align}
However, on the one hand, if $\kappa=0$, \eqref{damx} does not seem to yield any dissipation for $x_w(t)$. On the other hand, as in the velocity estimate, the boundary interaction term
$$-\sqrt{2\pi}\CM x_w\int_{u_1-v_w(t)>0}f(t,1,u)\sqrt{\mu(u)}|u_1-v_w(t)|{\rm d}u$$
cannot be controlled directly, and remains problematic even when $\kappa>0$ at this level.

To overcome this difficulty, we exploit the kinetic mass flux by choosing a suitable test function, which leads to the identity
\begin{align}\label{mf-id}
\frac{{\rm d}}{{\rm d}t}&\int_{\Omega\times\mathbb R^3}x_w(t)(1+x_w(t))x_1v_1\sqrt{\mu(v)}f(x_1){\rm d}x_1{\rm d}v\notag\\	
&\qquad+x_w(t)\int_{\mathbb R^3}(v_1-v_w(t))^2\sqrt{\mu(v)}f(1){\rm d}v\notag\\
&\qquad-x_w(t)\int_{\Omega\times\mathbb R^3} v_1^2\sqrt{\mu(v)}f(x_1){\rm d}x_1{\rm d}v\notag\\
&=-x_w(t)v_w(t)\int_{\mathbb R^3}(v_1-v_w(t))\sqrt{\mu(v)}f(1){\rm d}v-x_w(t)v_w(t)\int_{\Omega\times\mathbb R^3} v_1\sqrt{\mu(v)}f(x_1){\rm d}x_1{\rm d}v\notag\\
&\qquad-x_w(t)v_w(t)\int_{\Omega\times\mathbb R^3} v_1x_1\sqrt{\mu(v)}f(x_1){\rm d}x_1{\rm d}v\notag\\
&\qquad+\int_{\Omega\times\mathbb R^3}\pa_t[x_w(t)(1+x_w(t))v_1x_1]\sqrt{\mu(v)}\textbf Pf{\rm d}x_1{\rm d}v.\notag
\end{align}
Here, the leading-order contribution of the boundary term $\int_{\mathbb R^3}x_w(v_1-v_w)^2\sqrt{\mu}f{\rm d}v$ at $x_1=1$ provides the crucial cancellation of the problematic interaction term in \eqref{damx}, see \eqref{x_wf x_1=1}. In contrast, the boundary term at $x_1=0$ vanishes due to the specific choice of the test function $x_w(t)(1+x_w(t))v_1x_1\sqrt{\mu(v)}$. Moreover, the third term on the left-hand side of \eqref{mf-id} yields strong damping for the free boundary position $x_w(t)$:
\begin{align}
	-x_w(t)&\int_{\Omega\times\mathbb R^3} v_1^2\sqrt{\mu(v)}f{\rm d}x_1{\rm d}v\notag\\
	&=-x_w(t)\int_0^1 a{\rm d}x_1-x_w\int_{\Omega\times\mathbb R^3} v_1^2\sqrt{\mu(v)}\left(\sqrt{\mu}\frac{|v|^2-3}{2}c+(\mathbf I-\mathbf P)f\right){\rm d}x_1{\rm d}v\notag\\
	&=\frac{|x_w(t)|^2}{1+x_w}-x_w\int_{\Omega\times\mathbb R^3} v_1^2\sqrt{\mu(v)}\left(\sqrt{\mu}\frac{|v|^2-3}{2}c+(\mathbf I-\mathbf P)f\right){\rm d}x_1{\rm d}v\notag\\
	&\geq\frac{1}{2}|x_w(t)|^2-C|x_w(t)|(\|c\|_2+\|(\mathbf I-\mathbf P)f\|_2),
\end{align}
where we have used the mass conservation relation 
$$
\int_0^1 a{\rm d}x_1=-\frac{x_w}{1+x_w}.
$$

This analysis shows that the damping of the free boundary position is not merely a direct consequence of Hooke's law, but rather emerges from a subtle interplay between the kinetic mass flux and the conservation of mass. This combined mechanism is crucial for closing the energy estimates in our proof.

\medskip
\noindent$\bullet$ \textit{Macroscopic dissipation coupled to the free boundary position}

In case of fixed bounded domains, one may make use of dual argument on the thirteen-moments equations developed in \cite{EGKM-13} to capture the macroscopic dissipation of $\FP f$ or equivalently $[a,\Fb,c]$ on $\Omega=(0,1)$. The issue is more subtle in case of free boundary domains due to gas-mass interactions. In particular, we have to sacrifice the bounds of $[b_1,c]$ to obtain the dissipation of the free boundary position $x_w(t)$, see \eqref{xwf-int} in Proposition \ref{mix-eng}. Thus, it is important to deal with the dissipation of $[b_1,c]$ without relying on $x_w(t)$. Indeed, we are able to obtain
\[
\frac{\rd}{\rd t} N^{(c)}_f(t) +\|c\|_2^2\leq o(1) \|b_1\|_2^2 +C (|v_w(t)|^2+\|\{\FI-\FP\}f\|_2^2)+\cdots
\]
\[
\frac{\rd}{\rd t} N^{(b_1)}_f(t) +\|b_1\|_2^2\leq o(1) \|a\|_2^2 +C (\|c\|_2^2+|v_w(t)|^2+\|\{\FI-\FP\}f\|_2^2)+\cdots
\]
and
\[
\frac{\rd}{\rd t} N^{(b_i)}_f(t) +\|b_i\|_2^2\leq C (|v_w(t)|^2+\|\{\FI-\FP\}f\|_2^2)+\cdots,\ i=2,3.
\]
Since $o(1)$ is an arbitrary constant that can be sufficiently small, the macroscopic dissipation of $[\Fb,c]$ only relies on the free boundary velocity $v_w(t)$ and the microscopic component $\{\FI-\FP\}f$ that has been obtained in \eqref{dis-xv-mif}. To obtain the dissipation of the remaining macroscopic component $a$, we deduce that
\[
\frac{\rd}{\rd t} N^{(a)}_f(t) +\|a\|_2^2\leq C (|x_w(t)|^2+|v_w(t)|^2+\|b_1\|_2^2 + \|c\|_2^2+\|\{\FI-\FP\}f\|_2^2)+\cdots
\]
This indicates that only the dissipation of $a$ is connected with the free boundary position $x_w(t)$. Therefore, a suitable linear combination of the above estimates with \eqref{xwf-int} leads to the dissipation of all the macroscopic components $[a,\Fb,c]$ together with the free boundary position $x_w(t)$, see the desired estimate \eqref{es Pf} in Lemma \ref{lem es Pf}. This is a key step achieving the exponential dissipation structure for the coupled Boltzmann-Newton system in the $L^2$ setting.  

\medskip
\noindent$\bullet$ \textit{Local-in-time regularity and uniqueness.}

Due to the appearance of free boundary pair in transport term of the perturbed equation
\begin{eqnarray}
		&\dis \pa_t f+\frac{v_1-v_w(t)x_1}{1+x_w(t)}\pa_{x_1} f+Lf=\Ga(f,f),\notag
	\end{eqnarray}
the proof of uniqueness for $L^\infty$ solutions requires additional regularity beyond the standard energy framework. Following the strategy developed in \cite{cao kim, guo-17-inv}, we introduce a kinetic distance weight $W_c$
defined in \eqref{wc}, and establish local-in-time $W^{1,p}$ for $2<p<+\infty$. These estimates yield an $L^{1+\de}$ stability result for some $0<\de\ll 1$, which in turn implies the uniqueness of solutions. The details of this argument are presented in Sections~\ref{sec-loc-reg} and~\ref{sec-sta}.

\section{Boundary interaction between $x_w$, $v_w$ and $f$}\label{inter-bd-sec}
In this section, we derive the key energy estimates for the free boundary pair $[x_w(t),v_w(t)]$, in particular, we establish a dissipation mechanism for this pair through a newly constructed quadratic form that couples the free boundary dynamics, the boundary contributions from the energy of distribution function, the associated mass flux and the conservation of mass at kinetic level.

We begin by the {\it a priori} estimates for the perturbed problem
	\begin{eqnarray}\label{PBE-se2}
		&\dis \pa_t f+\frac{v_1-v_w(t)x_1}{1+x_w(t)}\pa_{x_1} f+Lf=\Ga(f,f),\ x_1\in(0,1),\ t>0,
	\end{eqnarray}
and
\begin{align}\label{trb-ode-se2}
\left\{\begin{array}{rll}
		\frac{{\rm d}}{{\rm d}t}x_w(t)&=v_w(t),\\[2mm]
		\frac{{\rm d}}{{\rm d}t}v_w(t)&=-\kappa x_w(t)-\mathcal M\left(P_r[\mu+\sqrt{\mu}f]-P_l[\mu+\sqrt{\mu}f] \right),
\end{array}\right.
	\end{align}
with the initial condition
	\begin{align}\label{icf-se2}
		f(0,x_1,v)=f_0(x_1,v),\ (x_w,v_w)(0)=(x_{w0},v_{w0}),
	\end{align}
and the boundary condition
	\begin{align}
\left\{\begin{array}{rll}
		&f(t,0,v)|_{v_1>0}=P_\gamma f(t,0,v),\label{bc-f-se2}\\[2mm]
		&f(t,1,v)|_{v_1-v_w(t)<0}=P_\gamma f(t,1,v)+r(t,v).
\end{array}\right.
	\end{align}
In this direction, we have the following result.

\begin{proposition}\label{mix-eng}
Let the conditions listed in Theorem \ref{main result} be satisfied. Assume $[f,x_w,v_w]$ is a strong solution to \eqref{PBE-se2}, \eqref{trb-ode-se2},
\eqref{icf-se2} and \eqref{bc-f-se2}, and also suppose
\begin{align}
\|wf(t)\|_{\infty}+|v_w(t)|+|x_w(t)|\leq 2\vps_0,\label{apas}
\end{align}
for all $t\in[0,\infty)$, where $\vps_0>0$ is sufficiently small. 
Then
there exist constants $\la_1>0$ and $C>0$ and an energy functional $\CE_0$ which satisfies $$\CE_{0}(t)\thicksim \| f(t)\|_{2}^2+\kappa|x_w(t)|^2+|v_w(t)|^2$$
 such that
\begin{align}\label{dis-xv-mif}
\frac{{\rm d}}{{\rm d}t}\CE_0(t)&+\la_1\left\{\left|\left(\FI-\frac{\sqrt{\mu_w}}{\sqrt{\mu}} P_\gamma \right)f\right|_{2,+,1}^2+|(\mathbf I-P_\gamma)f|_{2,+,0}^2\right\}\notag\\
&+\la_1\left\|(\nu(1+x_w))^{\frac{1}{2}}(\FI-\FP)f\right\|_{2}^2+\la_1|v_w(t)|^2\notag\\
\leq& C\left\| \nu^{-\frac{1}{2}}\Gamma(f,f)\right\|_{2}^2+CP(|x_w(t)|,|v_w(t)|,\|wf(t)\|_\infty),
\end{align}
where
\begin{equation*}
P(|x_w(t)|,|v_w(t)|,\|wf(t)\|_\infty):=(|x_w(t)|+|v_w(t)|)(|v_w(t)|^2+\|wf(t)\|_\infty^2).
\end{equation*}
Moreover, there exists a constant $C_1>0$ such that
\begin{align}\label{xwf-int}
	\frac{{\rm d}}{{\rm d}t}&(x_wv_w)
	+\mathcal M\frac{{\rm d}}{{\rm d}t}\int_{\Omega\times\mathbb R^3}x_w(1+x_w)x_1v_1\sqrt{\mu(v)}f{\rm d}x_1{\rm d}v\notag\\
	&
	+\frac{1}{2}\left(\kappa+\frac{\mathcal M}{2}\right)|x_w(t)|^2\notag\\
	\leq &CP(|x_w(t)|,|v_w(t)|,\|wf(t)\|_\infty)+C(|v_w(t)|^2+\|(\mathbf I-\mathbf P)f\|_2^2)+C_1(\|b_1(t)\|_2^2+\|c(t)\|_2^2).
\end{align}
\end{proposition}

\begin{proof}
The proof is divided into the following three steps.

\noindent\underline{{\it Step 1. Energy estimates for the free boundary pair $[x_w,v_w].$}} Recalling \eqref{trb-ode-se2},
we begin by the analysis of $P_r[\mu+\sqrt{\mu}f]-P_l[\mu+\sqrt{\mu}f]$ in \eqref{trb-ode-se2}.
Note that
\begin{align}\label{P_rf-P_lf}
	P_r[\mu&+\sqrt{\mu}f]-P_l[\mu+\sqrt{\mu}f]\notag\\
	&=\int_{v_1-v_w(t)<0}\mu(v)|v_1-v_w(t)|^2{\rm d}v+\frac{\sqrt{2\pi}}{2}\int_{v_1-v_w(t)<0}\mu(v)|v_1-v_w(t)|{\rm d}v\notag\\
	&\quad-\int_{v_1-v_w(t)>0} \mu(v)|v_1-v_w(t)|^2{\rm d}v-\frac{\sqrt{2\pi}}{2}\int_{v_1-v_w(t)>0} \mu(v)|v_1-v_w(t)|{\rm d}v\notag\\
	&\quad-\int_{v_1-v_w(t)>0} \sqrt{\mu(v)}f(t,1,v)|v_1-v_w(t)|^2{\rm d}v\notag\\&\quad-\frac{\sqrt{2\pi}}{2}\int_{v_1-v_w(t)>0}\sqrt{\mu(v)}f(t,1,v)|v_1-v_w(t)|{\rm d}v.
\end{align}
To proceed, we now denote
\begin{equation*}
	I_1=\int_{v_1-v_w(t)<0}[v_1-v_w(t)]^2\mu(v){\rm d}v-\int_{v_1-v_w(t)>0}[v_1-v_w(t)]^2\mu(v){\rm d}v,
\end{equation*}
and
\begin{equation*}
	I_2=-\frac{\sqrt{2\pi}}{2}\int_{v_1-v_w(t)<0}[v_1-v_w(t)]\mu(v){\rm d}v-\frac{\sqrt{2\pi}}{2}\int_{v_1-v_w(t)>0}[v_1-v_w(t)]\mu(v){\rm d}v.	\end{equation*}
Then, under the {\it a priori} assumption \eqref{apas}, there  exist $\xi_1,\xi_2\in(0,1)$ such that
\begin{align*}
	I_1&=\int_{v_1-v_w(t)<0}[v_1-v_w(t)]^2\mu(v){\rm d}v-\int_{v_1-v_w(t)>0}[v_1-v_w(t)]^2\mu(v){\rm d}v\notag\\
	&=\frac{1}{\sqrt{2\pi}}\int_{-\infty}^0v_1^2\Big[e^{-\frac{1}{2}(v_1+v_w(t))^2}-e^{-\frac{1}{2}(v_1-v_w(t))^2}\Big]{\rm d}v_1\notag\\	&=\frac{1}{\sqrt{2\pi}}\int_{-\infty}^0v_1^2\Big[\Big(e^{-\frac{1}{2}v_1^2}-v_w(t)v_1e^{-\frac{1}{2}v_1^2}\Big)
-\Big(e^{-\frac{1}{2}v_1^2}+v_w(t)v_1e^{-\frac{1}{2}v_1^2}\Big)\Big]{\rm d}v_1\notag\\
&\quad+\frac{1}{2\sqrt{2\pi}}|v_w|^2\int^0_{-\infty}v_1^2\Big[e^{-\frac{1}{2}(v_1+\xi_1v_w(t))^2}\Big(-1+(v_1+\xi_1v_w(t))^2\Big)\notag\\
&\qquad-e^{-\frac{1}{2}(v_1-\xi_2v_w(t))^2}\Big(-1+(v_1-\xi_2v_w(t))^2\Big)\Big]{\rm d}v_1\notag\\
&	=\frac{4}{\sqrt{2\pi}}v_w(t)+\frac{1}{2\sqrt{2\pi}}|v_w(t)|^2\int^0_{-\infty}v_1^2\Big[e^{-\frac{1}{2}(v_1+\xi_1v_w(t))^2}\Big(-1+(v_1+\xi_1v_w(t))^2\Big)\notag\\
	&\qquad-e^{-\frac{1}{2}(v_1-\xi_2v_w(t))^2}\Big(-1+(v_1-\xi_2v_w(t))^2\Big)\Big]{\rm d}v_1\notag\\
&	=:\frac{4}{\sqrt{2\pi}}v_w(t)+\mathcal I_1(t)|v_w(t)|^2,
\end{align*}
and
\begin{align*}
I_2&=-\frac{\sqrt{2\pi}}{2}\times(2\pi)^{-3/2}\int_{v_1-v_w(t)<0}[v_1-v_w(t)]e^{-\frac{1}{2}|v|^2}{\rm d}v\notag\\
	&\quad-\frac{\sqrt{2\pi}}{2}\times(2\pi)^{-3/2}\int_{v_1-v_w(t)>0}[v_1-v_w(t)]e^{-\frac{1}{2}|v|^2}{\rm d}v\notag\\	
&	=-\frac{1}{2}\int_{-\infty}^0v_1\Big[e^{-\frac{1}{2}(v_1+v_w(t))^2}-e^{-\frac{1}{2}(v_1-v_w(t))^2}\Big]{\rm d}v_1\notag\\ &=-\frac{1}{2}\int_{-\infty}^0v_1\Big[\Big(e^{-\frac{1}{2}v_1^2}-v_w(t)v_1e^{-\frac{1}{2}v_1^2}\Big)-\Big(e^{-\frac{1}{2}v_1^2}+v_w(t)v_1e^{-\frac{1}{2}v_1^2}\Big)\Big]{\rm d}v_1\notag\\
	&\quad-\frac{1}{4}|v_w|^2\int^0_{-\infty}v_1\Big[e^{-\frac{1}{2}(v_1+\xi_1v_w(t))^2}\Big(-1+(v_1+\xi_1v_w(t))^2\Big)\notag\\
	&\qquad-e^{-\frac{1}{2}(v_1-\xi_2v_w(t))^2}\Big(-1+(v_1-\xi_2v_w(t))^2\Big)\Big]{\rm d}v_1\notag\\
	&=\frac{\sqrt{2\pi}}{2}v_w(t)-\frac{1}{4}|v_w(t)|^2\int^0_{-\infty}v_1\Big[e^{-\frac{1}{2}(v_1+\xi_1v_w(t))^2}\Big(-1+(v_1+\xi_1v_w(t))^2\Big)\notag\\
	&\qquad-e^{-\frac{1}{2}(v_1-\xi_2v_w(t))^2}\Big(-1+(v_1-\xi_2v_w(t))^2\Big)\Big]{\rm d}v_1\notag\\
	&=:\frac{\sqrt{2\pi}}{2}v_w(t)+\mathcal I_2(t)|v_w(t)|^2.
\end{align*}
Plugging $I_1$ and $I_2$ into \eqref{P_rf-P_lf}, we conclude that
\begin{align}
P_r[\mu&+\sqrt{\mu}f]-P_l[\mu+\sqrt{\mu}f]\notag\\
=&\left(\frac{4}{\sqrt{2\pi}}+\frac{\sqrt{2\pi}}{2}\right)v_w(t)
-\int_{v_1-v_w(t)>0}[v_1-v_w(t)]^2\sqrt{\mu(v)}f(t,1,v){\rm d}v\notag\\
	&-\frac{\sqrt{2\pi}}{2}\int_{v_1-v_w(t)>0}[v_1-v_w(t)]\sqrt{\mu(v)}f(t,1,v){\rm d}v\notag\\
	&+(\mathcal I_1(t)+\mathcal I_2(t))|v_w(t)|^2.\label{sumGpm}
\end{align}
Next, we multiply $\eqref{trb-ode-se2}_1$ by $\kappa x_w(t)$ and $\eqref{trb-ode-se2}_2$ by $v_w(t)$, respectively. Taking the summation of the resulting identities and applying \eqref{sumGpm}, we obtain
\begin{align}\label{d_tx_w^2} \frac{1}{2}\frac{{\rm d}}{{\rm d}t}&(\kappa|x_w(t)|^2+|v_w(t)|^2)+\left(\frac{4}{\sqrt{2\pi}}+\frac{\sqrt{2\pi}}{2}\right)\mathcal M|v_w(t)|^2\notag\\
&
\underbrace{-\mathcal Mv_w\int_{v_1-v_w(t)>0}[v_1-v_w(t)]^2\sqrt{\mu(v)}f(t,1,v){\rm d}v}_{\eqref{d_tx_w^2}_1}\notag\\
	&-\frac{\sqrt{2\pi}}{2}\mathcal Mv_w\int_{v_1-v_w(t)>0}[v_1-v_w(t)]\sqrt{\mu(v)}f(t,1,v){\rm d}v\leq C|v_w(t)|^3.
\end{align}
The term \eqref{d_tx_w^2}$_1$ can be rewritten as
\begin{align}
	\eqref{d_tx_w^2}_1=&-\mathcal Mv_w\int_{v_1-v_w(t)>0}[v_1-v_w(t)]^2\sqrt{\mu(v)}\Big[f(t,1,v)\notag\\
	&-\sqrt{2\pi\mu_w(v)}\int_{u_1-v_w(t)>0}f(t,1,u)\sqrt{\mu(u)}|u_1-v_w(t)|{\rm d}u\Big]{\rm d}v\notag\\
	&-\sqrt{2\pi}\mathcal M v_w\int_{v_1-v_w(t)>0}[v_1-v_w(t)]^2\sqrt{\mu(v)}\sqrt{\mu_w(v)}\int_{u_1-v_w(t)>0}f(t,1,u)\sqrt{\mu(u)}|u_1-v_w(t)|{\rm d}u{\rm d}v\notag\\
	=&-\mathcal Mv_w\int_{v_1-v_w(t)>0}[v_1-v_w(t)]^2\sqrt{\mu(v)}\left(\FI-\frac{\sqrt{\mu_{w}}}{\sqrt{\mu}} P_\gamma \right)f(1){\rm d}v\notag\\
	&-\frac{\sqrt{2\pi}}{2}\mathcal Mv_w\int_{u_1-v_w(t)>0}f(t,1,u)\sqrt{\mu(u)}|u_1-v_w(t)|{\rm d}u\notag\\
	&-\sqrt{2\pi}\mathcal M v_w\int_{v_1-v_w(t)>0}[v_1-v_w(t)]^2[\sqrt{\mu(v)}\sqrt{\mu_w(v)}-\mu_w(v)]{\rm d}v\notag\\
	&\quad\times\int_{u_1-v_w(t)>0}f(t,1,u)\sqrt{\mu(u)}|u_1-v_w(t)|{\rm d}u.\notag
\end{align}
Substituting the above into \eqref{d_tx_w^2}, we deduce
\begin{align}\label{es x_w^2}
\frac{{\rm d}}{{\rm d}t}&\left\{\kappa|x_w(t)|^2+|v_w(t)|^2\right\}+\left(\frac{8}{\sqrt{2\pi}}+\sqrt{2\pi}\right)\mathcal M|v_w(t)|^2\notag\\
&-2\mathcal Mv_w\int_{v_1-v_w(t)>0}[v_1-v_w(t)]^2\sqrt{\mu(v)}\left(\FI-\frac{\sqrt{\mu_{w}}}{\sqrt{\mu}} P_\gamma \right)f(1){\rm d}v\notag\\
	&-2\sqrt{2\pi}\mathcal Mv_w\int_{u_1-v_w(t)>0}f(1)\sqrt{\mu(u)}|u_1-v_w(t)|{\rm d}u\leq CP(|v_w(t)|,\|wf(t)\|_\infty).
\end{align}
On the other hand,  multiplying $\eqref{trb-ode-se2}_1$ by $v_w(t)$ and $\eqref{trb-ode-se2}_2$ by $x_w(t)$, respectively, one arrives at
\begin{align}\label{d_t(x_wv_w)}
\frac{{\rm d}}{{\rm d}t}(x_wv_w)&+\kappa|x_w(t)|^2-\mathcal Mx_w\int_{v_1-v_w(t)>0}[v_1-v_w(t)]^2\sqrt{\mu(v)}\left(\FI-\frac{\sqrt{\mu_{w}}}{\sqrt{\mu}} P_\gamma \right)f(1){\rm d}v\notag\\
	&-\sqrt{2\pi}\mathcal Mx_w\int_{u_1-v_w(t)>0}f(t,1,u)\sqrt{\mu(u)}|u_1-v_w(t)|{\rm d}u\notag\\
= &|v_w(t)|^2-\left(\frac{4}{\sqrt{2\pi}}+\frac{\sqrt{2\pi}}{2}\right)\mathcal Mx_wv_w-\mathcal M(\mathcal I_1(t)+\mathcal I_2(t))x_w|v_w(t)|^2\notag\\
&
+\sqrt{2\pi}\mathcal M x_w\int_{v_1-v_w(t)>0}[v_1-v_w(t)]^2[\sqrt{\mu(v)}\sqrt{\mu_w(v)}-\mu_w(v)]{\rm d}v\notag\\
&\quad\times\int_{u_1-v_w(t)>0}f(t,1,u)\sqrt{\mu(u)}|u_1-v_w(t)|{\rm d}u.
\end{align}
We note that the problematic term
\begin{equation*}
	-2\sqrt{2\pi}\mathcal Mv_w\int_{u_1-v_w(t)>0}f(1)\sqrt{\mu(u)}|u_1-v_w(t)|{\rm d}u
\end{equation*}
in \eqref{es x_w^2} is similar to
\begin{equation*}
	-\sqrt{2\pi}\mathcal Mx_w\int_{u_1-v_w(t)>0}f(1)\sqrt{\mu(u)}|u_1-v_w(t)|{\rm d}u
\end{equation*}
in \eqref{d_t(x_wv_w)}. In Step 2 and Step 3 below, we shall develop two different cancellation mechanisms to deal with these two linear terms, respectively.

\noindent\underline{{\it Step 2. Dissipation of $v_w$ and energy estimates for $f$.}}
To handle the interaction terms in \eqref{es x_w^2}, we next derive the following energy estimate for $f$:
\begin{align}\label{beng-f}
\frac{{\rm d}}{{\rm d}t}&\left\| (1+x_w)^{\frac{1}{2}}f(t)\right\|_{2}^2+\delta\left\|(\nu(1+x_w))^{\frac{1}{2}}(\textbf I-\textbf P)f\right\|_{2}^2+\left|\left(\FI-\frac{\sqrt{\mu_w}}{\sqrt{\mu}} P_\gamma \right)f\right|_{2,+,1}^2+|(\mathbf I-P_\gamma)f|_{2,+,0}^2\notag\\
&-(2\pi)^{-1/2}(2+\frac{3\pi}{2})|v_w(t)|^2+2\sqrt{2\pi}v_w\int_{v_1-v_w>0}f(1)\sqrt{\mu}|v_1-v_w|{\rm d}v\notag\\
\leq& C\left\|\nu^{-\frac{1}{2}} \Gamma(f,f)\right\|_{2}^2+CP(|x_w(t)|,|v_w(t)|,\|wf(t)\|_\infty).
\end{align}
Indeed, from Lemma \ref{lem green fun} we obtain
\begin{align}\label{bcl2}
	\frac{{\rm d}}{{\rm d}t}&\left\| (1+x_w)^{\frac{1}{2}}f(t)\right\|_{2}^2+|f|_{2,+}^2+\delta\left\|(\nu(1+x_w))^{\frac{1}{2}}(\textbf I-\textbf P)f\right\|_{2}^2\notag\\
	&\leq |f|_{2,-}^2+C\int_{\Omega\times\mathbb R^3}|\Gamma(f,f)||(\textbf I-\textbf P)f|{\rm d}x_1{\rm d}v.
\end{align}
We now compute the boundary contributions.
It is clear that at $x_1=0$,
\begin{align}\label{0bd}
|f|_{2,+,0}^2-|f|_{2,-,0}^2=|(\mathbf I-P_\gamma)f|_{2,+,0}^2,
\end{align}
and at $x_1=1$,
\begin{align}\label{1bd}
|f|_{2,+,1}^2-|f|_{2,-,1}^2=&\int_{v_1-v_w(t)>0}|v_1-v_w(t)|f^2(t,1,v){\rm d}v\notag\\
&-\int_{v_1-v_w(t)<0}|v_1-v_w(t)|f^2(t,1,v){\rm d}v.
\end{align}
We shall use \eqref{drb-f} to treat the negative term of \eqref{1bd}. In fact, by Taylor's expansion, we can write
\begin{align}\label{taylor.add}
\sqrt{2\pi}&\frac{\mu_{w}(v)}{\sqrt{\mu(v)}}\int_{v_1-v_w(t)>0}\mu(v)|v_1-v_w(t)|{\rm d}v-\sqrt{\mu(v)}=\sqrt{\mu(v)}(v_1-\frac{\sqrt{2\pi}}{2})v_w(t)+\mathcal T(t,v)|v_w(t)|^2, 
\end{align}
where $\CT(t,v)$ is defined as
\begin{align*}
\CT(t,v):&=-\frac{\sqrt{2\pi}}{2}\frac{\mu(v_1-\xi_3v_w(t),v_2,v_3)}{\sqrt{\mu(v)}}\Big[\Big(1-(v_1-\xi_3v_w(t))^2\Big)\int_{v_1-\xi_3v_w(t)>0}\mu(v)|v_1-\xi_3v_w(t)|{\rm d}v\notag\\
&\qquad-2(v_1-\xi_3v_w(t))\int_{v_1-\xi_3v_w(t)>0}\mu(v){\rm d}v-e^{-\frac{1}{2}(v_1-\xi_3v_w(t))^2}\Big].
\end{align*}
with $\xi_3\in(0,1)$. Therefore, using \eqref{drb-f}, \eqref{taylor.add} and the {\it a priori} assumption \eqref{apas}, one obtains
\begin{align}
\int_{v_1-v_w(t)<0}&|v_1-v_w(t)|f^2(t,1,v){\rm d}v\notag\\
=&\int_{v_1-v_w(t)<0}|v_1-v_w(t)|\Big[\sqrt{2\pi}\frac{\mu_w(v)}{\sqrt{\mu}(v)}\int_{u_1-v_w(t)>0}f(1)\sqrt{\mu(u)}|u_1-v_w(t)|{\rm d}u\notag\\
&\qquad+\sqrt{\mu(v)}(v_1-\frac{\sqrt{2\pi}}{2})v_w(t)+\mathcal T(t,v)|v_w(t)|^2\Big]{\rm d}v\notag\\
=&\int_{v_1-v_w(t)<0}|v_1-v_w(t)|\Big(\sqrt{2\pi}\frac{\mu_w(v)}{\sqrt{\mu}(v)}\Big)^2\Big(\int_{u_1-v_w(t)>0}f(1)\sqrt{\mu(u)}|u_1-v_w(t)|{\rm d}u\Big)^2{\rm d}v\notag\\
&+\int_{v_1-v_w(t)<0}|v_1-v_w(t)|(v_1-\frac{\sqrt{2\pi}}{2})^2|v_w(t)|^2\mu(v){\rm d}v\notag\\
&+|v_w(t)|^4\int_{v_1-v_w(t)<0}|v_1-v_w(t)|\mathcal T^2(t,v){\rm d}v\notag\\
&+2\int_{v_1-v_w(t)<0}|v_1-v_w(t)|\sqrt{2\pi}\Big((v_1-\frac{\sqrt{2\pi}}{2})v_w(t)\mu_w(v)+\frac{\mu_w(v)}{\sqrt{\mu(v)}}\mathcal T(t,v)|v_w(t)|^2\Big){\rm d}v\notag\\
&\quad\times\int_{v_1-v_w(t)>0}f(1)\sqrt{\mu(v)}|v_1-v_w(t)|{\rm d}v\notag\\
&+2v_w|v_w(t)|^2\int_{v_1-v_w(t)<0}|v_1-v_w(t)|\sqrt{\mu(v)}(v_1-\frac{\sqrt{2\pi}}{2})\mathcal T(t,v){\rm d}v\notag\\
=&\int_{v_1-v_w(t)<0}|v_1-v_w(t)|2\pi\mu_w(v)\Big(\int_{u_1-v_w(t)>0}f(1)\sqrt{\mu(u)}|u_1-v_w(t)|{\rm d}u\Big)^2{\rm d}v\notag\\
&+2\pi\int_{v_1-v_w(t)<0}|v_1-v_w(t)|\Big[\Big(\frac{\mu_w(v)}{\sqrt{\mu}(v)}\Big)^2-\mu_w(v)\Big]\notag\\
&\quad\times\Big(\int_{u_1-v_w(t)>0}f(1)\sqrt{\mu(u)}|u_1-v_w(t)|{\rm d}u\Big)^2{\rm d}v+\frac{1}{\sqrt{2\pi}}(2+\frac{3\pi}{2})|v_w(t)|^2\notag\\
&+|v_w(t)|^2\int_{v_1-v_w(t)<0}|v_1-v_w(t)|\Big[(v_1-\frac{\sqrt{2\pi}}{2})^2\mu(v)-(v_1-v_w(t)-\frac{\sqrt{2\pi}}{2})^2\mu_w(v)\Big]{\rm d}v\notag\\
&+|v_w(t)|^4\int_{v_1-v_w(t)<0}|v_1-v_w(t)|\mathcal T^2(t,v){\rm d}v\notag\\
&+(-2\sqrt{2\pi}v_w+2|v_w(t)|^2)\int_{v_1-v_w(t)>0}f(1)\sqrt{\mu(v)}|v_1-v_w(t)|{\rm d}v\notag\\
&+2\sqrt{2\pi}|v_w(t)|^2\int_{v_1-v_w(t)<0}|v_1-v_w(t)|\frac{\mu_w(v)}{\sqrt{\mu(v)}}\mathcal T(t,v){\rm d}v\notag\\
&\quad\times\int_{v_1-v_w(t)>0}f(1)\sqrt{\mu(v)}|v_1-v_w(t)|{\rm d}v\notag\\
&+2v_w|v_w(t)|^2\int_{v_1-v_w(t)<0}|v_1-v_w(t)|\mu_w(v)(v_1-\frac{\sqrt{2\pi}}{2})\mathcal T(t,v){\rm d}v,\label{1bd-cp}
\end{align}
Then \eqref{beng-f} follows by plugging \eqref{0bd}, \eqref{1bd}, \eqref{drb-f} and \eqref{1bd-cp} into \eqref{bcl2}.

Now combining \eqref{es x_w^2}  with $\mathcal M\times$\eqref{beng-f} gives that
\begin{align}\label{dis-vw}
\mathcal M\frac{{\rm d}}{{\rm d}t}&\left\| (1+x_w)^{\frac{1}{2}}f(t)\right\|_{2}^2+\frac{{\rm d}}{{\rm d}t}(\kappa|x_w(t)|^2+|v_w(t)|^2)
+\mathcal M\left(\left|\left(\FI-\frac{\sqrt{\mu_{w}}}{\sqrt{\mu}} P_\gamma \right)f\right|_{2,+,1}^2+|(\mathbf I-P_\gamma)f|_{2,+,0}^2\right)\notag\\
&+\delta\mathcal M\left\|(\nu(1+x_w))^{\frac{1}{2}}(\textbf I-\textbf P)f\right\|_{2}^2+\mathcal M\left[\left(\frac{8}{\sqrt{2\pi}}+\sqrt{2\pi}\right)-\frac{1}{\sqrt{2\pi}}(2+\frac{3\pi}{2})\right]|v_w(t)|^2\notag\\
&-2\mathcal Mv_w\int_{v_1-v_w(t)>0}[v_1-v_w(t)]^2\sqrt{\mu(v)}\left(\FI-\frac{\sqrt{\mu_w}}{\sqrt{\mu}}P_\gamma\right)f{\rm d}v\notag\\
\leq& C\left\|\nu^{-\frac{1}{2}} \Gamma(f,f)\right\|_{2}^2
+CP(|x_w(t)|,|v_w(t)|,\|wf(t)\|_\infty),
\end{align}
in which the problematic term $2\sqrt{2\pi}\mathcal Mv_w\int_{v_1-v_w(t)>0}f(1)\sqrt{\mu(v)}|v_1-v_w(t)|{\rm d}v$ has been cancelled.

Next, Cauchy-Schwarz's inequality directly gives
\begin{align*}
\Big|v_w&\int_{v_1-v_w(t)>0}[v_1-v_w(t)]^2\sqrt{\mu(v)}\left(\FI-\frac{\sqrt{\mu_{w}}}{\sqrt{\mu}} P_\gamma \right)f{\rm d}v\Big|\notag\\
\leq& |v_w(t)|\Big(\int_{v_1-v_w(t)>0}[v_1-v_w(t)]^3\mu(v){\rm d}v\Big)^{\frac{1}{2}}\Big(\int_{v_1-v_w(t)>0}[v_1-v_w(t)]\Big[\left(\FI-\frac{\sqrt{\mu_{w}}}{\sqrt{\mu}} P_\gamma \right)f\Big]^2{\rm d}v\Big)^{\frac{1}{2}}\\
\leq&\Big(\frac{2}{\sqrt{2\pi}}\Big)^{\frac{1}{2}}|v_w(t)|\Big(\int_{v_1-v_w(t)>0}[v_1-v_w(t)]\left[\left(\FI-\frac{\sqrt{\mu_{w}}}{\sqrt{\mu}} P_\gamma \right)f\right]^2{\rm d}v\Big)^{\frac{1}{2}}\notag\\
&+CP(|x_w(t)|,|v_w(t)|,\|wf(t)\|_\infty).
\end{align*}
Since
\begin{equation*}
\Big[\frac{8}{\sqrt{2\pi}}+\sqrt{2\pi}-\frac{1}{\sqrt{2\pi}}(2+\frac{3\pi}{2})\Big]-\frac{2}{\sqrt{2\pi}}=\frac{8+\pi}{2\sqrt{2\pi}}>0,
\end{equation*}
there exists a constant $C_0>0$ such that
\begin{align*}
&\left|\left(\FI-\frac{\sqrt{\mu_{w}}}{\sqrt{\mu}} P_\gamma \right)f\right|_{2,+,1}^2+\Big[\frac{8}{\sqrt{2\pi}}+\sqrt{2\pi}-\frac{1}{\sqrt{2\pi}}(2+\frac{3\pi}{2})\Big]|v_w(t)|^2\notag\\
&\qquad-2v_w\int_{v_1-v_w(t)>0}[v_1-v_w(t)]^2\sqrt{\mu(v)}\left(\FI-\frac{\sqrt{\mu_{w}}}{\sqrt{\mu}} P_\gamma \right)f{\rm d}v\notag\\
&\quad\geq C_0 \left(|v_w(t)|^2+\left|\left(\FI-\frac{\sqrt{\mu_{w}}}{\sqrt{\mu}} P_\gamma \right)f\right|_{2,+,1}^2\right)-CP(|v_w(t)|,\|wf(t)\|_\infty).
\end{align*}
Hence \eqref{dis-vw} is now reduced to
\begin{align}\label{dis-vw-2}
\mathcal M\frac{{\rm d}}{{\rm d}t}&\left\| (1+x_w)^{\frac{1}{2}}f(t)\right\|_{2}^2+\frac{{\rm d}}{{\rm d}t}(\kappa|x_w(t)|^2+|v_w(t)|^2)
+\mathcal M|(\mathbf I-P_\gamma)f|_{2,+,0}^2\notag\\
&+\de\mathcal M\left\|(\nu(1+x_w))^{\frac{1}{2}}(\textbf I-\textbf P)f\right\|_{2}^2+C_0\mathcal M |v_w(t)|^2+C_0\mathcal M\left|\left(\FI-\frac{\sqrt{\mu_{w}}}{\sqrt{\mu}} P_\gamma \right)f\right|_{2,+,1}^2\notag\\
\leq& C\left\|\nu^{-\frac{1}{2}} \Gamma(f,f)\right\|_{2}^2
+CP(|x_w(t)|,|v_w(t)|,\|wf(t)\|_\infty).
\end{align}
Finally, let us define
\begin{equation*}
\mathcal E_0(t):=\mathcal M	\left\| (1+x_w)^{\frac{1}{2}}f(t)\right\|_{2}^2+\kappa|x_w(t)|^2+|v_w(t)|^2,
\end{equation*}
then it is straightforward to verify that $\mathcal E_0(t)\sim\|f(t)\|_2^2+\kappa|x_w(t)|^2+|v_w(t)|^2$. Hence \eqref{dis-xv-mif} follows directly from \eqref{dis-vw-2}.

\noindent\underline{{\it Step 3. Dissipation of $x_w$.}} If $\kappa>0$, the dissipation of $x_w(t)$ is derived directly from the interaction energy estimate \eqref{d_t(x_wv_w)}. In order to treat the boundary terms in \eqref{d_t(x_wv_w)}, we proceed as in \eqref{dis-vw}, employing the mass flux of $f$ to cancel the troublesome contributions such as
\begin{equation*}
	x_w\sqrt{2\pi}\mathcal M\int_{u_1-v_w(t)>0}f(t,1,u)\sqrt{\mu(u)}|u_1-v_w(t)|{\rm d}u.
\end{equation*}
Fortunately, by choosing a suitable test function,  we can eliminate the boudary term in \eqref{d_t(x_wv_w)}, and derive the dissipation of $x_w(t)$ at the same time. Therefore, the dissipation of $x_w(t)$ is available even though $\kappa=0$.
To do this, we multiply $x_w(t)(1+x_w(t))v_1x_1\sqrt{\mu(v)}$ on both sides of \eqref{PBE-se2} to have
\begin{align}
	0=&\int_{\Omega\times\mathbb R^3}x_w(t)v_1x_1\sqrt{\mu(v)}[(1+x_w(t))\pa_tf+(v_1-v_w(t)x_1)\pa_{x_1}f]{\rm d}x_1{\rm d}v\notag\\
	=&\frac{{\rm d}}{{\rm d}t}\int_{\Omega\times\mathbb R^3}x_w(t)(1+x_w(t))x_1v_1\sqrt{\mu(v)}f{\rm d}x_1{\rm d}v\notag\\
    &-\int_{\Omega\times\mathbb R^3}\pa_t[x_w(t)(1+x_w(t))v_1x_1]\sqrt{\mu(v)}\textbf Pf{\rm d}x_1{\rm d}v\notag\\
	&+\int_{\mathbb R^3}x_w(t)v_1(v_1-v_w(t))\sqrt{\mu(v)}f(1){\rm d}v\notag\\
	&+\int_{\Omega\times\mathbb R^3} x_w(t)v_w(t)v_1x_1\sqrt{\mu(v)}f{\rm d}x_1{\rm d}v-\int_{\Omega\times\mathbb R^3} x_w(t)(v_1-v_w(t)x_1)v_1\sqrt{\mu(v)}f{\rm d}x_1{\rm d}v.\notag
\end{align}
This gives that
\begin{align}\label{x_wf1}
\frac{{\rm d}}{{\rm d}t}&\int_{\Omega\times\mathbb R^3}x_w(t)(1+x_w(t))x_1v_1\sqrt{\mu(v)}f{\rm d}x_1{\rm d}v\notag\\	
&\qquad+x_w(t)\int_{\mathbb R^3}(v_1-v_w(t))^2\sqrt{\mu(v)}f(1){\rm d}v\notag\\
&\qquad-x_w(t)\int_{\Omega\times\mathbb R^3} v_1^2\sqrt{\mu(v)}f{\rm d}x_1{\rm d}v\notag\\
&=-x_w(t)v_w(t)\int_{\mathbb R^3}(v_1-v_w(t))\sqrt{\mu(v)}f(1){\rm d}v-2x_w(t)v_w(t)\int_{\Omega\times\mathbb R^3} v_1x_1\sqrt{\mu(v)}f{\rm d}x_1{\rm d}v\notag\\
&\qquad+\int_{\Omega\times\mathbb R^3}\pa_t[x_w(t)(1+x_w(t))v_1x_1]\sqrt{\mu(v)}\textbf Pf{\rm d}x_1{\rm d}v.
\end{align}
For the boundary term with $x_1=1$
\begin{align}\label{x_wf x_1=1}
	x_w&\int_{\mathbb R^3}(v_1-v_w(t))^2\sqrt{\mu(v)}f(1){\rm d}v\notag\\
	=&x_w\int_{\mathbb R^3}(v_1-v_w(t))^2\sqrt{\mu(v)}[P_\gamma f+{(\FI-P_\gamma)f\textbf 1_{\gamma^1_+}+r\textbf 1_{\gamma_-^1}}]{\rm d}v\notag\\
	=&\sqrt{2\pi} x_w\int_{\mathbb R^3}(v_1-v_w(t))^2\mu_w(v)\Big(\int_{u_1-v_w(t)>0}f(1)\sqrt{\mu(u)}|u_1-v_w(t)|{\rm d}u\Big){\rm d}v\notag\\
	&+\sqrt{2\pi} x_w\int_{\mathbb R^3}(v_1-v_w(t))^2\big(\mu(v)-\mu_w(v)\big)\Big(\int_{u_1-v_w(t)>0}f(1)\sqrt{\mu(u)}|u_1-v_w(t)|{\rm d}u\Big){\rm d}v\notag\\
	&+x_w\int_{v_1-v_w(t)>0}(v_1-v_w(t))^2\sqrt{\mu(v)}\left(\FI-\frac{\sqrt{\mu_w}}{\sqrt{\mu}}P_\gamma\right)f(1){\rm d}v\notag\\
	&+x_w\int_{v_1-v_w(t)>0}(v_1-v_w(t))^2(\sqrt{\mu_w(v)}-\sqrt{\mu(v)})P_\gamma f(1){\rm d}v\notag\\
    &+x_w\int_{v_1-v_w(t)<0}(v_1-v_w(t))^2\sqrt{\mu_w(v)}r{\rm d}v\notag\\
    &+x_w\int_{v_1-v_w(t)>0}(v_1-v_w(t))^2(\sqrt{\mu(v)}-\sqrt{\mu_w(v)})r{\rm d}v\notag\\
	=&\sqrt{2\pi} x_w\int_{u_1-v_w(t)>0}f(1)\sqrt{\mu(u)}|u_1-v_w(t)|{\rm d}u\notag\\
    &+x_w\int_{v_1-v_w(t)>0}(v_1-v_w(t))^2\sqrt{\mu(v)}\left(\FI-\frac{\sqrt{\mu_w}}{\sqrt{\mu}}P_\gamma\right)f(1){\rm d}v\notag\\
	&+x_w\int_{v_1-v_w(t)<0}(v_1-v_w(t))^2\sqrt{\mu_w(v)}r{\rm d}v\notag\\
	&+\sqrt{2\pi} x_w\int_{\mathbb R^3}(v_1-v_w(t))^2\big(\mu(v)-\mu_w(v)\big)\Big(\int_{u_1-v_w(t)>0}f(1)\sqrt{\mu(u)}|u_1-v_w(t)|{\rm d}u\Big){\rm d}v\notag\\
	&+x_w\int_{v_1-v_w(t)>0}(v_1-v_w(t))^2(\sqrt{\mu_w(v)}-\sqrt{\mu(v)})P_\gamma f(1){\rm d}v\notag\\
	&+x_w\int_{v_1-v_w(t)>0}(v_1-v_w(t))^2(\sqrt{\mu(v)}-\sqrt{\mu_w(v)})r{\rm d}v\notag\\
	=&\sqrt{2\pi} x_w\int_{u_1-v_w(t)>0}f(1)\sqrt{\mu(u)}|u_1-v_w(t)|{\rm d}u\notag\\
	&\quad+x_w\int_{v_1-v_w(t)>0}(v_1-v_w(t))^2\sqrt{\mu(v)}\left(\FI-\frac{\sqrt{\mu_w}}{\sqrt{\mu}}P_\gamma\right)f(1){\rm d}v\notag\\
	&\quad+x_w\int_{v_1-v_w(t)<0}(v_1-v_w(t))^2\sqrt{\mu_w(v)}r{\rm d}v+\mathcal I_3(t)x_w(t)v_w(t),
\end{align}
where we have defined
\begin{align*}
\mathcal I_3(t)&x_w(t)v_w(t)\notag\\
=&\sqrt{2\pi} x_w\int_{\mathbb R^3}(v_1-v_w(t))^2\big(\mu(v)-\mu_w(v)\big)\Big(\int_{u_1-v_w(t)>0}f(1)\sqrt{\mu(u)}|u_1-v_w(t)|{\rm d}u\Big){\rm d}v\notag\\
&+x_w\int_{v_1-v_w(t)>0}(v_1-v_w(t))^2(\sqrt{\mu_w(v)}-\sqrt{\mu(v)})P_\gamma f(1){\rm d}v\notag\\
&+x_w\int_{v_1-v_w(t)>0}(v_1-v_w(t))^2(\sqrt{\mu(v)}-\sqrt{\mu_w(v)})r{\rm d}v.	
\end{align*}
Using the  decomposition
\begin{equation*}
f=\sqrt{\mu}\left(a+\textbf b\cdot v+\frac{|v|^2-3}{2}c\right)+(\mathbf I-\mathbf P)f,
\end{equation*}
and $\int_0^1a{\rm d}x_1=-\frac{x_w}{1+x_w}$ (see \eqref{a-con} in next section for the details),
we can bound the third line on the left hand side of \eqref{x_wf1} as
\begin{align}\label{damp-x}
	-x_w(t)&\int_{\Omega\times\mathbb R^3} v_1^2\sqrt{\mu(v)}f{\rm d}x_1{\rm d}v\notag\\
	&=-x_w(t)\int_0^1 a{\rm d}x_1-x_w\int_{\Omega\times\mathbb R^3} v_1^2\sqrt{\mu(v)}\left(\sqrt{\mu}\frac{|v|^2-3}{2}c+(\mathbf I-\mathbf P)f\right){\rm d}x_1{\rm d}v\notag\\
	&=\frac{|x_w(t)|^2}{1+x_w}-x_w\int_{\Omega\times\mathbb R^3} v_1^2\sqrt{\mu(v)}\left(\sqrt{\mu}\frac{|v|^2-3}{2}c+(\mathbf I-\mathbf P)f\right){\rm d}x_1{\rm d}v\notag\\
	&\geq\frac{1}{2}|x_w(t)|^2-C|x_w(t)|(\|c\|_2+\|(\mathbf I-\mathbf P)f\|_2).
\end{align}
The last term on the right hand side  of \eqref{x_wf1} can be bounded by
\begin{align}\label{last-x_wf1}
	\int_{\Omega\times\mathbb R^3}&\pa_t[x_w(t)(1+x_w(t))v_1x_1]\sqrt{\mu(v)}\textbf Pf{\rm d}x_1{\rm d}v\notag\\
	&=\int_\Omega\pa_t[x_w(t)(1+x_w(t))x_1]b_1{\rm d}x_1\leq C|v_w(t)|\|b_1\|_2.
\end{align}
Plugging \eqref{x_wf x_1=1}, \eqref{damp-x} and \eqref{last-x_wf1} into \eqref{x_wf1} yields that
\begin{align}\label{d_t(x_wf)}
	\frac{{\rm d}}{{\rm d}t}&\int_{\Omega\times\mathbb R^3}x_w(t)(1+x_w(t))x_1v_1\sqrt{\mu(v)}f{\rm d}x_1{\rm d}v\notag\\
	&\quad+\sqrt{2\pi} x_w\int_{u_1-v_w(t)>0}f(1)\sqrt{\mu(u)}|u_1-v_w(t)|{\rm d}u\notag\\
	&\quad+x_w\int_{v_1-v_w(t)>0}(v_1-v_w(t))^2\sqrt{\mu(v)}\left(\FI-\frac{\sqrt{\mu_{w}}}{\sqrt{\mu}} P_\gamma \right)f(1){\rm d}v+\frac{1}{2}|x_w(t)|^2\notag\\
	&\leq {C|x_w(t)|\left(|v_w(t)|+\|c\|_2+\|(\mathbf I-\mathbf P)f\|_2\right)+C|v_w(t)|\|b_1\|_2}+CP(|x_w|,|v_w|,\|wf\|_\infty).
\end{align}
Then \eqref{xwf-int} follows from \eqref{d_t(x_wv_w)} and \eqref{d_t(x_wf)}.
This completes the proof of Proposition \ref{mix-eng}.
\end{proof}

\section{Macroscopic estimates and $L^2$ decay}\label{mac-l2-sec}
In this section, we derive an $L^2$ estimate for the macroscopic component of $f$. Combined with Proposition \ref{mix-eng} in Section \ref{inter-bd-sec}, this will lead to the exponential decay in time of both the $L^2$ norm of $f$ and the point-wise behavior of the free boundary pair $[x_w,v_w](t)$. The analysis is based on the duality argument developed in \cite{EGKM-13}, with an additional difficulty arising from the non-null integral of $a$.

Our first result is stated as follows.

\begin{lemma}\label{lem es Pf}
Suppose that $f\in L^2$ is a solution to \eqref{PBE}, \eqref{trb-ode}, \eqref{ic f}, \eqref{drbl-f} and \eqref{drbl-f}, and that it satisfies
\eqref{apas} for any $t\in[0,\infty)$. Then there exists a functional $N_f(t)$ such that, for all  $t\geq0$, the following holds
	
	({\romannumeral1}) $|N_f(t)|\leq C\left\|f(t)\right\|_{2}^2$;
	
	({\romannumeral2}) there exists $\la_2>0$ such that
	\begin{align}\label{es Pf}
\frac{{\rm d}}{{\rm d}t}&(x_wv_w)
+\mathcal M\frac{{\rm d}}{{\rm d}t}\int_{\Omega\times\mathbb R^3}x_w(1+x_w)x_1v_1\sqrt{\mu(v)}f{\rm d}x_1{\rm d}v+\frac{{\rm d}}{{\rm d}t}N_f(t)\notag\\
&\qquad
+\frac{1}{4}\left(\kappa+\frac{\mathcal M}{2}\right)|x_w(t)|^2+\la_2\|\FP f\|_{2}^2\notag\\
	&	\leq C|v_{w}(t)|^2+C\|(\FI-\FP)f\|_{2}^2+C\|\nu^{-\frac{1}{2}}\Gamma(f,f)\|_{2}^2\notag\\
		&\qquad+C\left|\left( \FI-\frac{\sqrt{\mu_w}}{\sqrt{\mu}} P_\gamma\right)f\right|_{2,+,1}^2+C|(\mathbf I-P_\gamma)f|_{2,+,0}^2.
	\end{align}
\end{lemma}
Before proving Lemma \ref{lem es Pf}, we first establish the following conservation law of mass.
\begin{lemma}\label{mass-lem}
	Suppose that $F$ is a strong solution to \eqref{BE}, \eqref{ODE}, \eqref{+0-0 data->0} and \eqref{initial F}. Let $f$ be the corresponding solution to \eqref{PBE}, \eqref{trb-ode}, \eqref{ic f}, \eqref{drbl-f} and \eqref{drbl-f}, with initial data satisfying \eqref{to-id-drbcl}, Then it holds that
	\begin{equation}
		\label{lem.mceq}
		\frac{{\rm d}}{{\rm d}t}\int_{-1}^{x_w(t)}{\rm d}X_1 \int_{\R^3}{\rm d}v\,F(t,X_1,v)=0,
	\end{equation}
and
\begin{equation}\label{a-con}
\int_{\Omega}a(t,x_1){\rm d}x_1=	\int_{\Omega\times\mathbb R^3}\sqrt{\mu}f(t,x_1,v){\rm d}v{\rm d}x_1=-\frac{x_w}{1+x_w}.
\end{equation}
\end{lemma}
\begin{proof}
A direct calculation gives
	\begin{align}
		&\frac{{\rm d}}{{\rm d}t}\int_{-1}^{x_w(t)}{\rm d}X_1 \int_{\R^3}{\rm d}v\,F(t,X_1,v)   \notag\\
		& =\int_{-1}^{x_w(t)} {\rm d}X_1 \int_{\R^3}{\rm d}v\pa_t F(t,X_1,v) +\int_{\R^3}{\rm d}v\,F(t,x_w(t),v) \dot{x}_w(t).\label{mcp-td1}
	\end{align}
	For the first term on the right hand side of \eqref{mcp-td1}, applying \eqref{BE}, we obtain
	\begin{align}
		&\int_{-1}^{x_w(t)} {\rm d}X_1 \int_{\R^3}{\rm d}v\pa_t F(t,X_1,v)   \notag\\
		& =\int_{-1}^{x_w(t)} {\rm d}X_1 \int_{\R^3}{\rm d}v [-v_1\pa_{X_1}F(t,X_1,v)+Q(F,F)(t,X_1,v)] \notag\\
		&=\int_{\R^3}v_1F(t,-1,v){\rm d}v -\int_{\R^3}v_1F(t,x_w(t),v){\rm d}v.\notag
	\end{align}
Substituting this into \eqref{mcp-td1} and using $\dot{x}_w(t)=v_w(t)$, we further obtain
	\begin{align}
		&\frac{{\rm d}}{{\rm d}t}\int_{-1}^{x_w(t)}{\rm d}X_1 \int_{\R^3}{\rm d}v\,F(t,X_1,v)
		 =\int_{\R^3}v_1F(t,-1,v){\rm d}v-\int_{\R^3} [v_1-v_w(t)]F(t,x_w(t),v){\rm d}v.\notag
	\end{align}
Next, the diffusive reflection boundary conditions \eqref{drbcl} and \eqref{drbcr} imply
	\begin{equation}
		\int_{\R^3}v_1F(t,-1,v){\rm d}v=0,\
		\int_{\R^3} [v_1-v_w(t)]F(t,x_w(t),v){\rm d}v=0.\notag
	\end{equation}
Therefore, \eqref{lem.mceq} follows. As a consequence of \eqref{lem.mceq} and \eqref{to-id-drbcl}, we have
\begin{equation*}
\int_{-1}^{x_w(t)}\int_{\mathbb R^3}F(t,X_1,v){\rm d}v{\rm d}X_1=1=\int_{\Omega\times\mathbb R^3}(1+x_w(t))(\mu+\sqrt{\mu}f){\rm d}v{\rm d}x_1.
\end{equation*}
which yields \eqref{a-con}.
This completes the proof of Lemma \ref{mass-lem}.
\end{proof}

With \eqref{a-con} in hand, we now turn to complete the proof of Lemma \ref{lem es Pf}.

\begin{proof}[The proof of Lemma \ref{lem es Pf}]
Let $\Psi=\Psi(t,x_1,v)\in C^\infty((0,+\infty)\times\Om\times\R^3)$ be a test function, then from \eqref{PBE}, it follows
\begin{align}\label{weak-f-eq}
		\frac{{\rm d}}{{\rm d}t}&((1+x_w(t))f,\Psi)-(f,\pa_t[(1+x_w(t))\Psi])
+((1+x_w(t))Lf,\Psi)-((v_1-v_w(t)x_1)f,\partial_{x_1}\Psi)\notag\\
		&+(v_w(t)f,\Psi)+\int_\gamma f\Psi{\rm d}\tilde\gamma=((1+x_w(t))\Gamma(f,f),\Psi).
	\end{align}
By the decomposition $f=\FP f+(\FI-\FP)f$, we thus have
\begin{equation}\label{Psi_i}
	-\int_{\Omega\times\mathbb R^3}(v_1-v_w(t)x_1)\FP f\partial_{x_1}\Psi{\rm d}x_1{\rm d}v
=-\frac{{\rm d}}{{\rm d}t}((1+x_w(t))f,\Psi)+S_1+S_2+S_3+S_4+S_5,
\end{equation}
with
\begin{align*}
S_1&=\int_{\Omega\times\mathbb R^3}(v_1-v_w(t)x_1)(\textbf I-\textbf P) f\partial_{x_1}\Psi{\rm d}x_1{\rm d}v,\
S_2=-\int_{\Omega\times\mathbb R^3}(1+x_w(t))Lf\Psi{\rm d}x_1{\rm d}v,\\
S_3&=\int_{\Omega\times\mathbb R^3}(1+x_w(t))f\partial_t \Psi{\rm d}x_1{\rm d}v,\
S_4=\int_{\Omega\times\mathbb R^3}(1+x_w(t))\Gamma(f,f)\Psi{\rm d}x_1{\rm d}v,
\end{align*}
and
\begin{align*}
S_5&=-\int_\gamma f\Psi{\rm d}\tilde\gamma,
\end{align*}
where we have used a cancellation:
\begin{align}
\int_{\Omega\times\mathbb R^3}f\partial_tx_w(t)\Psi{\rm d}x_1{\rm d}v-\int_{\Omega\times\mathbb R^3}v_w(t)f\Psi{\rm d}x_1{\rm d}v=0.\notag
\end{align}

Recalling the definition \eqref{mac-p}, we have
\begin{equation*}
	\textbf P f=[a(t,x_1)+\textbf b(t,x_1)\cdot v+c(t,x_1)\frac{|v|^2-3}{2}]\sqrt{\mu(v)},\ \Fb=(b_1,b_2,b_3).
\end{equation*}
We now proceed to estimate $a$, $\mathbf{b}$, and $c$ separately.

\noindent\textbf {Estimate for $a$.}
We choose the test function $\Psi$ as
\begin{equation}\label{Psi_a}
	\Psi=\Psi_a=(|v|^2-10)v_1\partial_{x_1}\phi_a\sqrt{\mu},
\end{equation}
where $\phi_a$ solves
\begin{equation}\label{phi-a-ep}
	-\partial_{x_1}^2\phi_a=a+\frac{x_w}{1+x_w},\ \pa_{x_1}\phi_a(t,0)=\pa_{x_1}\phi_a(t,1)=0,\ \int_0^1\phi_a{\rm d}x_1=0.
\end{equation}
By standard elliptic estimate with $\int_0^1\left(a+\frac{x_w}{1+x_w}\right){\rm d}x_1=0$, it follows that
\begin{equation}\label{es phi_a}
	\left\|\phi_a(t)\right\|_{H_{x_1}^2}+|\phi_a(t,1)|+|\phi_a(t,0)|\lesssim \left\|a(t)\right\|_{2}+|x_w(t)|.
\end{equation}
Substituting \eqref{Psi_a} into \eqref{Psi_i} and using the oddness in $v$ together with the {\it a priori} assumption \eqref{apas}, we obtain
\begin{align}\label{Psi_a,left}
	\int_{\Omega\times\mathbb R^3}&(v_1-v_w(t)x_1)\FP f\partial_{x_1}\Psi_a{\rm d}x_1{\rm d}v\notag\\
	&=-5\int_{\Omega}a(s,x_1)\partial_{x_1}^2\phi_a{\rm d}x_1+5\int_{\Omega}b_1(s,x_1)v_w(t)x_1\partial_{x_1}^2\phi_a{\rm d}x_1\notag\\
	&=5\|a\|_{2}^2-\frac{|x_w(t)|^2}{1+x_w}-5\int_{\Omega}\left(a+\frac{x_w}{1+x_w}\right)b_1v_w(t)x_1{\rm d}x_1\notag\\
	&\geq \frac{5}{2}\|a\|_{2}^2-C\vps\| b_1\|_{2}^2-C|x_w(t)|^2.
\end{align}
Next, using \eqref{es phi_a} and Cauchy-Schwarz's inequality, $S_1$ and $S_2$ can be bounded by
\begin{align}\label{Psi_a,23}
|S_1|+|S_2|&\leq C(\left\|a\right\|_{2}+|x_w(t)|)\left\|(\textbf I-\textbf P)f\right\|_{2}
\leq \frac{5}{64}\left\|a\right\|_{2}^2+C\left\|(\textbf I-\textbf P)f\right\|_{2}^2+C|x_w(t)|^2.
	\end{align}
Similarly, for $S_4$, it follows
\begin{equation}\label{Psi_a,5}
|S_4|\leq \frac{5}{64}\left\|a\right\|_{2}^2+C\left\|\nu^{-\frac{1}{2}}\Gamma(f,f)\right\|_{2}^2+C|x_w(t)|^2.	
\end{equation}
For the boundary term $S_5$, using \eqref{phi-a-ep} we have
\begin{equation}\label{s5-a}
	S_5=-\int_{\mathbb R^3}f\Psi_a(1,v) (v_1-v_w(t)){\rm d}v+\int_{\mathbb R^3}f\Psi_a(0,v) v_1{\rm d}v=0.
\end{equation}
For $S_3$, applying \eqref{mac-p} and using the oddness in $v$, we obtain
\begin{equation}\label{s3a}
	|S_3|\leq C[\left\| b_1\right\|_{2}+\left\|(\textbf I-\textbf P)f\right\|_{2}]\left\|\partial_{x_1}\partial_t\phi_a\right\|_{2}.
\end{equation}
To estimate $\|\partial_{x_1}\partial_t \phi_a\|_{2}^2$, we take $\Psi(x_1,v) = \phi(x_1)\sqrt{\mu(v)} \in H_{x_1}^1$ in \eqref{weak-f-eq}, integrate the resulting identity over the interval $[t, t+\eps]$ for small $\eps \neq 0$, and then divide by $\eps$ to obtain
\begin{align}
	\frac{1}{\eps}&\left\{\int_{\Omega}[(1+x_w)a](t+\eps)\phi{\rm d}x_1-\int_{\Omega}[(1+x_{w})a](t)\phi{\rm d}x_1\right\}\notag\\
	=&\frac{1}{\eps}\int_t^{t+\eps}\int_{\Omega}(-v_w(s)x_1a+b_1)\partial_{x_1}\phi{\rm d}x_1{\rm d}s+\frac{1}{\eps}\int_t^{t+\eps}\int_\gamma f\sqrt{\mu(v)}\phi{\rm d}\tilde\gamma{\rm d}s\notag\\
	&+\frac{1}{\eps}\int_t^{t+\eps}\int_{\Omega\times\mathbb R^3}(v_1-v_w(s)x_1)\textbf (\textbf I-\textbf P)f\sqrt{\mu(v)}\partial_{x_1}\phi{\rm d}x_1{\rm d}v{\rm d}s.\label{dual-pta}
\end{align}
Using the boundary decomposition
\begin{equation}
f(t,0,v)=P_\gamma f+(\FI-P_\gamma)f \textbf1_{\gamma^0_+},\	f(t,1,v)=P_\gamma f+(\FI-P_\gamma)f \textbf1_{\gamma^1_+}+r\textbf 1_{\gamma^1_-},\label{bd-dec}
\end{equation}
one has
\begin{align}
\int_\gamma f\sqrt{\mu(v)}\phi{\rm d}\tilde\gamma=&\int_{\R^3}f(1)(v_1-v_w(s))\sqrt{\mu(v)}\phi(1){\rm d}v-\int_{\R^3}f(0)v_1\sqrt{\mu(v)}\phi(0){\rm d}v\notag\\
=&-v_w\int_{\R^3}P_\gamma f\sqrt{\mu(v)}\phi(1){\rm d}v+\int_{v_1-v_w(s)>0}(\FI-P_\ga)f(1)(v_1-v_w(s))\sqrt{\mu(v)}\phi(1){\rm d}v\notag\\
&+\int_{v_1-v_w(s)<0}r(v_1-v_w(s))\sqrt{\mu(v)}\phi(1){\rm d}v\notag\\
&-\int_{v_1<0}(\mathbf I-P_\gamma)f(0)v_1\sqrt{\mu(v)}\phi(0){\rm d}v.\notag
\end{align}
Letting $\eps\rightarrow0$ and applying Cauchy-Schwarz's inequality, we get
\begin{align}
	\Big|\int_0^1(1+x_w)\partial_ta\phi{\rm d}x_1\Big|
	\leq& C\left\|\phi\right\|_{H_{x_1}^1}\Bigg(|v_w(t)|\left\|a\right\|_{2}+\left\|b_1\right\|_{2}+\left\|(\textbf I-\textbf P)f\right\|_{2}+\left|\left(\FI-\frac{\sqrt{\mu_w}}{\sqrt{\mu}} P_\gamma \right)f\right|_{2,+,1}\notag\\
	&\qquad\qquad\quad+|(\mathbf I-P_\gamma)f|_{2,+,0}+|v_{w}(t)|\Bigg),\label{a-dual}
\end{align}
where we have also used the trace inequality
$$
|[\phi(0),\phi(1)]|\lesssim\|\phi\|_{H_{x_1}^1}.
$$
Consequently, \eqref{a-dual} and \eqref{apas} imply
\begin{align}
\left\|\partial_ta\right\|_{( H_{x_1}^1)^*}\leq& C\Bigg(|v_w(t)|\left\|a\right\|_{2}+\left\|b_1\right\|_{2}+\left\|(\textbf I-\textbf P)f\right\|_{2}+\left|\left(\FI-\frac{\sqrt{\mu_w}}{\sqrt{\mu}} P_\gamma \right)f\right|_{2,+,1}\notag\\
	&\qquad+|(\mathbf I-P_\gamma)f|_{2,+,0}+|v_{w}(t)|\Bigg),\label{a-dual-2}
\end{align}
where $(H_{x_1}^1)^*$ denotes the dual of the set of functions in $H_{x_1}^1$.

Next, note that $\Phi=\partial_t\phi_a$ satisfies
\begin{equation*}
	-\partial_{x_1}^2\Phi(t,x_1)=\partial_ta(t,x_1)+\frac{v_w}{(1+x_w)^2},\ \pa_{x_1}\Phi(t,0)=\pa_{x_1}\Phi(t,1)=0,\ \int_{\Om}\Phi(t,x_1){\rm d}x_1=0.
\end{equation*}
Thus, it follows from \eqref{a-dual-2} that
\begin{align*}
\left\|\partial_t\partial_{x_1}\phi_a\right\|_{2}\leq& C\left\|\Phi\right\|_{H_{x_1}^1}\leq C\left\|\partial_ta+\frac{v_w}{(1+x_w)^2}\right\|_{( H_{x_1}^1)^*}\\
	\leq& C\Bigg(|v_w(t)|\left\|a\right\|_{2}+\left\|b_1\right\|_{2}+\left\|(\textbf I-\textbf P)f\right\|_{2}+\left|\left(\FI-\frac{\sqrt{\mu_w}}{\sqrt{\mu}} P_\gamma \right)f\right|_{2,+,1}\\
	&\qquad+|(\mathbf I-P_\gamma)f|_{2,+,0}+|v_{w}(t)|\Bigg).
\end{align*}
Hence, we get by plugging the above into \eqref{s3a} that
\begin{align}\label{Psi_a,4}
|S_3|
\leq& C\Bigg(|v_w(t)|^2\left\|a\right\|_{2}^2+\left\|b_1\right\|_{2}^2+\left\|(\textbf I-\textbf P)f\right\|_{2}^2+\left|\left(\FI-\frac{\sqrt{\mu_w}}{\sqrt{\mu}} P_\gamma \right)f\right|_{2,+,1}^2\notag\\
&\qquad+|(\mathbf I-P_\gamma)|_{2,+,0}^2+|v_{w}(t)|^2\Bigg).
\end{align}
Finally,
inserting \eqref{Psi_a,left}, \eqref{Psi_a,23}, \eqref{Psi_a,5}, \eqref{s5-a}, and \eqref{Psi_a,4} into \eqref{Psi_i}, one sees that there exist constants $C_a>0$ and $C>0$ such that
\begin{align}\label{es a}
\frac{{\rm d}}{{\rm d}t}N_f^{(a)}(t)&+\frac{5}{4}\|a(t)\|_{2}^2\notag\\
\leq&C_a(|x_w(t)|^2+\|b_1(t)\|_{2}^2+\|c(t)\|_{2}^2)+C\left(|v_w(t)|^2+\left\|(\textbf I- \textbf P)f\right\|_{2}^2+\left\|\nu^{-\frac{1}{2}}\Gamma(f,f)\right\|_{2}^2\right)\notag\\
&+C\left(\left|\left(\FI-\frac{\sqrt{\mu_w}}{\sqrt{\mu}} P_\gamma \right)f\right|_{2,+,1}^2+|(\mathbf I-P_\gamma)f|_{2,+,0}^2
\right),
\end{align}
where
\[N_f^{(a)}(t)=-((1+x_w(t))f,\Psi_a)=-((1+x_w(t))f,(|v|^2-10)v_1\partial_{x_1}\phi_a\sqrt{\mu}).
\]

\noindent\textbf{Estimate for $b$.} Note that $\mathbf{b} = [b_1, b_2, b_3]$, and the estimate for $b_1$ differs slightly from those for $b_2$ and $b_3$ due to the presence of the relative velocity $v_1 - v_w$. To handle this, we choose the test function
\begin{equation}
	\Psi=\Psi_{b_1}=(v_1^2-1)\partial_{x_1}\phi_{b_1}\sqrt{\mu},\notag
\end{equation}
where $\phi_{b_1}$ satisfies
\begin{equation*}
	-\partial_{x_1}^2\phi_{b_1}=b_1,\ \phi_{b_1}(t,1)=\phi_{b_1}(t,0)=0.
\end{equation*}
By the standard  elliptic estimate, it follows that
\begin{equation}\label{es phi_b1}
	\left\|\phi_{b_1}(t)\right\|_{H_{x_1}^2}+|\pa_{x_1}\phi_{b_1}(t,0)|+|\pa_{x_1}\phi_{b_1}(t,1)|\lesssim \left\|b_1(t)\right\|_{2}.
\end{equation}
We now compute the left-hand side of \eqref{Psi_i} with $\Phi$ replaced by $\Psi_{b_1}$ as
\begin{align}
	-\int_{\Omega\times\mathbb R^3}&\textbf Pf(v_1-v_w(t)x_1)\partial_{x_1}\Psi_{b_1}{\rm d}x_1{\rm d}v\notag\\
	=&-\int_{\Omega}\textbf b(t,x_1)\cdot\Big(\int_{\mathbb R^3}v(v_1^2-1)(v_1-v_w(t)x_1)\mu(v){\rm d}v\Big)\partial_{x_1}^2\phi_{b_1}{\rm d}x_1\notag\\
	&-\int_{\Omega}c(t,x_1)\Big(\int_{\mathbb R^3}\frac{|v|^2-3}{2}(v_1^2-1)(v_1-v_w(t)x_1)\mu(v){\rm d}v\Big)\partial_{x_1}^2\phi_{b_1}{\rm d}x_1\notag\\
\geq& \|b_1\|_{2}^2-C\left\| c\right\|_{2}^2.\label{b1-dis}
\end{align}
Next, we estimate the right hand side of \eqref{Psi_i} term by term.
By \eqref{es phi_b1} and Cauchy-Schwarz's inequality, $S_1$ and $S_2$ can be dominated by
\begin{align}\label{Psi_b1,23}
	|S_1|+|S_2|&\leq C\left\|b_1\right\|_{2}\left\|(\textbf I-\textbf P)f\right\|_{2}
	\leq \frac{1}{8}\left\|b_1\right\|_{2}^2+C \left\|(\textbf I-\textbf P)f\right\|_{2}^2.	
\end{align}
For $S_4$, using Cauchy-Schwarz's inequality again, we obtain
\begin{equation}\label{Psi_b1,5}
	|S_4|\leq \frac{1}{8}\left\|b_1\right\|_{2}^2+C\left\|\nu^{-\frac{1}{2}}\Gamma(f,f)\right\|_{2}^2.	
\end{equation}
For the boundary term $S_5$, applying \eqref{bd-dec} once again together with the estimate \eqref{es phi_b1}, we obtain
\begin{align}\label{s5-b1}
	S_5=&\int_{\mathbb R^3}f\Psi_{b_1}(t,1,v) (v_1-v_w(t)){\rm d}v-\int_{\mathbb R^3}f\Psi_{b_1}(t,0,v)v_1{\rm d}v\notag\\
=&\int_{\mathbb R^3}(v_1^2-1)\sqrt{\mu(v)}\partial_{x_1}\phi_{b_1}(t,1)((\FI-P_\gamma)f \textbf1_{\gamma^1_+}+r\textbf 1_{\gamma^1_-})(v_1-v_w(t)){\rm d}v-\int_{\mathbb R^3}f\Psi_{b_1}(t,0,v)v_1{\rm d}v\notag\\
\leq& \frac{1}{8}\left\|b_1\right\|_{2}^2+C\left(|v_w(t)|^2+\left|\left(\FI-\frac{\sqrt{\mu_w}}{\sqrt{\mu}} P_\gamma \right)f\right|_{2,+,1}^2+|(\mathbf I-P_\gamma)f|_{2,+,0}^2\right),
\end{align}
where, as in the derivation of \eqref{Psi_a,4}, the following estimates are used:
\begin{align*}
	\Big|\int_{\mathbb R^3}f\Psi_{b_1}(t,0,v)v_1{\rm d}v\Big|
	\leq& C\|b_1\|_{2}|(\mathbf I-P_\gamma)f|_{2,+,0}
	\leq \frac{1}{32}\|b_1\|_{2}^2+C|(\mathbf I-P_\gamma)f|_{2,+,0}^2,
\end{align*}
\begin{align*}
\int_{\mathbb R^3}&(v_1^2-1)\sqrt{\mu(v)}\partial_{x_1}\phi_{b_1} (v_1-v_w(t))P_\gamma f
	{\rm d}v=0,
\end{align*}
\begin{align*}
\int_{\mathbb R^3}&(v_1^2-1)\sqrt{\mu(v)}\partial_{x_1}\phi_{b_1}(\FI-P_\gamma)f \textbf1_{\gamma^1_+}
	 (v_1-v_w(t)){\rm d}v\\
	=&\int_{\mathbb R^3}(v_1^2-1)\sqrt{\mu(v)}\partial_{x_1}\phi_{b_1}\left(\FI-\frac{\sqrt{\mu_w}}{\sqrt{\mu}}P_\gamma\right)f \textbf1_{\gamma^1_+}
	 (v_1-v_w(t)){\rm d}v\\
	&-\int_{\mathbb R^3}(v_1^2-1)\sqrt{\mu(v)}\partial_{x_1}\phi_{b_1}\left(\FI-\frac{\sqrt{\mu_w}}{\sqrt{\mu}} \right)P_\gamma f \textbf1_{\gamma^1_+}
	 (v_1-v_w(t)){\rm d}v\\
	\leq& \frac{1}{32}\left\|b_1\right\|_{2}^2
+C\left|\left(\FI-\frac{\sqrt{\mu_w}}{\sqrt{\mu}}  P_\gamma\right) f\right|_{2,+,1}^2+C|v_w(t)|^2,
\end{align*}
and the contribution of $r$
\begin{align*}
\int_{\mathbb R^3}&v_1(|v|^2-1)\sqrt{\mu(v)}\partial_{x_1}\phi_{b_1}r\textbf 1_{\gamma^1_-}
	 (v_1-v_w(t)){\rm d}v\leq C|v_w(t)|\left\|b_1\right\|_{2}\leq \frac{1}{32}\left\|b_1\right\|_{2}^2
+C|v_w(t)|^2.
\end{align*}
For $S_3$, using the oddness in $v$, one has
\begin{align*}
	S_3=&\int_{\Omega\times\mathbb R^3}(1+x_w(t))(v_1^2-1)\sqrt{\mu}\{\FP f+(\textbf I-\textbf P)f\}\partial_t\partial_{x_1}\phi_{b_1}{\rm d}x_1{\rm d}v\notag\\
\leq& C\left(\|c\|_{2}+\left\|(\textbf I-\textbf P)f\right\|_{2}\right)\left\|\partial_{x_1}\partial_t\phi_{b_1}\right\|_{2}.
\end{align*}
It remains to estimate $\left\|\partial_{x_1}\partial_t\phi_{b_1}\right\|_{2}$. To this end, following an argument similar to the derivation of \eqref{dual-pta}, we apply \eqref{weak-f-eq} with $\Psi(x_1,v) = \phi(x_1)v_1\sqrt{\mu(v)} \in H_{x_1}^1$ to obtain
\begin{align*}
	\frac{1}{\eps}&\left\{\int_{\Omega}[(1+x_w)b_1](t+\eps)\phi{\rm d}x_1-\int_{\Omega}[(1+x_{w})b_1](t)\phi{\rm d}x_1\right\}\\
	=&\frac{1}{\eps}\int_t^{t+\eps}\int_{\Omega}(a-v_w(s)x_1b_1+2c)\partial_{x_1}\phi{\rm d}x_1{\rm d}s+\frac{1}{\eps}\int_t^{t+\eps}\int_{\Omega\times\mathbb R^3}\textbf (\textbf I-\textbf P)fv^2_1\sqrt{\mu(v)}\partial_{x_1}\phi{\rm d}x_1{\rm d}v{\rm d}s\\
	&+\frac{1}{\eps}\int_t^{t+\eps}\int_\gamma fv_1\sqrt{\mu(v)}\phi{\rm d}\tilde\gamma{\rm d}s.
\end{align*}
Letting $\eps\rightarrow0$ gives
\begin{align*}
	\Big|\int_0^1&(1+x_w(t))\partial_tb_1\phi{\rm d}x_1\Big|\\
	\leq& C\left\|\phi\right\|_{H_{x_1}^1}\Bigg(\left\|a\right\|_{2}+|v_w(t)|\left\|b_1\right\|_{2}+\left\|c\right\|_{2}+\left\|(\textbf I-\textbf P)f\right\|_{2}+\left|\left(\FI-\frac{\sqrt{\mu_w}}{\sqrt{\mu}} P_\gamma \right)f\right|_{2,+,1}\\
	&\qquad\qquad\quad+|(\mathbf I-P_\gamma)f|_{2,+,0}+|v_{w}(t)|\Bigg).
\end{align*}
Therefore,
\begin{align*}
	\left\|\partial_tb_1\right\|_{(H_{x_1}^1)^*}\leq& C\Bigg(\left\|a\right\|_{2}+|v_w(t)|\left\|b_1\right\|_{2}+\left\|c\right\|_{2}+\left\|(\textbf I-\textbf P)f\right\|_{2}+\left|\left(\FI-\frac{\sqrt{\mu_w}}{\sqrt{\mu}} P_\gamma \right)f\right|_{2,+,1}\\
	&\qquad\qquad\quad+|(\mathbf I-P_\gamma)f|_{2,+,0}+|v_{w}(t)|\Bigg).
\end{align*}
Combining this with
\begin{align*} \left\|\partial_t\partial_{x_1}\phi_{b_1}\right\|_{2}\leq C\left\|\partial_tb_1\right\|_{(H_{x_1}^1)^*},
\end{align*}
we obtain that $S_3$ can be controlled by
\begin{align}\label{Psi_b1,4}
	|S_3|\leq&\eta\left\|a\right\|_{2}^2+C_\eta\left(\left\|c\right\|_{2}^2+\left\|(\textbf I-\textbf P)f\right\|_{2}^2\right) +C|v_w|\left\|b_1\right\|_{2}^2+C\left|\left(\FI-\frac{\sqrt{\mu_w}}{\sqrt{\mu}} P_\gamma \right)f\right|_{2,+,1}^2\notag\\
	&+C|(\mathbf I-P_\gamma)f|_{2,+,0}^2+C|v_{w}(t)|^2.
\end{align}
Plugging \eqref{b1-dis}, \eqref{Psi_b1,23}, \eqref{Psi_b1,5}, \eqref{s5-b1}, and \eqref{Psi_b1,4} into \eqref{Psi_i}, we have
\begin{align}\label{es b_1}
	\frac{{\rm d}}{{\rm d}t}N_f^{(b_1)}(t)+\frac{1}{4}\left\|b_1(t)\right\|_{2}^2\leq& \eta\left\|a\right\|_{2}^2+C_\eta(\left\|c\right\|_{2}^2+\left\|(\textbf I-\textbf P)f\right\|_{2}^2)+C\left\|\nu^{-\frac{1}{2}}\Gamma(f,f)\right\|_{2}^2\notag\\&+C\left|\left(\FI-\frac{\sqrt{\mu_w}}{\sqrt{\mu}} P_\gamma \right)f\right|^2_{2,+,1}+C|(\mathbf I-P_\gamma)f|^2_{2,+,0}+C|v_{w}(t)|^2,
\end{align}
where
$$
N_f^{(b_1)}(t)=((1+x_w(t))f,\Psi_{b_1})=((1+x_w(t))f,(v_1^2-1)\partial_{x_1}\phi_{b_1}\sqrt{\mu}).
$$
To estimate $b_i$ with $i=2,3$, we choose a test function
\begin{equation}
	\Psi=\Psi_{b_i}=(v_1^2-1)v_1v_i\partial_{x_1}\phi_{b_i}\sqrt{\mu}.\notag
\end{equation}
Here, $\phi_{b_i}$ is a solution to
\begin{equation*}
	-\partial_{x_1}^2\phi_{b_i}=b_i,\ \phi_{b_i}(t,1)=\phi_{b_i}(t,0)=0.
\end{equation*}
From the standard  elliptic estimate, we have
\begin{equation}\label{es phi_bi}
	\left\|\phi_{b_i}(t)\right\|_{H_{x_1}^2}+|\pa_{x_1}\phi_{b_i}(t,1)|+|\pa_{x_1}\phi_{b_i}(t,0)|\lesssim \left\|b_i(t)\right\|_{2}.
\end{equation}
By the oddness in $v$, we compute the left-hand side of \eqref{Psi_i} with $\Phi$ replaced by $\Psi_{b_i}$ as follows
\begin{align*}
	-\int_{\Omega\times\mathbb R^3}&\textbf Pf(v_1-v_w(t)x_1)\partial_{x_1}\Psi{\rm d}x_1{\rm d}v=\|b_i\|_2^2,\ i=2,3.
\end{align*}
Next, we estimate the right hand side of \eqref{Psi_i} individually.
By \eqref{es phi_bi} and Cauchy-Schwarz's inequality, $S_1$ and $S_2$ can be bounded as
\begin{align}\label{Psi_bi,23}
	|S_1|+|S_2|&\leq C\left\|b_i\right\|_{2}\left\|(\textbf I-\textbf P)f\right\|_{2}\leq \frac{1}{8}\left\|b_i\right\|_{2}^2+C \left\|(\textbf I-\textbf P)f\right\|_{2}^2.	
\end{align}
For $S_4$, it follows
\begin{equation}\label{Psi_bi,5}
	|S_4|\leq \frac{1}{8}\left\|b_i\right\|_{2}^2+C\left\|\nu^{-\frac{1}{2}}\Gamma(f,f)\right\|_{2}^2.	
\end{equation}
For $S_5$ and $S_3$, by arguments similar to those used for $b_1$ in \eqref{s5-b1} and \eqref{Psi_b1,4}, we obtain
\begin{align}\label{Psi_bi 6}
	|S_5|	\leq& \frac{1}{8}\left\|b_i\right\|_{2}^2+C|v_w(t)|^2+C\left(\left|\left(\FI-\frac{\sqrt{\mu_w}}{\sqrt{\mu}} P_\gamma \right)f\right|_{2,+,1}^2+|(\mathbf I-P_\gamma)f|_{2,+,0}^2\right),
\end{align}
and
\begin{align}\label{Psi_bi,4}
	|S_3|\leq& C\left\|(\textbf I-\textbf P)f\right\|_{2}\left\|\partial_{x_1}\partial_t\phi_{b_i}\right\|_{2}\notag\\
	\leq& C\Bigg(|v_w(t)|^2\left\|b_i\right\|_{2}^2+\left\|(\textbf I-\textbf P)f\right\|_{2}^2+\left|\left(\FI-\frac{\sqrt{\mu_w}}{\sqrt{\mu}} P_\gamma \right)f\right|_{2,+,1}^2\notag\\
	&+|(\mathbf I-P_\gamma)f|_{2,+,0}^2+|v_{w}(t)|^2\Bigg).
\end{align}
Substituting \eqref{Psi_bi,23}, \eqref{Psi_bi,5}, \eqref{Psi_bi 6}, and \eqref{Psi_bi,4} into \eqref{Psi_i}, we conclude that
\begin{align}\label{es b_i}
	\frac{{\rm d}}{{\rm d}t}N_f^{(b_i)}(t)+\frac{1}{2}\left\|b_i(t)\right\|_{2}^2
	\leq &C\left\|(\textbf I-\textbf P)f\right\|_{2}^2+C\left\|\nu^{-\frac{1}{2}}\Gamma(f,f)\right\|_{2}^2\notag\\&+C\left(\left|\left(\FI-\frac{\sqrt{\mu_w}}{\sqrt{\mu}} P_\gamma \right)f\right|_{2,+,1}^2+|(\mathbf I-P_\gamma)f|_{2,+,0}^2+|v_w(t)|^2\right),\ i=2,3,
\end{align}
with
$$
N_f^{(b_i)}(t)=((1+x_w(t))f,\Psi_{b_i})=((1+x_w(t))f,(v_1^2-1)v_1v_i\partial_{x_1}\phi_{b_i}\sqrt{\mu}).
$$

\noindent\textbf {Estimate for $c$}. We choose a test function
\begin{equation}
	\Psi=\Psi_c=(|v|^2-5)v_1\partial_{x_1}\phi_c\sqrt{\mu}.\notag
\end{equation}
Here, $\phi_c$ solves
\begin{equation*}
	-\partial_{x_1}^2\phi_c=c,\ \phi_c(t,1)=\phi_c(t,0)=0.
\end{equation*}
With this choice,
by performing calculations similar to those used in deriving the estimates for $a$ and $\mathbf{b}$ above, we obtain
\begin{align}\label{es c}
\frac{{\rm d}}{{\rm d}t}N_f^{(c)}(t)+\frac{1}{2}\left\|c(t)\right\|_{2}^2\leq& \eta_1\left\|b_1(t)\right\|_{2}^2
+C_{\eta_1}\left\|(\textbf I-\textbf P)f\right\|_{2}^2+C\left|\left(\FI-\frac{\sqrt{\mu_w}}{\sqrt{\mu}} P_\gamma \right)f\right|_{2,+,1}^2\notag\\ &+C\left(|(\mathbf I-P_\gamma)f|_{2,+,0}^2+\left\|\nu^{-\frac{1}{2}}\Gamma(f,f)\right\|_{2}^2+|v_{w}(t)|^2\right),
\end{align}
where
$$
N_f^{(c)}(t)=((1+x_w(t))f,\Psi_{c})=((1+x_w(t))f,(|v|^2-5)v_1\partial_{x_1}\phi_c\sqrt{\mu}).
$$
Choosing constants $1<C_1, C_a\ll C_2\ll C_3,$
we get from the summation of  \eqref{es a}, $C_2\times\eqref{es b_1}$, \eqref{es b_i}, and $C_3\times\eqref{es c}$ that
\begin{align}\label{d_t(N^a+N^b)}
\frac{{\rm d}}{{\rm d}t}&\left(N_f^{(a)}(t)+C_2N_f^{(b_1)}(t)+N_f^{(b_2)}(t)+N_f^{(b_3)}(t)+C_3N_f^{(c)}(t)\right)\notag\\
&+\left(\|a(t)\|_2^2+\frac{C_2}{8}\|b_1(t)\|_2^2+\frac{1}{2}(\|b_2(t)\|_2^2+\|b_3(t)\|_2^2)+\frac{C_3}{4}\|c\|_2^2\right)\notag\\
	\leq&C_a|x_w(t)|^2+ C\left\|(\textbf I-\textbf P)f\right\|_{2}^2+C\left|\left(\FI-\frac{\sqrt{\mu_w}}{\sqrt{\mu}} P_\gamma \right)f\right|_{2,+,1}^2\notag\\ &+C\left(|(\mathbf I-P_\gamma)f|_{2,+,0}^2+\left\|\nu^{-\frac{1}{2}}\Gamma(f,f)\right\|_{2}^2+|v_{w}(t)|^2\right).
\end{align}
Next we choose $\frac{16C_1}{C_2}<\eta_2<\frac{2\kappa+\mathcal M}{8C_a}$. Then the linear combination of $\eta_2\times\eqref{d_t(N^a+N^b)}+\eqref{xwf-int}$ gives that
\begin{align}
	\frac{{\rm d}}{{\rm d}t}&(x_wv_w)
	+\mathcal M\frac{{\rm d}}{{\rm d}t}\int_{\Omega\times\mathbb R^3}x_w(1+x_w)x_1v_1\sqrt{\mu(v)}f{\rm d}x_1{\rm d}v+\frac{{\rm d}}{{\rm d}t}N_f(t)\notag\\
	&
	\qquad+\frac{1}{4}\left(\kappa+\frac{\mathcal M}{2}\right)|x_w(t)|^2+\lambda_2\|[a,\mathbf b,c](t)\|_2^2\notag\\
	&\leq C\left\|(\textbf I-\textbf P)f\right\|_{2}^2
+C\left|\left(\FI-\frac{\sqrt{\mu_w}}{\sqrt{\mu}} P_\gamma \right)f\right|_{2,+,1}^2\notag\\ &\qquad+C\left(|(\mathbf I-P_\gamma)f|_{2,+,0}^2+\left\|\nu^{-\frac{1}{2}}\Gamma(f,f)\right\|_{2}^2+|v_{w}(t)|^2\right),\notag
\end{align}
where
\begin{align}
N_f(t)=\eta_2\left[N_f^{(a)}(t)+C_2N_f^{(b_1)}(t)+N_f^{(b_2)}(t)+N_f^{(b_3)}(t)+C_3N_f^{(c)}(t)\right].\notag
\end{align}
Thus, the desired estimate \eqref{es Pf} follows. This completes the proof of Lemma \ref{lem es Pf}.
\end{proof}

With Proposition \ref{mix-eng} and Lemma \ref{lem es Pf} in hand, we can now derive the following decay estimate for the solution $[f, x_w, v_w]$ to the system \eqref{PBE-se2}, \eqref{trb-ode-se2},
\eqref{icf-se2}, and \eqref{bc-f-se2}.

\begin{proposition}\label{decay-l2}
Let the conditions listed in Theorem \ref{main result} be satisfied. Assume $[f,x_w,v_w]$ is a strong solution of \eqref{PBE-se2}, \eqref{trb-ode-se2},
\eqref{icf-se2}, and \eqref{bc-f-se2}. Moreover, suppose that
\begin{align}
\|wf(t)\|_{\infty}+|v_w(t)|+|x_w(t)|\leq 2\vps_0,\notag
\end{align}
for $t\in[0,\infty)$, then there exists a constant $\la_0>0$ and an energy functional $\CE_1$ which satisfies $$\CE_{1}(t)\thicksim \| f(t)\|_{2}^2+|x_w(t)|^2+|v_w(t)|^2$$
 such that
\begin{align}\label{ex-decay}
\CE_1(t)&+\int_{0}^te^{-\la_0(t-s)}\left\{\left|\left(\FI-\frac{\sqrt{\mu_w}}{\sqrt{\mu}} P_\gamma \right)f\right|_{2,+,1}^2+|(\mathbf I-P_\gamma)f|_{2,+,0}^2\right\}{\rm d}s\notag\\
\leq& e^{-\la_0t}\CE_1(0)+C\int_0^te^{-\la_0(t-s)}\left\{\left\| \nu^{-\frac{1}{2}}\Gamma(f,f)\right\|_{2}^2+P(|x_w(s)|,|v_w(s)|,\|wf(s)\|_\infty)\right\}{\rm d}s,
\end{align}
for $t\in[0,\infty)$.
\end{proposition}
\begin{proof}
If $\kappa>0$, a linear combination of \eqref{dis-xv-mif} and \eqref{es Pf} implies that there exists a constant $\la_0>0$ such that
\begin{align}
\frac{{\rm d}}{{\rm d}t}\CE_1(t)&+\la_0 \CE_1(t)+\la_0 \left\{\left|\left(\FI-\frac{\sqrt{\mu_w}}{\sqrt{\mu}} P_\gamma \right)f\right|_{2,+,1}^2+|(\mathbf I-P_\gamma)f|_{2,+,0}^2\right\}\notag\\
\leq& C\left\{\left\| \nu^{-\frac{1}{2}}\Gamma(f,f)\right\|_{2}^2+P(|x_w(t)|,|v_w(t)|,\|wf(t)\|_\infty)\right\},\notag
\end{align}
which directly gives \eqref{ex-decay}.

If $\kappa=0$, we let $\eta_3>0$ be suitably small. By taking the sum of \eqref{dis-xv-mif} and $\eta_3\times \eqref{es Pf}$, we obtain that there exists a constant $\lambda_3>0$ such that
\begin{align}
	\frac{{\rm d}}{{\rm d}t}(\|f(t)\|_2^2&+|v_w(t)|^2)+\la_3 \CE_1(t)+\la_3
\left\{\left|\left(\FI-\frac{\sqrt{\mu_w}}{\sqrt{\mu}} P_\gamma \right)f\right|_{2,+,1}^2+|(\mathbf I-P_\gamma)f|_{2,+,0}^2\right\}\notag\\
	\leq&-\eta_3\frac{{\rm d}}{{\rm d}t}\left(x_wv_w+\mathcal M\int_{\Omega\times\mathbb R^3}x_w(1+x_w)x_1v_1\sqrt{\mu(v)}f{\rm d}v{\rm d}x_1+N_f(t)\right)\notag\\&+ C\left\{\left\| \nu^{-\frac{1}{2}}\Gamma(f,f)\right\|_{2}^2+P(|x_w(t)|,|v_w(t)|,\|wf(t)\|_\infty)\right\}.\label{ka=0-eng}
\end{align}
On the other hand, it follows that
\begin{equation}
	\frac{1}{2}\frac{{\rm d}}{{\rm d}t}|x_w(t)|^2=x_w(t)v_w(t).\label{xw-eng-p}
\end{equation}
Combining \eqref{ka=0-eng} and \eqref{xw-eng-p}, we conclude that \eqref{ex-decay} also holds for $\kappa=0$. This completes the proof of Proposition \ref{decay-l2}.
\end{proof}

\section{$L^\infty$ estimate}\label{lif-sec}
In this section, we derive the velocity-weighted $L^\infty$ estimate of the solution $f$. The argument is based on the method of characteristics. To this end,  we define
\begin{equation*}
	h(t,x_1,v):=w(v)f(t,x_1,v),
\end{equation*}
and introduce
\begin{equation*}
	K_{w}(\ \cdot \ ):=wK(\frac{1}{w}\ \cdot\ ),\ \tilde w=\frac{1}{w(v)\sqrt{\mu(v)}}.
\end{equation*}
Then \eqref{PBE}, \eqref{ic f}, \eqref{drbl-f} and \eqref{drb-f} can be rewritten as
\begin{align}
	\left\{	\begin{array}{l}
		\partial_th+\frac{v_1-v_w(t)x_1}{1+x_w(t)}\partial_{x_1}h+\nu h=K_{w}h+w\Gamma(\frac{h}{w},\frac{h}{w}),\\[2mm]
		h(0,x_1,v)=h_0(x_1,v),\\[2mm]
		h(t,0,v)=w(v)P_\gamma f(t,0,v),\ v_1>0, \\[2mm]
		h(t,1,v)=w(v) P_\gamma f(t,1,v)+w(v)r(t,v),\  v_1-v_w(t)<0.
	\end{array}\right.\notag
\end{align}
The main result of this section is stated as follows.

\begin{proposition}\label{prop es L^infty-drcl}
	Suppose that $[f,x_w,v_w]$ is a solution of \eqref{PBE}, \eqref{trb-ode}, \eqref{ic f},  \eqref{drbl-f}, and \eqref{drb-f}, and satisfies
	\begin{align}\label{apas-2-drbcl}
		\mathop{\sup}_{0\leq s\leq t}e^{\la_1 s}\{|x_w(s)|+|v_w(s)|+\left\|wf(s)\right\|_{\infty}\}\leq2\vps,
	\end{align}
	for any $t\in[0,+\infty)$ and suitably small $\vps>0$, where $\la_1<\frac{1}{2}\min\{\la_0,\nu_0\}$. Then, for all $t \ge 0$, it holds that
	\begin{equation}\label{L^inf es-drcl}
		\mathop{\sup}_{0\leq s\leq t}e^{\la_1 s}\|h(s)\|_\infty+\mathop{\sup}_{0\leq s\leq t}e^{\la_1 s}|h(s)|_\infty\leq C\Big\{ \|h_0\|_\infty
		+|x_{w0}|+|v_{w0}|\Big\}.
	\end{equation}
\end{proposition}
Before proving Proposition \ref{prop es L^infty-drcl}, let us first introduce the following notations.

Let $(x_1,v)\notin\gamma_0$ and $(t^0,x_1^0,v^0)=(t,x_1,v)$, define
\begin{align*}
	&t^1(t,x_1,v_1)=\inf\{s<t: X_1^f(s;t,x_1,v_1)\in\Omega\},\\
	&x_1^1=X_1^f(t^1(t,x_1,v_1);t,x_1,v_1),\\
	&v^1\in\mathcal V^1:=\{x^1_1=0,v_1^1<0,v^1\in\mathbb R^3\}\cup\{x^1_1=1,v_1^1-v_w(t^1)>0,v^1\in\mathbb R^3\},
\end{align*}
and inductively for $l\geq1$
\begin{align*}
	&t^l=\inf\{s<t^{l-1}: X_1^f(s;t^{l-1},x^{l-1}_1,v^{l-1}_1)\in\Omega\},\\
	&x_1^l=X_1^f(t^l;t^{l-1},x_1^{l-1},v_1^{l-1}),\\
	&v^l\in\mathcal V^l:=\{x^l_1=0,v_1^l<0,v^l\in\mathbb R^3\}\cup\{x^l_1=1,v_1^l-v_w(t^l)>0,v^l\in\mathbb R^3\}.
\end{align*}
Moreover, we define the following stochastic cycles
\begin{align}\label{def X_cl}
\left\{\begin{array}{rll}
	&X_{1,{\mathbf {cl}}}(s;t,x_1,v_1)=\sum_l \mathbf 1_{[t^{l+1},t^l)}(s)X_1^f(s;t^l,x_1^l,v^l_1),\\[2mm]
	&V_{{\mathbf {cl}}}(s;t,x_1,v_1)=\sum_l \mathbf 1_{[t^{l+1},t^l)}(s)v^l.
\end{array}\right.
\end{align}

Based on the above definitions, we obtain the following elementary estimates.
\begin{lemma}\label{lem es h-drcl}
	If $t^1\leq0$, then
	\begin{equation}\label{t1<0-drcl}
		|h(t,x_1,v)|\leq e^{-\nu(v) t}|h(0,X_1^f(0),v)|+\int_{0}^te^{-\nu(v)(t-s)}|[K_{w}h+w\Gamma(\frac{h}{w},\frac{h}{w})](s,X_1^f(s;t,x_1,v_1),v)|{\rm d}s.
	\end{equation}
	If $t^1\geq0$, we have
	\begin{align}\label{t1>0-drcl}
		|h(t,x_1,v)|&\leq e^{-\nu(v)(t-t^1)}|wr(t^1,x_1^1,v)|\mathbf 1_{\{x_1^1=1\}}\notag\\
		&\quad+\int_{t^1}^te^{-\nu(v)(t-s)}|[K_{w}h+w\Gamma(\frac{h}{w},\frac{h}{w})](s,X_1^f(s;t,x_1,v_1),v)|{\rm d}s\notag\\
		&\quad+\frac{e^{-\nu(v)(t-t^1)}}{\tilde w(v)}\int_{\prod_{j=1}^{k-1}\mathcal V^j}|H|,
	\end{align}
	where $|H|$ is bounded by
	\begin{align}\label{es H1}
		\sum_{l=1}^{k-1}&\mathbf 1_{\{t^{l+1}\leq0<t^l\}}|h(0,X_1^f(0;t^l,x_1^l,v_1^l),v^l)|{\rm d}\Sigma_l(0)\\\label{es H2}
		&+\sum_{l=1}^{k-1}\int_0^{t^l}\mathbf 1_{\{t^{l+1}\leq0<t^l\}}|[K_{w}h+w\Gamma(\frac{h}{w},\frac{h}{w})](s,X_1^f(s;t^l,x_1^l,v_1^l),v^l)|{\rm d}\Sigma_l(s){\rm d}s\\\label{es H3}
		&+\sum_{l=1}^{k-1}\int_{t^{l+1}}^{t^l}\mathbf 1_{\{0<t^l\}}|[K_{w}h+w\Gamma(\frac{h}{w},\frac{h}{w})](s,X_1^f(s;t^l,x_1^l,v_1^l),v^l)|{\rm d}\Sigma_l(s){\rm d}s\\\label{es H4}
		&+\sum_{l=1}^{k-1}\mathbf 1_{\{0<t^l\}}{\rm d}\Sigma_l^r\\\label{es H5}
		&+\mathbf 1_{\{0<t^k\}}|h(t^k,x_1^k,v^{k-1})|{\rm d}\Sigma_{k-1}(t^k).
	\end{align}
	We have denoted
	\begin{align}\label{dSigma}
		{\rm d}\sigma^l&=-\sqrt{2\pi}\mu(v^l)v^l_1{\rm d}v^l,\ {\rm if}\ x^l_1=0,\ {\rm and}\ {\rm d}\sigma^l=\sqrt{2\pi}\mu(v^l)(v^l_1-v_w(t^l)){\rm d}v^l,\ {\rm if}\ x_1^l=1,\notag\\
		{\rm d}\Sigma_l&=\{\prod_{j=l+1}^{k-1}{\rm d}\sigma^j\}\times\{\tilde w(v^l){\rm d}\sigma^l\}\times\{\prod_{j=1}^{l-1}{\rm d}\sigma^j\},\notag\\
		{\rm d}\Sigma_l(s)&=\{\prod_{j=l+1}^{k-1}{\rm d}\sigma^j\}\times\{e^{-\nu(v^l)(t^l-s)}\tilde w(v^l){\rm d}\sigma^l\}\times\{\prod_{j=1}^{l-1}e^{-\nu(v^j)(t^j-t^{j+1})}{\rm d}\sigma^j\},\notag\\	
		{\rm d}\Sigma_l^r&=\{\prod_{j=l+1}^{k-1}{\rm d}\sigma^j\}\times\{e^{-\nu(v^l)(t^l-t^{l+1})}\tilde w(v^l)w(v^l)r(t^{l+1},x_1^{l+1},v^l){\rm d}\sigma^l\}\notag\\
		&\quad\times\{\prod_{j=1}^{l-1}e^{-\nu(v^j)(t^j-t^{j+1})}{\rm d}\sigma^j\}.
	\end{align}
\end{lemma}
\begin{proof}
The proof of Lemma \ref{lem es h-drcl} is simpler than that of Lemma \ref{lem es h^ell+1} later, and is therefore omitted for brevity.
\end{proof}

With Lemma \ref{lem es h-drcl} in hand, we now establish a finite-time $L^\infty$ estimate.
\begin{lemma}\label{lem es ft}
Under the conditions listed in Proposition \ref{prop es L^infty-drcl},
	there exist constants $C>0$ and $T_0>0$ such that, for $k = C T_0^{5/4}$ in Lemma \ref{lem es h-drcl}, we have
	\begin{equation}\label{es ft}
	\|h(T_0)\|_\infty+|h(T_0)|_\infty\leq e^{-\la_2T_0}\|h_0\|_\infty+C(T_0)e^{-\frac{\nu_0}{2}T_0}\left(\mathop{\sup}_{0\leq s\leq T_0}|v_w(s)|+\int_0^{T_0}\|f(s)\|_2{\rm d}s\right),
	\end{equation}
where $0<\lambda_1<\lambda_2<\min\left\{\frac{\nu_0}{2},\frac{\la_0}{2}\right\}$.
\end{lemma}
\begin{proof}
To prove \eqref{es ft}, it suffices to estimate the terms on the right hand side of \eqref{t1<0-drcl} and \eqref{t1>0-drcl}. For the $r$-contribution in  \eqref{t1>0-drcl}, it is straightforward to obtain
	\begin{equation}
		e^{-\nu(v)(t-t^1(t,x_1,v_1))}	|w r(t^1,x_1^1,v)|\leq e^{-\nu_0(t-t^1)}|w r(t^1,x_1^1,v)|.	\notag
	\end{equation}
Since the exponential time-decay factor in ${\rm d}\Sigma_l^r$ is bounded by $e^{-\nu_0 (t - t^{l+1})}$, it follows from \eqref{dSigma}, together with the choice $k = C T_0^{5/4}$ ensuring that \eqref{int-ds^j3} in Lemma \ref{lem es} holds, that the $r$-contribution in \eqref{es H4} can be bounded by
\begin{align}
		Ck&\Big(1+\frac{1}{k}\Big)^{k-1}e^{-\frac{\nu_0}{2}t}\frac{1}{\tilde w(v)}|e^{\frac{\nu_0}{2}s}wr(s)|_\infty\leq Cke^{-\frac{\nu_0}{2}t}\mathop{\sup}_{0\leq s\leq t}|e^{\frac{\nu_0}{2}s}wr(s)|_\infty.\notag
	\end{align}
Next, we consider the $\Gamma(\frac{h}{w},\frac{h}{w})$-contribution in \eqref{t1<0-drcl}, \eqref{t1>0-drcl}, \eqref{es H2} and \eqref{es H3}. The exponential time-decay factor in $d\Sigma_l(s)$ is bounded by $e^{-\nu_0(t^1-s)}$. The $\Gamma(\frac{h}{w},\frac{h}{w})$-contribution can be bounded by
	\begin{align*}
		2\int_0^t&e^{-\nu(v)(t-s)}\left\| \frac{w \Gamma(\frac{h}{w},\frac{h}{w})}{\langle v\rangle}(s)\right\|_\infty{\rm d}s+\frac{e^{-\nu_0(t-t^1)}}{\tilde w(v)}\int_0^te^{-\nu_0(t^1-s)}\left\|\frac{w \Gamma(\frac{h}{w},\frac{h}{w})}{\langle v\rangle}(s)\right\|_\infty\\
		&\times\sum_{l=1}^{k-1}\Big\{\int \textbf 1_{\{t^{l+1}\leq0< t^l\}}\tilde w(v^l)\langle v^l\rangle\prod_{j=1}^{k-1}{\rm d}\sigma^j+\mathop{\max}_l\int \textbf 1_{\{0<t^{l+1}\}}\tilde w(v^l)\langle v^l\rangle\prod_{j=1}^{k-1}{\rm d}\sigma^j\Big\}{\rm d}s,
	\end{align*}
	where we have used
	\begin{align*}
		|w \Gamma(\frac{h}{w},\frac{h}{w})(s,X_1^f(s;t^l,x_1^l,v_1^l),v^l)|&=\frac{w \Gamma(\frac{h}{w},\frac{h}{w})(s,X_1^f(s;t^l,x_1^l,v_1^l),v^l)}{\langle v^l\rangle}\times\langle v^l\rangle\\
		&\leq\left\|\frac{w \Gamma(\frac{h}{w},\frac{h}{w})}{\langle v\rangle}(s,v)\right\|_\infty\langle v^l\rangle.
	\end{align*}
	Using Lemma \ref{lem es}, the second term can be bounded by
	\begin{align*}
		Ck& e^{-\nu_0(t-t^1)}\int_0^t e^{-\nu_0(t^1-s)}\left\|\frac{w \Gamma(\frac{h}{w},\frac{h}{w})}{\langle v\rangle}(s)\right\|_\infty{\rm d}s\\
		&\leq Ck e^{-\nu_0(t-t^1)}\int_0^t e^{-\frac{\nu_0}{2}(t^1-s)}e^{-\frac{\nu_0}{2}t^1}\mathop{\sup}_{0\leq s\leq t}\Big\{e^{\frac{\nu_0}{2}s}\left\|\frac{w \Gamma(\frac{h}{w},\frac{h}{w})}{\langle v\rangle}(s)\right\|_\infty\Big\}{\rm d}s\\
		&\leq Ck e^{-\frac{\nu_0}{2}t}\mathop{\sup}_{0\leq s\leq t}\Big\{e^{\frac{\nu_0}{2}s}\left\|\frac{w \Gamma(\frac{h}{w},\frac{h}{w})}{\langle v\rangle}(s)\right\|_\infty\Big\}.
	\end{align*}
	Therefore, the $\Gamma(\frac{h}{w},\frac{h}{w})$-contribution can be further bounded by
	\begin{equation}
		Cke^{-\frac{\nu_0}{2}t}\mathop{\sup}_{0\leq s\leq t}\Big\{e^{\frac{\nu_0}{2}s}\left\|\frac{w \Gamma(\frac{h}{w},\frac{h}{w})}{\langle v\rangle}(s)\right\|_\infty\Big\}.\notag
	\end{equation}
	
Next applying Lemma \ref{lem es} again, we can bound the first term on the right hand side of \eqref{t1<0-drcl}, \eqref{es H1} and \eqref{es H5} as
	\begin{align*}
		e^{-\nu(v)t}&|h(0,X_1^f(0;t,x_1,v_1),v)|\textbf 1_{\{t^1<0\}}\\
		&\quad+\frac{e^{-\nu(v)(t-t^1)}}{\tilde w(v)}\int_{\prod_{j=1}^{k-1}\mathcal V^j}\Big\{\sum_{l=1}^{k-1}\textbf 1_{\{t^{l+1}\leq0<t^l\}}|h(0,X_1^f(0;t^l,x_1^l,v_1^l),v^l)|{\rm d}\Sigma_l(0)\\
		&\qquad+\textbf 1_{\{0<t^k\}}|h(t^k.x_1^k,v^{k-1})|{\rm d}\Sigma_{k-1}(t^k)\Big\}\\
		&\leq C e^{-\nu_0t}\left\|h(0)\right\|_\infty+ e^{-\nu_0t}\int \sum_{l=1}^{k-1}\textbf 1_{\{t^{l+1}\leq0<t^l\}}{\rm d}\Sigma_l\left\|h(0)\right\|_\infty\\
		&\quad+ e^{-\frac{\nu_0}{2}t}\mathop{\sup}_{0\leq s\leq t}\big\{e^{\frac{\nu_0}{2}s}\left\| h(s)\right\|_\infty\Big\}\int\textbf 1_{\{0<t^k\}}{\rm d}\Sigma_{k-1}\\
		&\leq Ce^{-\frac{\nu_0}{2}t}\Big\{k\left\|h(0)\right\|_\infty+\Big\{\frac{1}{2}\Big\}^{k}\mathop{\sup}_{0\leq s\leq t}\big\{e^{\frac{\nu_0}{2}s}\left\| h(s)\right\|_\infty\big\}\Big\}.
	\end{align*}
Consequently, we arrive at
	\begin{align}\label{es 1-drcl}
		|h(t,x_1,v)|&\leq \textbf 1_{\{t^1<0\}}\int_{0}^te^{-\nu(v)(t-s)}|K_{w}h(s,X_1^f(s;t,x_1,v_1),v)|{\rm d}s\notag\\
		&\quad+\textbf 1_{\{t^1>0\}}\int_{t^1}^te^{-\nu(v)(t-s)}|K_{w}h(s,X_1^f(s;t,x_1,v_1),v)|{\rm d}s\notag\\
		&\quad+\frac{e^{-\nu(t-t^1)}}{\tilde w}\int_{\prod_{j=1}^{k-1}\mathcal V^j}\sum_{l=1}^{k-1}\Big\{\int_0^{t^l}\textbf 1_{\{t^{l+1}\leq0<t^l\}}|K_{w}h(s,X_{1,\mathbf{cl}}(s),v^l)|{\rm d}\Sigma_l(s){\rm d}s\notag\\
		&\qquad+\int_{t^{l+1}}^{t^l}\textbf 1_{\{0<t^{l+1}\}}|K_{w}h(s,X_{1,\mathbf{cl}}(s),v^l)|{\rm d}\Sigma_l(s){\rm d}s\Big\}\notag\\
		&\quad+e^{-\frac{\nu_0}{2}t}A(t,x_1,v_1),
	\end{align}
	where
	\begin{align}\label{def.A.dd}
		A(t,x_1,v_1)=&Ck\left\|h(0)\right\|_\infty+e^{\frac{\nu_0}{2}t^1}	|w r(t^1,x_1^1,v)|+Ck\mathop{\sup}_{0\leq s\leq t}|e^{\frac{\nu_0}{2}s}wr(s)|_\infty\notag\\
		&+Ck\mathop{\sup}_{0\leq s\leq t}\Big\{e^{\frac{\nu_0}{2}s}\left\|w \Gamma(\frac{h}{w},\frac{h}{w})(s)\right\|_\infty\Big\}+C\Big\{\frac{1}{2}\Big\}^{k}\mathop{\sup}_{0\leq s\leq t}\big\{e^{\frac{\nu_0}{2}s}\left\| h(s)\right\|_\infty\big\}.
	\end{align}
	Recalling the definition of $X_{1,\mathbf {cl}}(s;t,x_1,v_1)$ and $V_{\mathbf {cl}}(s;t,x_1,v_1)$ in \eqref{def X_cl}, we further iterate the $K_{w}$ operator in \eqref{es 1-drcl} as
	\begin{align}\label{Kh-drcl}
		K_{w}&h(s,X_{1,\mathbf {cl}}(s),v^l)\leq\int_{\mathbb R^3}\textbf k_w(v^l,v_*)|h(s,X_{1,\mathbf {cl}}(s),v_*){\rm d}v_\ast\notag\\
		&\leq\iint \textbf 1_{\{\bar t^1<0\}}\int_{0}^se^{-\nu(v_*)(s-s_1)}\textbf k_w(v^l,v_*)\textbf k_w(v_*,v_{**})\notag\\
		&\qquad\times|h(s_1,X_1^f(s_1;s,X_{1,{\bf cl}}(s),v_{*1}),v_{**})|{\rm d}s_1{\rm d}v_*{\rm d}v_{**}\notag\\
		&\quad+\iint \textbf 1_{\{\bar t^1>0\}}\int_{\bar t^1}^se^{-\nu(v_*)(s-s_1)}\textbf k_w(v^l,v_*)\textbf k_w(v_*,v_{**})\notag\\
		&\qquad\times|h(s_1,X_1^f(s_1;s,X_{1,\rm cl}(s),v_{*1}),v_{**})|{\rm d}s_1{\rm d}v_*{\rm d}v_{**}\notag\\
		&\quad+\iint\int_{\prod_{ j=1}^{k-1}\bar{\mathcal V}^j}\frac{e^{-\nu(v_*)(s-\bar t^1)}}{\tilde w(v_*)}\sum_{ l'=1}^{k-1}\Big\{\int_0^{\bar{t}^{l'}}\textbf 1_{\{\bar t^{l'+1}\leq0<\bar t^{l'}\}}\notag\\
		&\qquad\times\textbf k_w(v^l,v_*)\textbf k_w(\bar v^{l'},v_{**}) |h(s_1,X_1^f(s_1;\bar t^{l'},\bar x_1^{l'},\bar v_1^{l'}),v_{**})|{\rm d}\Sigma_{l'}(s_1){\rm d}s_1{\rm d}v_*{\rm d}v_{**}\notag\\
		&\quad+\iint\int_{\prod_{ j=1}^{k-1}\bar{\mathcal V}^j}\frac{e^{-\nu(v_*)(s-\bar t^1)}}{\tilde w(v_*)}\sum_{ l'=1}^{k-1}\Big\{\int_{\bar t^{l'+1}}^{\bar{t}^{l'}}\textbf 1_{\{\bar t^{l'+1}>0\}}\notag\\
		&\qquad\times\textbf k_w(v^l,v_*)\textbf k_w(\bar v^{l'},v_{**}) |h(s_1,X_1^f(s_1;\bar t^{l'},\bar x_1^{l'},\bar v_1^{l'}),v_{**})|{\rm d}\Sigma_{l'}(s_1){\rm d}s_1{\rm d}v_*{\rm d}v_{**}\notag\\
		&\quad+e^{-\frac{\nu_0}{2}s}\int_{\mathbb R^3}\textbf k_w(v^l,v_*)A(s,X_{1,\mathbf{cl}}(s),v_*){\rm d}v_*,
	\end{align}
where
\begin{align*}
	&\bar t^1=t^1(s,X_{1,\mathbf{cl}}(s),v_{*1}),\\
	&\bar x_1^1=X_1^f(\bar t^1;s,X_{1,\mathbf{cl}}(s),v_{*1}),\\
	&\bar v^1\in\bar{\mathcal V}^1:=\{\bar x^1_1=0,\bar v_1^1<0,\bar v^1\in\mathbb R^3\}\cup\{\bar x^1_1=1,\bar v_1^1-v_w(\bar t^1)>0,\bar v^1\in\mathbb R^3\},
\end{align*}
and inductively for $l\geq1$
\begin{align*}
	&\bar t^l=\sup\{0\leq s<\bar t^{l-1}: X_1^f(s;\bar t^{l-1},\bar x^{l-1}_1,\bar v^{l-1}_1)\in\Omega\},\\
	&\bar x_1^l=X_1^f(\bar t^l;\bar t^{l-1},\bar x_1^{l-1},\bar v_1^{l-1}),\\
	&\bar v^l\in\bar{\mathcal V}^l:=\{\bar x^l_1=0,\bar v_1^l<0,\bar v^l\in\mathbb R^3\}\cup\{\bar x^l_1=1,\bar v_1^l-v_w(\bar t^l)>0,\bar v^l\in\mathbb R^3\}.
\end{align*}
Now, plugging \eqref{Kh-drcl} into \eqref{es 1-drcl}, we estimate the right hand side of the resulting inequality as follows. By Lemma \ref{k-op}, the contribution from $A=A(t,x_1,v_1)$ by \eqref{def.A.dd} can be bounded by
	\begin{align}\label{es A-con-drcl}
		2A(t)&\int^t_0 e^{-\nu_0(t-s)}e^{-\frac{\nu_0}{2}s}{\rm d}s+e^{-\frac{\nu_0}{2}t} A\notag\\&+e^{-\frac{\nu_0}{2}t} A(t)\int_{\prod_{j=1}^{k-1}\mathcal V^j}\sum_{l=1}^{k-1}\Big\{\int_0^{t^l}\textbf 1_{\{t^{l+1}\leq0<t^l\}}+\int_{t^{l+1}}^{t^l}\textbf 1_{\{0<t^{l+1}\}}\Big\} e^{-\frac{\nu_0}{2}(t-s)}\tilde w(v^l){\rm d}\Sigma_l{\rm d}s\notag\\
		&\lesssim e^{-\frac{\nu_0}{2}t}\Big[k+\Big\{\frac{1}{2}\Big\}^{k}\Big]A(t)\notag\\
		&\lesssim e^{-\frac{\nu_0}{2}t}k\left[k\mathop{\sup}_{0\leq s\leq t}|e^{\frac{\nu_0}{2}s}wr(s)|_\infty+k\mathop{\sup}_{0\leq s\leq t}\Big\{e^{\frac{\nu_0}{2}s}\left\|\frac{w \Gamma(\frac{h}{w},\frac{h}{w})}{\langle v\rangle}(s)\right\|_\infty\Big\}\right]\notag\\
		&\quad+ ke^{-\frac{\nu_0}{2}t}\left\|h(0)\right\|_\infty+\Big\{\frac{1}{2}\Big\}^{k}e^{-\frac{\nu_0}{2}t}\mathop{\sup}_{0\leq s\leq t}\big\{e^{\frac{\nu_0}{2}s}\left\| h(s)\right\|_\infty\big\}.
	\end{align}
For the $K_wh$ contribution,  we focus on only the following delicate term:
	\begin{align}\label{es kkh-drcl}
		&\frac{e^{-\nu(t-t^1)}}{\tilde w}\int_{\prod_{j=1}^{k-1}\mathcal V^j}\sum_{l=1}^{k-1}\int_0^{t^l}\textbf 1_{\{t^{l+1}\leq0<t^l\}}\iint\int_{\prod_{ j=1}^{k-1}\bar{\mathcal V}^j}\frac{e^{-\nu(v_*)(s-\bar t^1)}}{\tilde w(v_*)}\sum_{ l'=1}^{k-1}\Big\{\int_0^{\bar{t}^{l'}}\textbf 1_{\{\bar t^{l'+1}\leq0<\bar t^{l'}\}}\notag\\
		&\times\textbf k_w(v^l,v_*)\textbf k_w(\bar v^{l'},v_{**}) |h(s_1,X_1^f(s_1;\bar t^{l'},\bar x_1^{l'},\bar v_1^{l'}),v_{**})|{\rm d}\Sigma_{l'}(s_1){\rm d}s_1{\rm d}v_*{\rm d}v_{**}{\rm d}\Sigma_l(s){\rm d}s,
	\end{align}
and the remaining terms can be treated in a similar manner.
To estimate \eqref{es kkh-drcl},  we first split the integral into two parts according  to whether $s_1\leq \bar t^{l'}-\eta_0$ and $s_1\geq \bar t^{l'}-\eta_0$ for $\eta_0>0$ suitably small. In the latter case, we obtain
	\begin{equation}\label{es h-con1}
		\eta_0^2 k^2e^{-\frac{\nu_0}{2}t}\mathop{\sup}_{0\leq s\leq t}\big\{e^{\frac{\nu_0}{2} s}\left\| h(s)\right\|_\infty\big\}.
	\end{equation}
	Next we turn to the case $s_1\leq \bar t^{l'}-\eta_0$, for which the computation is divided into the following subcases.
	
	\noindent{\textbf {Case 1:}} $|v^l|\geq N,$ or $|\bar v^{l'}|\geq N$.  Since
	\begin{equation*}
		\iint 	\textbf k_w(v^l,v_*)\textbf k_w(\bar v^{l'},v_{**}){\rm d}v_*{\rm d}v_{**}\lesssim \frac{1}{(1+|v^l|)(1+|\bar v^{l'}|)}\leq \frac{C}{N},
	\end{equation*}
	and
	\begin{equation*}
		\int_{\prod_{j=1}^{k-1}\mathcal V^j}{\rm d}\Sigma_l<\infty,\ 	\int_{\prod_{j=1}^{k-1}\bar{\mathcal V}^j}{\rm d}\Sigma_{l'}<\infty,
	\end{equation*}
	the contribution of \eqref{es kkh-drcl} for $|v^l|\geq N$ or $|\bar v^{l'}|\geq N$ can be bounded by
	\begin{align}\label{es case1}
		\frac{C}{N}k^2e^{-\frac{\nu_0}{2}t}\mathop{\sup}_{0\leq s\leq t}\big\{e^{\frac{\nu_0}{2}s}\left\| h(s)\right\|_\infty\big\}.
	\end{align}
	\noindent{\textbf{Case 2:}} $|v^l|\leq N$ with $|v_*|\geq 2N$ or  $|\bar v^{l'}|\leq N$ with $|v_{**}|\geq 2N$.   In this case, we have $|v^l-v_*|\geq N$ or $|\bar v^{l'}-v_{**}|\geq N$. Hence
	\begin{align*}
		|\textbf k_w(v^l,v_*)|\leq  e^{-\frac{\eps}{8}N^2}\left|\textbf k_w(v^l,v_*)e^{\frac{\eps}{8}|v^l-v_*|^2}\right|,|\textbf k_w(\bar v^{l'},v_{**})|\leq  e^{-\frac{\eps}{8}N^2}\left|\textbf k_w(\bar v^{l'},v_{**})e^{\frac{\eps}{8}|\bar v^{l'}-v_{**}|^2}\right|.
	\end{align*}
	Consequently, we get from Lemma \ref{k-op} that
	\begin{align}
		&\int_{\prod_{j=1}^{k-1}\mathcal V^j\cap\{|v^l|\leq N\}}\int_{|v_*|\geq2N}\int_{\prod_{ j=1}^{k-1}\bar{\mathcal V}^j}\textbf k_w(v^l,v_*) {\rm d}\Sigma_{l'}{\rm d}v_*{\rm d}\Sigma_l\notag\\
		&\qquad+\int_{\prod_{j=1}^{k-1}\mathcal V^j}\int_{|v_{**}|\geq2N}\int_{\prod_{ j=1}^{k-1}\bar{\mathcal V}^j\cap\{|v^l|\leq N\}}\textbf k_w(\bar v^{l'},v_{**}) {\rm d}\Sigma_{l'}{\rm d}v_*{\rm d}\Sigma_l\notag\\
		&\quad\lesssim e^{-\frac{\eps}{8}N^2}\int_{\prod_{j=1}^{k-1}\mathcal V^j\cap\{|v^l|\leq N\}}\int_{|v_*|\geq2N}\int_{\prod_{ j=1}^{k-1}\bar{\mathcal V}^j}\textbf k_w(v^l,v_*)e^{\frac{\eps}{8}|v^l-v_*|^2} {\rm d}\Sigma_{l'}{\rm d}v_*{\rm d}\Sigma_l\notag\\
		&\qquad+e^{-\frac{\eps}{8}N^2}\int_{\prod_{j=1}^{k-1}\mathcal V^j}\int_{|v_{**}|\geq2N}\int_{\prod_{ j=1}^{k-1}\bar{\mathcal V}^j\cap\{|v^l|\leq N\}}\textbf k_w(v^l,v_*)e^{\frac{\eps}{8}|\bar v^{l'}-v_{**}|^2} {\rm d}\Sigma_{l'}{\rm d}v_*{\rm d}\Sigma_l
		\lesssim e^{-\frac{\eps}{8}N^2}.\notag
	\end{align}
	Therefore, the contribution of \eqref{es kkh-drcl} corresponding to $|v^l-v_*|\geq N$ or $|\bar v^{l'}-v_{**}|\geq N$ can be bounded by
	\begin{equation}\label{es case2}
		e^{-\frac{\eps}{8}N^2} e^{-\frac{\nu_0}{2}t}\mathop{\sup}_{0\leq s\leq t}\big\{e^{\frac{\nu_0}{2}s}\left\| h(s)\right\|_\infty\big\}.
	\end{equation}
	\noindent{\textbf {Case 3:}}  $|v^l|\leq N$ with $|v_*|\leq 2N$ and  $|\bar v^{l'}|\leq N$ with $|v_{**}|\leq 2N$. In this stage, firstly, for any large $N\geq 1$, we choose a number $m=m(N)$ and define
	\begin{equation*}
		\textbf k_m(\bar v^{l'},v_{**})=\mathbf 1_{|\bar v^{l'}-v_{**}|\geq\frac{1}{m}, |v_{**}|\leq m}\textbf k_w(\bar v^{l'},v_{**}),
	\end{equation*}
	such that
	\begin{equation*}
		\mathop{\sup}_{\bar v^{l'}}\int_{\mathbb R^3}|\textbf k_m(\bar v^{l'},v_{**})-\textbf k_w(\bar v^{l'},v_{**})|{\rm d}v_{**}\leq\frac{1}{N}.
	\end{equation*}
	We split $\textbf k_w(\bar v^{l'},v_{**})=\{\textbf k_w(\bar v^{l'},v_{**})-\textbf k_m(\bar v^{l'},v_{**})\}+\textbf k_m(\bar v^{l'},v_{**})$, for which,  the first part leads to the following bound for the corresponding contribution of \eqref{es kkh-drcl}
	\begin{equation}\label{es h-con2}
		\frac{C}{N}k^2 e^{-\frac{\nu_0}{2}t}\mathop{\sup}_{0\leq s\leq t}\big\{e^{\frac{\nu_0}{2}s}\left\| h(s)\right\|_\infty\big\}.
	\end{equation}
	For the second part, note that in this situation
	\begin{equation*}
		|\textbf k_m(\bar v^{l'},v_{**})|\leq C_N,
	\end{equation*}
	we may make a change of variable $y_1=\frac{1+x_w(\bar t^{l'})}{1+x_w(s_1)}\bar x_1^{l'}+\frac{\bar v_1^{l'}(s_1-\bar t^{l'})}{1+x_w(s_1)}$ to obtain the estimate
	\begin{align}\label{l1-l2-drcl}
		\int_{|v_{**}|\leq m}&\int_{\bar{\mathcal  V}^{l'}}|h(s_1,X_1^f(s_1;\bar t^{l'},\bar x_1^{l'},\bar v_1^{l'}),v_{**})|e^{-(\frac{1}{4}+\zeta)|\bar v^{l'}|^2}|\bar v_1^{l'}-v_w(\bar t^{l'})\textbf 1_{\{\bar x_1^{l'}=1\}}|{\rm d}\bar v^{l'}{\rm d}v''\notag\\&\lesssim\int_{\mathbb R^2} e^{-(\frac{1}{4}+\zeta)(|\bar v_2^{l'}|^2+|\bar v_3^{l'}|^2)}{\rm d}\bar v_2^{l'}{\rm d}\bar v_3^{l'}\int_{|v_{**}|\leq m}{\rm d}v_{**} \int_{\mathbb R}|h(s_1,X_1^f(s_1;\bar t^{l'},\bar x_1^{l'},\bar v_1^{l'}),v_{**})|\textbf 1_{\{X_1^f\in\bar\Omega\}}{\rm d}\bar v_1^{l'}\notag\\
		&\lesssim {\frac{C}{\eta_0}}\int_{\Omega}\int_{|v_{**}|\leq m}|h(s_1,y_1,v_{**})|{\rm d}y_1{\rm d}v_{**}{\lesssim_{m,\eta_0}}\left\|\frac{h(s_1)}{w}\right\|_2,
	\end{align}
	for $s_1$ fixed, where we have used the fact that   $|\frac{d y_1}{d \bar v_1^{l'}}|\geq\frac{\eta_0}{2}$ for $\bar t^{l'}-s_1\geq\eta_0$. Combining \eqref{es h-con1}, \eqref{es case1}, \eqref{es case2}, \eqref{es h-con2} and \eqref{l1-l2-drcl}, we can bound \eqref{es kkh-drcl} as
	\begin{align}\label{B-drcl}
		\eqref{es kkh-drcl}
	&	\lesssim k^2e^{-\frac{\nu_0}{2}t}\left\{\Big[{\eta_0^2}+\frac{1}{N}+e^{-\frac{\eps N^2}{8}}\Big]\mathop{\sup}_{0\leq s\leq t}\big\{e^{\frac{\nu_0}{2}s}\left\| h(s)\right\|_\infty\big\}+{C_{\eta_0,m,N}}e^{\frac{\nu_0}{2}t}\int_0^t\left\|\frac{h(s)}{w}\right\|_2{\rm d}s\right\}\notag\\
	&	=:e^{-\frac{\nu_0}{2}t}B.
	\end{align}
Note that, after substituting \eqref{Kh-drcl}, all remaining terms in \eqref{es 1-drcl}, except for the contribution from $A$, admit the same bound as above.

	Plugging \eqref{es A-con-drcl} and \eqref{B-drcl} into \eqref{es 1-drcl} and applying Lemma \ref{Ga}, we obtain
	\begin{align*}
		|h(t,x_1,v)|&\lesssim e^{-\frac{\nu_0}{2}t}[A(t)+B(t)]\notag\\
		&\leq C_Kk^2e^{-\frac{\nu_0}{2} t}\left\|h(0)\right\|_\infty+C e^{-\frac{\nu_0}{2}t}k^2\Big[\mathop{\sup}_{0\leq s\leq t}\Big\{e^{\frac{\nu_0}{2}s}|w r(s)|_\infty\Big\}+\notag\\
		&\qquad+\big[\Big\{\frac{1}{2}\Big\}^{k}+{\eta_0^2}+\vps+\frac{1}{N}+e^{-\frac{\eps N^2}{8}}\big]\mathop{\sup}_{0\leq s\leq t}\big\{e^{\frac{\nu_0}{2}s}\left\| h(s)\right\|_\infty\big\}\Big]\notag\\
		&\qquad+{C_{\eta_0,m, N}}k^2\int_0^t\left\|\frac{h(s)}{w}\right\|_2{\rm d}s,
	\end{align*}
where $C_K>0$.

Choosing $\vps>0$, ${\eta_0>0}$ sufficiently small, and $N$, $k$ sufficiently large such that $2C_Kk^2 e^{-\frac{\nu_0}{2}T_0}\leq e^{-\la_2T_0}$, we further obtain
\begin{align}
&\mathop{\sup}_{0\leq t\leq T_0}\left\{e^{\frac{\nu_0}{2}t}\|h(t)\|_\infty\right\}+\mathop{\sup}_{0\leq t\leq T_0}\left\{e^{\frac{\nu_0}{2}t}|h(t)|_\infty\right\}\notag\\
&\quad\leq e^{(\frac{\nu_0}{2}-\la_2)T_0}\|h_0\|_\infty+C(T_0)\left(\mathop{\sup}_{0\leq t\leq T_0}|v_w(t)|+\int_0^{T_0}\left\|\frac{h(t)}{w}\right\|_2{\rm d}t\right).\label{h-if-l2}
\end{align}
Here we have used the estimate $|w r(t)| \lesssim |v_w(t)|$. Noting that $k \sim T_0^{\frac{5}{4}}$, it follows from \eqref{h-if-l2} that \eqref{es ft} holds. This completes the proof of Lemma \ref{lem es ft}.
\end{proof}

We now turn to complete the proof of Proposition \ref{prop es L^infty-drcl}.
\begin{proof}[The proof of Proposition \ref{prop es L^infty-drcl}]
Let us first define
	\begin{equation*}
		\mathcal R\equiv\|wf_0\|_\infty+|x_{w0}|+|v_{w0}|+\mathop{\sup}_{0\leq s\leq\infty}\left\{e^{\frac{3\lambda_1}{2}s}[|x_w(s)|^{\frac{1}{2}}+|v_w(s)|^{\frac{1}{2}}]\|wf(s)\|_\infty\right\}.
	\end{equation*}
By Proposition \ref{decay-l2} and the {\it a priori} assumption \eqref{apas-2-drbcl}, one has
\begin{equation}
	\|f(t)\|_2+|x_w(t)|+|v_w(t)|\lesssim e^{-\lambda_1 t}\mathcal R.\label{l2-decay-r}
\end{equation}
Let $m \in \mathbb{N}$. By induction on $m$, and using \eqref{es ft} together with \eqref{l2-decay-r}, we obtain
\begin{align}
&	\|h(\{m+1\}T_0)\|_\infty+|h(\{m+1\}T_0)|_\infty\notag\\
&\quad	\leq e^{-\la_2T_0}\|h(mT_0)\|_\infty+C(T_0)e^{-\la_1T_0}\left(\mathop{\sup}_{0\leq s\leq T_0}|v_w(mT_0+s)|+\int_0^{T_0}\|f(mT_0+s)\|_2{\rm d}s\right)\notag\\
&\quad	\leq e^{-\la_2T_0}\|h(mT_0)\|_\infty+C(T_0)e^{-\la_1T_0}e^{-m\la_1T_0}\Big(\mathop{\sup}_{0\leq s\leq T_0}\left\{e^{m\la_1T_0}|v_w(mT_0+s)|\right\}\notag\\
&\qquad+\int_0^{T_0}e^{m\la_1T_0}\|f(mT_0+s)\|_2{\rm d}s\Big)\notag\\
&\quad\leq e^{-\la_2T_0}\|h(mT_0)\|_\infty+C(T_0)(1+T_0)e^{-\la_1T_0}e^{-m\la_1T_0}\mathcal R\notag\\
&\quad\leq e^{-\la_2 T_0}\left[e^{-\lambda_2T_0}\|h(\{m-1\}T_0)\|_\infty+C(T_0)(1+T_0)e^{-\lambda_1T_0}e^{-(m-1)\la_1T_0}\mathcal R\right]\notag\\
&\qquad+C(T_0)(1+T_0)e^{-\la_1T_0}e^{-m\la_1T_0}\mathcal R\notag\\
&\quad\leq e^{-2\la_2T_0}\|h(\{m-1\}T_0)\|_\infty+C(T_0)(1+T_0)e^{-\lambda_1T_0}e^{-m\la_1T_0}\mathcal R\left(1+e^{-(\la_2-\la_1)T_0}\right)\notag\\
&\quad\leq e^{-3\la_2T_0}\|h(\{m-2\}T_0)\|_\infty+C(T_0)(1+T_0)e^{-\lambda_1T_0}e^{-m\la_1T_0}
\left(1+e^{-(\la_2-\la_1)T_0}+e^{-2(\la_2-\la_1)T_0}\right)\mathcal R\notag\\
&\quad\leq e^{-(m+1)\la_2T_0}\|h(0)\|_\infty+C(T_0)(1+T_0)e^{-\lambda_1T_0}e^{-m\la_1T_0}\sum\limits_{j\geq0}e^{-j(\la_2-\la_1)T_0}\mathcal R\notag\\
&\quad\leq C(T_0)e^{-(m+1)\la_1T_0}\mathcal R,\label{mp1-decay}
\end{align}
where we have used $0<\lambda_1<\lambda_2$.

On the other hand, \eqref{es ft} can be extended to
\begin{equation}\label{es ft-1}
	\|h(t_1)\|_\infty+|h(t_1)|_\infty\leq e^{-\la_2t_1}\|h_0\|_\infty+C(T_0)\left(\mathop{\sup}_{0\leq s\leq t_1}|v_w(s)|+\int_0^{t_1}\|f(s)\|_2{\rm d}s\right),
	\end{equation}
for any $0\leq t_1\leq T_0$.

Therefore, for any $t=mT_0+t_1$ with $t_1\in[0,T_0]$, it follows from \eqref{mp1-decay} and \eqref{es ft-1} that
$$
\|h(t)\|_\infty+|h(t)|_\infty\leq Ce^{-\lambda_1t}\mathcal R.
$$
This, together with the {\it a priori} assumption \eqref{apas-2-drbcl}, yields \eqref{L^inf es-drcl}. This finishes the proof of Proposition \ref{prop es L^infty-drcl}.
	\end{proof}

\section{Local-in-time weighted $W^{1,p}$ estimate}\label{sec-loc-reg}
In this section, we derive weighted first-order derivative estimates for the solution $[f, x_w, v_w]$. As demonstrated in \cite{guo-17-inv}, such estimates for the first-order derivatives can be constructed, at least locally in time. Moreover, these weighted $W^{1,p}$ estimate plays a crucial role in the proof of the $L^{1+\de}$ stability establish in Section \ref{sec-sta}.
We first present two basic estimates as follows.
\begin{proposition}\label{al-es}
Let $\alpha_\vps(s,x_1,u_1)$ be given as \eqref{def alpha}.
	For $0<\vho<1$, $N>1$, $s\geq0$ and $x_1\in\bar\Omega$, there exists a constant $C_{\vho,N}$, depending only on $\vho$ and $N$, such that
	\begin{equation}\label{u<N}
		\int_{|u_1|\leq N}\frac{1}{[\alpha_\vps(s,x_1,u_1)]^{\vho}}{\rm d}u_1\leq C_{\vho,N}(s+1),
	\end{equation}
and
\begin{equation}\label{u>N}
	\int_{|u_1|\geq N}\frac{e^{-C|v_1-u_1|^2}}{[\alpha_\vps(s,x_1,u_1)]^{\vho}}{\rm d}u_1\leq C_{\vho,N}.
\end{equation}
\end{proposition}
\begin{proof}
First, we prove \eqref{u<N}. For $x_1\in\bar\Omega$ and $s\geq0$, we can split
\begin{equation*}
	\{u_1\in\mathbb R: t_{\textbf b}(s,x_1,u_1)\leq s+1\}
\end{equation*}
into two parts
\begin{align*}
D_1:=\{u_1\in\mathbb R: t_{\textbf b}(s,x_1,u_1)\leq s+1,\ |u_1-v_w(s-t_{\textbf b})x_{1\textbf b}|\leq t_{\textbf b}\},\\
D_2:=\{u_1\in\mathbb R: t_{\textbf b}(s,x_1,u_1)\leq s+1,\ |u_1-v_w(s-t_{\textbf b})x_{1\textbf b}|> t_{\textbf b}\}.
\end{align*}
In part $D_1$, we define a map
\begin{equation}\label{map 8}
	u_1\in D_1\mapsto t_{\textbf b}\in\mathbb R.
\end{equation}
From \eqref{X} and \eqref{par t_b}, $t_{\textbf b}$ satisfies
\begin{equation}\label{Jac 8}
	|\frac{\partial t_{\textbf b}}{\partial u_1}|=|\frac{t_{\textbf b}}{u_1-v_w(s-t_{\textbf b})x_{1\textbf b}}|\geq 1.
\end{equation}
We use the map \eqref{map 8} and \eqref{Jac 8} to deduce that
\begin{align}
	\int_{\{|u_1|\leq N\}\cap D_1}\frac{1_{t_{\textbf b}\leq s+1}}{\alpha^\vho_\vps(s,x_1,u_1)}{\rm d}u_1&\leq \int\frac{|u_1-v_w(s-t_{\textbf b})x_{1\textbf b}|^{1-\vho}}{t_{\textbf b}}1_{|u_1-v_w(s-t_{\textbf b})x_{1\textbf b}|\leq t_{\textbf b}\leq s+1}{\rm d}t_{\textbf b}\notag\\
	&\leq \int\frac{1_{|u_1-v_w(s-t_{\textbf b})x_{1\textbf b}|\leq t_{\textbf b}\leq s+1}}{t_{\textbf b}^\vho}{\rm d}t_{\textbf b}\leq C_{\vho}(s+1).\notag
\end{align}
In part $D_2$, we define $y_1=u_1-v_w(s-t_{\textbf b})x_{1\textbf b}$, and deduce
\begin{align*}
2\geq\frac{{\rm d} y_1}{\rm d u_1}=1-v'_w(s-t_{\textbf b})\frac{t_{\textbf b}x_{1\textbf b}}{u_1-v_w(s-t_{\textbf b})x_{1\textbf b}}\geq\frac{1}{2}.
\end{align*}
By this, we then get the bound
\begin{equation*}
\int_{\{|u_1|\leq N\}\cap D_2}\frac{1_{t_{\textbf b}\leq s+1}}{[\alpha_\vps(s,x_1,u_1)]^{\vho}}{\rm d}u_1\leq 2\int_{|y_1|\leq N+1}\frac{1}{y_1^\vho}{\rm d}y_1\leq C_{N,\vho}.
\end{equation*}

Now we prove \eqref{u>N}. Note that for $|u_1|\geq N$,
\begin{equation}
	\alpha_\vps(s,x_1,u_1)\geq \textbf 1_{s-t_{\textbf b}\geq-\frac{\vps}{2}}|u_1-v_w(s-t_{\textbf b})x_{1\textbf b}|+\textbf 1_{s-t_{\textbf b}<-\frac{\vps}{2}},\notag
\end{equation}
which implies
\begin{align*}
	\int_{|u_1|\geq N}\frac{e^{-C|v_1-u_1|^2}}{[\alpha_\vps(s,x_1,u_1)]^{\vho}}{\rm d}u_1&\lesssim1+\int_{|u_1|\geq N}\frac{e^{-C|v_1-u_1|^2}}{|u_1-v_w(s-t_{\textbf b})x_{1\textbf b}|^\vho}{\rm d}u_1\lesssim1.
\end{align*}
This ends the proof of Proposition \ref{al-es}.
\end{proof}

The following lemma is devoted to the exponential weighted estimates for collision operators $K$ and $\Ga$.

\begin{lemma}[\cite{cao kim}]\label{es K}
For $0<\theta<\frac{1}{8}$, $0<\frac{\zeta}{2}<\theta$, $0<\tilde\theta<\theta-\frac{\zeta}{2}$, $0<\zeta'<\zeta$, it holds that for $\al\geq0$,
\begin{equation*}
	|e^{\zeta'|v|^2} K\partial^{\al}_{x_1}f(v)|\lesssim \int_{\mathbb R^3}{\bf k}_{\tilde\theta}(v,u)e^{\zeta'|u|^2}|\partial^{\al}_{x_1}f(u)|{\rm d}u,
\end{equation*}
\begin{equation*}
|e^{\zeta'|v|^2}\Gamma_{+}(\partial^{\al}_{x_1}f,f)|+|e^{\zeta'|v|^2}\Gamma_{+}(f,\partial^{\al}_{x_1}f)|\lesssim \left\|wf\right\|_\infty\int_{\mathbb R^3}{\bf k}_{\tilde\theta}(v,u)e^{\zeta|u|^2}|\partial^{\al}_{x_1}f(u)|{\rm d}u,	
\end{equation*}
\begin{equation*}
	|e^{\zeta'|v|^2}\Gamma_{-}(\partial^{\al}_{x_1}f,f)|\lesssim
\left\|wf\right\|_\infty\int_{\mathbb R^3}{\bf k}_{\tilde\theta}(v,u)e^{\zeta'|u|^2}|\partial^{\al}_{x_1}f(u)|{\rm d}u,
\end{equation*}	
\begin{equation*}
	|e^{\zeta'|v|^2}\Gamma_{-}(f,\partial^{\al}_{x_1}f)|\lesssim \langle v\rangle \left\|wf\right\|_\infty e^{\zeta'|v|^2} |\partial^{\al}_{x_1}f(v)|.
\end{equation*}
Here,
\begin{align*}
	{\bf k}_\theta(v,u):=\frac{1}{|v-u|}\exp\Big\{-\theta|v-u|^2-\theta\frac{||v|^2-|u|^2|^2}{|v-u|^2}\Big\}.
\end{align*}
\end{lemma}
Based on Proposition \ref{al-es} and Lemma \ref{es K}, we can now derive the following weighted local regularity.
\begin{proposition}\label{loc-reg}
Under the conditions listed in Theorem \ref{main result},
suppose that
$f$ is a solution  to the system \eqref{PBE}, \eqref{trb-ode}, \eqref{ic f}, \eqref{drbl-f} and \eqref{drb-f}, and satisfies \eqref{L^inf es-drcl}. Then there exists a constant $C>0$ such that
\begin{equation}\label{loc-reg-f}
\left\|{W_c}\partial_{x_1}f(t)\right\|_{p}\lesssim e^{Ct^{p}},\ 2<p<\infty.
\end{equation}

\end{proposition}
\begin{proof}
We first derive the basic $L^p$ estimate for the solution $f$. By Lemmas \ref{lem green fun} and \ref{es K}, one directly has
\begin{align}\label{inequ L^p}
	&\|f(t)\|_{p}^p+\int_0^t\|\nu^{1/p}f\|_{p}^p{\rm d}s+\int_0^t|f|_{p,+}^p{\rm d}s\notag\\
	&\quad\lesssim \|f(0)\|_{p}^p+(1+\|wf\|_\infty)\int_0^t\int_{\Omega\times\mathbb R^3}|f(v)|^{p-1}\int_{\mathbb R^3}\textbf k_{\tilde\theta}(v,u)|f(u)|{\rm d}u{\rm d}x_1{\rm d}v{\rm d}s+\int_0^t|f|_{p,-}^p{\rm d}s.
\end{align}
Applying H\"{o}lder's inequality, it follows that
\begin{align*}
	\int_{\mathbb R^3}&|f(v)|^{p-1}\int_{\mathbb R^3}\textbf k_{\tilde\theta}(v,u)|f(u)|{\rm d}u{\rm d}v\notag\\
	&\lesssim \|f\|_{L^p_v}^{\frac{1}{q-1}}\Big\|\int_{\mathbb R^3}\textbf k_{\tilde\theta}^{1/q}(v,u)\textbf k_{\tilde\theta}^{1/p}(v,u)|f(u)|{\rm d}u\Big\|_{L^p_v}\notag\\
	&\lesssim \|f\|_{L^p_v}^{\frac{1}{q-1}}\Big(\int_{\mathbb R^3}\textbf k_{\tilde\theta}(v,u){\rm d}u\Big)^{1/q}\Big\|\Big(\int_{\mathbb R^3}\textbf k_{\tilde\theta}(v,u)|f(u)|^p{\rm d}u\Big)^{1/p}\Big\|_{L^p_v}\notag\\
	&\lesssim \|f\|_{L^p_v}^{p}\Big(\int_{\mathbb R^3}\textbf k_{\tilde\theta}(v,u){\rm d}u\Big)^{1/q}\Big(\int_{\mathbb R^3}\textbf k_{\tilde\theta}(v,u){\rm d}u\Big)^{1/p}\lesssim \|f\|_{L^p_v}^{p},
\end{align*}
where $\frac{1}{p}+\frac{1}{q}=1.$ Thus,
the second term on the right hand side of \eqref{inequ L^p} is controlled by
\begin{equation}\label{es second term}
	(1+\|wf\|_\infty)\int_0^t\|f\|_{p}^p{\rm d}s.
\end{equation}
Next, we estimate $\int_0^t|f|_{p,-}^p{\rm d}s$. Decomposing the set $\{0<v_1-v_w(s)\}$ into $\{0<v_1-v_w(s)<\vps\}\cup\{v_1-v_w(s)\geq\vps\}$ and applying \eqref{es second term} together with Lemma \ref{lem trace}, we obtain
\begin{align}\label{p,-,1}
\int_0^t|f|_{p,-,1}^p{\rm d}s&\lesssim \int_0^t\Big(\int_{0<u_1-v_w(s)<\vps}f(s,1,u)\sqrt{\mu(u)}|u_1-v_w(s)|{\rm d}u\Big)^p{\rm d}s\notag\\
&\quad+\int_0^t\Big(\int_{u_1-v_w(s)\geq\vps}f(s,1,u)\sqrt{\mu(u)}|u_1-v_w(s)|{\rm d}u\Big)^p{\rm d}s+\int_0^t|v_w(s)|^p{\rm d}s\notag\\
&\lesssim o(1)\int_0^t|f|_{p,+,1}^p{\rm d}s+\int_0^t|v_w(s)|^p{\rm d}s\notag\\
&\qquad+\int_0^t\Big(\int_{u_1-v_w(s)\geq\vps}f(s,1,u)\sqrt{\mu(u)}|u_1-v_w(s)|{\rm d}u\Big)^p{\rm d}s\notag\\
&\lesssim o(1)\int_0^t|f|_{p.+,1}^p{\rm d}s+\int_0^t|v_w(s)|^p{\rm d}s+\|f(0)\|_{p}^p+	(1+\|wf\|_\infty)\int_0^t\|f\|_{p}^p{\rm d}s.
\end{align}
Similarly, by using the decomposition $\{-\vps<v_1<0\}\cup\{v_1\leq-\vps\}$, we have
\begin{equation*}
	\int_0^t|f|_{p,-,0}^p{\rm d}s
	\lesssim o(1)\int_0^t|f|_{p,+,0}^p{\rm d}s+\|f(0)\|_{p}^p+	(1+\|wf\|_\infty)\int_0^t\|f\|_{p}^p{\rm d}s.
\end{equation*}
Combining the above estimate yields
\begin{align}\label{es third term}
\int_0^t|f|_{p,-}^p{\rm d}s&=\int_0^t|f|_{p,-,0}^p{\rm d}s+\int_0^t|f|_{p,-,1}^p{\rm d}s\notag\\
&\lesssim 	o(1)\int_0^t|f|_{p,+}^p{\rm d}s+\int_0^t|v_w(s)|^p{\rm d}s
+\|f(0)\|_{p}^p+	(1+\|wf\|_\infty)\int_0^t\|f\|_{p}^p{\rm d}s.
\end{align}
Substituting \eqref{es second term} and \eqref{es third term} into \eqref{inequ L^p}, we conclude that
\begin{align}
	&\|f(t)\|_{p}^p+\int_0^t\|\nu^{1/p}f\|_{p}^p{\rm d}s+\int_0^t|f|_{p,+}^p{\rm d}s\lesssim\|f(0)\|_{p}^p+\int_0^t|v_w(s)|^p{\rm d}s+	(1+\|wf\|_\infty)\int_0^t\|f\|_{p}^p{\rm d}s,\notag
\end{align}
which, by Gr\"{o}nwall's inequality, further implies
\begin{align}\label{es L^p}
	&\|f(t)\|_{p}^p+\int_0^t\|\nu^{1/p}f\|_{p}^p{\rm d}s+\int_0^t|f|_{p,+}^p{\rm d}s\lesssim e^{Ct},
\end{align}
for some constant $C>0$.

We next perform the weighted $x_1$ derivative estimate for $f$. To this end,
using \eqref{par alpha}, we deduce from \eqref{PBE} that
\begin{align*}
	&\frac{1}{p}|{W_c}\partial_{x_1}f|^{p-1}[\partial_t+\frac{v_1-v_w(t)x_1}{1+x_w(t)}\pa_{x_1}+\nu]{W_c}\partial_{x_1}f\\
&\quad=\frac{1}{p}\alpha_\vps^{\vho p}e^{p\zeta'|v|^2}|\pa_{x_1}f|^{p-1}	[\partial_t+\frac{v_1-v_w(t)x_1}{1+x_w(t)}\pa_{x_1}+\nu]\partial_{x_1}f\\
&\quad=\frac{1}{p}\alpha_\vps^{\vho p}e^{p\zeta'|v|^2}|\partial_{x_1}f|^{p-1}\Big(\frac{v_w(t)}{1+x_w(t)}\partial_{x_1}f+\Gamma(\partial_{x_1}f,f)
+\Gamma(f,\partial_{x_1}f)+K\partial_{x_1}f\Big).
\end{align*}
Applying Lemma \ref{lem green fun} to the above equality yields
\begin{align}\label{1}
&\left\|{W_c}\partial_{x_1}f(t)\right\|_{p}^p+\int_0^t\left\|\nu^{1/p}{W_c}\partial_{x_1}f\right\|_{p}^p{\rm d}s
+\int_0^t|{W_c}\partial_{x_1}f|_{p,+}^p{\rm d}s\notag\\
&\quad \leq \left\|{W_c}\partial_{x_1}f(0)\right\|_{p}^p+\underbrace{\int_0^t|{W_c}\partial_{x_1}f|_{p,-}^p}_{\eqref{1}_{\gamma_-}}{\rm d}s\notag\\
&\qquad+\underbrace{\frac{1}{p}\int_0^t\int_{\Omega\times\mathbb R^3}e^{p\zeta'|v|^2}\alpha_\vps^{\vho p}|\partial_{x_1}f|^{p-1}
\Big(\frac{v_w(t)}{1+x_w(t)}\partial_{x_1}f+\Gamma(\partial_{x_1}f,f)+\Gamma(f,\partial_{x_1}f)
-K\partial_{x_1}f\Big){\rm d}x_1{\rm d}v{\rm d}s}_{\eqref{1}_{\mathcal R}}.
\end{align}
We now estimate \eqref{1}$_{\mathcal R}$. The contribution of term involving $|\partial_{x_1}f|$ in \eqref{1}$_{\mathcal R}$ can be bounded by
\begin{equation}\label{es R1}
\int_0^t|v_w|\left\|{W_c}\partial_{x_1}f\right\|_{p}^p{\rm d}s.	
	\end{equation}
In view of Lemma \ref{es K}, the contribution of $\Gamma(\partial_{x_1}f,f)+\Gamma_{+}(f,\partial_{x_1}f)-K\partial_{x_1}f$
is bounded by
\begin{align}\label{5.1}
(1+\mathop{\sup}_{0\leq s\leq t}&\|w f(s)\|_\infty)\int_0^t\int_{\Omega\times\mathbb R^3}|{W_c}\partial_{x_1}f(v)|^{p-1}\int_{\mathbb R^3}{\bf k}_\theta(v,u) {W_c}(v)|\partial_{x_1}f(u)|{\rm d}u{\rm d}v{\rm d}x_1{\rm d}s\notag\\
=&(1+\mathop{\sup}_{0\leq s\leq t}\|w f(s)\|_\infty)\int_0^t\int_{\Omega\times\mathbb R^3}
|{W_c}\partial_{x_1}f(v)|^{p-1}\int_{|u_1|\leq N}\textbf k_\theta(v,u) {W_c}(v)|\partial_{x_1}f(u)|{\rm d}u{\rm d}v{\rm d}x_1{\rm d}s\notag\\
&+(1+\mathop{\sup}_{0\leq s\leq t}\|w f(s)\|_\infty)\int_0^t\int_{\Omega\times\mathbb R^3}
|{W_c}\partial_{x_1}f(v)|^{p-1}\int_{|u_1|\geq N}\textbf k_\theta(v,u) {W_c}(v)|\partial_{x_1}f(u)|{\rm d}u{\rm d}v{\rm d}x_1{\rm d}s.
\end{align}
For the region $\{|u_1|\geq N\}$ and $0<\tilde\theta\ll\theta$, applying H\"{o}lder's inequality with $\frac{1}{q}+\frac{1}{p}=1$, we obtain
\begin{align*}
\int_{|u_1|\geq N}&\textbf k_{\tilde\theta}(v,u) \alpha_\vps^{\vho}(v_1)|\partial_{x_1}f(u)|{\rm d}u\\
&\leq \alpha_\vps^{\vho}(v_1)\Big(\int_{|u_1|\geq N}\textbf k_{\tilde\theta}(v,u)\frac{1}{[\alpha_\vps(u_1)]^{\vho q}}{\rm d}u \Big)^{1/q}\Big(\int_{|u_1|\geq N}\textbf k_{\tilde\theta}(v,u)|\alpha_\vps^\vho\partial_{x_1}f(u)|^p{\rm d}u\Big)^{1/p}\\
&\lesssim \alpha_\vps^{\vho}(v_1) \Big(\int_{|u_1|\geq N}\textbf k_{\tilde\theta}(v,u)|\alpha_\vps^\vho\partial_{x_1}f(u)|^p{\rm d}u\Big)^{1/p},
\end{align*}
where we have used \eqref{u>N} with $0<\vho q<1$. Hence, the contribution of the region $\{|u_1|\geq N\}$ in \eqref{5.1} can be bounded as
\begin{align}\label{5.2}
\int_0^t&\int_{\Omega\times\mathbb R^3}|\nu^{1/p}{W_c}\partial_{x_1}f(v)|^{p-1}\frac{\alpha_\vps^{\vho}(v_1)}{\nu(v)^{\frac{p-1}{p}}}\int_{|u_1|\geq N}\textbf k_\theta(v,u)\frac{e^{\zeta'|v|^2}}{e^{\zeta'|u|^2}} |e^{\zeta'|u|^2}\partial_{x_1}f(u)|{\rm d}u{\rm d}v{\rm d}x_1{\rm d}s	\notag\\
&\leq\int_0^t\int_\Omega\Big(\int_{\mathbb R^3}|\nu^{1/p}{W_c}\partial_{x_1}f(v)|^p{\rm d}v\Big)^{1/q}\Big(\int_{|u_1|\geq N}|{W_c}\partial_{x_1}f(u)|^p{\rm d}u\int_{\R^3}\textbf k_{\tilde\theta}(v,u){\rm d}v\Big)^{1/p}{\rm d}x_1\notag\\
&\leq \eta\int_0^t\|\nu^{1/p}{W_c}\partial_{x_1}f(s)\|_{p}^p{\rm d}s+C_\eta\int_0^t\|{W_c}\partial_{x_1}f(s)\|_{p}^p{\rm d}s,
\end{align}
due to $\frac{\alpha_\vps^{\vho}(v_1)}{\nu^{\frac{p-1}{p}}(v)}\lesssim 1$ for $0<\vho<\frac{p-1}{p}$ and Lemma \ref{es K}.

The contribution of $\{|u_1|\leq N\}$ in \eqref{5.1} is bounded by
\begin{align}\label{5.3}
	\int_0^t&\int_{\Omega\times\mathbb R^3}|\nu^{1/p}{W_c}\partial_{x_1}f(v)|^{p-1}\int_{|u_1|\leq N}\textbf k_\theta(v,u)\frac{e^{\zeta'|v|^2}}{e^{\zeta'|u|^2}}\frac{\alpha_\vps^\vho(v_1)}{\nu^{\frac{p-1}{p}}(v)}\frac{ |{W_c}\partial_{x_1}f(u)|}{\alpha_\vps^{\vho}(u_1)}{\rm d}u{\rm d}v{\rm d}x_1{\rm d}s\notag	\\
	\leq& \int_0^t\|\nu^{1/p}{W_c}\partial_{x_1}f(s)\|_p^{p-1}\times\Big[\int_{\Omega\times\mathbb R^3}\Big(\underbrace{\int_{|u_1|\leq N}\textbf k_{\tilde\theta}(v,u)\frac{|{W_c}\partial_{x_1}f(u)|}{\alpha_\vps^{\vho}(u_1)}{\rm d}u}_{\eqref{5.3}_*}\Big)^p{\rm d}v{\rm d}x_1\Big]^{1/p}{\rm d}s.
\end{align}
Applying H\"{o}lder's inequality again, \eqref{5.3}$_*$ can be bounded by
\begin{align*}
	\|{W_c}\partial_{x_1}f(\cdot)\|_{L^p(\mathbb R^3)}\Big(\int_{\mathbb R^3}\frac{e^{-q\tilde{\theta}|v-u|^2}}{|v-u|^{q}}\frac{1_{|u_1|\leq N}}{\alpha_\vps^{\vho q}(u_1)}{\rm d}u\Big)^{1/q}
	\lesssim \|{W_c}\partial_{x_1}f(\cdot)\|_{L^p(\mathbb R^3)}\Big\{\int_{\mathbb R}\frac{1_{|u_1|\leq N}}{\alpha_\vps^{\vho q}(u_1)}{\rm d}u_1\Big\}^{1/q},
\end{align*}
where we have used $\int_{\mathbb R^2}\frac{e^{-q\tilde{\theta}|v'-u'|^2}}{|v'-u'|^{q}}{\rm d}u'<\infty$ for $q=\frac{p}{p-1}<2$.
On the other hand, since $\varrho q<1$, from \eqref{u<N}, it follows that
\begin{equation*}
\Big(\int_{\mathbb R}\frac{1_{|v_1|\leq N}}{\alpha_\vps^{\vho q}(s,x_1,v_1)}{\rm d}v_1\Big)^{1/q}\lesssim (s+1)^{1/q}.	
	\end{equation*}
Consequently, we bound \eqref{5.3}$_*$ by
\begin{equation}
	\|{W_c}\partial_{x_1}f(\cdot)\|_{L^p(\mathbb R^3)}(s+1)^{1/q}.\notag
\end{equation}
Therefore,
\begin{align}\label{5.5}
\int_0^t&\int_{\Omega\times\mathbb R^3}|{W_c}\partial_{x_1}f(v)|^{p-1}\int_{|u_1|\leq N}\textbf k_\theta(v,u) {W_c}(v)|\partial_{x_1}f(u)|{\rm d}u{\rm d}v{\rm d}x_1{\rm d}s\notag\\
\leq&	 \eta\int_0^t\|\nu^{1/p}{W_c}\partial_{x_1}f(s)\|_{p}^p{\rm d}s+C_\eta\int_0^t(s+1)^{p-1}\|{W_c}\partial_{x_1}f(s)\|_{p}^p{\rm d}s.
	\end{align}
It remains now to compute the contribution of $|\Gamma_{-}(f,\partial_{x_1}f)|$. By using Proposition \ref{es K}, this term can be bounded by
\begin{equation}\label{5.6}
\mathop{\sup}_{0\leq s\leq t}\left\|w f(s)\right\|_\infty\int_0^t\|\nu^{1/p}{W_c}\partial_{x_1}f(s)\|_{p}^p{\rm d}s.
\end{equation}
Combining \eqref{es R1}, \eqref{5.2} \eqref{5.5}, and \eqref{5.6} together, we conclude
\begin{align}\label{es R}
\eqref{1}_{\mathcal R}&\lesssim (\eta+\mathop{\sup}_{0\leq s\leq t}\left\|w f(s)\right\|_\infty)\int_0^t\|\nu^{1/p}{W_c}\partial_{x_1}f(s)\|_{p}^p{\rm d}s\notag\\
&\qquad+C_\eta\int_0^t(s+1)^{p-1}\|{W_c}\partial_{x_1}f(s)\|_{p}^p{\rm d}s.
\end{align}

Now we turn to the estimate of $\eqref{1}_{\gamma_-}$. We first compute the $x_1$-derivative of the solution. Differentiating \eqref{drbl-f} and \eqref{drb-f} with respect to $t$ and using \eqref{PBE}, we obtain
\begin{align}\label{bc par_tf(0)}
	\partial_tf(t,0,v)|_{v_1>0}=&\sqrt{2\pi\mu (v)}\int_{u_1<0}\Big\{-\frac{u_1}{1+x_w}\partial_{x_1}f-Lf+\Gamma(f,f)\Big\}(t,0,u)\notag\\
&\qquad\qquad\times\sqrt{\mu(u)}|u_1|{\rm d}u,
\end{align}
and\begin{align}\label{bc par_tf(1)}
	\partial_tf(t,1,v)|_{v_1-v_w(t)<0}
	=&\sqrt{2\pi}\frac{\partial_t\mu_{w}(v)}{\sqrt{\mu (v)}}\int_{u_1-v_w(t)>0}f(t,1,u)\sqrt{\mu(u)}|u_1-v_w(t)|{\rm d}u\notag\\
	&-\sqrt{2\pi}\frac{\mu_{w}(v)}{\sqrt{\mu (v)}}v_w'(t)\int_{u_1-v_w(t)>0}f(t,1,u)\sqrt{\mu(u)}{\rm d}u\notag\\
	&+\sqrt{2\pi}\frac{\mu_{w}(v)}{\sqrt{\mu (v)}}\int_{u_1-v_w(t)>0}\Big\{-\frac{u_1-v_w}{1+x_w}\partial_{x_1}f-Lf+\Gamma(f,f)\Big\}(t,1,u)\notag\\
	&\qquad\qquad\times\sqrt{\mu(u)}|u_1-v_w(t)|{\rm d}u\notag\\ &+\sqrt{2\pi}\frac{\partial_t\mu_{w}(v)}{\sqrt{\mu (v)}}\int_{u_1-v_w(t)>0}\mu(u)|u_1-v_w(t)|{\rm d}u\notag\\
    &-\sqrt{2\pi}\frac{\mu_{w}(v)}{\sqrt{\mu (v)}}v_w'(t)\int_{u_1-v_w(t)>0}\mu(u){\rm d}u.
\end{align}
From \eqref{bc par_tf(0)}, \eqref{bc par_tf(1)} and\begin{equation*}
\partial_{x_1}f=-\frac{1+x_w}{v_1-v_wx_1}\Big\{\partial_t f+Lf-\Gamma(f,f)\Big\},
\end{equation*}
we have that for $(x_1,v)\in\gamma_-$
\begin{align}\label{es par_x f(0)}
\left|\partial_{x_1}f(t,0,v)|_{v_1>0}\right|&\lesssim \frac{1}{|v_1|}\left\{|(\pa_tf+\nu f)(t,0,v)|_{v_1>0}|+|\Gamma(f,f)|+|Kf|\right\}\notag\\
&\lesssim  \langle v\rangle\sqrt{\mu(v)}\frac{1}{|v_1|}
\int_{\gamma^0_+}\Big\{\langle u_1\rangle|\partial_{x_1}f|+ \langle u\rangle|f|\notag\\
&\qquad\qquad\qquad+(1+\left\|w f\right\|_\infty)\int_{\mathbb R^3}\textbf k_\theta(u,u_*)|f(u_*)|{\rm d}u_*\Big\}\sqrt{\mu(u)}{{\rm d}\tilde\gamma}\notag\\
&\quad+\langle v\rangle w^{-1}(v)\frac{1}{|v_1|}\|wf(t)\|_\infty,
\end{align}
and
\begin{align}\label{es par_x f(1)}
	\left|\partial_{x_1}f(t,1,v)|_{v_1-v_w(t)<0}\right|&\lesssim \frac{1}{|v_1-v_w(t)|}\left\{|(\pa_tf+\nu f)(t,1,v)|_{v_1-v_w(t)<0}|+|\Gamma(f,f)|+|Kf|\right\}\notag\\
	&\lesssim  \langle v\rangle\sqrt{\mu(v)}\frac{1}{|v_1-v_w(t)|}
	\int_{\gamma^1_+}\Big\{\langle u_1\rangle|\partial_{x_1}f|+ \langle u\rangle|f|\notag\\
	&\qquad\qquad\qquad+(1+\left\|w f\right\|_\infty)\int_{\mathbb R^3}\textbf k_\theta(u,u_*)|f(u_*)|{\rm d}u_*\Big\}\sqrt{\mu(u)}{\rm d}\tilde\gamma\notag\\
	&\quad+\frac{1}{|v_1-v_w(t)|}\left\{\langle v\rangle w^{-1}(v)\|wf(t)\|_\infty+\sqrt{\mu(v)}|v_w'(t)|\right\}.
\end{align}
As a consequence of \eqref{al on gamma}, \eqref{es par_x f(0)}, and \eqref{es par_x f(1)}, and by applying Lemma \ref{es K}, we obtain
\begin{align}\label{5.7}
	\int_{\gamma_-}&|v_1-v_w(t)x_1|^{\vho p}|e^{\zeta'|v|^2}\partial_{x_1}f(t,x_1,v)|^p|v_1-v_w(t)x_1|{\rm d}v\notag\\
	&\lesssim\int_{\gamma_-}e^{p\zeta'|v|^2} \mu^{\frac{p}{2}}(v)|v_1-v_w(t)x_1|^{(\vho-1) p+1}\Big|\int_{\gamma_+}\Big\{\langle u_1\rangle|\partial_{x_1}f|+{\langle u\rangle |f|}\notag\\
	&\qquad+(1+\left\|w f\right\|_\infty)\int_{\mathbb R^3}\textbf k_\theta(u,u_*)|f(u_*)|{\rm d}u_*\Big\}\sqrt{\mu(u)}{{\rm d}\tilde\gamma+\langle v\rangle\mu^{-\frac{1}{2}}w^{-1}\|wf(t)\|_\infty}+|v_w'(t)|\Big|^p{\rm d}v.
\end{align}
Note that since $(\vho-1)p+1>-1$, we have $|v_1-v_w(t)x_1|^{(\vho-1)p+1}\in L^1_{loc}(\mathbb R^3)$.
For the first term on the right hand side of \eqref{5.7}, we split the region $\ga_+$ into $\gamma_+^\vps\cup\gamma_+\setminus \gamma_+^\vps$,  where $\gamma_+^\vps$ is defined as \eqref{def gamma_+^vps}. By H\"{o}lder's inequality,
\begin{align}\label{5.8}
\int_{\gamma_-}&e^{p\zeta'|v|^2} \mu^{\frac{p}{2}}(v)|v_1-v_w(t)x_1|^{(\vho-1) p+1}\Big|\int_{\gamma_+}\langle u_1\rangle|\partial_{x_1}f|\sqrt{\mu(u)}{\rm d}\tilde\gamma\Big|^p{\rm d}v\notag\\	\lesssim&\Big\{\int_{\gamma_+}|{W_c}\partial_{x_1}f(s,x_1,u)|W^{-1}_c(s,x_1,u)\langle u_1\rangle\sqrt{\mu(u)}|u_1-v_w(t)x_1|{\rm d}u\Big\}^p\notag\\
	\lesssim& \Big\{\int_{\gamma_+^\vps}|{W_c}\partial_{x_1}f(s,x_1,u)|^p|u_1-v_w(t)x_1|{\rm d}u\Big\}\notag\\
	&\qquad\quad\times\Big\{\int_{\gamma_+^\vps}\alpha^{-\vho q}_\vps(u_1)|u_1-v_w(t)x_1|\mu^{\frac{q}{4}}(u){\rm d}u\Big\}^{p/q}\notag\\
	&+\Big\{\int_{\gamma_+\setminus\gamma_+^\vps}|{W_c}\partial_{x_1}f(s,x_1,u)|^p\mu^{\frac{p}{8}}(u)|u_1-v_w(t)x_1|{\rm d}u\Big\}\notag\\
	&\qquad\quad\times \Big\{\int_{\gamma_+\setminus\gamma_+^\vps}\alpha^{-\vho q}_\vps(u_1)|u_1-v_w(t)x_1|\mu^{\frac{q}{8}}(u){\rm d}u\Big\}^{p/q},
\end{align}
where $q=\frac{p}{p-1}$. Note that, in general, $\alpha_\vps(s,x_1,u_1) \neq |u_1 - v_w(s)x_1|$ for $(x_1,u)\in \gamma_+$.
By \eqref{lb-tb} in Lemma \ref{lem trace}, for $(x_1,u)\in \gamma_+ \setminus \gamma_+^\vps$, we have
\[
t_b(s,x_1,u_1) \geq \sqrt{\tfrac{2}{5}}\,\vps =: \vps_1.
\]
Using the change of variable
\begin{equation*}
	u_1\in\{u_1\in\mathbb R:t_{\textbf b}(s,x_1,u_1)\leq s+1\}\rightarrow t_{\textbf b}\in\mathbb R,
\end{equation*}
we obtain, for $0<\vho<1$,
\begin{align*}
\int_{\gamma_+\setminus\gamma_+^\vps\cap\{|u_1|\leq N\}}\frac{1_{t_{\textbf b}\leq s+1}}{\alpha^\vho_\vps(s,x_1,u_1)}{\rm d}u_1&\leq \int\frac{|u_1-v_w(s-t_{\textbf b})x_{1\textbf b}|^{1-\vho}}{t_{\textbf b}}1_{\vps_1\leq t_{\textbf b}\leq s+1}{\rm d}t_{\textbf b}\\
&\leq \int_{\vps_1}^{s+1}\Big[\frac{1}{t_{\textbf b}^{2-\vho}}+\frac{|v_w(s-t_{\textbf b})|}{t_{\textbf b}}\Big]{\rm d}t_{\textbf b}\lesssim C_\vps,
\end{align*}
where we have used the fact that
\begin{equation*}
u_1=\frac{1}{t_{\textbf b}}	\big((1+x_w(s-t_{\textbf b}))x_{1\textbf b}-(1+x_w(s))x_1\big).
	\end{equation*} Since $1_{\gamma_+^\vps}\downarrow0$ and $0<\vho q<1$, applying the dominant convergence theorem and the above inequality, we further have
\begin{align}\label{5.9}
	\eqref{5.8}&\lesssim o(1)\int_{\gamma_+^\vps}|{W_c}\partial_{x_1}f(s,x_1,u)|^p|u_1-v_w(t)x_1|{\rm d}u\notag\\
	&\quad+\int_{\gamma_+\setminus\gamma_+^\vps}|{W_c}\partial_{x_1}f(s,x_1,u)|^p\mu^{\frac{p}{8}}(u)|u_1-v_w(t)x_1|{\rm d}u.
\end{align}
By Lemma \ref{lem trace}, we bound the second term on the right hand side of \eqref{5.9} as 
\begin{align}\label{5.10}
\int_0^t&\int_{\gamma_+\setminus\gamma_+^\vps}|{W_c}\partial_{x_1}f(s,x_1,u)|^p\mu^{\frac{p}{8}}(u)|u_1-v_w(t)x_1|{\rm d}u{\rm d}s\notag\\
&\lesssim \|{W_c}\partial_{x_1}f(0)\mu^{\frac{1}{8}}\|_{p}^p	+\int_0^t\|{W_c}\partial_{x_1}f\mu^{\frac{1}{8}}\|_{p}^p{\rm d}s\notag\\
&\quad+\int_0^t\int_{\Omega\times\mathbb R^3}(1+x_w(s))[\partial_s+\frac{v_1-v_w(s)x_1}{1+x_w(s)}\partial_{x_1}+\nu](\mu^{\frac{p}{8}}|{W_c}\partial_{x_1}f|^p){\rm d}v{\rm d}x_1{\rm d}s\notag\\
&\lesssim \|{W_c}\partial_{x_1}f(0)\mu^{\frac{1}{8}}\|_{p}^p	+\int_0^t\|{W_c}\partial_{x_1}f\mu^{\frac{1}{8}}\|_{p}^p{\rm d}s\notag\\
&\quad+\int_0^t\int_{\Omega\times\mathbb R^3}p{W_c}^p\mu^{\frac{p}{8}}|\partial_{x_1}f|^{p-1}\Big|\frac{v_w(s)}{1+x_w(s)}\partial_{x_1}f+\Gamma(\partial_{x_1}f,f)+\Gamma(f,\partial_{x_1}f)
-K\partial_{x_1}f\Big|{\rm d}v{\rm d}x_1{\rm d}s\notag\\
&\lesssim \|{W_c}\partial_{x_1}f(0)\mu^{\frac{1}{8}}\|_{p}^p	+\int_0^t\|{W_c}\partial_{x_1}f\mu^{\frac{1}{8}}\|_{p}^p{\rm d}s
+\int_0^t(s+1)^{p-1}\|{W_c}\partial_{x_1}f(s)\|_{p}^p{\rm d}s\notag\\
&\quad+(o(1)+\mathop{\sup}_{0\leq s\leq t}\left\|w f(s)\right\|_\infty)\int_0^t\|\nu^{1/p}{W_c}\partial_{x_1}f(s)\|_{p}^p{\rm d}s,
\end{align}
where we have employed a similar calculation as in the derivation of \eqref{5.5}.

For the remaining terms in \eqref{5.7}, by H\"{o}lder's inequality with $\frac{1}{p}+\frac{1}{q}=1$, we obtain
\begin{align}\label{5.11}
	\Big\{\int_{\gamma_+}&\Big({\langle u\rangle} |f|
+(1+\left\|w f\right\|_\infty)\int_{\mathbb R^3}\frac{\textbf k^{1/{q}}_\theta(u,u_*)}{|u_{*1}-v_w(s)x_1|^{1/p}}\textbf k^{1/p}_\theta(u,u_*)|f(u_*)|\notag\\
&\quad\times|u_{*1}-v_w(s)x_1|^{1/p}{\rm d}u_*\Big)\sqrt{\mu(u)}{\rm d}\tilde\gamma\Big\}^p\notag\\
&\lesssim \int_{\gamma_+}|f|^p|u_1-v_w(s)x_1|{\rm d}u+(1+\left\|w f\right\|_\infty)\Big(\int_{\mathbb R^3}\textbf k_\theta(u,u_*)|u_{*1}-v_w(s)x_1|^{-q/p}{\rm d}u_*\Big)^{p/{q}}\notag\\
&\quad\times\int_{\mathbb R^3}\int_{\mathbb R^3}\textbf k_\theta(u,u_*)|f(u_*)|^p|u_{*1}-v_w(s)x_1|{\rm d}u_*{\rm d}u\notag\\
&\lesssim (1+\left\|w f\right\|_\infty)\int_{\gamma_+}|f|^p|u_1-v_w(s)x_1|{\rm d}u.
\end{align}
Combining \eqref{5.7}, \eqref{5.9}, \eqref{5.10} and \eqref{5.11}, we arrive at
\begin{align}\label{es gamma_-}
\eqref{1}_{\gamma_-}\lesssim& \|{W_c}\partial_{x_1}f(0)\|_p^p\notag\\&+ o(1)\int_0^t\int_{\gamma_+^\vps}|{W_c}\partial_{x_1}f(s,x_1,u)|^p|u_1-v_w(s)x_1|{\rm d}u{\rm d}s\notag\\
&+(o(1)+\mathop{\sup}_{0\leq s\leq t}\left\|w f(s)\right\|_\infty)\int_0^t\|\nu^{1/p}{W_c}\partial_{x_1}f(s)\|_p^p{\rm d}s+\int_0^t(s+1)^{p-1}\|{W_c}\partial_{x_1}f(s)\|_p^p{\rm d}s\notag\\&+(1+\left\|w f\right\|_\infty)\int_0^t\int_{\gamma_+}|f|^p|u_1-v_w(s)x_1|{\rm d}u{\rm d}s+{\int_0^t\|wf(s)\|_\infty^p{\rm d}s}+\int_0^t|v_w'(s)|^p{\rm d}s.
\end{align}
From  \eqref{es R} and \eqref{es gamma_-}, we get from \eqref{1} that
\begin{align}\label{5.26}
&\left\|{W_c}\partial_{x_1}f(t)\right\|_{p}^p+\int_0^t\left\|\nu^{1/p}{W_c}\partial_{x_1}f\right\|_{p}^p{\rm d}s+\int_0^t|{W_c}\partial_{x_1}f|_{p,+}^p{\rm d}s\notag\\
&\quad\lesssim \|{W_c}\partial_{x_1}f(0)\|_p^p+\int_0^t(s+1)^{p-1}\|{W_c}\partial_{x_1}f(s)\|_{p}^p{\rm d}s\notag\\
&\qquad+(1+\left\|w f\right\|_\infty)\int_0^t\int_{\gamma_+}|f|^p|u_1-v_w(s)x_1|{\rm d}u{\rm d}s+{\int_0^t\|wf(s)\|_\infty^p{\rm d}s}	+\int_0^t|v_w'(s)|^p{\rm d}s.
\end{align}
Combining this with the bound for $\int_0^t|f|^p_{p,+}{\rm d}s$ obtained in \eqref{es L^p} and applying Gr\"{o}nwall's inequality, we deduce \eqref{loc-reg-f}.
Hence, the proof of Proposition \ref{loc-reg} is complete.
\end{proof}

\section{$L^{1+\de}$ stability}\label{sec-sta}
In this section, we derive an {\it a priori} $L^{1+\de}$ stability estimate for the solution of the system \eqref{PBE}, \eqref{trb-ode}, \eqref{ic f}, \eqref{drbl-f} and \eqref{drb-f}, where $0<\de\ll 1$. As a consequence, this stability estimate implies the uniqueness of the solution constructed in Section~\ref{glex-sec}.
We begin by establishing the following $L^{1+\de}$ stability result for the free boundary $[x_w,v_w]$.
\begin{lemma}\label{lem es px_w}
	Assume $[x_{w,f},v_{w,f}]$ satisfies
	\begin{align}\label{equ x_wf}
		\frac{{\rm d}}{{\rm d}t}x_{w,f}(t)&=v_{w,f}(t),\notag\\
		\frac{{\rm d}}{{\rm d}t}v_{w,f}(t)&=-\kappa x_{w,f}(t)-\mathcal M\left(P_r[\mu+\sqrt{\mu}f]-P_l[\mu+\sqrt{\mu}f] \right),\notag\\
		x_{w,f}(0)&=x_{w,f,0}, v_{w,f}(0)=v_{w,f,0},
	\end{align}
where $f$ is given.
Then for $0<\delta\ll 1$, there exists a constant $C>0$ such that
\begin{align}\label{es px_w1}
|x_{w,f}(t)&-x_{w,\tilde f}(t)|^{1+\delta}+|v_{w,f}(t)-v_{w,\tilde f}(t)|^{1+\delta}\notag\\
&\lesssim e^{Ct}(|x_{w,f,0}-x_{w,\tilde f,0}|^{1+\delta}+|v_{w,f,0}-v_{w,\tilde f,0}|^{1+\delta})\notag\\
&\quad+e^{Ct}\int_0^t\int_{v_1-v_{w,f}(s)>0}|f(s,1,v)-\tilde f(s,1,v)|^{1+\delta}|v_1-v_{w,f}(s)|{\rm d}v{\rm d}s,
\end{align}
where $[x_{w,\tilde{f}},v_{w,\tilde{f}}]$ denotes the solution to \eqref{equ x_wf} with $f$ replaced by $\tilde{f}.$
\end{lemma}
\begin{proof}
Note that
\begin{align}\label{pequ x_wf}
	\frac{{\rm d}}{{\rm d}t}[x_{w,f}-x_{w,\tilde f}](t)&=[v_{w,f}-v_{w,\tilde f}](t),\notag\\
	\frac{{\rm d}}{{\rm d}t}[v_{w,f}-v_{w,\tilde f}](t)&=-\kappa[x_{w,f}-x_{w,\tilde f}](t)-\mathcal M\left(P_r[\mu+\sqrt{\mu}f]-P_l[\mu+\sqrt{\mu}f] \right)\notag\\
	&\quad+\mathcal M\left(P_r[\mu+\sqrt{\mu}\tilde f]-P_l[\mu+\sqrt{\mu}\tilde f] \right),\notag\\
	[x_{w,f}-x_{w,\tilde f}](0)&=x_{w,f,0}-x_{w,\tilde f,0}, [v_{w,f}-v_{w,\tilde f}](0)=v_{w,f,0}-v_{w,\tilde f,0}.
\end{align}
We now compute the boundary terms on the right hand side of \eqref{pequ x_wf}.
By \eqref{P_rf-P_lf}, we have
\begin{align*}
	P_r&[\mu+\sqrt{\mu}f]-P_l[\mu+\sqrt{\mu}f]\notag\\
	&=\int_{v_1-v_{w,f}(t)<0}[v_1-v_{w,f}(t)]^2\mu(v){\rm d}v-\frac{\sqrt{2\pi}}{2}\int_{v_1-v_{w,f}(t)<0}[v_1-v_{w,f}(t)]\mu(v){\rm d}v\notag\\
	&\quad-\int_{v_1-v_{w,f}(t)>0}[v_1-v_{w,f}(t)]^2\mu(v){\rm d}v-\frac{\sqrt{2\pi}}{2}\int_{v_1-v_{w,f}(t)>0}[v_1-v_{w,f}(t)]\mu(v){\rm d}v\notag\\
	&\quad-\int_{v_1-v_{w,f}(t)>0}[v_1-v_{w,f}(t)]^2\sqrt{\mu(v)}f(t,1,v){\rm d}v\notag\\
	&\quad-\frac{\sqrt{2\pi}}{2}\int_{v_1-v_{w,f}(t)>0}[v_1-v_{w,f}(t)]\sqrt{\mu(v)}f(t,1,v){\rm d}v.
\end{align*}
Hence, we get the following bound
\begin{align}
|P_r&[\mu+\sqrt{\mu}f]-P_l[\mu+\sqrt{\mu}f]-P_r[\mu+\sqrt{\mu}\tilde f]+P_l[\mu+\sqrt{\mu}\tilde f]|\notag\\&\leq C|v_{w,f}(t)-v_{w,\tilde f}(t)|\notag\\
&\quad+\Big|\int_{v_1-v_{w,f}(t)>0}[v_1-v_{w,f}(t)]^2\sqrt{\mu(v)}f(t,1,v){\rm d}v\notag\\
&\qquad-\int_{v_1-v_{w,\tilde f}(t)>0}[v_1-v_{w,\tilde f}(t)]^2\sqrt{\mu(v)}\tilde f(t,1,v){\rm d}v\Big|\notag\\
&\quad+\frac{\sqrt{2\pi}}{2}\Big|\int_{v_1-v_{w,f}(t)>0}[v_1-v_{w,f}(t)]\sqrt{\mu(v)}f(t,1,v){\rm d}v\notag\\
&\qquad-\int_{v_1-v_{w,\tilde f}(t)>0}[v_1-v_{w,\tilde f}(t)]\sqrt{\mu(v)}\tilde f(t,1,v){\rm d}v\Big|.\label{prl-bd}
\end{align}
For the second term on the right hand side of \eqref{prl-bd}, by h\"{o}lder's inequality, it follows that
\begin{align*}
	&\Big|\int_{v_1-v_{w,f}(t)>0}[v_1-v_{w,f}(t)]^2\sqrt{\mu(v)}f(t,1,v){\rm d}v\notag\\
	&\qquad\qquad-\int_{v_1-v_{w,\tilde f}(t)>0}[v_1-v_{w,\tilde f}(t)]^2\sqrt{\mu(v)}\tilde f(t,1,v){\rm d}v\Big|\\
	&\quad\leq 	\Big|\int_{v_1-v_{w,f}(t)>0}[v_1-v_{w,f}(t)]^2\sqrt{\mu(v)}f(t,1,v){\rm d}v\notag\\
	&\qquad\qquad-\int_{v_1-v_{w,f}(t)>0}[v_1-v_{w,f}(t)]^2\sqrt{\mu(v)}\tilde f(t,1,v){\rm d}v\Big|\\
	&\qquad+\Big|\int_{v_1-v_{w,f}(t)>0}[v_1-v_{w,f}(t)]^2\sqrt{\mu(v)}\tilde f(t,1,v){\rm d}v\notag\\
	&\qquad\qquad-\int_{v_1-v_{w,\tilde f}(t)>0}[v_1-v_{w,\tilde f}(t)]^2\sqrt{\mu(v)}\tilde f(t,1,v){\rm d}v\Big|\\
	&\qquad\lesssim\left\|w \tilde f(t)\right\|_\infty|v_{w,f}(t)-v_{w,\tilde f}(t)|+C\Big\{\int_{v_1-v_{w,f}(t)>0}|f(t,1,v)-\tilde f(t,1,v)|^{1+\delta}|v_1-v_{w,f}(t)|{\rm d}v\Big\}^{\frac{1}{1+\delta}}.
\end{align*}
Similarly, we bound the third term on the right hand side of \eqref{prl-bd} as
\begin{align*}
	&\frac{\sqrt{2\pi}}{2}\Big|\int_{v_1-v_{w,f}(t)>0}[v_1-v_{w,f}(t)]\sqrt{\mu(v)}f(t,1,v){\rm d}v\notag\\
	&\qquad-\int_{v_1-v_{w,\tilde f}(t)>0}[v_1-v_{w,\tilde f}(t)]\sqrt{\mu(v)}\tilde f(t,1,v){\rm d}v\Big|\\
	&\quad\lesssim \left\|w \tilde f(t)\right\|_\infty|v_{w,f}(t)-v_{w,\tilde f}(t)|+C\Big\{\int_{v_1-v_{w,f}(t)>0}|f(t,1,v)-\tilde f(t,1,v)|^{1+\delta}|v_1-v_{w,f}(t)|{\rm d}v\Big\}^{\frac{1}{1+\delta}}.
\end{align*}
Consequently, it follows that
\begin{align}
|P_r&[\mu+\sqrt{\mu}f]-P_l[\mu+\sqrt{\mu}f]-P_r[\mu+\sqrt{\mu}\tilde f]+P_l[\mu+\sqrt{\mu}\tilde f]|\notag\\&\leq C|v_{w,f}(t)-v_{w,\tilde f}(t)|\notag\\
&\quad+	C\Big\{\int_{v_1-v_{w,f}(t)>0}|f(t,1,v)-\tilde f(t,1,v)|^{1+\delta}|v_1-v_{w,f}(t)|{\rm d}v\Big\}^{\frac{1}{1+\delta}}.\label{prl-bd2}
\end{align}
Multiplying \eqref{pequ x_wf}$_1$ by $|x_{w,f}-x_{w,\tilde f}|^{\delta-1}(x_{w,f}-x_{w,\tilde f})$ and \eqref{pequ x_wf}$_2$ by $|v_{w,f}-v_{w,\tilde f}|^{\delta-1}(v_{w,f}-v_{w,\tilde f})$, and applying Young's inequality together with \eqref{prl-bd2}, we obtain
\begin{align}
\frac{1}{1+\delta}&\frac{{\rm d}}{{\rm d}t}(|x_{w,f}(t)-x_{w,\tilde f}(t)|^{1+\delta}+|v_{w,f}(t)-v_{w,\tilde f}(t)|^{1+\delta})\notag\\
&\lesssim (|x_{w,f}(t)-x_{w,\tilde f}(t)|^{1+\delta}+|v_{w,f}(t)-v_{w,\tilde f}(t)|^{1+\delta})\notag\\
&\quad+\int_{v_1-v_{w,f}(t)>0}|f(t,1,v)-\tilde f(t,1,v)|^{1+\delta}|v_1-v_{w,f}(t)|{\rm d}v.\notag
\end{align}
An application of Gr\"onwall's inequality then yields \eqref{es px_w1}. This completes the proof of Lemma \ref{lem es px_w}.
\end{proof}
\begin{proposition}\label{prop es f-g}
Suppose that both $[f,x_{w_f},v_{w,f}]$ and $[\tilde{f},x_{w,\tilde{f}},v_{w,\tilde{f}}]$ solve \eqref{PBE}, \eqref{trb-ode}, \eqref{drbl-f}, and \eqref{drb-f}, and satisfy \eqref{es decay-drbcl} and \eqref{es par f-drbcl}, with the initial data
$[f_0(x_1,v),(x_{w,f,0},v_{w,f,0})]$ and $[\tilde{f}_0(x_1,v),(x_{w,\tilde{f},0},v_{w,\tilde{f},0})]$, respectively. Then it holds that
for $0<\de\ll1$
\begin{align}
&\left\|f(t)-\tilde f(t)\right\|_{1+\delta}^{1+\delta}+\int_0^t\int_{\gamma_{+,f}}|f-\tilde f|^{1+\delta}{\rm d}\tilde{\gamma}_f{\rm d}s\notag\\
&\quad\lesssim\left\|f_0-\tilde f_0\right\|_{1+\delta}^{1+\delta}+e^{Ct}(|x_{w,f,0}-x_{w,\tilde f,0}|^{1+\delta}+|v_{w,f,0}-v_{w,\tilde f,0}|^{1+\delta})\notag\\
&\quad\quad+te^{Ct}\int_0^t\int_{\gamma_{+,f}}|f-\tilde f|^{1+\delta}{\rm d}\tilde{\gamma}_f{\rm d}s+\int_0^t e^{C(s+1)^{p}}\|f-\tilde f\|_{1+\delta}^{1+\delta}{\rm d}s,\notag
\end{align}
where $C>0$,
\begin{equation*}
	\gamma_{+,f}=\{x_1\in\{0, 1\}, v\in\mathbb R^3:(v_1-v_{w,f}(t)x_1)n(x_1)>0\},
\end{equation*}
and the boundary measure ${\rm d}\tilde{\gamma}_f$ is defined by
\begin{equation*}
{\rm d}\tilde{\gamma}_f=|v_1-v_{w,f}(t)x_1|{\rm d}v, \ x_1=0,1.
\end{equation*}
\end{proposition}

\begin{proof}
	Assume that both $f$ and $\tilde f $ are solutions of \eqref{PBE} and \eqref{trb-ode}. Then it holds
 \begin{align}
\pa_t [f-\tilde f]+\frac{v_1-v_{w,f}(t)x_1}{1+x_{w,f}(t)}\pa_{x_1} [f-\tilde f]+\nu[f-\tilde f]&=\Big[\frac{v_1-v_{w,\tilde f}(t)x_1}{1+x_{w,\tilde f}(t)}-\frac{v_1-v_{w,f}(t)x_1}{1+x_{w,f}(t)}\Big]\pa_{x_1}\tilde f\notag\\
&\quad+K[f-\tilde f]+\Ga(f,f)-\Ga(\tilde f,\tilde f).	\notag	
	\end{align}
Define
\begin{equation*}
	\tilde{\mathcal A}:=\frac{v_1-v_{w,\tilde f}(t)x_1}{1+x_{w,\tilde f}(t)}-\frac{v_1-v_{w,f}(t)x_1}{1+x_{w,f}(t)}.
\end{equation*}
By Lemma \ref{lem green fun} with $p=1+\delta$, we obtain
\begin{align}
	&\left\|f(t)-\tilde f(t)\right\|_{1+\delta}^{1+\delta}+\int_0^t\int_{\gamma_{+,f}}|f-\tilde f|^{1+\delta}{\rm d}\tilde{\gamma}_f{\rm d}s+\int_0^t\left\|\nu^{\frac{1}{1+\delta}}[f-\tilde f]\right\|_{1+\delta}^{1+\delta}{\rm d}s\notag\\
	&\quad\lesssim \left\|f_0-\tilde f_0\right\|_{1+\delta}^{1+\delta}+\int_0^t\int_{\gamma_{-,f}}|f-\tilde f|^{1+\delta}{\rm d}\tilde{\gamma}_f{\rm d}s\notag\\
	&\qquad+\int_0^t\int_{\Omega\times\mathbb R^3}|\tilde{\mathcal A}\pa_{x_1}\tilde f+K[f-\tilde f]+\Ga(f,f)-\Ga(\tilde f,\tilde f)||f-\tilde f|^\delta{\rm d}x_1{\rm d}v{\rm d}s,\label{f-fti}
\end{align}
where
\begin{equation*}
	\gamma_{-,f}=\{x_1\in\{0, 1\}, v\in\mathbb R^3:(v_1-v_{w,f}(t)x_1)n(x_1)<0\}.
\end{equation*}
For $0<\delta\ll1$, by H\"{o}lder's inequality with $\frac{p-1-\delta}{p(1+\delta)}+\frac{1}{p}+\frac{\delta}{1+\delta}=1$ $(2<p<\infty)$, we have
\begin{align*}
\int_0^t&\int_{\Omega\times\mathbb R^3}|\tilde{\mathcal A}|\pa_{x_1}\tilde f|	|f-\tilde f|^\delta{\rm d}x_1{\rm d}v{\rm d}s\notag\\
&\lesssim \int_0^t\int_{\Omega\times\mathbb R^3}\langle v_1\rangle(|x_{w,f}(s)-x_{w,\tilde f}(s)|+|v_{w,f}(s)-v_{w,\tilde f}(s)|)|\pa_{x_1}\tilde f||f-\tilde f|^{\delta}{\rm d}x_1{\rm d}v{\rm d}s\notag\\
	&\lesssim \int_0^t(|x_{w,f}(s)-x_{w,\tilde f}(s)|+|v_{w,f}(s)-v_{w,\tilde f}(s)|)\left\|\frac{\langle v_1\rangle}{{W_c}(v)}\right\|_{L^{\frac{p(1+\delta)}{p-1-\delta}}_{x_1,v}}\|{W_c}(v)\pa_{x_1}\tilde f\|_{p}\|f-\tilde f\|_{1+\delta}^\delta{\rm d}s.\notag
\end{align*}
Since $0<\varrho\frac{p(1+\delta)}{p-1-\delta}<1$ for $0<\delta\ll1$, by \eqref{u<N}, we obtain
\begin{equation*}
\left\|\frac{\langle v_1\rangle}{{W_c}(v)}\right\|_{L^{\frac{p(1+\delta)}{p-1-\delta}}_{x_1,v}}\leq C(s+1)^{	\frac{p-1-\delta}{p(1+\delta)}}\lesssim s+1.
\end{equation*}
Therefore, employing \eqref{es par f-drbcl}, we get
\begin{align}\label{es mathA}
	\int_0^t\int_{\Omega\times\mathbb R^3}|\tilde{\mathcal A}|\pa_{x_1}\tilde f|	|f-\tilde f|^\delta{\rm d}x_1{\rm d}v{\rm d}s\lesssim&\int_0^t(|x_{w,f}(s)-x_{w,\tilde f}(s)|^{1+\delta}+|v_{w,f}(s)-v_{w,\tilde f}(s)|^{1+\delta}){\rm d}s\notag\\&+\int_0^t e^{C(s+1)^p}\|f-\tilde f\|_{1+\delta}^{1+\delta}{\rm d}s.
\end{align}
Applying Lemma \ref{es K} with $\al=0$ and using \eqref{es decay-drbcl}, one has
\begin{align}\label{es kf-g}
	\int_0^t&\int_{\Omega\times\mathbb R^3}|K[f-\tilde f]+\Ga(f,f)-\Ga(\tilde f,\tilde f)||f-\tilde f|^\delta{\rm d}x_1{\rm d}v{\rm d}s\notag\\
	&\lesssim \mathop{\sup}_{0\leq s\leq t}(1+\|w f(s)\|_\infty+\|w \tilde f(s)\|_\infty)\int_0^t\|f-\tilde f\|_{1+\delta}^{1+\delta}{\rm d}s.
\end{align}
Finally we turn to estimate the incoming boundary term on the right hand side of \eqref{f-fti}. For the boundary at $x_1=0$, by using Lemma \ref{lem trace}, it is clear that
\begin{align}\label{es bi x=0}
\int_0^t&\int_{v_1>0}|(f-\tilde f)(0)|^{1+\delta}|v_1|{\rm d}v{\rm d}s\notag\\
&\leq o(1)\int_0^t\int_{v_1<0}|(f-\tilde f)(t,0,v)|^{1+\delta}|v_1|{\rm d}v+\|(f-\tilde f)(0)\|_{1+\delta}^{1+\delta}\notag\\
&\qquad+\mathop{\sup}_{0\leq s\leq t}(1+\|w f(s)\|_\infty+\|w \tilde f(s)\|_\infty)\int_0^t\|f-\tilde f\|_{1+\delta}^{1+\delta}{\rm d}s.
\end{align}
For the boundary at $x_1=1$, we decompose as
\begin{align*}
	\int_0^t&\int_{v_1-v_{w,f}(s)<0}|(f-\tilde f)(1)|^{1+\delta}|v_1-v_{w,f}(s)|{\rm d}v{\rm d}s\notag\\
	&=\int_0^t\textbf 1_{\{s: v_{w,f}(s)\leq v_{w,\tilde f}(s)\}}\int_{v_1-v_{w,f}(s)<0}|(f-\tilde f)(1)|^{1+\delta}|v_1-v_{w,f}(s)|{\rm d}v{\rm d}s\notag\\
	&\quad+\int_0^t\textbf 1_{\{s: v_{w,f}(s)> v_{w,\tilde f}(s)\}}\int_{v_1-v_{w,f}(s)<0}|(f-\tilde f)(1)|^{1+\delta}|v_1-v_{w,f}(s)|{\rm d}v{\rm d}s.
\end{align*}
Next, in the case $v_{w,f}(s)\leq v_{w,\tilde f}(s)$, using a similar analysis as in the proof of Lemma \ref{lem es px_w}, we have
\begin{align}\label{ftf-bd}
	|(f-\tilde f)(1)|_{v_1-v_{w,f}(s)<0}|&\lesssim \sqrt{\mu(v)}|v_{w,f}(s)-v_{w,\tilde f}(s)|\notag\\
	&\quad+\sqrt{\mu(v)}\Big\{\int_{u_1-v_{w,f}(s)>0}|f(s,1,u)-\tilde f(s,1,u)|\sqrt{\mu(u)}|u_1-v_{w,f}(s)|{\rm d}u\Big\}.
\end{align}
For the case $v_{w,f}(s)>v_{w,\tilde f}(s)$, it follows that
\begin{align}
	\int_0^t&\int_{v_1-v_{w,f}(s)<0}|(f-\tilde f)(1)|^{1+\delta}|v_1-v_{w,f}(s)|{\rm d}v{\rm d}s\notag\\
=&\int_0^t\int_{v_1-v_{w,\tilde f}(s)<0}|(f-\tilde f)(1)|^{1+\delta}|v_1-v_{w,f}(s)|{\rm d}v{\rm d}s
\notag\\
&+\int_0^t\int_{v_{w,\tilde f}(s)<v_1<v_{w,f}(s)}|(f-\tilde f)(1)|^{1+\delta}|v_1-v_{w,f}(s)|{\rm d}v{\rm  d}s\notag\\
\leq&\int_0^t\int_{v_1-v_{w,\tilde f}(s)<0}|(f-\tilde f)(1)|^{1+\delta}|v_1-v_{w,f}(s)|{\rm d}v{\rm d}s\notag\\
&+\int_0^t\int_{v_{w,\tilde f}(s)<v_1<v_{w,f}(s)}|(f-\tilde f)(1)|^{1+\delta}|v_{w,f}-v_{w,\tilde f}(s)|{\rm d}v{\rm  d}s\notag\\
\leq&\int_0^t\textbf 1_{\{s: v_{w,f}(s)> v_{w,\tilde f}(s)\}}\int_{v_1-v_{w,\tilde f}(s)<0}|(f-\tilde f)(1)|^{1+\delta}|v_1-v_{w,f}(s)|{\rm d}v{\rm d}s\notag\\
&+(\|wf\|^{1+\delta}_\infty+\|w\tilde f\|^{1+\delta}_\infty)\int_0^t\textbf 1_{\{s: v_{w,f}(s)> v_{w,\tilde f}(s)\}}\int_{v_{w,\tilde f}(s)<v_1<v_{w,f}(s)}e^{-\zeta(v_2^2+v_3^2)}|v_{w,f}(s)-v_{w,\tilde f}(s)|{\rm d}v{\rm  d}s\notag\\
\lesssim& \int_0^t|v_{w,f}(s)-v_{w,\tilde f}(s)|^{1+\delta}{\rm  d}s+\int_0^t\Big\{\int_{u_1-v_{w,f}(s)>0}|f(s,1,u)-\tilde f(s,1,u)|\sqrt\mu|u_1-v_{w,f}(s)|{\rm d}u\Big\}^{1+\delta}{\rm d}s,\label{ftf-bd2}
\end{align}
where, in the last inequality, we have applied the same bound for $|(f-\tilde f)|_{v_1-v_{w,\tilde f}<0}|$ as in \eqref{ftf-bd}. In addition, we also use the fact that $|v_{w,f}-v_{w,\tilde f}|^2\leq |v_{w,f}-v_{w,\tilde f}|^{1+\delta}$, which follows from $|v_{w,f}-v_{w,\tilde f}|\leq |v_{w,f}|+|v_{w,\tilde f}|\ll1$ and $\delta\ll1$.

Utilizing \eqref{ftf-bd} and \eqref{ftf-bd2},  we then deuce
\begin{align}\label{es bi x=1}
\int_0^t&\int_{v_1-v_{w,f}(s)<0}|(f-\tilde f)(1)|^{1+\delta}|v_1-v_{w,f}(s)|{\rm d}v{\rm d}s\notag\\
&\lesssim \int_0^t|v_{w,f}(s)-v_{w,\tilde f}(s)|^{1+\delta}{\rm d}s\notag\\
&\quad+\int_0^t\Big\{\int_{u_1-v_{w,f}(s)>0}|f(s,1,u)-\tilde f(s,1,u)|\sqrt\mu|u_1-v_{w,f}(s)|{\rm d}u\Big\}^{1+\delta}{\rm d}s\notag\\
&\lesssim \int_0^t|v_{w,f}(s)-v_{w,\tilde f}(s)|^{1+\delta}{\rm d}s\notag\\
&\quad+\int_0^t\Big\{\int_{v_{w,f}(s)<u_1<v_{w,f}(s)+\vps}|f(s,1,u)-\tilde f(s,1,u)|\sqrt{\mu(u)}|u_1-v_{w,f}(s)|{\rm d}u\Big\}^{1+\delta}{\rm d}s\notag\\
&\quad+\int_0^t\Big\{\int_{u_1-v_{w,f}(s)-\vps>0}|f(s,1,u)-\tilde f(s,1,u)|\sqrt{\mu(u)}|u_1-v_{w,f}(s)|{\rm d}u\Big\}^{1+\delta}{\rm d}s\notag\\
&\lesssim o(1)\int_0^t\int_{v_1-v_{w,f}(s)>0}|(f-\tilde f)(1)|^{1+\delta}|v_1-v_{w,f}(s)|{\rm d}v{\rm d}s+\|f_0-\tilde f_0\|_{1+\delta}^{1+\delta}\notag\\
&\quad+\int_0^t|v_{w,f}(s)-v_{w,\tilde f}(s)|^{1+\delta}{\rm d}s\notag\\
&\quad+\mathop{\sup}_{0\leq s\leq t}(1+\|w f(s)\|_\infty
+\|w \tilde f(s)\|_\infty)\int_0^te^{C(s+1)^{p}}\|(f-\tilde f)(s)\|_{1+\delta}^{1+\delta}{\rm d}s,
	\end{align}
where we have used Lemma \ref{lem trace} to control the integral over the domain $v_{w,f}(s)<u_1<v_{w,f}(s)+\vps$.

Now putting \eqref{es mathA}, \eqref{es kf-g}, \eqref{es bi x=0} and \eqref{es bi x=1} together, and using Lemma \ref{lem es px_w}, we obtain
\begin{align*}
	\|f(t)&-\tilde f(t)\|_{1+\delta}^{1+\delta}+\int_0^t\int_{\gamma_{+,f}}|f-\tilde f|^{1+\delta}{\rm d}\tilde{\gamma}_f{\rm d}s+\int_0^t\left\|\nu^{\frac{1}{1+\delta}}[f-\tilde f]\right\|_{1+\delta}^{1+\delta}{\rm d}s\notag\\
\lesssim& \left\|f_0-\tilde f_0\right\|_{1+\delta}^{1+\delta}+\int_0^t(|x_{w,f}(s)-x_{w,\tilde f}(s)|^{1+\delta}+|v_{w,f}(s)-v_{w,\tilde f}(s)|^{1+\delta}){\rm d}s\notag\\
&+\int_0^t e^{C(s+1)^{p}}\|f-\tilde f\|_{1+\delta}^{1+\delta}{\rm d}s\notag\\
	\lesssim&\left\|f_0-\tilde f_0\right\|_{1+\delta}^{1+\delta}+e^{Ct}(|x_{w,f,0}-x_{w,\tilde f,0}|^{1+\delta}+|v_{w,f,0}-v_{w,\tilde f,0}|^{1+\delta})\notag\\
	&+te^{Ct}\int_0^t\int_{\gamma_{+,f}}|f-\tilde f|^{1+\delta}{\rm d}\tilde{\gamma}_f{\rm d}s+\int_0^t e^{C(s+1)^{p}}\|f-\tilde f\|_{1+\delta}^{1+\delta}{\rm d}s.
\end{align*}
This completes the proof of Proposition \ref{prop es f-g}.
\end{proof}

\section{The global existence}\label{glex-sec}
We are able to complete the proof of Theorem \ref{main result} regarding the global existence of the free boundary problem \eqref{PBE}, \eqref{trb-ode}, \eqref{ic f}, \eqref{drbl-f} and \eqref{drb-f} by using the {\it a priori} estimate \eqref{L^inf es-drcl} together with the local existence and the standard continuation argument.

For completeness, it suffices to prove the following result for the local existence.

\begin{theorem}\label{loc}
	Under the conditions listed in Theorem \ref{main result}, there exists a unique non-negative solution $F(t,X_1,v)=\bar F(t,x_1,v)=\mu+\sqrt{\mu}f\geq0$	to the problem \eqref{BE}, \eqref{initial F}, \eqref{drbcl} and \eqref{drbcr} in $[0,T^*(\vps_0))\times\Omega\times\mathbb R^3$ for some $T^*(\vps_0)>0$. Additionally, the solution to the ODEs \eqref{ODE}, \eqref{initial F} exists on the interval on $[0,T^*(\vps_0))$.  Moreover, the pair $[f,x_w,v_w]$ satisfies
	\begin{align}
		\mathop{\sup}_{0\leq t\leq T^*}\{\left\|f(t)\right\|_{2}+\left\|	w f(t)\right\|_\infty+|x_w(t)|+|v_{w}(t)|\}\leq 2\vps_0,\notag
	\end{align}
	and
	\begin{equation}\label{loc es par f}
		\mathop{\sup}_{0\leq t\leq T^*}\Big\{\left\|{W_c}\partial_{x_1}f(t)\right\|_{p}^p+\int_0^t|{W_c}\pa_{x_1}f(s)|_{p,+}^p{\rm d}s\Big\}< \infty.
	\end{equation}
Furthermore, the quantities $\|w f(t)\|_\infty$, $\|{W_c}\pa_{x_1}f(t)\|_{p}^p$, and $\int_0^t|{W_c}\pa_{x_1}f|_{p,+}^p{\rm d}s$ are continuous in $t$.
\end{theorem}

To prove Theorem \ref{loc} above, we begin with the following approximation regime
	\begin{equation}\label{equ bar F^ell+1}
		\left\{	\begin{array}{l}
			\partial_t\bar F^{\ell+1}+\frac{v_1-v^{\ell}_w(t)x_1}{1+x^{\ell}_w(t)}\partial_{x_1}\bar F^{\ell+1}=Q_{+}(\bar F^\ell,\bar F^\ell)-Q_{-}(\bar F^\ell,\bar F^{\ell+1}),\ \ell\geq0,\vspace{1ex}\\ \bar F^{\ell+1}(0,x_1,v)=\bar F_0(x_1,v),\vspace{1ex}\\
		\end{array}\right.
	\end{equation}
	with the following boundary condition on $\gamma_-$:
	\begin{align*}
		\bar F^{\ell+1}(t,0,v)|_{v_1>0}&=\sqrt{2\pi}\mu(v)\int_{u_1<0}\bar F^\ell(t,0,u)|u_1|{\rm d}u,\\
		\bar	F^{\ell+1}(t,1,v)|_{v_1-v_w^\ell(t)<0}&=\sqrt{2\pi}\mu(v_1-v^\ell_w(t),v_2,v_3)\int_{u_1-v^\ell_w(t)>0}\bar F^\ell(t,1,u)|u_1-v^\ell_w(t)|{\rm d}u,
	\end{align*}
	and the free boundary pair $[x^{\ell+1},v^{\ell+1}]$ satisfies the iteration scheme
	\begin{align*}
		\left\{	\begin{array}{l}
			\frac{{\rm d}}{{\rm d}t}x^{\ell+1}_w(t)=v^{\ell+1}_w(t),\vspace{1ex}\\
			\frac{{\rm d}}{{\rm d}t}v^{\ell+1}_w(t)=-\kappa x_w^{\ell+1}(t)-\mathcal M\left(P_r^{\ell+1}[\bar{F}^{\ell}]-P_l^{\ell+1}[\bar{F}^\ell] \right),\vspace{1ex}\\
			x^{\ell+1}_w(0)=x_{w0}, v^{\ell+1}_w(0)=v_{w0},
		\end{array}\right.
	\end{align*}
	where
	\begin{align*}
		P_r^{\ell+1}[\bar{F}^{\ell}]&-P_l^{\ell+1}[\bar{F}^\ell]\notag\\ &=\int_{v_1-v_w^{\ell+1}(t)<0}\mu(v)|v_1-v_w^{\ell+1}|^2{\rm d}v+\frac{\sqrt{2\pi}}{2}\int_{v_1-v_w^{\ell+1}(t)<0}\mu(v)|v_1-v_w^{\ell+1}(t)|{\rm d}v\notag\\
		&\quad-\int_{\mathbb R^3}\bar F^\ell(t,1,v)|v_1-v_w^{\ell+1}(t)|^2{\rm d}v.
	\end{align*}
	Additionally, we set $\bar F^0(t,x_1,v)=\bar F_0(x_1,v)$, $x^{0}_w(0)=x_{w0}$ and $v^{0}_w(0)=v_{w0}$. The proof is then divided into the following four steps.

\medskip
	\noindent {\bf Step 1. Non-negativity.} We aim to prove that
	\begin{equation}\label{posi F^ell+1}
		\textrm{if}\ \bar F^\ell\geq0,\ \textrm{then}\ \bar F^{\ell+1}\geq0.
	\end{equation}
	We define the functional $\nu(F)$ as
	\begin{equation*}
		\nu(F):=\int_{\mathbb R^3\times\mathbb S^2}|(v-u)\cdot\omega|F(u){\rm d}\omega{\rm d}u.
	\end{equation*}
	Since $\bar F^\ell\geq 0$, we have
	\begin{equation}\label{gain>0-drbcl}
		\nu(F^\ell)\geq0,\ Q_{+}(\bar F^\ell,\bar F^\ell)\geq0\ \textrm{and}\ \bar F^{\ell+1}\geq0\ \textrm{on}\ \gamma_-.
	\end{equation}
	Next, consider the characteristic line defined by
	\begin{equation}
		X_1^\ell(s;t,x_1,v_1)=\frac{1+x_w^\ell(t)}{1+x_w^\ell(s)}x_1+\frac{v_1(s-t)}{1+x_w^\ell(s)},\notag
	\end{equation}
	which solves
	\begin{equation}
		\frac{\rm d}{\rm ds}X_1^\ell(s;t,x_1,v_1)=\frac{v_1-v_w^\ell(s)X_1^\ell(s;t,x_1,v_1)}{1+x_w^\ell(s)}.\label{l-ch}
	\end{equation}
	From \eqref{equ bar F^ell+1}, one has by using \eqref{l-ch}
	\begin{align}
		\frac{\rm d}{\rm ds}&\Big\{e^{-\int_s^t\nu(\bar F^\ell)(\tau,X_1^\ell(\tau),v){\rm d}\tau}\bar F^{\ell+1}(s,X_1^\ell(s),v)\Big\}
		=e^{-\int_s^t\nu(\bar F^\ell)(\tau,X_1^\ell(\tau),v){\rm d}\tau}Q_{+}(\bar F^\ell,\bar F^\ell)(s,X_1^\ell(s),v).\notag
	\end{align}
This together with \eqref{gain>0-drbcl} and $\bar F_0(x_1,v)\geq0$ implies $\bar F^{\ell+1}\geq0$. Thus \eqref{posi F^ell+1} holds true.
	
\medskip
\noindent {\bf Step 2. Uniform bound.} By choosing $\vps_0\ll1$ and $T=T(\vps_0)\ll 1$, we will show that
	\begin{equation}\label{es L^infty h^ell}
		\mathop{\sup}_{0\leq t\leq T}\mathop{\max}_n\|w f^n(t)\|_\infty+\mathop{\sup}_{0\leq t\leq T}\mathop{\max}_n|x_w^n(t)|_\infty+\mathop{\sup}_{0\leq t\leq T}\mathop{\max}_n|v_w^n(t)|_\infty\leq 2\vps_0.
	\end{equation}
	To this end, set $\bar F^{\ell+1}=\mu+\sqrt{\mu}f^{\ell+1}$, then $(f^{\ell+1},x_w^{\ell+1},v_w^{\ell+1})$ solves
	\begin{equation}\label{PBE F^ell+1}
		\left\{	\begin{array}{l}
			\partial_tf^{\ell+1}+\frac{v_1-v^{\ell}_w(t)x_1}{1+x^{\ell}_w(t)}\partial_{x_1}f^{\ell+1}+\nu f^{\ell+1}=K f^\ell+\Gamma_{+}(f^\ell,f^\ell)-\Gamma_{-}(f^\ell,f^{\ell+1}),\vspace{1ex}\\ f^{\ell+1}(0,x_1,v)=f_0(x_1,v),
		\end{array}\right.
	\end{equation}
	and
	\begin{align}\label{ODE x^ell+1}
		\left\{	\begin{array}{l}
			\frac{{\rm d}}{{\rm d}t}x^{\ell+1}_w(t)=v^{\ell+1}_w(t),\vspace{1ex}\\
			\frac{{\rm d}}{{\rm d}t}v^{\ell+1}_w(t)=-\kappa x_w^{\ell+1}(t)-\mathcal M\left(P_r^{\ell+1}[\bar{F}^{\ell}]-P_l^{\ell+1}[\bar{F}^\ell] \right),\vspace{1ex}\\
			x^{\ell+1}_w(0)=x_{w0},\ v^{\ell+1}_w(0)=v_{w0},
		\end{array}\right.
	\end{align}
	with the boundary conditions
	\begin{align}
		f^{\ell+1}(t,0,v)|_{v_1>0}=\sqrt{2\pi\mu(v)}\int_{u_1<0}f^\ell(t,0,u)\sqrt{\mu(u)}|u_1|{\rm d}u\eqdef:P_{\gamma^\ell} f^\ell(0,v),\notag
	\end{align}
	and
	\begin{align}
		&f^{\ell+1}(t,1,v)|_{v_1-v^\ell_w(t)<0}\notag\\
		&\quad=\sqrt{2\pi\mu(v)}\int_{u_1-v^\ell_w(t)>0}f^\ell(t,1,u)\sqrt{\mu(u)}|u_1-v^\ell_w(t)|{\rm d}u\notag\\
		&\qquad +\sqrt{2\pi}\frac{\mu(v_1-v_w^{\ell}(t),v_2,v_3)-\mu(v)}{\sqrt{\mu (v)}}\int_{u_1-v_w^{\ell}(t)>0}f^{\ell}(t,1,u)\sqrt{\mu(u)}|u_1-v_w^{\ell}(t)|{\rm d}u\notag\\
		&\qquad+\sqrt{2\pi}\frac{\mu(v_1-v^{\ell}_w(t),v_2,v_3)-\mu(v)}{\sqrt{\mu (v)}}\int_{u_1-v_w^\ell(t)>0}\mu(u)|u_1-v^{\ell}_w(t)|{\rm d}u\notag\\
		&\qquad+\sqrt{2\pi \mu(v)}\int_{u_1-v_w^{\ell}(t)>0}[\mu(u)-\mu(u_1-v^{\ell}_w(t),u_2,u_3)]|u_1-v^{\ell}_w(t)|{\rm d}u\notag\\
		&\quad \eqdef:P_{\gamma^\ell} f^\ell(1,v)+r^\ell,\label{fk1-bd-drbcl}
	\end{align}
	Here,
	$$
	\gamma^\ell=\gamma_+^\ell	\cup\gamma_-^\ell, \ \gamma^0=\gamma_+	\cup\gamma_-, \ \ell\geq0,
	$$
	with
	\begin{align*}
		&\gamma_-^\ell=\{x_1\in\{0,1\}, v\in\R^3: [v_1-v_w^\ell(t)x_1]n(x_1)<0\},\\
		&\gamma_+^\ell=\{x_1\in\{0,1\}, v\in\R^3: [v_1-v_w^\ell(t)x_1]n(x_1)>0\}.
	\end{align*}
	Moreover, we denote
	$P_{\gamma^\ell} f^\ell(1,v)$ as the first term on the right hand side of \eqref{fk1-bd-drbcl} and $r^\ell$ as the remaining terms.

	We prove \eqref{es L^infty h^ell} by the argument of induction. First, \eqref{es L^infty h^ell} is true for $n=0$ by the initial condition given in \eqref{to-id-drbcl}.
	We then assume \eqref{es L^infty h^ell} is valid for $n\leq \ell$, and we aim to prove that \eqref{es L^infty h^ell} holds for $n=\ell+1$. Let $(x_1,v)\notin\gamma_0$ and $(t^0,x_1^0,v^0)=(t,x_1,v)$, define the following quantities:
	\begin{align*}
		&t^\ell_1(t,x_1,v_1)=\inf\{s<t: X_1^\ell(s;t,x_1,v_1)\in\Omega\},\\
		&x_{1,1}^\ell=X_1^\ell(t^\ell_1(t,x_1,v_1);t,x_1,v_1),\\
		&v^\ell_{1}\in\mathcal V^\ell_1:=\{x^\ell_{1,1}=0,v_{1,1}^\ell<0,v^\ell_1\in\mathbb R^3\}\cup\{x^\ell_{1,1}=1,v_{1,1}^\ell-v^{\ell}_w(t_1^\ell)>0,v^\ell_1\in\mathbb R^3\},\ \ell\geq1,\\
		&t^{\ell-1}_2(t,x_1,v_1,v_{1,1}^\ell)=\inf\{s<t_1^\ell: X_1^{\ell-1}(s;t^\ell_1,x^\ell_{1,1},v^\ell_{1,1})\in\Omega\},\\
		&x_{2,1}^{\ell-1}=X_1^{\ell-1}(t^{\ell-1}_2(t,x_1,v_1,v_{1,1}^\ell);t^\ell_1,x^\ell_{1,1},v^\ell_{1,1}),\\
		&v^{\ell-1}_{2}\in\mathcal V^{\ell-1}_2:=\{x^{\ell-1}_{2,1}=0,v_{2,1}^{\ell-1}<0,v^{\ell-1}_2\in\mathbb R^3\}\cup\{x^{\ell-1}_{1}=1,v_{2,1}^{\ell-1}-v^{\ell-1}_w(t_2^{\ell-1})>0,v^{\ell-1}_2\in\mathbb R^3\},
	\end{align*}
	and inductively
	\begin{align*}
		&t^{\ell-(l-1)}_l(t,x_1,v_1,v_{1,1}^\ell,\cdots,v_{l-1,1}^{\ell-(l-2)})\\
		&\quad=\inf\{s<t_{l-1}^{\ell-(l-2)}: X_1^{\ell-(l-1)}(s;t_{l-1}^{\ell-(l-2)},x_{l-1,1}^{\ell-(l-2)},v_{l-1,1}^{\ell-(l-2)})\in\Omega\},\ l\geq2,\\
		&x_{l,1}^{\ell-(l-1)}(t,x_1,v_1,v_{1,1}^\ell,\cdots,v_{l-1,1}^{\ell-(l-2)})\\
		&\quad=X_1^{\ell-(l-1)}(t^{\ell-(l-1)}_l;t_{l-1}^{\ell-(l-2)},x_{l-1,1}^{\ell-(l-2)},v_{l-1,1}^{\ell-(l-2)}),\\
		&v_{l}^{\ell-(l-1)}(t,x_1,v_1,v_{1,1}^\ell,\cdots,v_{l-1,1}^{\ell-(l-2)})\in\mathcal V^{\ell-(l-1)}_k\\
		&\quad:=\{x_{l,1}^{\ell-(l-1)}=0,v_{l,1}^{\ell-(l-1)}<0,v_{l}^{\ell-(l-1)}\in\mathbb R^3\}\\
		&\qquad\cup\{x_{l,1}^{\ell-(l-1)}=1,v_{l,1}^{\ell-(l-1)}-v^{\ell-(l-1)}_w(t_l^{\ell-(l-1)})>0,v_{l}^{\ell-(l-1)}\in\mathbb R^3\}.
	\end{align*}
	Next, setting
	\begin{equation*}
		h^{\ell+1}(t,x_1,v):=w(v)f^{\ell+1}(t,x_1,v),
	\end{equation*}
	one can rewrite \eqref{PBE F^ell+1} as
	\begin{align}\label{equ h^ell+1}
		\left\{	\begin{array}{l}
			\partial_th^{\ell+1}+\frac{v_1-v^\ell_w(t)x_1}{1+x^\ell_w(t)}\partial_{x_1}h^{\ell+1}+\nu h^{\ell+1}=K_{w}h^\ell+w\Gamma_{+}(\frac{h^\ell}{w},\frac{h^\ell}{w})-w\Gamma_{-}(\frac{h^\ell}{w},\frac{h^{\ell+1}}{w}),\vspace{1ex}\\ h^{\ell+1}(0,x_1,v)=h_0(x_1,v),\vspace{1ex}\\
			h^{\ell+1}(t,0,v)=w(v)P_{\gamma^\ell} f^\ell(t,0,v),\ v_1>0,\vspace{1ex}\\
			h^{\ell+1}(t,1,v)=w(v) P_{\gamma^\ell} f^\ell(t,1,v)+w(v)r^\ell(t,v),\ v_1-v_w^\ell(t)<0.
		\end{array}\right.
	\end{align}
	We then define
	\begin{equation}\label{es nu^ell}
		\nu^\ell(t,x_1,v)=\nu+\nu(\sqrt{\mu}\frac{h^\ell}{w})=\nu(v)+\int_{\R^3}\int_{\S^2}B(v-u,\theta)\sqrt{\mu}\frac{h^\ell}{w}{\rm d}\omega {\rm d}u.
	\end{equation}
	Since $\left\|wf^\ell(s)\right\|_{\infty}\leq 2\vps_0\ll1$, we have $\nu^\ell(t,x_1,v)\geq\la_1 \langle v\rangle$.
	
	Denote $H^\ell=\Gamma_{+}(\frac{h^\ell}{w},\frac{h^\ell}{w})$, it follows from Lemma \ref{Ga} that
	\begin{equation}\label{es |Gamma|-drbcl}
		|w H^\ell|\lesssim \langle v\rangle \|h^\ell\|_\infty^2.
	\end{equation}
	The following lemma is crucial for estimating $|h^{\ell+1}|$.
	\begin{lemma}\label{lem es h^ell+1}
	Let $\ell\geq0$. If $t^\ell_1\leq0$, then
		\begin{equation}\label{t^ell_1<0}
			|h^{\ell+1}(t,x_1,v)|\leq e^{-\int_0^t\nu^\ell {\rm d}s}|h^{\ell+1}(0,X_1^\ell(0),v)|+\int_{0}^te^{-\int_s^t\nu^\ell {\rm d}\tau}|[K_{w}h^\ell+wH^\ell](s,X_1^\ell(s;t,x_1,v_1),v)|{\rm d}s.
		\end{equation}
		If $t^\ell_1\geq0$, we have
		\begin{align}\label{t^ell_1>0}
			|h^{\ell+1}(t,x_1,v)|&\leq {\bf 1}_{t^\ell_1\geq0}e^{-\int^t_{t^\ell_1}\nu^\ell{\rm d}s}|wr^\ell(t^\ell_1,x^\ell_{1,1},v)|\mathbf 1_{\{x^\ell_{1,1}=1\}}\notag\\
			&\quad+\int_{t^\ell_1}^te^{-\int^t_{s}\nu^\ell{\rm d}\tau}|[K_{w}h^\ell+wH^\ell](s,X_1^\ell(s;t,x_1,v_1),v)|{\rm d}s\notag\\
			&\quad+\frac{e^{-\int^t_{t^\ell_1}\nu^\ell{\rm d}s}}{\tilde w(v)}\int_{\prod_{j=1}^{k-1}\mathcal V^{\ell-(j-1)}_j}|\tilde H|,
		\end{align}
		where $|\tilde H|$ is bounded by
		\begin{align}
			\sum_{l=1}^{k-1}&\mathbf 1_{\{t^{\ell-l}_{l+1}\leq0<t^{\ell-(l-1)}_l\}}|h^{\ell+1-l}(0,X_1^{\ell-l}(0;t^{\ell-{l-1}}_l,x^{\ell-(l-1)}_{l,1},v^{\ell-(l-1)}_{l,1}),v^{\ell-(l-1)}_{l})|{\rm d}\tilde\Sigma_l(0)\notag\\
			&+\sum_{l=1}^{k-1}\int_0^{t^{\ell-(l-1)}_{l}}\mathbf 1_{\{t^{\ell-l}_{l+1}\leq0<t^{\ell-(l-1)}_{l}\}}\notag\\
			&\quad\times|[K_{w}h^{\ell-l}+wH^{\ell-l}](s,X_1^{\ell-l}(s;t^{\ell-(l-1)}_{l},x^{\ell-(l-1)}_{l,1},v^{\ell-(l-1)}_{l,1}),v^{\ell-(l-1)}_{l})|{\rm d}\tilde\Sigma_l(s){\rm d}s\notag\\
			&+\sum_{l=1}^{k-1}\int_{t^{\ell-l}_{l+1}}^{t^{\ell-(l-1)}_{l}}\mathbf 1_{\{0<t^{\ell-(l-1)}_{l}\}}\notag\\ &\quad\times|[K_{w}h^{\ell-l}+wH^{\ell-l}](s,X_1^{\ell-l}(s;t^{\ell-(l-1)}_{l},x^{\ell-(l-1)}_{l,1},v^{\ell-(l-1)}_{l,1}),v^{\ell-(l-1)}_{l})|{\rm d}\tilde\Sigma_l(s){\rm d}s\notag\\
			&+\sum_{l=1}^{k-2}\mathbf 1_{\{0<t^{\ell-l}_{l+1}\}}{\rm d}\tilde\Sigma_l^r\notag\\
			&+\mathbf 1_{\{0<x^{\ell-(k-1)}_{k}\}}|h^{\ell+2-k}(t^{\ell-(k-1)}_{k},x^{\ell-(k-1)}_{k,1},v^{\ell-(k-2)}_{k-1})|{\rm d}\tilde\Sigma_{k-1}(t^{\ell-(k-1)}_{k}).\label{es tildeH5}
		\end{align}
		We define
		\begin{align}
			{\rm d}\tilde\sigma^l&=-\sqrt{2\pi}\mu(v^{\ell-(l-1)}_{l})v^{\ell-(l-1)}_{l,1}{\rm d}v^{\ell-(l-1)}_{l}\ {\rm if}\ x^{\ell-(l-1)}_{l,1}=0,\notag\\
			{\rm d}\tilde\sigma^l&=\sqrt{2\pi}\mu(v^{\ell-(l-1)}_{l})(v^{\ell-(l-1)}_{l,1}-v^{\ell+1-l}_w(t^{\ell-(l-1)}_{l})){\rm d}v^{\ell-(l-1)}_{l}\ {\rm if}\ x^{\ell-(l-1)}_{l,1}=1,\notag\\
			{\rm d}\tilde\Sigma_l&=\{\prod_{j=l+1}^{k-1}{\rm d}\tilde{\sigma}^j\}\times\{\tilde w(v^{\ell-(l-1)}_{l}){\rm d}\tilde{\sigma}^l\}\times\{\prod_{j=1}^{l-1}{\rm d}\tilde{\sigma}^j\},\notag\\
			{\rm d}\tilde\Sigma_l(s)&=\{\prod_{j=l+1}^{k-1}{\rm d}\tilde{\sigma}^j\}\times\{e^{-\int_s^{t^{\ell-(l-1)}_{l}}\nu^{\ell-l}{\rm d}\tau}\tilde w(v^{\ell-(l-1)}_{l}){\rm d}\tilde{\sigma}^l\}\times\{\prod_{j=1}^{l-1}e^{-\int_{t^{\ell-j}_{j+1}}^{t^{\ell-(j-1)}_{j}}\nu^{\ell-(j-1)}{\rm d}\tau}{\rm d}\tilde{\sigma}^j\},\notag\\	
			{\rm d}\tilde{\Sigma}_l^r&=\{\prod_{j=l+1}^{k-1}{\rm d}\tilde{\sigma}^j\}\times\{e^{-\int_{t^{\ell-l}_{l+1}}^{t^{\ell-(l-1)}_{l}}\nu^{\ell-l}{\rm d}\tau}\tilde w(v^{\ell-(l-1)}_l)w(v^{\ell-(l-1)}_l)r^{\ell-l}(t^{\ell-l}_{l},x^{\ell-l}_{l+1,1},v^{\ell-(l-1)}_{l}){\bf 1}_{x^{\ell-l}_{l+1,1}=1}{\rm d}\tilde{\sigma}^l\}\notag\\
			&\quad\times\{\prod_{j=1}^{l-1}e^{-\int_{t^{\ell-j}_{j+1}}^{t^{\ell-(j-1)}_{j}}\nu^{\ell-(j-1)}{\rm d}\tau}{\rm d}\tilde{\sigma}^j\},\  \ k\geq3,\ l\geq1.\notag
		\end{align}
	\end{lemma}
	\begin{proof}
		If $t^\ell_1\leq0$, then the estimate \eqref{t^ell_1<0} directly follows from \eqref{equ h^ell+1}.
		
		If $t^\ell_1\geq0$, we have
		\begin{equation*}
			|h^{\ell+1}(t,x_1,v)|\leq e^{-\int_{t_1^\ell}^t\nu^\ell{\rm d}s}|h^{\ell+1}(t^\ell_1,x_{1,1}^\ell,v)|+\int_{t^\ell_1}^te^{-\int_s^t\nu^\ell{\rm d}\tau}|[K_{w}h^\ell+wH^\ell]
(s,X_1^\ell(s;t,x_1,v_1),v)|{\rm d}s.	
		\end{equation*}
	At this stage,	we have $v_1\geq0, x_{1,1}^\ell=0$ or $v_1-v^\ell_w(t^\ell_1)\leq0, x_{1,1}^\ell=1$, then $h^{\ell+1}$ satisfies
		\begin{equation*}
			h^{\ell+1}(t^\ell_1,x_{1,1}^\ell,v)=wf^{\ell+1}(t^\ell_1,x_{1,1}^\ell,v)=w P_{\gamma^\ell} f^\ell+wr^\ell\textbf 1_{x_{1,1}^\ell=1},
		\end{equation*}
		with
		\begin{align*}
			w P_{\gamma^\ell} f^{\ell}&=\sqrt{2\pi} w\sqrt{\mu}(v)\int_{\mathcal V^\ell_1}\sqrt{\mu}f^\ell(t^\ell_1,x_{1,1}^\ell,v^\ell_1)|v^\ell_{1,1}-v^\ell_w(t^\ell_1)\textbf 1_{x_{1,1}^\ell=1}|{\rm d}v^\ell_1\\
			&=\frac{\sqrt{2\pi}}{ \tilde w(v)}\int_{\mathcal V^\ell_1}(\tilde w\mu)(v^\ell_1) h^\ell(t^\ell_1,x_{1,1}^\ell,v^\ell_1)|v^\ell_{1,1}-v^\ell_w(t^\ell_1)\textbf 1_{x_{1,1}^\ell=1}|{\rm d}v^\ell_1.
		\end{align*}
		Next we continue to express $h^\ell(t^\ell_1,x_{1,1}^\ell,v^\ell_1)$ as
		\begin{align*}
			|h^\ell(t^\ell_1,x_{1,1}^\ell,v^\ell_1)|&\leq \textbf 1_{\{t^{\ell-1}_2\leq0<t^\ell_1\}}\Bigg\{e^{-\int_0^{t_1^\ell}\nu^{\ell-1}{\rm d}s }|h^\ell(0,X_1^{\ell-1}(0;t^\ell_1,x_{1,1}^\ell,v_{1,1}^\ell),v^\ell_1)|\\
			&\qquad+\int_{0}^{t^\ell_1}e^{-\int_s^{t^\ell_1}\nu^{\ell-1}{\rm d}\tau}|[K_{w}h^{\ell-1}+wH^{\ell-1}](s,X_1^{\ell-1}(s;t^\ell_1,x_{1,1}^\ell,v^\ell_{1,1}),v^\ell_1)|{\rm d}s\Bigg\}\\
			&\quad+\textbf 1_{\{t^{\ell-1}_2\geq0\}}\Bigg\{e^{-\int_{t^{\ell-1}_2}^{t_1^\ell}\nu^{\ell-1}{\rm d}s}|h^\ell(t^{\ell-1}_2,x_{2,1}^{\ell-1},v^\ell_1)|\\
			&\qquad+\int_{t^{\ell-1}_2}^{t^\ell_1}e^{-\int^{t^\ell_1}_{s}\nu^{\ell-1}{\rm d}\tau}|[K_{w}h^{\ell-1}+wH^{\ell-1}](s,X_1^{\ell-1}(s;t^\ell_1,x^\ell_{1,1},v_{1,1}^\ell),v^\ell_1)|{\rm d}s\Bigg\}.
		\end{align*}
		Therefore, $h^{\ell+1}(t,x_1,v)$ can be further bounded as
		\begin{align}
			&h^{\ell+1}(t,x_1,v)\notag\\
			&\leq \int_{t^\ell_1}^te^{-\int_s^t\nu^\ell{\rm d}\tau}|[K_{w}h^\ell+wH^\ell](s,X_1^\ell(s;t,x_1,v_1),v)|{\rm d}s\notag\\
			&\quad+e^{-\int_{t_1^\ell}^t\nu^\ell{\rm d}s}|wr^\ell(t_1^\ell,x_{1,1}^\ell,v)|\textbf 1_{x_{1,1}^\ell=1}\notag\\
			&\quad+\sqrt{2\pi}\frac{e^{-\int_{t_1^\ell}^t\nu^\ell{\rm d}s}}{ \tilde w(v)}\int_{\mathcal V^\ell_1}\textbf 1_{\{t^{\ell-1}_2\leq0<t^\ell_1\}}e^{-\int_0^{t_1^\ell}\nu^{\ell-1}{\rm d}\tau}|h^\ell(0,X_1^{\ell-1}(0;t^\ell_1,x_{1,1}^\ell,v_{1,1}^\ell),v^\ell_1)|\notag\\
			&\qquad\times(\tilde w\mu)(v^\ell_1)|v^\ell_{1,1}-v^\ell_w(t^\ell_1)\textbf 1_{x_{1,1}^\ell=1}|{\rm d}v^\ell_1\notag\\
			&\quad+\sqrt{2\pi}\frac{e^{-\int_{t_1^\ell}^t\nu^\ell{\rm d}s}}{ \tilde w(v)}\int_{\mathcal V^\ell_1}\textbf 1_{\{t^{\ell-1}_2\leq0<t^\ell_1\}}\int_{0}^{t^\ell_1}e^{-\int_s^{t^\ell_1}\nu^{\ell-1}{\rm d}\tau}|[K_{w}h^{\ell-1}+wH^{\ell-1}](s,X_1^{\ell-1}(s;t^\ell_1,x_{1,1}^\ell,v^\ell_{1,1}),v^\ell_1)|\notag\\
			&\qquad\times(\tilde w\mu)(v^\ell_1)|v^\ell_{1,1}-v^\ell_w(t^\ell_1)\textbf 1_{x_{1,1}^\ell=1}|{\rm d}v^\ell_1{\rm d}s\notag\\
			&\quad+\sqrt{2\pi}\frac{e^{-\int_{t_1^\ell}^t\nu^\ell{\rm d}s}}{ \tilde w(v)}\int_{\mathcal V^\ell_1}\textbf 1_{\{t^{\ell-1}_2\geq0\}}\int_{t^{\ell-1}_2}^{t^\ell_1}e^{-\int^{t^\ell_1}_{s}\nu^{\ell-1}{\rm d}\tau}|[K_{w}h^{\ell-1}+wH^{\ell-1}](s,X_1^{\ell-1}(s;t^\ell_1,x^\ell_{1,1},v_{1,1}^\ell),v^\ell_1)|\notag\\
			&\qquad\times(\tilde w\mu)(v^\ell_1)|v^\ell_{1,1}-v^\ell_w(t^\ell_1)\textbf 1_{x_{1,1}^\ell=1}|{\rm d}v^\ell_1{\rm d}s\notag\\
			&\quad+\sqrt{2\pi}\frac{e^{-\int_{t_1^\ell}^t\nu^\ell{\rm d}s}}{ \tilde w(v)}\int_{\mathcal V^\ell_1}\textbf 1_{\{t^{\ell-1}_2\geq0\}}|h^\ell(t^{\ell-1}_2,x_{2,1}^{\ell-1},v^\ell_1)|e^{-\int_{t^{\ell-1}_2}^{t_1^\ell}\nu^{\ell-1}{\rm d}\tau}(\tilde w\mu)(v^\ell_1)|v^\ell_{1,1}-v^\ell_w(t^\ell_1)\textbf 1_{x_{1,1}^\ell=1}|{\rm d}v^\ell_1.\notag
		\end{align}
	Thus \eqref{t^ell_1>0} holds for $k=2$.
		
		Now we assume that the inequality \eqref{t^ell_1>0} holds for an arbitrary $k\in\mathbb N$. We proceed to verify its validity for $k+1$.
		Note that $v_{k-1,1}^{\ell-(k-2)}\leq0, x_{k,1}^{\ell-(k-1)}=0$ or $v_{k-1,1}^{\ell-(k-2)}-v^{\ell+1-k}_w(t_{k}^{\ell-(k-1)})\geq0, x_{k,1}^{\ell-(k-1)}=1$, we apply the boundary condition from \eqref{equ h^ell+1} to the final term \eqref{es tildeH5} to obtain
		\begin{align}\label{h(t^l)}
			h^{\ell+2-k}&(t_{k}^{\ell-(k-1)},x_{k,1}^{\ell-(k-1)},v_{k-1}^{\ell-(k-2)})\notag\\
			&=wf^{\ell+2-k}(t_{k}^{\ell-(k-1)},x_{k,1}^{\ell-(k-1)},v_{k-1}^{\ell-(k-2)})\notag\\
			&=w(v_{k-1}^{\ell-(k-2)}) P_{\gamma^{\ell+2-(k+1)}} f^{\ell+2-(k+1)}+w(v_{k-1}^{\ell-(k-2)})r^{\ell+2-(k+1)}\textbf 1_{x_{k,1}^{\ell-(k-1)}=1},
		\end{align}
		with
		\begin{align*}
			wP_{\gamma^{\ell+2-(k+1)}} f^{\ell+2-(k+1)}&=\frac{\sqrt{2\pi}}{ \tilde w(v_{k-1,1}^{\ell-(k-2)})}\int_{\mathcal V^{\ell-(k-1)}_k}\tilde w\mu h^{\ell+2-(k+1)}(t_{k}^{\ell-(k-1)},x_{k,1}^{\ell-(k-1)},v_{k}^{\ell-(k-1)})\\
			&\qquad\times|v_{k,1}^{\ell-(k-1)}-v^{\ell+2-(k+1)}_w(t_{k}^{\ell-(k-1)})\textbf 1_{x_{k,1}^{\ell-(k-1)}=1}|{\rm d}v_{k}^{\ell-(k-1)}.
		\end{align*}
		Next we continue to express $h^{\ell+2-(k+1)}(t_{k}^{\ell-(k-1)},x_{k,1}^{\ell-(k-1)},v_{k}^{\ell-(k-1)})$ as
		\begin{align}\label{es h(t^l)}
			|h^{\ell+2-(k+1)}&(t_{k}^{\ell-(k-1)},x_{k,1}^{\ell-(k-1)},v_{k}^{\ell-(k-1)})|\notag\\
			&\leq \textbf 1_{\{t_{k+1}^{\ell-k}\leq0<t_{k}^{\ell-(k-1)}\}}\Bigg\{e^{-\int_0^{t_{k}^{\ell-(k-1)}}\nu^{\ell-k} {\rm d}s}\notag\\
			&\quad\qquad\times|h^{\ell+2-(k+1)}(0,X_1^{\ell-k}(0;t_{k}^{\ell-(k-1)},x_{k,1}^{\ell-(k-1)},v_{k,1}^{\ell-(k-1)}),v_{k}^{\ell-(k-1)})|\notag\\
			&\qquad+\int_{0}^{t_{k+1}^{\ell-k}}e^{-\int_s^{t_{k}^{\ell-(k-1)}}\nu^{\ell-k}{\rm d}\tau}\notag\\
			&\quad\qquad\times|[K_{w}h^{\ell-k}+wH^{\ell-k}](s,X_1^{\ell-k}(s;t_{l}^{\ell-(k-1)},x_{k,1}^{\ell-(k-1)},v_{k,1}^{\ell-(k-1)}),v_{k}^{\ell-(k-1)})|{\rm d}s\Bigg\}\notag\\
			&\quad+\textbf 1_{\{t_{k+1}^{\ell-k}\geq0\}}\Bigg\{e^{-\int_{t_{k+1}^{\ell-k}}^{t_{k}^{\ell-(k-1)}}\nu^{\ell-k}{\rm d}s}|h^{\ell+2-(k+1)}(t_{k+1}^{\ell-k},x_{k+1,1}^{\ell-k},v_{k}^{\ell-(k-1)})|\notag\\
			&\qquad+\int_{t_{k+1}^{\ell-k}}^{t_{k}^{\ell-(k-1)}}e^{-\int_s^{t_{k}^{\ell-(k-1)}}\nu^{\ell-k}{\rm d}\tau}\notag\\
			&\quad\qquad\times|[K_{w}h^{\ell-k}+wH^{\ell-k}](s,X_1^{k-l}(s;t_{k}^{\ell-(k-1)},x_{k,1}^{\ell-(k-1)},v_{k,1}^{\ell-(k-1)}),v_{k}^{\ell-(k-1)})|{\rm d}s\Bigg\}.
		\end{align}
On the other hand,	from \eqref{h(t^l)}, we also have
		\begin{align*}
			&\frac{e^{-\int_{t^\ell_1}^t\nu^\ell{\rm d}s}}{\tilde {w}(v)}\int_{\prod_{j=1}^{k-1}\mathcal V^{\ell-(j-1)}_j}\textbf 1_{\{0<t_{k}^{\ell-(k-1)}\}}|h^{\ell+2-k}(t_{k}^{\ell-(k-1)},x_{k,1}^{\ell-(k-1)},v_{k}^{\ell-(k-1)})|{\rm d}
\tilde{\Sigma}_{k-1}(t_k^{\ell-(k-1)})\notag\\
			&\quad\leq \frac{e^{-\int_{t^\ell_1}^t\nu^\ell{\rm d}s}}{\tilde {w}(v)}\int_{\prod_{j=1}^{k-1}\mathcal V^{\ell-(j-1)}_j}\textbf 1_{\{0<t_{k}^{\ell-(k-1)}\}}{\rm d}\tilde{\Sigma}_{k-1}^r\notag\\
			&\qquad+\frac{e^{-\int_{t^\ell_1}^t\nu^\ell{\rm d}s}}{\tilde {w}(v)}\int_{\prod_{j=1}^{k-1}\mathcal V^{\ell-(j-1)}_j}\textbf 1_{\{0<t_{k}^{\ell-(k-1)}\}}|h^{\ell+2-(k+1)}(t_{k}^{\ell-(k-1)},x_{k,1}^{\ell-(k-1)},v_{k}^{\ell-(k-1)})|{\rm d}\tilde{\Sigma}_k(t_{k}^{\ell-(k-1)}).
		\end{align*}
Applying \eqref{es h(t^l)}, we then see that the last term above can be bounded by
		\begin{align*}
			&\frac{e^{-\int_{t^\ell_1}^t\nu^\ell{\rm d}s}}{\tilde{w}(v)}\int_{\prod_{j=1}^{k}\mathcal V^{\ell-(j-1)}_j}\frac{\tilde w(v_{k}^{\ell-(k-1)})}{\tilde w(v_{k-1}^{\ell-(k-2)})}{\rm d}\tilde\sigma^k\\
			&\quad\times \tilde w(v_{k-1}^{\ell-(k-2)})e^{-\int^{t_{k-1}^{\ell-(k-2)}}_{t_{k}^{\ell-(k-1)}}\nu^{\ell-(k-2)}{\rm d}s}{\rm d}\tilde\sigma^{k-1}\prod_{j=1}^{k-2}e^{-\int^{t_{j}^{\ell-(j-1)}}_{t_{j+1}^{\ell-j}}\nu^{\ell-(l-1)}{\rm d}s}{\rm d}\tilde\sigma^{j}\\
			&\quad\times\textbf 1_{\{t_{k+1}^{\ell-k}\leq0<t_{k}^{\ell-(k-1)}\}}\Bigg\{e^{-\int_0^{t_{k}^{\ell-(k-1)}}\nu^{\ell-l}{\rm d}s}|h^{\ell+2-(k+2)}(0,X_1^{\ell-{k+1}}(0;t^k,x^{\ell-(k-1)}_{k,1},v^{\ell-(k-1)}_{k,1}),v^{\ell-(k-1)}_{k})|\notag\\
			&\quad\qquad+\int_{0}^{t^{\ell-k}_{k+1}}e^{-\int_s^{t^{\ell-(k-1)}_{k}}\nu^{\ell-k}{\rm d}\tau}\\
			&\qquad\qquad\times|[K_{w}h^{\ell-(k+1)}+wH^{\ell-(k+1)}](s,X_1^{\ell-(k+1)}
(s;t^{\ell-(k-1)}_{k},x^{\ell-(k-1)}_{k,1},v^{\ell-(k-1)}_{k,1}),v^{\ell-(k-1)}_{k})|{\rm d}s\Bigg\}\notag\\
			&\qquad+\textbf 1_{\{t^{\ell-k}_{k+1}\geq0\}}\Bigg\{e^{-\int_{t^{\ell-k}_{k+1}}^{t^{\ell-(k-1)}_{k}}\nu^{\ell-k}{\rm d}s}|h^{\ell+2-(k+2)}(t^{\ell-k}_{k+1},x^{\ell-k}_{k+1,1},v^{\ell-(k-1)}_{k})|\notag\\
			&\quad\qquad+\int_{t^{\ell-k}_{k}}^{t^{\ell-(k-1)}_{k}}e^{-\int_s^{t^{\ell-(k-1)}_{k}}\nu^{\ell-k}{\rm d}\tau}\\
			&\qquad\qquad\times|[K_{w}h^{\ell-(k+1)}+wH^{\ell-(k+1)}](s,X_1^{\ell-(k+1)}
(s;t^{\ell-(k-1)}_{k},x^{\ell-(k-1)}_{k,1},v^{\ell-(k-1)}_{k,1}),v^{\ell-(k-1)}_{k})|{\rm d}s\Bigg\}.
		\end{align*}
Therefore, \eqref{es tildeH5} can be further expanded as
\begin{align*}
\frac{e^{-\int_{t^\ell_1}^t\nu^\ell{\rm d}s}}{\tilde {w}(v)}&\int_{\prod_{j=1}^{k-1}\mathcal V^{\ell-(j-1)}_j}\textbf 1_{\{0<t_{k}^{\ell-(k-1)}\}}{\rm d}\tilde{\Sigma}_{k-1}^r\notag\\
&+\frac{e^{-\int_{t^\ell_1}^t\nu^\ell{\rm d}s}}{\tilde{w}(v)}\int_{\prod_{j=1}^{k}\mathcal V^{\ell-(j-1)}_j}\tilde w(v_{k}^{\ell-(k-1)}){\rm d}\tilde\sigma^k\\
			&\quad\times e^{-\int^{t_{k-1}^{\ell-(k-2)}}_{t_{k}^{\ell-(k-1)}}\nu^{\ell-(k-2)}{\rm d}s}{\rm d}\tilde\sigma^{k-1}\prod_{j=1}^{k-2}e^{-\int^{t_{j}^{\ell-(j-1)}}_{t_{j+1}^{\ell-j}}\nu^{\ell-(l-1)}{\rm d}s}{\rm d}\tilde\sigma^{j}\\
			&\quad\times\textbf 1_{\{t_{k+1}^{\ell-k}\leq0<t_{k}^{\ell-(k-1)}\}}\Bigg\{e^{-\int_0^{t_{k}^{\ell-(k-1)}}\nu^{\ell-l}{\rm d}s}|h^{\ell+2-(k+2)}(0,X_1^{\ell-{k+1}}(0;t^k,x^{\ell-(k-1)}_{k,1},v^{\ell-(k-1)}_{k,1}),v^{\ell-(k-1)}_{k})|\notag\\
			&\quad\qquad+\int_{0}^{t^{\ell-k}_{k+1}}e^{-\int_s^{t^{\ell-(k-1)}_{k}}\nu^{\ell-k}{\rm d}\tau}\\
			&\qquad\qquad\times|[K_{w}h^{\ell-(k+1)}+wH^{\ell-(k+1)}](s,X_1^{\ell-(k+1)}
(s;t^{\ell-(k-1)}_{k},x^{\ell-(k-1)}_{k,1},v^{\ell-(k-1)}_{k,1}),v^{\ell-(k-1)}_{k})|{\rm d}s\Bigg\}\notag\\
			&\qquad+\textbf 1_{\{t^{\ell-k}_{k+1}\geq0\}}\Bigg\{e^{-\int_{t^{\ell-k}_{k+1}}^{t^{\ell-(k-1)}_{k}}\nu^{\ell-k}{\rm d}s}|h^{\ell+2-(k+2)}(t^{\ell-k}_{k+1},x^{\ell-k}_{k+1,1},v^{\ell-(k-1)}_{k})|\notag\\
			&\quad\qquad+\int_{t^{\ell-k}_{k}}^{t^{\ell-(k-1)}_{k}}e^{-\int_s^{t^{\ell-(k-1)}_{k}}\nu^{\ell-k}{\rm d}\tau}\\
			&\qquad\qquad\times|[K_{w}h^{\ell-(k+1)}+wH^{\ell-(k+1)}](s,X_1^{\ell-(k+1)}
(s;t^{\ell-(k-1)}_{k},x^{\ell-(k-1)}_{k,1},v^{\ell-(k-1)}_{k,1}),v^{\ell-(k-1)}_{k})|{\rm d}s\Bigg\}.\notag
\end{align*}
	Hence, \eqref{t^ell_1>0} holds for $k+1$. The proof of Lemma \ref{lem es h^ell+1} is complete.
	\end{proof}

The following lemma, which provides an estimate of the measure of stochastic cycles, is adapted from \cite[Lemma 23, pp.781]{Guo-2010} with a slight modification.
	\begin{lemma}\label{lem es}
		For $T_0>0$ sufficiently large, there exist constants $C_4$, $C_5>0$ independent of $T_0$, such that for $k=C_4T_0^{5/4}$, and all $(t,x_1,v)\in[0,T_0]\times\bar\Omega\times\mathbb R^3$
		\begin{equation}\label{int-ds^j1}
			\int_{\prod_{j=1}^{k-1}\mathcal V^{\ell-(j-1)}_j}\mathbf 1_{t_{k}^{\ell-(k-1)}>0}\prod_{j=1}^{k-1}{\rm d}\tilde\sigma^j\leq\left\{\frac{1}{2}\right\}^{C_5T_0^{5/4}}.
		\end{equation}
		Suppose
\begin{align}
\mathop{\sup}_s\mathop{\max}_\ell|v_w^\ell(s)|\leq \frac{1}{\sqrt{2\pi}k},\label{vw-k}
\end{align}
we also have
		\begin{equation}\label{int-ds^j2}
			\int_{\prod_{j=1}^{k-1}\mathcal V^{\ell-(j-1)}_j}\sum_{l=1}^{k-1}\mathbf 1_{\{t^{\ell-l}_{l+1}\leq0<t^{\ell-(l-1)}_{l}\}}\tilde w(v^{\ell-(l-1)}_{l})\langle v^{\ell-(l-1)}_{l}\rangle\prod_{j=1}^{k-1}{\rm d}\tilde\sigma^j\leq Ck,
		\end{equation}
		\begin{equation}\label{int-ds^j3}
			\int_{\prod_{j=1}^{k-1}\mathcal V^{\ell-(j-1)}_j}\mathbf 1_{\{0<t^{\ell-l}_{l+1}\}}\tilde w(v^{\ell-(l-1)}_{l})\langle v^{\ell-(l-1)}_{l}\rangle\prod_{j=1}^{k-1}{\rm d}\tilde\sigma^j\leq C\left(1+\frac{1}{k}\right)^{k-1},
		\end{equation}
		for all $l=1,2,\cdots,k-1$.
	\end{lemma}
	\begin{proof}
To prove \eqref{int-ds^j1},
		choose $0<\delta$ sufficiently small and define
		\begin{align*}
			\mathcal V^{\ell-(j-1)}_{j,\delta}&:=\{v^{\ell-(j-1)}_{j}\in\mathcal V^{\ell-(j-1)}_j: x^{\ell-(j-1)}_{j,1}=0, v^{\ell-(j-1)}_{j,1}\leq-\delta;\\
			&\quad x^{\ell-(j-1)}_{j,1}=1, v^{\ell-(j-1)}_{j,1}-v^{\ell+1-j}_w(t^{\ell-(j-1)}_{j})\geq\delta\}\cap\{v^{\ell-(j-1)}_{j}\in\mathcal V^{\ell-(j-1)}_{j}: |v^{\ell-(j-1)}_{j,1}|\leq\delta^{-1}\}.
		\end{align*}
Then we have
		\begin{align*}
			\int_{\mathcal V^{\ell-(j-1)}_{j}\backslash\mathcal V^{\ell-(j-1)}_{j,\delta}}{\rm d}\tilde\sigma^j&\leq\Big(\int_{x^{\ell-(j-1)}_{j,1}=0, -\delta<v^{\ell-(j-1)}_{j,1}\leq0}{\rm d}\tilde\sigma^j+\int_{x^{\ell-(j-1)}_{j,1}=1, 0<v^{\ell-(j-1)}_{j,1}-v^{\ell+1-j}_w(t^{\ell-(j-1)}_{j})\leq\delta}{\rm d}\tilde\sigma^j\\
			&\quad+\int_{|v^{\ell-(j-1)}_{j,1}|\geq\delta^{-1}}{\rm d}\tilde\sigma^j\Big)\\
			&\leq C\delta,
		\end{align*}
		where $C $ is independent of $j$. On the other hand, if $v_j^{\ell-(j-1)}\in\mathcal V_{j,\delta}^{\ell-(j-1)}$, from
		\begin{equation*}
			x^{\ell-j}_{j+1,1}=\frac{1+x_w^{\ell-j}(t^{\ell-(j-1)}_{j})}{1+x^{\ell-j}_w(t^{\ell-j}_{j+1})}x^{\ell-(j-1)}_{j,1}+\frac{v^{\ell-(j-1)}_{j,1}(t^{\ell-j}_{j+1}-t^{\ell-(j-1)}_{j})}{1+x^{\ell-j}_w(t^{\ell-j}_{j+1})},
		\end{equation*}
		we have
		\begin{equation*}
			|t^{\ell-(j-1)}_{j}-t^{\ell-j}_{j+1}|=\frac{|(1+x^{\ell-j}_w(t^{\ell-j}_{j+1})x^{k-j}_{j+1,1}-(1+x^{\ell-j}_w(t^{\ell-(j-1)}_{j})x^{\ell-(j-1)}_{j,1}|}{|v^{\ell-(j-1)}_{j,1}|}\geq\frac{1}{2|v^{\ell-(j-1)}_{j,1}|}\geq\frac{\delta}{2},
		\end{equation*}
		since $(x_{j,1}^{\ell-(j-1)},x_{j+1,1}^{\ell-j})=(0,1)$ or $(1,0)$.
		Therefore, if $t^{\ell-(k-1)}_{k}>0$, there can be at most $\left[\frac{2T_0}{\delta}\right]+1$ number of $v^{\ell-(j-1)}_{j}\in\mathcal V^{\ell-(j-1)}_{j,\delta}$ for $1\leq j\leq k-1$. There exist as least $k-1-\left[\frac{2T_0}{\delta}\right]$ number $v^{\ell-(m-1)}_{m}\in \mathcal V^{\ell-(m-1)}_{m}\backslash\mathcal V^{\ell-(m-1)}_{m,\delta}$. We have
		\begin{align*}
			\int_{\prod_{j=1}^{k-1}\mathcal V^{k-(j-1)}_{j}}&\textbf 1_{\{t^{\ell-(k-1)}_{l}>0\}}\prod_{j=1}^{k-1}{\rm d}\tilde\sigma^j\\
			&\leq\sum_{m=1}^{\left[\frac{2T_0}{\delta}\right]+1}\int_{\{{\rm There\ are}\ m\ {\rm of}\ v^{\ell-(j-1)}_{j}\in\mathcal V^{\ell-(j-1)}_{j,\delta},\ k-1-m\ {\rm of}\ v^{\ell-(j-1)}_{j}\in \mathcal V^{\ell-(j-1)}_{j}\backslash\mathcal V^{\ell-(j-1)}_{j,\delta} \}}\prod_{j=1}^{k-1}{\rm d}\tilde\sigma^j\\
			&\leq \sum_{m=1}^{\left[\frac{2T_0}{\delta}\right]+1}\binom{k-1}{m}|\mathop{\sup}_{j}\int_{\mathcal V^{\ell-(j-1)}_{j,\delta}}{\rm d}\tilde\sigma^j|^m\Big\{\mathop{\sup}_j\int_{\mathcal V^{\ell-(j-1)}_{j}\backslash\mathcal V^{\ell-(j-1)}_{j,\delta}}{\rm d}\tilde\sigma^j\Big\}^{k-1-m}.
		\end{align*}
If $ x_{j,1}^{\ell-(j-1)}=0$, ${\rm d}\tilde\sigma^j$ is a probability measure and $\int_{\mathcal V^{\ell-(j-1)}_{j}}{\rm d}\tilde\sigma^j=1$. However, if 	$ x_{j,1}^{\ell-(j-1)}=1$,
\begin{align}
\int_{\mathcal V^{\ell-(j-1)}_{j}}{\rm d}\tilde\sigma^j&=\int_{v_{j,1}^{\ell-(j-1)}-v_w^{\ell-(j-1)}(t_j^{\ell-(j-1)})>0}e^{-\frac{1}{2}|v_{j,1}^{\ell-(j-1)}|^2}\left(v_{j,1}^{\ell-(j-1)}-v_w^{\ell-(j-1)}(t_j^{\ell-(j-1)})\right){\rm d}v_{j,1}^{\ell-(j-1)}\notag\\
&=\int_{v_{j,1}^{\ell-(j-1)}>0}e^{-\frac{1}{2}|v_{j,1}^{\ell-(j-1)}+v_w^{\ell-(j-1)}(t_j^{\ell-(j-1)})|^2}v_{j,1}^{\ell-(j-1)}{\rm d}v_{j,1}^{\ell-(j-1)}\notag\\
&=\int_{v_{j,1}^{\ell-(j-1)}>0}e^{-\frac{1}{2}|v_{j,1}^{\ell-(j-1)}|^2}\left(1-v_{j,1}^{\ell-(j-1)}v_w^{\ell-(j-1)}(t_j^{\ell-(j-1)})\right)v_{j,1}^{\ell-(j-1)}{\rm d}v_{j,1}^{\ell-(j-1)}\notag\\
&\quad+O(1)|v_w^{\ell-(j-1)}(t_j^{\ell-(j-1)})|^2\notag\\
&=1-{\frac{\sqrt{2\pi}}{2}}v_w^{\ell-(j-1)}(t_j^{\ell-(j-1)})+O(1)|v_w^{\ell-(j-1)}(t_j^{\ell-(j-1)})|^2\notag\\
&\leq 1+\sqrt{2\pi}|v_w^{\ell-(j-1)}(t_j^{\ell-(j-1)})|.\notag
\end{align}
Therefore,
		\begin{equation}\label{es-int-dsi^j}
			\int_{\mathcal V^{\ell-(j-1)}_{j}}{\rm d}\tilde\sigma^j\leq 1+\sqrt{2\pi}|v_w^{\ell+1-j}(t^{\ell-(j-1)}_{j})|
		\end{equation}
Since $|v_w^{\ell+1-j}(t^{\ell-(j-1)}_{j})|\ll1$, we have
		\begin{equation*}
			\left|\int_{\mathcal V^{\ell-(j-1)}_{j,\delta}}{\rm d}\tilde\sigma^j\right|^m< 2^m\leq 2^{\left[\frac{2T_0}{\delta}\right]+1},
		\end{equation*} and
		\begin{align*}
			\Big\{\mathop{\sup}_j\int_{\mathcal V^{\ell-(j-1)}_{j}\backslash\mathcal V^{\ell-(j-1)}_{j,\delta}}{\rm d}\tilde\sigma^j\Big\}^{k-1-m}\leq \Big\{\mathop{\sup}_j\int_{\mathcal V^{\ell-(j-1)}_{j}\backslash\mathcal V^{\ell-(j-1)}_{j,\delta}}{\rm d}\tilde\sigma^j\Big\}^{k-2-\left[\frac{2T_0}{\delta}\right]}\leq (C\delta)^{k-2-\left[\frac{2T_0}{\delta}\right]}.
		\end{align*}
		From $\binom{k-1}{m}\leq (k-1)^m\leq (k-1)^{\left[\frac{2T_0}{\delta}\right]+1}$, we deduce that
		\begin{align*}
			&\int_{\prod_{j=1}^{k-1}\mathcal V^{\ell-(j-1)}_{j}}\textbf 1_{t^{\ell-(k-1)}_{k}>0}\prod_{j=1}^{k-1}{\rm d}\tilde\sigma^j
			\leq \left(\left[\frac{2T_0}{\delta}\right]+1\right)[2(k-1)]^{\left[\frac{2T_0}{\delta}\right]+1}(C\delta)^{k-2-\left[\frac{2T_0}{\delta}\right]}.
		\end{align*}
		Let $k-2=N\Big\{\left[\frac{2T_0}{\delta}\right]+1\Big\}$, so that if $\frac{2T_0}{\delta}\geq1$ and $C\delta<1$, then
		\begin{align*}
			\int_{\prod_{j=1}^{k-1}\mathcal V^{\ell-(j-1)}_{j}}\textbf 1_{t_k^{\ell-(k-1)}>0}\prod_{j=1}^{k-1}{\rm d}\tilde\sigma^j
			&\leq \Big[2(N+1)\left(\left[\frac{2T_0}{\delta}\right]+1\right)\Big]^{\left[\frac{2T_0}{\delta}\right]+1}(C\delta)^{(N-1)\big\{\left[\frac{2T_0}{\delta}\right]+1\big\}}\\
			&=\Big[2(N+1)\left(\left[\frac{2T_0}{\delta}\right]+1\right)(C\delta)^{N-1}\Big]^{\left[\frac{2T_0}{\delta}\right]+1}\\
			&:=\Big[C_NT_0\delta^{N-2}\Big]^{\left[\frac{2T_0}{\delta}\right]+1}.
		\end{align*}
		Choose $C_NT_0\delta^{N-2}=\frac{1}{2}$, so that $\delta=\left\{\frac{1}{2C_NT_0}\right\}^{\frac{1}{N-2}}$ is small for $T_0$ large and $N\geq2$. Moreover,
		\begin{equation*}
			\left[\frac{2T_0}{\delta}\right]+1\sim C_NT_0^{1+\frac{1}{N-2}},
		\end{equation*}
		and  $\frac{2T_0}{\delta}+1\geq2$. Finally, by choosing $N=6$, then
	$
			\left[\frac{2T_0}{\delta}\right]+1\sim C_NT_0^{5/4}.
$
Thus \eqref{int-ds^j1} is valid.
	
		Next, we prove the remaining results. We note that for $1\leq l\leq k-1$
			\begin{align*}
			\int_{\prod_{j=1}^{k-1}\mathcal V^{\ell-(j-1)}_{j}}&\tilde w(v^{\ell-(l-1)}_{l})\langle v^{\ell-(l-1)}_{l}\rangle\prod_{j=1}^{k-1}{\rm d}\tilde\sigma^j\\
			&\leq \prod_{j\neq l}\left(\int_{\mathcal V^{\ell-(j-1)}_{j}}{\rm d}\tilde\sigma^j\right)\times\left(\int_{\mathcal V^{\ell-(l-1)}_{l}}\tilde w(v^{\ell-(l-1)}_{l})\langle v^{\ell-(l-1)}_{l}\rangle{\rm d}\tilde\sigma^l\right)\\
				&\leq \prod_{j\neq l}\left(\int_{\mathcal V^{\ell-(j-1)}_{j}}{\rm d}\tilde\sigma^j\right)\\
				&\qquad\times\left(\int_{\mathcal V^{\ell-(l-1)}_{l}}\frac{\langle v^{\ell-(l-1)}_{l}\rangle}{ w(v^{\ell-(l-1)}_{l})}|v_{l,1}^{\ell-(l-1)}-v_w^{\ell-(l-1)}(t^{\ell-(l-1)})x_{l,1}^{\ell-(l-1)}|{\rm d}v_l^{\ell-(l-1)}\right)\\
		&\leq \prod_{j\neq l}\left(\int_{\mathcal V^{\ell-(j-1)}_{j}}{\rm d}\tilde\sigma^j\right)\times\left(\int_{\mathcal V^{\ell-(l-1)}_{l}}e^{-\frac{\zeta}{2}|v_l^{\ell-(l-1)}|^2}{\rm d}v_l^{\ell-(l-1)}\right)\\			
			&\leq C\left(1+\frac{1}{k}\right)^{k-1}<\infty,
		\end{align*}
where we have used \eqref{es-int-dsi^j} and \eqref{vw-k}. We can thus bound the left hand side of \eqref{int-ds^j2}	as
		\begin{equation*}
				\int_{\prod_{j=1}^{k-1}\mathcal V^{\ell-(j-1)}_{j}}\sum_{l=1}^{k-1}\textbf 1_{\{t^{\ell-l}_{l+1}\leq0<t^{\ell-(l-1)}_{l}\}}\tilde w(v^{\ell-(l-1)}_{l})\langle v^{\ell-(l-1)}_{l}\rangle\prod_{j=1}^{k-1}{\rm d}\tilde\sigma^j\leq Ck.
		\end{equation*}
		Similarly, for \eqref{int-ds^j3}, we have
		\begin{align*}
			\int_{\prod_{j=1}^{k-1}\mathcal V^{\ell-(j-1)}_{j}}&\textbf 1_{\{0<t^{\ell-(l-1)}_{l}\}}\tilde w(v^{\ell-(l-1)}_{l})\langle v^{\ell-(l-1)}_{l}\rangle\prod_{j=1}^{k-1}{\rm d}\tilde\sigma^j\\
			&\leq \int_{\prod_{j=1}^{k-1}\mathcal V^{\ell-(j-1)}_{j}}\tilde w(v^{\ell-(l-1)}_{l})\langle v^{\ell-(l-1)}_{l}\rangle\prod_{j=1}^{k-1}{\rm d}\tilde\sigma^j\\
			&\leq \prod_{j\neq l}\left(\int_{\mathcal V^{\ell-(j-1)}_{j}}{\rm d}\tilde\sigma^j\right)\times\left(\int_{\mathcal V^{\ell-(l-1)}_{l}}\tilde w(v^{\ell-(l-1)}_{l})\langle v^{\ell-(l-1)}_{l}\rangle{\rm d}\tilde\sigma^l\right)
			\leq C\left(1+\frac{1}{k}\right)^{k-1}.
		\end{align*}
		This completes the proof of Lemma \ref{lem es}.
	\end{proof}
With Lemmas \ref{lem es h^ell+1} and \ref{lem es} established, we now return to the proof of \eqref{es L^infty h^ell}.
	From \eqref{es nu^ell}, \eqref{es |Gamma|-drbcl}, \eqref{t^ell_1<0} and \eqref{t^ell_1>0}, by choosing $k$ large, we obtain
	\begin{align}
		|h^{\ell+1}&(t,x_1,v)|\notag\\
		&\leq C_k\|e^{-\frac{\nu_0}{2}t}h_0\|_\infty+Ce^{-\frac{\nu_0}{2}(t-t_1^\ell)}|v_w^\ell(t_1^\ell)|\notag\\
		&\quad+\int_{\max\{t_1^\ell,0\}}^te^{-\frac{\nu_0}{2}(t-s)}\int_{\mathbb R^3}\textbf k_w(v,v_*)|h^\ell(s,X_1^\ell(s),v_*)|{\rm d}v_*{\rm d}s\notag\\
		&\quad+C_k\mathop{\sup}_l\int_{\max \{t_{l+1}^{\ell-l},0\}}^te^{-\frac{\nu_0}{2}(t-s)}\int_{\mathbb R^3}\int_{\mathbb R^3}\textbf k_w(v_l^{\ell-(l-1)},v_*)\notag\\
		&\qquad\times|h^{\ell-l}(s,X_1^{\ell-l}(s;v^{\ell-(l-1)}_{l}),v_*)|\{v^{\ell-(l-1)}_{l,1}-v_w^{\ell-(l-1)}\textbf 1_{x^{\ell-(l-1)}_{l,1}=1}\}\frac{\sqrt{\mu(v^{\ell-(l-1)}_{l})}}{w(v^{\ell-(l-1)}_{l})}{\rm d}v^{\ell-(l-1)}_{l}{\rm d}v_*{\rm d}s\notag\\
		&\quad+C\int_{\max\{t_1^\ell,0\}}^t\langle v\rangle e^{-\int_s^t\frac{\nu^\ell(\tau)}{2}{\rm d}\tau}\|e^{-\frac{\nu_0}{2}(t-s)}h^\ell(s)\|_\infty^2\notag\\
		&\quad+C_k\mathop{\sup}_l\int_{\max\{t_{l+1}^{\ell-l},0\}}^{t_l^{\ell-(l-1)}}\langle v^{\ell-(l-1)}_{l}\rangle e^{-\int_s^{t^{\ell-(l-1)}_{l}}\frac{\nu^{\ell-l}(\tau)}{2}{\rm d}\tau}\|e^{-\frac{\nu_0}{2}(t-s)}h^{\ell-l}(s)\|_\infty^2\notag\\
		&\quad+\Big\{\frac{1}{2}\Big\}^{k}\|e^{-\frac{\nu_0}{2}(t-t_k^{\ell-(k-1)})}h^{\ell+2-k}(t_k^{\ell-(k-1)})\|_\infty
+Cke^{-\frac{\nu_0}{2}t}\mathop{\max}_{1\leq l\leq k-1}\mathop{\sup}_{0\leq s\leq t}|v_w^{\ell-l}(s)|.\notag
	\end{align}
	On the other hand, from \eqref{ODE x^ell+1}, we get
	\begin{equation}
		|x_w^{\ell+1}(t)|+|v_w^{\ell+1}(t)|\lesssim (|x_{w0}|+|v_{w0}|)+\int_0^t(|x_w^{\ell+1}(s)|+|v_w^{\ell+1}(s)|+\|wf^{\ell}(s)\|_\infty){\rm d}s.\notag
	\end{equation}
Combining the above estimates, we derive
	\begin{equation} \|h^{\ell+1}(t)\|_\infty+|x_w^{\ell+1}(t)|+|v_w^{\ell+1}(t)|
\lesssim_k\|h(0)\|_\infty+(|x_{w0}|+|v_{w0}|)+\left(\left\{\frac{1}{2}\right\}^k+t\right)\vps_0+\vps_0^2.\notag
	\end{equation}
	By taking $T\ll1$, $k$  sufficiently large, and $0\leq t\leq T$, we conclude \eqref{es L^infty h^ell}.

\medskip
\noindent {\bf Step 3. Local regularity.} We begin by noting that
\begin{align*}
	\gamma^\ell_{+}&=\{x_1=0,1,  v\in\mathbb R^3\big| [v_1-v^\ell_w(t)x_1]n(x_1)>0\},\\
	\gamma_-^{\ell}&=\{x_1=0,1,  v\in\mathbb R^3\big|  [v_1-v^\ell_w(t)x_1]n(x_1)<0\},
\end{align*}
	and
	\begin{equation*}
		{\rm d}\tilde{\gamma}^\ell=|v_1|{\rm d}v,\ \textrm{if}\ x_1=0,\ \textrm{and}\ {\rm d}\tilde{\gamma}^\ell=|v_1-v_{w}^\ell(t)|{\rm d}v,\ \textrm{if}\ x_1=1.
	\end{equation*}
	We define the following kinetic distance weight
	\begin{align*}
		{W_c}^{f^{\ell}}(t,x_1,v)&:=\alpha_\vps^{f^{\ell}}(t,x_1,v_1)e^{\zeta'|v|^2}\\
		&:=\chi(\frac{t-t_{\textbf b}^{f^{\ell}}(t,x_1,v_1)+\vps}{\vps})|v_1-v^{\ell-1}_w(t-t_{\textbf b}^{f^{\ell}}(t,x_1,v_1))x_{1\textbf b}^{f^\ell}|\\
		&\qquad+\Big[1-\chi(\frac{t-t_{\textbf b}^{f^{\ell}}(t,x_1,v_1)+\vps}{\vps})\Big]	e^{\zeta'|v|^2},
	\end{align*}
	where
	\begin{equation*}
		t_{\textbf b}^{f^{\ell}}(t,x_1,v_1)=\sup\{s\geq0 : X_1^{\ell-1}(\tau;t,x_1,v_1)\in(0,1)\ for\ all\ \tau\in(t-s,t)\}.
	\end{equation*}
	We also define the following norm:
	\begin{align*}
		\mathcal E^\ell(t):&=\|f^\ell(t)\|_{p}^p+\int_0^t\int_{\gamma_+^{\ell-1}}|f^{\ell}|^p{\rm d}\tilde{\gamma}^{\ell-1}{\rm d}s\notag\\
		&\quad+\|W^{f^\ell}_c\pa_{x_1}f^\ell(t)\|_{p}^p+\int_0^t\int_{\gamma_+^{\ell-1}}|W^{f^\ell}_c\pa_{x_1}f^\ell(t)|^p{\rm d}\tilde{\gamma}^{\ell-1}{\rm d}s,\ p\in(2,+\infty).
	\end{align*}
	We then claim that there exists $T^*\ll1$ ($T^*\leq T$ ) such that
	\begin{align}\label{es E^ell}
		\mathop{\max}_{\ell \geq m\geq0}\mathop{\sup}_{0\leq t\leq T^*}\mathcal E^m(t)&\leq C\|f_0\|_{p}^p+C\|W^{f_0}_c\pa_{x_1}f_0\|_{p}^p+C\vps_0^p.
	\end{align}

	The proof of \eqref{es E^ell} follows similarly to that of \eqref{loc-reg-f} in Proposition \ref{loc-reg}. First, \eqref{es E^ell} holds for
	$m=0$ by \eqref{to-id-drbcl}.
	Assume now that there exists $T^*<T$ such that \eqref{es E^ell} is valid for $0\leq m\leq \ell$.
	From Lemmas \ref{lem green fun} and \ref{es K}, we obtain
	\begin{align*}
		\| f^{\ell+1}(t)\|_{p}^p&+\int_0^t\int_{\gamma_+^\ell}|f^{\ell+1}|^p{\rm d}\tilde{\gamma}^\ell{\rm d}s+\int_0^t\left\|\nu^{\frac{1}{p}}f^{\ell+1}\right\|_{p}^p{\rm d}s\notag\\
		\leq& \left\| f(0)\right\|_{p}^p+\int_0^t\int_{\gamma_-^\ell}|f^{\ell+1}|^p{\rm d}\tilde{\gamma}^\ell{\rm d}s+\int_0^t\int_{\Omega\times\mathbb R^3}(1+x^\ell_w(s))|Kf^\ell||f^{\ell+1}|^{p-2}f^{\ell+1}{\rm d}x_1{\rm d}v{\rm d}s\\
		&+\int_0^t\int_{\Omega\times\mathbb R^3}(1+x_w^\ell(s))\big|\Gamma_{+}(f^\ell,f^\ell)-\Gamma_{-}(f^\ell,f^{\ell+1})\big||f^{\ell+1}|^{p-2}f^{\ell+1}{\rm d}x_1{\rm d}v{\rm d}s\\
		\leq& C\left\| f(0)\right\|_{p}^p+\int_0^t\int_{\gamma_-^\ell}|f^{\ell+1}|^p{\rm d}\tilde{\gamma}^\ell{\rm d}s+CT^*(1+\mathop{\sup}_{0\leq s\leq t}\|w f^\ell(s)\|_\infty)\mathop{\sup}_{0\leq s\leq t}\|(f^\ell,f^{\ell+1})(s)\|_{p}^p.
	\end{align*}
	Now, focusing on $\int_0^t\int_{\gamma_-^\ell}|f^{\ell+1}|^p{\rm d}\tilde{\gamma}^\ell{\rm d}s$, by Lemma \ref{lem trace}, we proceed with a similar calculation as for deriving \eqref{p,-,1} to obtain
	\begin{align}
		\int_0^t\int_{\gamma_-^\ell}|f^{\ell+1}|^p{\rm d}\tilde{\gamma}^\ell{\rm d}s\lesssim& o(1)\int_0^t\int_{\gamma_+^\ell}|f^{\ell+1}|^p{\rm d}\tilde{\gamma}^\ell{\rm d}s+\|f_0\|_{p}^p+\int_0^t\|f^{\ell+1}\|_{p}^p{\rm d}s\notag\\
		&+\int_0^t\int_{\Omega\times\mathbb R^3}\Big|\Big[\partial_t+\frac{v_1-v^\ell_w(t)x_1}{1+v^\ell_w(t)}\partial_{x_1}\Big]f^{\ell+1}|f^{\ell+1}|^{p-2}f^{\ell+1}\Big
		|{\rm d}x_1{\rm d}v{\rm d}s+\int_0^t|v_w^\ell(s)|^p{\rm d}s\notag\\
		\lesssim& \|f_0\|_{p}^p+CT^*(1+\mathop{\sup}_{0\leq s\leq t}\|w f^\ell(s)\|_\infty)\mathop{\sup}_{0\leq s\leq t}\|(f^\ell,f^{\ell+1})(s)\|_{p}^p\notag\\
		&+o(1)\int_0^t\int_{\gamma_+^\ell}|f^{\ell+1}|^p{\rm d}\tilde{\gamma}^\ell{\rm d}s
		+\int_0^t|v_w^\ell(s)|^p{\rm d}s.\notag
	\end{align}
	As a consequence, by the induction hypothesis, we arrive at
	\begin{align}\label{es f^ell L^p}
		\mathop{\sup}_{0\leq t\leq T^*}&\left\| f^{\ell+1}(t)\right\|_{p}^p	+\int_0^{T^*}\int_{\gamma_+^\ell}|f^{\ell+1}|^p{\rm d}\tilde{\gamma}^\ell{\rm d}s\notag\\
		\lesssim& \|f_0\|_{p}^p+CT^*(1+\mathop{\sup}_{0\leq t\leq T^*}\|w f^\ell(t)\|_\infty)\mathop{\sup}_{0\leq t\leq T^*}\|f^\ell(t)\|_{p}^p+\int_0^{T^*}|v_w^\ell(s)|^p{\rm d}s\notag\\
		\lesssim& \|f_0\|_{p}^p+\|W^{f_0}_c\pa_{x_1}f_0\|_{p}^p+\vps_0^p.
	\end{align}
	
	Next, applying $\pa_{x_1}$ to \eqref{PBE F^ell+1}, one gets
	\begin{align}
		(\partial_t&+\frac{v_1-v_w^\ell(t)x_1}{1+x^\ell_w(t)}\partial_{x_1}+\nu)(\partial_{x_1}f^{\ell+1})\notag\\
		&\quad=\frac{v^\ell_w(t)}{1+x^\ell_w(t)}\partial_{x_1}f^{\ell+1}+\pa_{x_1}\big(\Gamma_{+}(f^\ell,f^\ell)-\Gamma_{-}(f^\ell,f^{\ell+1})\big)
		-K\partial_{x_1}f^{k}.\notag
	\end{align}
	Using Lemma \ref{lem green fun}, we can derive the estimate
	\begin{align}\label{es pa f^ell L^p}
		\left\| {W_c}^{f^{\ell+1}}\pa_{x_1}f^{\ell+1}(t)\right\|_{p}^p&+\int_0^t\int_{\gamma_+^\ell}|{W_c}^{f^{\ell+1}}\pa_{x_1}f^{\ell+1}|^p{\rm d}\tilde{\gamma}^\ell{\rm d}s+\int_0^t\left\|\nu^{\frac{1}{p}}{W_c}^{f^{\ell+1}}\pa_{x_1}f^{\ell+1}\right\|_{p}^p{\rm d}s\notag\\
		\lesssim& \|{W_c}^{f_0}\pa_{x_1}f_0\|_{p}^p+\int_0^t\int_{\gamma_-^\ell}|W^{f^{\ell+1}}_c\pa_{x_1}f^{\ell+1}|^p{\rm d}\tilde{\gamma}^\ell{\rm d}s\notag\\
		&+CT^*(1+\mathop{\sup}_{0\leq s\leq t}\|w f^\ell(s)\|_\infty+\mathop{\sup}_{0\leq s\leq t}\|w f^{\ell+1}(s)\|_\infty)\mathop{\sup}_{0\leq s\leq t}\mathop{\max}_{m=\ell,\ell+1}\mathcal E^m.
	\end{align}
	Following the same reasoning as in the derivation of \eqref{es gamma_-}, we get the bound:
	\begin{align}\label{es f^ell ga_-}
		\int_0^t&\int_{\gamma_-^\ell}|W^{f^{\ell+1}}_s\pa_{x_1}f^{\ell+1}|^p{\rm d}\tilde{\gamma}^\ell{\rm d}s\notag\\
		\lesssim& \|{W_c}^{f_0}\pa_{x_1}f_0\|_{p}^p+o(1)\mathop{\max}_{0\leq k\leq \ell}\mathop{\sup}_{0\leq s\leq t}\mathcal E^k(s)+t(1+\|w f^\ell\|_\infty)\mathop{\max}_{0\leq k\leq \ell}\mathop{\sup}_{0\leq s\leq t}\mathcal E^k(s)\notag\\
		&+t\mathop{\sup}_{0\leq s\leq t}|[v^\ell_w]'(s)|^p+t\mathop{\sup}_{0\leq s\leq t}\|wf^\ell(s)\|_\infty^p.
	\end{align}
	Combining \eqref{es f^ell L^p}, \eqref{es pa f^ell L^p}, and \eqref{es f^ell ga_-}, we can conclude \eqref{es E^ell} by the induction hypothesis and by choosing $0<T^*\ll1$.

\medskip
\noindent{\bf Step 4. Convergence.}
	From \eqref{PBE F^ell+1} and \eqref{ODE x^ell+1}, it follows
	\begin{align*}
		\left\{	\begin{array}{l}
			\partial_t[f^{\ell+1}-f^\ell]+\frac{v_1-v^{\ell}_w(t)x_1}{1+x^{\ell}_w(t)}\partial_{x_1}[f^{\ell+1}-f^{\ell}]+L[f^{\ell+1}-f^{\ell}]\vspace{1ex}\\
			\quad=\Big[\frac{v_1-v^{\ell-1}_w(t)x_1}{1+x^{\ell-1}_w(t)}-\frac{v_1-v^{\ell}_w(t)x_1}{1+x^{\ell}_w(t)}\Big]\partial_{x_1}f^\ell\vspace{1ex}\\
			\qquad+\Gamma_{+}(f^\ell,f^\ell)-\Gamma_{-}(f^\ell,f^{\ell+1})-\Gamma_{+}(f^{\ell-1},f^{\ell-1})+\Gamma_{-}(f^{\ell-1},f^{\ell}),\\ (f^{\ell+1}-f^\ell)(0,x_1,v)=0,\vspace{1ex}
		\end{array}\right.
	\end{align*}
	and
	\begin{align*}
		\left\{	\begin{array}{l}
			\frac{{\rm d}}{{\rm d}t}[x^{\ell+1}_w(t)-x_w^\ell(t)]=v^{\ell+1}_w(t)-v_w^\ell(t),\vspace{1ex}\\	\frac{{\rm d}}{{\rm d}t}[v^{\ell+1}_w(t)-v_w^\ell(t)]=-\kappa[x^{\ell+1}_w(t)-x_w^\ell(t)]\\
			\qquad\qquad-\mathcal M(P_r^{\ell+1}[\bar{F}^\ell]-P_l^{\ell+1}[\bar{F}^\ell])
			+\mathcal M(P_r^{\ell}[\bar{F}^{\ell-1}]-P_l^{\ell}[\bar{F}^{\ell-1}]),\vspace{1ex}\\
			x_w^{\ell+1}(0)-x_w^\ell(0)=0, v_w^{\ell+1}(0)-v_w^\ell(0)=0.
		\end{array}\right.
	\end{align*}
	On $\gamma_-^\ell$, if $x_1=0$, one has
	\begin{equation*}
		f^{\ell+1}-f^\ell	=P_{\gamma}(f^{\ell}-f^{\ell-1}).
	\end{equation*}
	If $x_1=1$, we write
	\begin{equation}\label{bc f ell+1-ell}
		[f^{\ell+1}-f^\ell]_{\gamma_-}	=
		\left\{ \begin{array}{l}
			P_{\gamma^\ell}f^{\ell}-P_{\gamma^{\ell-1}}f^{\ell-1}+[r^{\ell}-r^{\ell-1}](t,v),\ \textrm{if}\ v_1<v_{w}^\ell(t)\leq v_w^{\ell-1}(t),\notag\\[2mm]
			\quad\textrm{or}\ \ v_1\leq v_{w}^{\ell-1}(t)< v_w^{\ell}(t),	\\[2mm]
			P_{\gamma^\ell} f^\ell+r^{\ell}-f^\ell,\ \textrm{if}\ v_w^{\ell-1}(t)\leq v_1\leq v_w^\ell(t).
		\end{array}\right.
	\end{equation}
	We then modify the proof of Lemma \ref{lem es px_w} and Proposition \ref{prop es f-g} to obtain the following estimates
	\begin{align}\label{es f-drbcl k+1-k}
		&\left\|f^{\ell+1}(t)-f^\ell(t)\right\|_{1+\delta}^{1+\delta}+\int_0^t\int_{\gamma_{+}^\ell}|f^{\ell+1}-f^\ell|^{1+\delta}{\rm d}\tilde{\gamma}^\ell{\rm d}s\notag\\
		&\quad\lesssim \int_0^t e^{C(1+s)^p}\|f^{\ell+1}-f^\ell\|_{1+\delta}^{1+\delta}{\rm d}s+\int_0^te^{C(1+s)^p}\|f^{\ell}-f^{\ell-1}\|_{1+\delta}^{1+\delta}{\rm d}s\notag\\
		&\qquad+\int_0^t(|x_w^\ell(s)-x_w^{\ell-1}(s)|^{1+\delta}+|v_w^\ell(s)-v_w^{\ell-1}(s)|^{1+\delta}){\rm d}s,
	\end{align}
	and
	\begin{align}\label{es x_w-drbcl k+1-k}
		&|x_{w}^{\ell+1}(t)-x_{w}^{\ell}(t)|^{1+\delta}+|v_{w}^{\ell+1}(t)-v_{w}^\ell(t)|^{1+\delta}\notag\\
		&\lesssim e^{Ct}\Big\{\int_0^t|v_{w}^{\ell}(s)-v_{w}^{\ell-1}(s)|^{1+\delta}{\rm d}s\notag\\
		&\qquad+\int_0^t\int_{v_1-v^{\ell-1}_{w}(s)>0}|f^{\ell}(s,1,v)- f^{\ell-1}(s,1,v)|^{1+\delta}|v_1-v^\ell_{w}(s)|{\rm d}v{\rm d}s\}.
	\end{align}
	Taking an appropriate linear combination of \eqref{es f-drbcl k+1-k} and $\eqref{es x_w-drbcl k+1-k}$ (with $\ell$ replaced by $\ell+1$), we obtain
	\begin{align}\label{es f-drbcl x_w k+1-k}
		&|x_{w}^{\ell+2}(t)-x_{w}^{\ell+1}(t)|^{1+\delta}+|v_{w}^{\ell+2}(t)-v_{w}^{\ell+1}(t)|^{1+\delta}\notag\\
		&\qquad+\left\|f^{\ell+1}(t)-f^\ell(t)\right\|_{1+\delta}^{1+\delta}+\int_0^t\int_{\gamma_{+}^\ell}|f^{\ell+1}-f^\ell|^{1+\delta}{\rm d}\tilde{\gamma}^\ell{\rm d}s\notag\\
		&\quad\lesssim \int_0^t e^{C(1+s)^p}\|f^{\ell+1}-f^\ell\|_{1+\delta}^{1+\delta}{\rm d}s+\int_0^te^{C(1+s)^p}\|f^{\ell}-f^{\ell-1}\|_{1+\delta}^{1+\delta}{\rm d}s\notag\\
		&\qquad+\int_0^t(|x_w^\ell(s)-x_w^{\ell-1}(s)|^{1+\delta}+|v_w^\ell(s)-v_w^{\ell-1}(s)|^{1+\delta}){\rm d}s+e^{Ct}\int_0^t|v_{w}^{\ell+1}(s)-v_{w}^{\ell}(s)|^{1+\delta}{\rm d}s.
	\end{align}
	Next, we define
	\begin{align*}
		\CA^\ell(t)&:=\mathop{\sup}_{0\leq s\leq t}\Big\{|x_{w}^{\ell+2}(s)-x_{w}^{\ell+1}(s)|^{1+\delta}+|v_{w}^{\ell+2}(s)-v_{w}^{\ell+1}(s)|^{1+\delta}
		\\&\qquad+\|f^{\ell+1}(s)-f^\ell(s)\|_{1+\delta}^{1+\delta}
		+\int_0^s\int_{\gamma_{+}^\ell}|f^{\ell+1}-f^\ell|^{1+\delta}{\rm d}\tilde{\gamma}^\ell{\rm d}\tau\Big\}.
	\end{align*}
	By Gr\"{o}nwall's inequality, we get from \eqref{es f-drbcl x_w k+1-k} that
	\begin{align}
		\CA^\ell(t) \leq& Ct e^{C(1+t)^p}(\CA^{\ell-1}(t)+\CA^{\ell-2}(t))\leq \frac{1}{2}(\CA^{\ell-1}(t)+\CA^{\ell-2}(t)),\label{ak-ite-drbcl}
	\end{align}
	where we have chosen $0<t\ll1$ $(t<T^*)$ such that $Ct e^{C(1+t)^p}<\frac{1}{2}$. Let us assume that $T^*$ satisfies $CT^* e^{C(1+T^*)^p}<\frac{1}{2}$.	\eqref{ak-ite-drbcl} further implies that both $\{f^\ell\}_{\ell=1}^\infty$ and $\{[x_w^\ell,v_w^\ell]\}_{\ell=1}^\infty$ are Cauchy sequences, thus there exists $f\in L^{1+\delta}(\Omega\times\mathbb R^3)$ and $[x_w,v_w]\in L^\infty([0,T^*])$ such that
	\begin{equation}\label{f^ell+1-f^ell limit}
		\int_0^t\int_{\gamma_{+}^\ell}|f^{\ell+1}-f^\ell|^{1+\delta}{\rm d}\tilde{\gamma}^\ell{\rm d}s\leq \frac{1}{2}(\CA^{\ell-1}(t)+\CA^{\ell-2}(t)),
	\end{equation}
	and
	\begin{align}\label{strongly 1-drbcl}
		\left\{\begin{array}{rll}
			&f^\ell\rightarrow f\ \textrm{strongly in}\ L^{1+\de}(\Omega\times\mathbb R^3),\\
			& \textrm{and}\
			[x_w^\ell,v_w^\ell]\rightarrow [x_w,v_w]\ \textrm{strongly in}\ L^\infty([0,T^*]),\ \textrm{for small}\ T^*>0.
		\end{array}\right.
	\end{align}
	On the other hand, using \eqref{ak-ite-drbcl} and \eqref{strongly 1-drbcl}, we have
	\begin{align}\label{strongly ga+-drbcl}
		\int_0^t&\int_{\gamma_{+}}|f^{\ell+1}-f^{\ell}|^{1+\delta}{\rm d}\tilde{\gamma}{\rm d}s\notag\\
		&=\Big(\int_0^t\int_{\gamma_{+}}|f^{\ell+1}-f^\ell|^{1+\delta}{\rm d}\tilde{\gamma}{\rm d}s-\int_0^t\int_{\gamma_{+}^\ell}|f^{\ell+1}-f^\ell|^{1+\delta}{\rm d}\tilde{\gamma}^\ell{\rm d}s\Big)+\int_0^t\int_{\gamma_{+}^\ell}|f^{\ell+1}-f^\ell|^{1+\delta}{\rm d}\tilde{\gamma}^\ell{\rm d}s\notag\\
		&\leq C|v_w^{\ell}(t)-v_w(t)|+\int_0^t\int_{\gamma_{+}^\ell}|f^{\ell+1}-f^\ell|^{1+\delta}{\rm d}\tilde{\gamma}^\ell{\rm d}s,
	\end{align}
and
	\begin{align}\label{v^l-v}
	|v^\ell_w(t)-v_w(t)|&=\mathop{\lim}_{n\to\infty}|v^\ell_w(t)-v^n_w(t)|\leq \mathop{\lim}_{n\to\infty}\sum_{j=\ell}^{n-1}|v^j_w(t)-v^{j+1}_w(t)|\notag\\
	&\leq C\mathop{\lim}_{n\to\infty}\sum_{j=\ell}^{n-1}\CA^{j-1}(t)\leq C(\CA^{\ell-1}(t)+\CA^\ell(t)).
\end{align}
	By the boundary condition \eqref{bc f ell+1-ell}, applying \eqref{f^ell+1-f^ell limit}, \eqref{strongly ga+-drbcl}, \eqref{v^l-v} and \eqref{strongly 1-drbcl}, we conclude that
	\begin{equation}\label{strongly 2-drbcl}
		f^\ell\rightarrow f\ \textrm{strongly in}\ L^\infty([0,T],L^{1+\de}(\Omega\times\mathbb R^3))\cap L^{1+\de}([0,T^*]\times \gamma).
	\end{equation}

	On the other hand, from \eqref{es L^infty h^ell} and \eqref{strongly 2-drbcl} we obtain a unique weak-$*$ convergence (up to subsequence if necessary) $f^\ell\overset{\ast}{\rightarrow}f$ in $L^\infty([0,T^*]\times\Omega\times\mathbb R^3)$.
	We now show that $[f,x_w,v_w]$ constructed above is indeed a weak solution of the system \eqref{PBE}, \eqref{trb-ode}, \eqref{ic f}, \eqref{drbl-f} and \eqref{drb-f}. To this end,
	for $\varphi\in C_c^\infty(\mathbb R\times\Omega\times\mathbb R^3)$, we have
	\begin{align}\label{7.3}
		\int_0^{T^*}& \int_{\Omega\times\mathbb R^3}(1+x_w^\ell(t))f^{\ell+1}\Big[\partial_t+\frac{v_1-v_w^\ell(t)}{1+x_w^\ell(t)}\pa_{x_1}+\nu\Big]\varphi{\rm d}x_1{\rm d}v{\rm d}t\notag\\
		&=\int_0^{T^*}\int_{\Omega\times\mathbb R^3}(1+x_w^\ell(t))[Kf^\ell+\Gamma_{+}(f^\ell,f^\ell)-\Gamma_{-}(f^\ell,f^{\ell+1})]\varphi{\rm d}x_1{\rm d}v{\rm d}t\notag\\
		&\quad+\int_0^{T^*}\int_{\gamma_+^\ell}f^{\ell+1}\varphi{\rm d}\tilde\gamma^\ell{\rm d}t-\int_0^{T^*}\int_{\gamma_-^\ell}f^{\ell+1}\varphi{\rm d}\tilde\gamma^\ell{\rm d}t.
	\end{align}
	Except $\Gamma_{+}(f^\ell,f^\ell)$ and $\Gamma_{-}(f^\ell.k^{\ell+1})$ in \eqref{7.3}, all other terms converges to their limits with $[f,x_w,v_w]$ instead of $[f^\ell,x_w^\ell,v_w^\ell]$ or $[f^{\ell+1},x_w^\ell,v_w^\ell]$. To handle these two remaining terms, we follow the approach introduced in \cite{li-wang}.
	
	First, we consider the term involving $\Gamma_{-}(f^\ell,f^{\ell+1})$. It follows that
	\begin{align*}
		\Big|\int_0^{T^*}&\int_{\Omega\times\mathbb R^3}(1+x_w^\ell(t))\Gamma_{-}(f^\ell,f^{\ell+1})\varphi{\rm d}x_1{\rm d}v{\rm d}t-\int_0^{T^*}\int_{\Omega\times\mathbb R^3}(1+x_w(t))\Gamma_{-}(f,f)\varphi{\rm d}x_1{\rm d}v{\rm d}t\Big|\\
		&\leq\Big|\int_0^{T^*}\int_{\Omega\times\mathbb R^3}(1+x_w^\ell(t))\Gamma_{-}(f^\ell,f^{\ell+1})\varphi{\rm d}x_1{\rm d}v{\rm d}t-\int_0^{T^*}\int_{\Omega\times\mathbb R^3}(1+x_w(t))\Gamma_{-}(f^\ell,f^{\ell+1})\varphi{\rm d}x_1{\rm d}v{\rm d}t\Big|\\\
		&\quad+\Big|\int_0^{T^*}\int_{\Omega\times\mathbb R^3}(1+x_w(t))\int_{\mathbb R^3}|v-u|\{f^\ell(u)-f(u)\}\sqrt{\mu(u)}{\rm d}u f^{\ell+1}(v)\varphi(t,x_1,v){\rm d}x_1{\rm d}v{\rm d}t\Big|\\
		&\quad+\Big|\int_0^{T^*}\int_{\Omega\times\mathbb R^3}(1+x_w(t))\int_{\mathbb R^3}|v-u|f(u)\sqrt{\mu(u)}{\rm d}u\{f^{\ell+1}(v)-f(v)\} \varphi(t,x_1,v){\rm d}x_1{\rm d}v{\rm d}t\Big|.
	\end{align*}
	The first term above converges to zero from the strong convergence $x_w^\ell$ in $L^\infty([0,{T^*}])$ and uniform bound \eqref{es L^infty h^ell}. The third term converges to zero from the weak$-*$ convergence $f^{\ell+1}$ in $L^\infty([0,{T^*}]\times\Omega\times\mathbb R^3)$ and uniform bound \eqref{es L^infty h^ell}.
	For the second term, by changing the variables $(u,v)\leftrightarrow(v,u)$, we obtain
	\begin{align}
		\int_{\mathbb R^3\times\mathbb R^3}&|v-u|\{f^\ell(u)-f(u)\}\sqrt{\mu(u)} f^{\ell+1}(v)\varphi(t,x_1,v){\rm d}u{\rm d}v\notag\\
		&=\int_{\mathbb R^3\times\mathbb R^3}|v-u|\{f^\ell(v)-f(v)\}\sqrt{\mu(v)} f^{\ell+1}(u)\varphi(t,x_1,u){\rm d}u{\rm d}v\notag\\
		&=\int_{\mathbb R^3}\{f^\ell(v)-f(v)\}\Big(\int_{\mathbb R^3}|v-u|\sqrt{\mu(v)} f^{\ell+1}(u)\varphi(t,x_1,u){\rm d}u\Big){\rm d}v\notag\\
		&=:\int_{\mathbb R^3}\{f^\ell(v)-f(v)\} \mathcal S_1(t,x_1,v){\rm d}v.\notag
	\end{align}
	We have $|\mathcal S_1|\leq \sqrt{\mu(v)}\int_{\mathbb R^3}|v-u|w^{-1}(u)\|wf^{\ell+1}\|_{\infty}|\varphi(t,x_1,u)|{\rm d}u\in L^1_{t,x_1,v}$. Therefore, from the weak$-*$ convergence $f^{\ell+1}$ in $L^\infty([0,{T^*}]\times\Omega\times\mathbb R^3)$ and uniform bound \eqref{es L^infty h^ell}, we have
	\begin{equation*}
		\int_0^{T^*}\int_{\Omega\times\mathbb R^3}(1+x_w(t))	\{f^\ell(v)-f(v)\}\mathcal S_1(t,x_1,v){\rm d}x_1{\rm d}v{\rm d}t\rightarrow0,
	\end{equation*}
	and
	\begin{equation*}
		\int_0^{T^*}\int_{\Omega\times\mathbb R^3}(1+x_w^\ell(t))\Gamma_{-}(f^\ell,f^{\ell+1})\varphi{\rm d}x_1{\rm d}v{\rm d}t\rightarrow\int_0^{T^*}\int_{\Omega\times\mathbb R^3}(1+x_w(t))\Gamma_{-}(f,f)\varphi{\rm d}x_1{\rm d}v{\rm d}t.
	\end{equation*}
	
	We now turn to consider the term involving $\Gamma_{+}(f^\ell,f^\ell)$. By a standard change of variables $(v,u)\mapsto(v',u')$ and $(v,u)\mapsto(u',v')$, we get
	\begin{align}
		\int_0^{T^*}&(1+x_w^\ell(t))\int_{\Omega\times\mathbb R^3}\Gamma_{+}(f^\ell,f^\ell)\varphi{\rm d}x_1{\rm d}v{\rm d}t-\int_0^{T^*}(1+x_w(t))\int_{\Omega\times\mathbb R^3}\Gamma_{+}(f,f)\varphi{\rm d}x_1{\rm d}v{\rm d}t\notag\\
		&=\int_0^{T^*}(1+x_w^\ell(t))\int_{\Omega\times\mathbb R^3}\Gamma_{+}(f^\ell,f^\ell)\varphi{\rm d}x_1{\rm d}v{\rm d}t-\int_0^{T^*}(1+x_w(t))\int_{\Omega\times\mathbb R^3}\Gamma_{+}(f^\ell,f^\ell)\varphi{\rm d}x_1{\rm d}v{\rm d}t\notag\\
		&\quad+\int_0^{T^*}(1+x_w(t))\int_{\Omega\times\mathbb R^3}\Gamma_{+}(f^\ell-f,f^\ell)\varphi{\rm d}x_1{\rm d}v{\rm d}t+\int_0^{T^*}(1+x_w(t))\int_{\Omega\times\mathbb R^3}\Gamma_{+}(f,f^\ell-f)\varphi{\rm d}x_1{\rm d}v{\rm d}t\notag\\ \label{gain 1-drbcl}
		&=	\int_0^{T^*}(x_w^\ell(t)-x_w(t))\int_{\Omega\times\mathbb R^3}\Gamma_{+}(f^\ell,f^\ell)\varphi{\rm d}x_1{\rm d}v{\rm d}t\\
		&\quad+\int_0^{T^*}(1+x_w(t))\int_{\Omega\times\mathbb R^3}\Big(\int_{\mathbb R^3\times \S^2}(f^\ell(t,x_1,u)-f(t,x_1,u))\sqrt{\mu(u')}|(v-u)\cdot\omega|\varphi(t,x_1,v'){\rm d}u{\rm d}\omega\Big)\notag\\ \label{gain 2-drbcl}
		&\qquad\times f^\ell(t,x_1,v){\rm d}x_1{\rm d}v{\rm d}t\\
		&\quad+\int_0^{T^*}(1+x_w(t))\int_{\Omega\times\mathbb R^3}\Big(\int_{\mathbb R^3\times \S^2}(f^\ell(t,x_1,u)-f(t,x_1,u))\sqrt{\mu(v')}|(v-u)\cdot\omega|\varphi(t,x_1,u'){\rm d}u{\rm d}\omega\Big)\notag\\ \label{gain 3-drbcl}
		&\qquad\times f(t,x_1,v){\rm d}x_1{\rm d}v{\rm d}t.
	\end{align}
Note that \eqref{gain 1-drbcl} vanishes as $\ell\rightarrow\infty$ according to the strong convergence \eqref{strongly 1-drbcl} and the uniform bound \eqref{es L^infty h^ell}.
	
	Next, for \eqref{gain 2-drbcl}, we  have
	\begin{align*}
		&\int_{\mathbb R^3\times\mathbb R^3\times \S^2}(f^\ell(t,x_1,u)-f(t,x_1,u))\sqrt{\mu(u')}|(v-u)\cdot\omega|f^\ell(t,x_1,v)\varphi(t,x_1,v'){\rm d}u{\rm d}v{\rm d}\omega\notag\\
		&\quad=\int_{\mathbb R^3}(f^\ell(t,x_1,u)-f(t,x_1,u))\Big(\int_{\mathbb R^3\times \S^2}\sqrt{\mu(u')}|(v-u)\cdot\omega|f^\ell(t,x_1,v)\varphi(t,x_1,v'){\rm d}v{\rm d}\omega\Big){\rm d}u\notag\\
		&\quad=:\int_{\mathbb R^3}(f^\ell(t,x_1,u)-f(t,x_1,u))\mathcal S_2(t,x_1,u){\rm d}u.
	\end{align*}
	Because $\varphi(t,x_1,v')$ has compact support, we have
	\begin{equation*}
		\sqrt{\mu(u')}=e^{-\frac{1}{4}(|u|^2+|v|^2-|v'|^2)}\lesssim \sqrt{\mu(u)}\sqrt{\mu(v)},\ {\rm for}\ |v'|\leq C,
	\end{equation*}
	where $C>0$ is a constant. Hence $|\mathcal S_2|\leq \sqrt{\mu(u)}\int_{\mathbb R^3}|v-u|\sqrt{\mu(v)}\|f^{\ell}\|_{\infty}|\varphi(t,x_1,v')|{\rm d}v\in L^1_{t,x_1,u}$. We obtain that $\eqref{gain 2-drbcl}\rightarrow0$ as $\ell\rightarrow\infty$. Similarly, one can show $\eqref{gain 3-drbcl}\rightarrow0$ as $\ell\rightarrow\infty$.
	This proves the existence of a (weak) solution of $f\in L^\infty$.
	
	Next we prove \eqref{loc es par f}. Modify the proof of \eqref{5.26}, we can obtain
	\begin{equation}
		\mathop{\sup}_\ell\mathop{\sup}_{0\leq t\leq T^*}\Big\{\|W_{c}^{f^{\ell+1}}\pa_{x_1}f^{\ell+1}(t)\|_{p}^p+\int_0^t|W_{c}^{f^{\ell+1}}\pa_{x_1}f^{\ell+1}(\tau)|_{L_{\tilde{\gamma}^\ell}^p,+}^p{\rm d}\tau\Big\}<\infty,\notag
	\end{equation}
where $|\cdot|^p_{L_{\tilde{\gamma}^\ell}^p,+}=\int_{\ga^\ell_+}|\cdot|^p{\rm d}\tilde{\ga}^\ell=\int_{\ga^\ell_+}|\cdot|^p|v_1-v^\ell_w(t)x_1|{\rm d}v$ for $p\geq1$.

By the lower semi-continuity of $L^p$, we know that (if necessary we extract a subsequence)
	\begin{equation*}
		W_{c}^{f^{\ell+1}}\pa_{x_1}f^{\ell+1}\rightharpoonup\mathcal F,
	\end{equation*}
	and
	\begin{equation*}
		\mathop{\sup}_{0\leq t\leq T^*}\|\mathcal F(t)\|_{p}^p\leq \lim\inf\mathop{\sup}_{0\leq t\leq T^*}\|W_{c}^{f^{\ell+1}}\pa_{x_1}f^{\ell+1}(t)\|_{p}^p,
	\end{equation*}
	and
	\begin{equation*}
		\int_0^{T^*}|\mathcal F(t)|_{p}^p{\rm d}t\leq \lim\inf\int_0^{T^*}|W_{c}^{f^{\ell+1}}\pa_{x_1}f^{\ell+1}(t)|_{L_{\tilde{\gamma}^\ell}^p,+}^p{\rm d}t.
	\end{equation*}
	It remains now to prove
	\begin{equation}\label{equ mathF-drbcl}
		\mathcal F={W_c}\pa_{x_1}f\ \textrm{almost everywhere except}\ \gamma_0.
	\end{equation}
	To this end,
	for $N\in\mathbb N$, we define a set
	\begin{align}
		\mathcal S_N:=\Big\{(x_1,v)\in\bar\Omega\times\mathbb R^3:\min\{|x_1|,|x_1-1|\}\leq\frac{1}{N}\ \textrm{and}\ |v_1-v_w(t)x_1|\leq\frac{1}{N}\Big\}\cup\{|v_1|> N\}.\notag
	\end{align}
	For a given test function $\psi\in C_c^\infty(\bar\Omega\times\mathbb R^3\setminus\gamma_0)$, we can always find $N\gg1$ such that
	\begin{equation}
		supp(\psi)\subset(\mathcal S_N)^c:=\bar\Omega\times\mathbb R^3\setminus\mathcal S_N.\notag
	\end{equation}
	By an integration by parts, we obtain
	\begin{align}
		\int_0^{T^*}&\int_{\Omega\times\mathbb R^3}e^{\zeta'|v|^2}(\alpha_{\vps}^{f^{\ell+1}})^\vho\pa_{x_1}f^{\ell+1}\psi{\rm d}x_1{\rm d}v{\rm d}t\notag\\\label{7.23-drbcl}
		&=-\int_0^{T^*}\int_{\Omega\times\mathbb R^3}e^{\zeta'|v|^2}(\alpha_{\vps}^{f^{\ell+1}})^\vho f^{\ell+1}\pa_{x_1}\psi{\rm d}x_1{\rm d}v{\rm d}t\\ \label{7.24-drbcl}
		&\quad+\int_0^{T^*}\int_{\mathbb R^3}e^{\zeta'|v|^2}(\alpha_{\vps}^{f^{\ell+1}})^\vho f^{\ell+1}\psi(t,1,v){\rm d}v{\rm d}t-\int_0^{T^*}\int_{\mathbb R^3}e^{\zeta'|v|^2}(\alpha_{\vps}^{f^{\ell+1}})^\vho f^{\ell+1}\psi(t,0,v){\rm d}v{\rm d}t\\
		\label{7.25-drbcl}
		&\quad-\int_0^{T^*}\int_{\Omega\times\mathbb R^3}e^{\zeta'|v|^2}\pa_{x_1}(\alpha_{\vps}^{f^{\ell+1}})^\vho f^{\ell+1}\psi{\rm d}x_1{\rm d}v{\rm d}t.
	\end{align}
	From \eqref{def alpha} and \eqref{es L^infty h^ell}, if $(x_1,v)\in(\mathcal S_N)^c$, then
	\begin{equation*}
		\mathop{\sup}_k|\alpha_{\vps,k}^\vho(t,x_1,v)|\lesssim |v_1|^\vho+1\lesssim N^\vho+1<\infty.
	\end{equation*}
	Hence, we can extract the subsequence $\{k_N\}$ such that $(\alpha_{\vps}^{f^{\ell+1}})^\vho\overset{\ast}{\rightharpoonup} \overline{\alpha_{\vps}^\vho}\in L^\infty$ weak-$*$ in $L^\infty((0,T^*)\times(\mathcal S_N)^c)\cap L^\infty((0,T^*)\times(\gamma\cap(\mathcal S_N)^c))$. Note that $(\alpha_{\vps}^{f^{\ell+1}})^\vho$ satisfies that $\Big[\pa_t+\frac{v_1-v_w^\ell x_1}{1+x_w^\ell}\pa_{x_1}\Big](\alpha_{\vps}^{f^{\ell+1}})^\vho=0$ and $(\alpha_{\vps}^{f^{\ell+1}})^\vho|_{\gamma_-}=|v_1-v_w^\ell(t)x_1|^\vho$. By passing a limit in the weak formulation, we conclude  that $\Big[\pa_t+\frac{v_1-v_wx_1}{1+x_w}\pa_{x_1}\Big]\overline{\alpha_{\vps}^\vho}=0$ and $\overline{\alpha_{\vps}^\vho}|_{\gamma_-}=|v_1-v_w^\ell(t)x_1|^\vho$. Hence we derive $\overline{\alpha_{\vps}^\vho}=\alpha_{\vps}^\vho$ almost everywhere and
	\begin{equation}\label{alpha weak*-drbcl}
		(\alpha_{\vps}^{f^{\ell+1}})^\vho\overset{\ast}{\rightharpoonup}\alpha_\vps^\vho\ weakly-*\ \textrm{in}\ L^\infty((0,T^*)\times(\mathcal S_N)^c)\cap L^\infty((0,T^*)\times(\gamma\cap(\mathcal S_N)^c)).
	\end{equation}
	By the strong convergence \eqref{strongly 2-drbcl} and the weak-$*$ convergence \eqref{alpha weak*-drbcl},
	\begin{align}\label{conver par f^ell 1}
		\eqref{7.23-drbcl}&+\eqref{7.24-drbcl}\notag\\
		&\rightarrow-\int_0^{T^*}\int_{\Omega\times\mathbb R^3}e^{\zeta'|v|^2}\alpha_{\vps}^\vho f\pa_{x_1}\psi{\rm d}x_1{\rm d}v{\rm d}t\notag\\
		&\quad+\int_0^{T^*}\int_{\mathbb R^3}e^{\zeta'|v|^2}\alpha_{\vps}^\vho f\psi(t,1,v){\rm d}v{\rm d}t-\int_0^{T^*}\int_{\mathbb R^3}e^{\zeta'|v|^2}\alpha_{\vps}^\vho f\psi(t,0,v){\rm d}v{\rm d}t.
	\end{align}
	
	For the convergence of \eqref{7.25-drbcl}, we choose $(x_1,v)\in(\mathcal S_N)^c$. From \eqref{def chi}, we derive that
	\begin{equation*}
		\textrm{if}\ t_{\textbf b}^{f^{\ell+1}}\geq t+\vps,\ \textrm{then}\ \alpha_{\vps}^{f^{\ell+1}}(t,x_1,v)=1.
	\end{equation*}
	Now we consider the case of $t_{\textbf b}^{f^{\ell+1}}\leq t+\vps$. If $|v_1-v_w^\ell(t)|\geq 3\mathop{\sup}_t\mathop{\max}_\ell|v_w^\ell(t)|$, then
	\begin{align}\label{7.28-drbcl}
		|v_1-v_w^\ell(t-t_{\textbf b}^{f^{\ell+1}})x_{1\textbf b}^{f^{\ell+1}}|&\geq |v_1|-|v_w^\ell(t-t_{\textbf b}^{f^{\ell+1}})x_{1\textbf b}^{f^{\ell+1}}|\notag\\
		&\geq 2\mathop{\sup}_t\mathop{\max}_\ell|v_w^\ell(t)|-\mathop{\sup}_t\mathop{\max}_\ell|v_w^\ell(t)|\geq \mathop{\sup}_t\mathop{\max}_\ell|v_w^\ell(t)|.
	\end{align}
	If $|v_1-v_w^\ell(t)|\leq 3\mathop{\sup}_t\mathop{\max}_\ell|v_w^\ell(t)|$, then
	\begin{align}\label{7.29-drbcl}
		&|v_1-v_w^\ell(t-t_{\textbf b}^{f^{\ell+1}})x_{1\textbf b}^{f^{\ell+1}}|\notag\\
		&\quad\geq |v_1-v_w^\ell(t)x_1|-|(v_1-v_w^\ell(t))x_1-(v_1-v_w^\ell(t-t_{\textbf b}^{f^{\ell+1}}))x_1|\notag\\
		&\qquad-|(v_1-v_w^\ell(t-t_{\textbf b}^{f^{\ell+1}}))x_1-(v_1-v_w^\ell(t-t_{\textbf b}^{f^{\ell+1}}))x_{1\textbf b}^{f^{\ell+1}}|\notag\\
		&\quad\geq \frac{1}{N}-\mathop{\sup}_t\mathop{\max}_\ell|[v_w^\ell]'(t)|t_{\textbf b}^{f^{\ell+1}}-5\mathop{\sup}_t\mathop{\max}_\ell|v_w^\ell(t)|\notag\\
		&\quad\geq \frac{1}{N}-C\vps_0.
	\end{align}
	From \eqref{par t_b}, \eqref{7.28-drbcl} and \eqref{7.29-drbcl}, we have
	\begin{equation*}
		\mathop{\sup}_{\substack{k\in\mathbb N, (x_1,v)\in(\mathcal S_N)^c\\-\vps\leq t-t_{\textbf b}^{f^{\ell+1}}\leq t\leq T^*}}|\pa_{x_1}(\alpha_{\vps}^{f^{\ell+1}})^\vho|\lesssim \frac{1}{|v_1-v_w^\ell(t-t_{\textbf b}^{f^{\ell+1}})x_{1\textbf b}^{f^{\ell+1}}|^{2-\vho}}\lesssim C_{\vps,T^*,N}.
	\end{equation*}
	Hence we extract  another subsequence such that
	\begin{equation*}
		\pa_{x_1}(\alpha_{\vps}^{f^{\ell+1}})^\vho\overset{\ast}{\rightharpoonup}\pa_{x_1}(\alpha_\vps^\vho)\ \textrm{weak}-*\ in\ L^\infty((-\vps,T^*)\times(\mathcal S_N)^c).
	\end{equation*}
	Clearly,
	\begin{equation}\label{conver par f^ell 2}
		\eqref{7.25-drbcl}\rightarrow-\int_0^{T^*}\int_{\Omega\times\mathbb R^3}e^{\zeta'|v|^2}\pa_{x_1}(\alpha_{\vps}^\vho) f\psi{\rm d}x_1{\rm d}v{\rm d}t.
	\end{equation}
	
	From \eqref{conver par f^ell 1} and \eqref{conver par f^ell 2}, we prove \eqref{equ mathF-drbcl}.
	
	The proof of the continuity of $h^\ell$ is similar to Lemma 26 in \cite{Guo-2010}. As the limit of $h^\ell$, $h$ preserves the continuity in $[0,T^*]\times\{\bar\Omega\times\mathbb R^3\setminus \gamma_0\}$.  Similar to the derivation of \eqref{5.26}, we can obatin the continuity of $\|{W_c}\pa_{x_1}f(t)\|_p^p+\int_0^t|{W_c}\pa_{x_1}f|_{p,+}^p{\rm d}\tau$. Therefore, the proof of Theorem \ref{loc} is complete.
\qed


\section{Appendix}\label{Appendix}

In this appendix, we collect several significant estimates and identities that have been used in the previous sections.

We first derive the explicit form of the drag force caused by the pressure difference on the two sides of the free boundary.
\begin{proof}[\bf{Derivation of \eqref{P_rF-P_lF}}] By \eqref{def-PF} and \eqref{+0-0 data->0}, we can further write $P_r[F]-P_r[F]$ as
\begin{align}
P_r[F]-P_r[F]&=\int_{\mathbb R^3} [v_1-v_w(t)]^2F(t,x_w(t)+0,v){\rm d}v-\int_{\mathbb R^3} [v_1-v_w(t)]^2F(t,x_w(t)-0,v){\rm d}v\notag\\
&=\int_{v_1-v_w(t)<0}[v_1-v_w(t)]^2\mu(v){\rm d}v\notag\\
&\quad+\sqrt{2\pi}\int_{v_1-v_w(t)>0}[v_1-v_w(t)]^2\mu_w(v)\left(\int_{u_1-v_w(t)<0}|u_1-v_w(t)|\mu(u){\rm d}u\right){\rm d}v\notag\\
&\quad-\int_{\mathbb R^3} [v_1-v_w(t)]^2F(t,x_w(t),v){\rm d}v\notag\\
&=\int_{v_1-v_w(t)<0}[v_1-v_w(t)]^2\mu(v){\rm d}v+\frac{\sqrt{2\pi}}{2}\int_{v_1-v_w(t)<0}|v_1-v_w(t)|\mu(v){\rm d}v\notag\\
&\quad-\int_{\mathbb R^3} [v_1-v_w(t)]^2F(t,x_w(t),v){\rm d}v.\label{der.ppm}
\end{align}
This proves \eqref{P_rF-P_lF}.
\end{proof}

We next show the basic estimates for the linearized Boltzmann collision operators $K$ and $L$ and the nonlinear collision operator $\Ga$. This is established in the following three lemmas.

The first lemma is concerned with the integral operator $K$ given by \eqref{k-def}, and its proof has been given by \cite[Lemma 3, pp.727]{Guo-2010}.
	\begin{lemma}\label{k-op}
		Let $K$ be defined as \eqref{k-def}. Denote
		\begin{align}
			\Fk(v,v_\ast)=\Fk_2(v,v_\ast)-\Fk_1(v,v_\ast),\
			\Fk_w(v,v_\ast)=e^{\zeta |v|^2}\Fk(v,v_\ast)e^{-\zeta |v_\ast|^2}\notag
		\end{align}
		with  $0\leq \zeta\ll1$.
		Then, it holds that
		\begin{equation}\notag
			\int_{\R^3} \Fk_w(v,v_\ast)e^{\frac{\eps|v-v_\ast|^2}{8}}dv_\ast\leq \frac{C}{1+|v|},
		\end{equation}
		for $\eps=0$ or any $\eps> 0$ small enough.
		
		Moreover, it holds that
		\begin{align}
			|e^{\zeta |v|^2} (Kf)|\leq C\|e^{\zeta |v|^2} f\|_{L^\infty}.\notag
		\end{align}
		\end{lemma}

\begin{lemma}{\cite[pp.639, Lemma 3.3]{Guo-2006}}
		There is a constant $\de_0>0$ such that
		\begin{align}
			(Lf,f)=( L\{\FI-\FP\}f,\{\FI-\FP\}f)\geq\de_0\|\{\FI-\FP\}f\|_\nu^2,\notag
		\end{align}
		where $\|\cdot\|_\nu=\|\nu^{\frac{1}{2}}\cdot\|_2.$
	\end{lemma}
For polynomial weights, the following lemma was proved in {\cite[pp.~276, Lemma~2.6]{UY-06}} and \cite[pp.730, Lemma 5]{Guo-2010}. The corresponding estimate for exponential weights, with sufficiently small $\zeta$, can be obtained in a similar manner.
\begin{lemma}\label{Ga}
		Let $0\leq \zeta\ll1$, it holds that
		\begin{align}
			\|e^{\zeta |v|^2}\nu^{-1}\Ga(f,g)\|_{L^\infty}\leq C
			\left\{\|e^{\zeta |v|^2}f\|_{L^\infty}\|e^{\zeta |v|^2}g\|_{L^\infty}+\|e^{\zeta |v|^2}f\|_{L^\infty}
			\|e^{\zeta |v|^2} g\|_{L^\infty}\right\}.\notag
		\end{align}

			\end{lemma}

\begin{lemma}\label{Lemma 2.1}
		For fixed $t$, a map
		\begin{equation}\label{map1}
			(x_1,v)\in\Omega\times\mathbb R^3\mapsto(t-t_{\textbf b}(t,x_1,v_1),x_{1\textbf b},v)\in\mathbb R\times\gamma_-		
		\end{equation}
		is one-to-one and
		\begin{equation}\label{Jac det1}
			\Big|\det\Big(\frac{\partial(t-t_{\textbf b}(t,x_1,v_1),v)}{\partial(x_1,v)}\Big)\Big|=\frac{1+x_w(t)}{|v_1-v_w(t-t_{\textbf b})x_{1\textbf b}|}.
		\end{equation}
		Also, a map
		\begin{equation}\label{map2}
			(t,x_1,v)\in\mathbb R\times\gamma_+\mapsto(t-t_{\textbf b}(t,x_1,v_1),x_{1\textbf b},v)\in\mathbb R\times\gamma_-		
		\end{equation}
		is one-to-one and
		\begin{equation}\label{Jac det2}
			\Big|\det\Big(\frac{\partial(t-t_{\textbf b}(t,x_1,v),v)}{\partial(t,v)}\Big)\Big|=\Big|\frac{v_1-v_w(t)x_1}{v_1-v_w(t-t_{\textbf b})x_{1\textbf b}}\Big|.
		\end{equation}
	\end{lemma}
	\begin{proof}
Clearly, maps \eqref{map1} and \eqref{map2} is one-to-one since the characteristic
is a solution of \eqref{ODE X}. We only need to  prove \eqref{Jac det1} and \eqref{Jac det2}.

		First, denote
		\begin{align*}
			\partial_sX_1^f&=\frac{\partial X_1^f(s;t,x_1,v_1)}{\partial s},\	\partial_tX_1^f=\frac{\partial X_1^f(s;t,x_1,v_1)}{\partial t},\\ 	\partial_{x_1}X_1^f&=\frac{\partial X_1^f(s;t,x_1,v_1)}{\partial x_1},\ 	\partial_{v_1}X_1^f=\frac{\partial X_1^f(s;t,x_1,v_1)}{\partial v_1}.
		\end{align*}
		Differentiate $x_{1\textbf b}=X_1^f(t-t_{\textbf b};t,x_{1},v_1)$ with respect to $x_1$ and $v_1$, and use \eqref{ODE X}, \eqref{X}, we have
		\begin{align}\label{par t_b}
				\partial_{x_1}t_{\textbf b}&=\frac{\partial_{x_1}X_1^f(t-t_{\textbf b})}{\partial_{s}X_1^f(t-t_{\textbf b})}=\frac{1+x_w(t)}{v_1-v_w(t-t_{\textbf b})x_{1\textbf b}},\notag\\
				\
				\partial_{v_1}t_{\textbf b}&=\frac{\partial_{v_1}X_1^f(t-t_{\textbf b})}{\partial_{s}X_1^f(t-t_{\textbf b})}=\frac{-t_{\textbf b}}{v_1-v_w(t-t_{\textbf b})x_{1\textbf b}}.
		\end{align}
		Next, it follows that
		\begin{equation*}
			\frac{\partial(t-t_{\textbf b}(t,x_1,v_1),v)}{\partial(x_1,v)}=
			\begin{bmatrix}
				-\partial_{x_1}t_{\textbf b}&-\partial_{v_1}t_{\textbf b}&0&0\\
				0&1&0&0\\
				0&0&1&0\\
				0&0&0&1\\
			\end{bmatrix},
		\end{equation*}
		this together with \eqref{par t_b} gives
		\begin{equation*}
			\det\Big(\frac{\partial(t-t_{\textbf b}(t,x_1,v_1),v)}{\partial(x_1,v)}\Big)=-\partial_{x_1}t_{\textbf b}=-\frac{1+x_w(t)}{v_1-v_w(t-t_{\textbf b})x_{1\textbf b}}.
		\end{equation*}
This proves \eqref{Jac det1}.

		Next, we prove the second result. It is easy to check that
		\begin{equation*}
			\det\Big(\frac{\partial(t-t_{\textbf b}(t,x_1,v_1),v)}{\partial(t,v)}\Big)=1-\partial_t t_{\textbf b}.
		\end{equation*}
		Note that for either $0<\vps\ll1$ or $0<-\vps\ll1$, with $X_1^f(t+\vps;t,x_1,v_1)\in[0,1]$, it holds that
		\begin{equation}\label{t_b+vps}
			t_{\textbf b}(t+\vps,X_1^f(t+\vps;t,x_1,v_1),v_1)	=t_{\textbf b}(t,x_1,v_1)+\vps.
		\end{equation}
		Then differentiate \eqref{t_b+vps} with respect to $\vps$ and take $\vps=0$,
		\begin{equation*}
			(\partial_t+\frac{v_1-v_w(t)x_1}{1+x_w(t)}\partial_{x_1})t_{\textbf b}(t,x_1,v_1)=1.
		\end{equation*}
		This gives that
		\begin{align*}
			\det\Big(\frac{\partial(t-t_{\textbf b}(t,x_1,v_1),v)}{\partial(t,v)}\Big)&=1-\partial_t t_{\textbf b}=\frac{v_1-v_w(t)x_1}{1+x_w(t)}\partial_{x_1}t_{\textbf b}\\
			&=\frac{v_1-v_w(t)x_1}{1+x_w(t)}\times\frac{1+x_w(t)}{v_1-v_w(t-t_{\textbf b})x_{1\textbf b}}\\
			&=\frac{v_1-v_w(t)x_1}{v_1-v_w(t-t_{\textbf b})x_{1\textbf b}}.
		\end{align*}
		This completes the proof of Lemma \ref{Lemma 2.1}.
	\end{proof}
	
	We define the forward exit time
	\begin{equation*}
		t_{\textbf f}(t,x_1,v_1)=\sup\{s\geq0: X_1^f(\tau;t,x_1,v_1)\in\Omega\ \textrm{for\ all}\ \tau\in(t,t+s)\}.
	\end{equation*}
	
	\begin{lemma}\label{Lemma 2.2}
		For fixed $s>0$ so that $t-t_{\textbf b}<s<t$, the map
		\begin{equation}\label{map3}
			(t,x_1,v)\in[0,T]\times\gamma_+\mapsto(X_1^f(s;t,x_1,v_1),v)\in(0,1)\times\mathbb R^3		
		\end{equation}
		is injective. For fixed $s>0$ so that $t<s<t+t_f$, the map
		\begin{equation}\label{map4}
			(t,x_1,v)\in[0,T]\times\gamma_-\mapsto(X_1^f(s;t,x_1,v_1),v)\in(0,1)\times\mathbb R^3		
		\end{equation}
		is also injective. For maps \eqref{map3} and \eqref{map4}, they have the same Jacobian
		\begin{equation}\label{Jac det3}
			\Big|\det\Big(\frac{\partial(X_1^f(s;t,x_1,v_1),v)}{\partial(t,v)}\Big)\Big|=\frac{|v_1-v_w(t)x_1|}{1+x_w(s)}.
		\end{equation}
		Moreover,  the maps
		\begin{equation}\label{map5}
			\begin{split}
				(t,s&,x_1,v)\in[0,T]\times\{-\min\{t,t_{\textbf b}(t,x_1,v_1)\}<s<0\}\times\gamma_+\\
				&\mapsto(t+s,X_1^f(t+s;t,x_1,v_1),v)\in[0,T]\times\Omega\times\mathbb R^3,
			\end{split}	
		\end{equation}
		and
		\begin{equation}\label{map6}
			\begin{split}
				(t,s&,x_1,v)\in[0,T]\times\{0<s<\min\{t_f(t,x_1,v_1),T-t\}\}\times\gamma_-\\
				&\mapsto(t+s,X_1^f(t+s;t,x_1,v_1),v)\in[0,T]\times\Omega\times\mathbb R^3,
			\end{split}	
		\end{equation}
		are both injecture, and the Jacobian of the maps of \eqref{map5} and \eqref{map6} are
		\begin{equation}\label{Jac det4}
			\Big|\det\Big(\frac{\partial(t+s,X_1^f(s;t,x_1,v_1),v)}{\partial(t,s,v)}\Big)\Big|=\frac{|v_1-v_w(t)x_1|}{1+x_w(t+s)}.
		\end{equation}
In \eqref{Jac det3} and \eqref{Jac det4}, $x_1=0$ or $1$.
	\end{lemma}
	\begin{proof}
		The maps are injective since the characteristic is a solution of \eqref{ODE X}.
		
		First, we prove \eqref{Jac det3}. By a direct  computation, we have
		\begin{equation}\label{pa_t X}
			\det \Big(\frac{\partial(X_1^f(s;t,x_1,v_1),v)}{\partial(t,v)}\Big)=\partial_t X_1^f(s;t,x_1,v_1).
		\end{equation}
		Also we have
		\begin{equation*}
			X_1^f(s;t+\vps,X_1^f(t+\vps;t,x_1,v_1),v_1)=X_1^f(s;t,x_1,v_1).
		\end{equation*}
		Differentiate the above equality with respect to $\vps$ and take $\vps=0$, we get
		\begin{equation}\label{X equ}
			[\partial_t+\frac{v_1-v_w(t)x_1}{1+x_w(t)}\partial_{x_1}]X_1^f(s;t,x_1,v_1)=0.
		\end{equation}
From \eqref{pa_t X} and \eqref{X equ}, it holds that
		\begin{align*}
			\det \Big(\frac{\partial(X_1^f(s;t,x_1,v_1),v)}{\partial(t,v)}\Big)&=\partial_t X_1^f(s;t,x_1,v_1)\\
			&=-\frac{v_1-v_w(t)x_1}{1+x_w(t)}\partial_{x_1}X_1^f(s;t,x_1,v_1)\\
			&=-\frac{v_1-v_w(t)x_1}{1+x_w(t)}\cdot\frac{1+x_w(t)}{1+x_w(s)}=-\frac{v_1-v_w(t)x_1}{1+x_w(s)}.
		\end{align*}
Then \eqref{Jac det3} holds.
		
		For \eqref{Jac det4},
		\begin{equation*}
			\det\Big(\frac{\partial(t+s,X_1^f(s;t,x_1,v_1),v)}{\partial(t,s,v)}\Big)=-\partial_t X_1^f(t+s;t,x_1,v_1).
		\end{equation*}
		Similarly, we obtain
		\begin{equation*}
			\det\Big(\frac{\partial(t+s,X_1^f(s;t,x_1,v_1),v)}{\partial(t,s,v)}\Big)=\frac{v_1-v_w(t)x_1}{1+x_w(t+s)}.
		\end{equation*}
		This proves \eqref{Jac det4} and hence completes the proof of Lemma \ref{Lemma 2.2}.
	\end{proof}
	
	\begin{lemma}\label{lem h}
		Suppose $h(t,x_1,v)\in L^1([0,T]\times\Omega\times\mathbb R^3)$, it holds that
		\begin{align}\label{h}
			\int_0^T&\int_{\Omega\times\mathbb R^3}h(t,x_1,v){\rm d}v{\rm d}x_1{\rm d}t\nonumber\\
			&=\int_{\Omega\times\mathbb R^3}\int_{-\min\{T,t_{\textbf b}(T,x_1,v_1)\}}^0 h(T+s,X_1^f(T+s;t,x_1,v_1),v)\frac{1+x_w(T)}{1+x_w(T+s)}{\rm d}s{\rm d}x_1{\rm d}v\nonumber\\
			&\quad+\int_0^T\int_{\gamma_+}\int_{-\min\{t,t_{\textbf b}(t,x_1,v_1)\}}^0 h(t+s,X_1^f(t+s;t,x_1,v_1),v)\frac{1}{1+x_w(t+s)}{\rm d}s{\rm d}\tilde\gamma{\rm d}t.
		\end{align}
	\end{lemma}
	\begin{proof}
		First, the region $\{(t,x_1,v)\in[0,T]\times\Omega\times\mathbb R^3\}$ can be split into two disjoint parts
		\begin{align*}
			A:&=\{(t,x_1,v)\in[0,T]\times\Omega\times\mathbb R^3:t_{\textbf f}(t,x_1,v_1)+t\leq T\},\\
			B:&=\{(t,x_1,v)\in[0,T]\times\Omega\times\mathbb R^3:t_{\textbf f}(t,x_1,v_1)+t> T\}.
		\end{align*}
		We also define
		\begin{align*}
			A'&=\{(t,s,x_1,v)\in[0,T]\times\mathbb R\times\gamma_+:-\min\{t_{\textbf b}(t,x_1,v_1),t\}\leq s\leq0\}\\
			B'&=\{(s,x_1,v)\in[0,T]\times\Omega\times\mathbb R^3:s\leq t_{\textbf b}(T,x_1,v_1)\}.
		\end{align*}
		Denote
		\begin{align*}
			\mathcal A:\ &A'\rightarrow A\\
			&(t,s,x_1,v)\in[0,T]\times\{-\min\{t_{\textbf b}(t,x_1,v_1),t\}\leq s\leq0\}\times\gamma_+\\
			&\quad\mapsto (t+s,X_1^f(t+s;t.x_1,v_1),v)\in[0,T]\times\Omega\times\mathbb R^3.
		\end{align*}
		Since $t_{\textbf f}(t+s,X_1^f(t+s;t,x_1,v_1),v_1)+(t+s)=-s+(t+s)=t\leq T$ if $(x_1,v)\in\gamma_+$ and $-\min\{t_{\textbf b}(t,x_1,v_1),t\}\leq s\leq0$, the map $\mathcal A$ is well-defined. Applying the change of variable $\mathcal A$ and \eqref{Jac det4}, we have
		\begin{equation}\label{h A}
			\begin{split}
				\int_A&h(t,x_1,v){\rm d}v{\rm d}x_1{\rm d}t\nonumber\\
				&\quad=\int_0^T\int_{\gamma_+}\int_{-\min\{t,t_{\textbf b}(t,x_1,v_1)\}}^0 h(t+s,X_1^f(t+s;t,x_1,v_1),v) \frac{1}{1+x_w(t+s)}{\rm d}s{\rm d}\tilde\gamma{\rm d}t.
			\end{split}	
		\end{equation}
		
		Next consider the map
		\begin{equation}\label{map7}
			(s,x_1,v)\in B'\mapsto(T-s,X_1^f(T-s;T,x_1,v_1),v)\in B.
		\end{equation}
		By computations and using \eqref{X}, the Jacobian of map \eqref{map7} is
		\begin{align*}
			\det\Big(\frac{\partial(T-s,X_1^f(T-s;T,x_1,v_1),v)}{\partial(s,x_1,v)}\Big)&=-\partial_{x_1}X_1^f(T-s;T,x_1,v_1)=-\frac{1+x_w(T)}{1+x_w(T-s)}.
		\end{align*}
		Then by the change of variable \eqref{map7} and $s\mapsto-s$, we get
		\begin{align}\label{h B}
			\int_B&h(t,x_1,v){\rm d}v{\rm d}x_1{\rm d}t\notag\\
				&=\int_{\Omega\times\mathbb R^3}\int_{-\min\{T,t_{\textbf b}(T,x_1,v_1)\}}^0 h(T+s,X_1^f(T+s;t,x_1,v_1),v) \frac{1+x_w(T)}{1+x_w(T+s)}{\rm d}s{\rm d}x_1{\rm d}v.
		\end{align}
		Combining \eqref{h A} with \eqref{h B}, one has \eqref{h}. Then, the proof of Lemma \ref{lem h} is complete.
	\end{proof}
	
	\begin{lemma}\label{lem green fun}
		For $p\in[1,\infty)$, we assume $f\in L^p_{loc}(\mathbb R_+\times\Omega\times\mathbb R^3)$ satisfies
		\begin{equation*}
			\partial_t f+\frac{v_1-v_w(t)x_1}{1+x_w(t)}\partial_{x_1}f\in L^p_{loc}(\mathbb R_+\times\Omega\times\mathbb R^3),\ f\in L^p_{loc}(\mathbb R_+;L^p(\gamma_+)).
		\end{equation*}
Then $f\in C_{loc}^0(\mathbb R_+;L^p(\Omega\times\mathbb R^3))$ and $f\in L_{loc}^p(\mathbb R_+;L^p(\gamma_+))$.		Moreover,
		\begin{align}\label{equ green fun}
				p\int_0^T&\int_{\Omega\times\mathbb R^3}(1+x_w(t))\{\partial_t+\frac{v_1-v_w(t)x_1}{1+x_w(t)}\partial_{x_1}\}f|f|^{p-2}f{\rm d}x_1{\rm d}v{\rm d}t\notag\\
				&=\left\|((1+x_w)^{1/p}f)(T)\right\|_{2}^p-\left\|((1+x_w)^{1/p}f)(0)\right\|_{2}^p\notag\\
				&\quad+\int_0^T\int_{\gamma_+}|f(t)|^p{\rm d}\tilde\gamma{\rm d}t-\int_0^T\int_{\gamma_-}|f(t)|^p{\rm d}\tilde\gamma{\rm d}t.	
		\end{align}
	\end{lemma}
	
	\begin{proof}
		By Lemma \ref{lem h}, it holds
		\begin{align*}
			p\int_0^T&\int_{\Omega\times\mathbb R^3}(1+x_w(t))\{\partial_t+\frac{v_1-v_w(t)x_1}{1+x_w(t)}\partial_{x_1}\}f|f|^{p-2}f{\rm d}x_1{\rm d}v{\rm d}t\notag\\
			&=p\int_{\Omega\times\mathbb R^3}\int_{-\min\{T,t_{\textbf b}(T,x_1,v_1)\}}^0 \{\partial_t+\frac{v_1-v_w(t)x_1}{1+x_w(t)}\partial_{x_1}\}f|f|^{p-2}f(T+s,X_1^f(T+s;t,x_1,v_1),v)\nonumber\\
			&\qquad\times (1+x_w(T)){\rm d}s{\rm d}x_1{\rm d}v\nonumber\\
			&\quad+p\int_0^T\int_{\gamma_+}\int_{-\min\{t,t_{\textbf b}(t,x_1,v_1)\}}^0 \{\partial_t+\frac{v_1-v_w(t)x_1}{1+x_w(t)}\partial_{x_1}\}f|f|^{p-2}f(t+s,X_1^f(t+s;t,x_1,v_1),v)\nonumber\\
			&\qquad\times {\rm d}s{\rm d}\tilde\gamma{\rm d}t.
		\end{align*}
		Note that
		\begin{align*}
			\frac{d}{ds}&|f(t+s,X_1^f(t+s;t.x_1,v_1),v)|^p\\
			&=p\{\partial_tf+\frac{v_1-v_w(t)x_1}{1+x_w(t)}\partial_{x_1}f\}|f|^{p-2}f(t+s,X_1^f(t+s;t,x_1,v_1),v)	.
		\end{align*}
		Then,
		\begin{align}\label{equ 2.22}
			p\int_0^T&\int_{\Omega\times\mathbb R^3}(1+x_w(t))\{\partial_t+\frac{v_1-v_w(t)x_1}{1+x_w(t)}\partial_{x_1}\}f|f|^{p-2}f{\rm d}x_1{\rm d}v{\rm d}t\notag\\
			&=\int_{\Omega\times\mathbb R^3}(1+x_w(T))|f(T,x_1,v)|^p{\rm d}x_1{\rm d}v+\int_0^T\int_{\gamma_+}|f(t,x_1,v)|^p{\rm d}\tilde\gamma{\rm d}t\nonumber\\
			&\quad-\underbrace{\int_{\Omega\times\mathbb R^3}1_{\{T\geq t_{\textbf b}(T,x_1,v_1)\}}(1+x_w(T))|f(T-t_{\textbf b},x_{1\textbf b},v)|^p{\rm d}x_1{\rm d}v}_{\eqref{equ 2.22}_1}\nonumber\\
			&\quad-\underbrace{\int_{\Omega\times\mathbb R^3}1_{\{T< t_{\textbf b}(T,x_1,v_1)\}}(1+x_w(T))|f(0,X_1^f(0;T,x_1,v_1),v)|^p{\rm d}x_1{\rm d}v}_{\eqref{equ 2.22}_2}\nonumber\\
			&\quad-\underbrace{\int_0^T\int_{\gamma_+}1_{\{t\geq t_{\textbf b}(t,x_1,v_1)\}}|f(t-t_{\textbf b},x_{1\textbf b},v)|^p{\rm d}\tilde\gamma{\rm d}t}_{\eqref{equ 2.22}_3}\nonumber\\
			&\quad-\underbrace{\int_0^T\int_{\gamma_+}1_{\{t< t_{\textbf b}(t,x_1,v_1)\}}|f(0,X_1^f(0;t,x_1,v_1),v)|^p{\rm d}\tilde\gamma{\rm d}t}_{\eqref{equ 2.22}_4}.
		\end{align}
		In $\eqref{equ 2.22}_2$ and $\eqref{equ 2.22}_4$, we use the changes of variable $(x_1,v)\mapsto(X_1^f(0;T,x_1,v),v)$ and $(t,v)\mapsto(X_1^f(0;t,x_1,v),v)$ respectively to get
		\begin{equation*}
			\eqref{equ 2.22}_2+\eqref{equ 2.22}_4=\int_{\Omega\times\mathbb R^3}(1+x_{w0})|f(0,x_1,v)|^p{\rm d}x_1{\rm d}v.
		\end{equation*}
		By \eqref{Jac det1} and \eqref{Jac det2} in Lemma \ref{Lemma 2.1}, we can claim that
		\begin{equation*}
			\eqref{equ 2.22}_1+\eqref{equ 2.22}_3=\int_0^T\int_{\gamma_-}|f|^p{\rm d}\tilde\gamma{\rm d}t.
		\end{equation*}
		This proves \eqref{equ green fun} and hence completes the proof of Lemma \ref{lem green fun}.
	\end{proof}
	
We define
\begin{equation}\label{def gamma_+^vps}
	\gamma_+^\vps:=\{(x_1,v)\in\gamma_+:|v_1-v_w(t)x_1|\leq \vps\ or\ |v_1|\geq\frac{1}{\vps}\}.
\end{equation}

\begin{lemma}\label{lem trace}
	There exists a constant $C_{\Omega}>0$	such that
	\begin{align}\label{equ trace}
			\int_0^t\int_{\gamma_+\setminus\gamma_+^\vps}|h|{\rm d}\tilde\gamma{\rm d}s\notag
			&\leq C_\Omega\Big\{\left\|(1+x_{w0}) h_0\right\|_{L^1_{x_1,v}}+\int_0^t\left\|(1+x_w(s))h(s)\right\|_{L^1_{x_1,v}}{\rm d}s\notag\\
			&\quad+\int_0^t\left\|(1+x_w)[\partial_t+\frac{v_1-v_w(s)x_1}{1+v_w(t)}\partial_{x_1}]h(s)\right\|_{L^1_{x_1,v}}{\rm d}s\Big\}.
	\end{align}
\end{lemma}
\begin{proof}
	Suppose $h$ solves
	\begin{equation*}
		\partial_th+\frac{v_1-v_w(t)x_1}{1+x_w(t)}\partial_{x_1}h+\Psi h=N.	
	\end{equation*}
	Alone the characteristic, we have
	\begin{align*}
		h(t,x_1,v)&=h(t+s,X_1^f(t+s),v)e^{-\int_{t+s}^t\Psi(\tau',X_1^f(\tau'),v){\rm d}\tau'}\\
		&\quad+\int_{t+s}^t e^{-\int_\tau^t\Psi(\tau',X_1^f(\tau'),v){\rm d}\tau'}N(\tau,X_1^f(\tau),v){\rm d}\tau,
	\end{align*}
	for $(t,x_1,v)\in[0,T]\times\gamma_+$ and $-\min\{t,t_{\textbf b}(t,x_1,v_1)\}\leq s\leq0$. Here, $X_1^f(s):=X_1^f(s;t,x_1,v_1)$. Then
	\begin{align*}
		\min\{&t,t_{\textbf b}(t,x_1,v_1)\}\times h(t,x_1,v)\\
		&=\int_{-\min\{t,t_{\textbf b}(t,x_1,v_1)\}}^0| h(t,x_1,v)|{\rm d}s\\
		&\leq \int_{-\min\{t,t_{\textbf b}(t,x_1,v_1)\}}^0| h(t+s,X_1^f(t+s),v)|{\rm d}s\\
		&\quad+\int_{-\min\{t,t_{\textbf b}(t,x_1,v_1)\}}^0\int_{t+s}^t| N(\tau,X_1^f(\tau),v)|{\rm d}\tau{\rm d}s\\
		&\leq \int_{-\min\{t,t_{\textbf b}(t,x_1,v_1)\}}^0| h(t+s,X_1^f(t+s),v)|{\rm d}s\\
		&\quad+t_{\textbf b}(t,x_1,v_1)\int_{t-\min\{t,t_{\textbf b}(t,x_1,v_1)\}}^t| N(\tau,X_1^f(\tau),v)|{\rm d}\tau.
	\end{align*}

	Now we bound $t_{\textbf b}$. For $t-t_{\textbf b}(t,x_1,v_1)\leq s\leq t$, $|v_1|\geq\vps$, $\vps>4\delta_1$, if $2(|x_{w0}|+|v_{w0}|+C\vps)\leq \delta_1$,
	\begin{align*}
		|X_1^f(t-t_\textbf{b};t,x_1,v_1)-x_1|&\geq \Big|\int_{t-t_\textbf{b}}^t\frac{v_1}{1+x_w(\tau)}{\rm d}\tau\Big|-\Big|\int_{t-t_\textbf{b}}^t\frac{v_w(\tau)x_1}{1+x_w(\tau)}{\rm d}\tau\Big|\\
		&\geq \frac{1}{2}| v_1|t_b-\delta_1 t_b\geq\frac{1}{4}|v_1|t_b.
	\end{align*}
	This implies that $t_b\leq
	\frac{1}{\delta_1}$. On the other hand, for $(t,x_1,v)\in [0,T]\times \gamma_+\setminus\gamma_+^\vps$, one has
	\begin{align*}
		1&=|x_1-x_{1\textbf b}|^2=\Big|\int_{t-t_b}^t\frac{v_1-v_w(\tau)x_1}{1+x_w(\tau)}{\rm d}\tau\Big|^2
		\leq 2t_{\textbf b}^2|v_1|^2+2\delta_1^2t_{\textbf b}^2\\
		&\leq 2t_{\textbf b}^2|v_1|^2+\frac{1}{2}t_{\textbf b}^2|v_1|^2\leq \frac{5}{2}t_{\textbf b}^2|v_1|^2,
	\end{align*}
	which gives the lower bound of $t_{\textbf b}$ as
	\begin{equation}\label{lb-tb}
		t_{\textbf b}\geq\sqrt{\frac{2}{5}}\vps=:\vps_1.
	\end{equation}
	
	Next, for $(t,x_1,v)\in [\vps_1,T]\times \gamma_+\setminus\gamma_+^\vps$, it holds
	\begin{align}\label{t_b large}
		\vps_1\int_{\vps_1}^T&\int_{\gamma_+\setminus\gamma_+^\vps}|h(t,x_1,v)|{\rm d}\tilde\gamma{\rm d}t\nonumber\\
		&\leq\int_{\vps_1}^T\int_{\gamma_+\setminus\gamma_+^\vps}\mathop{\min}_{[\vps_1,T]\times\gamma_+\setminus\gamma_+^\vps}\{t,t_{\textbf b}(t,x_1,v_1)\}\times |h(t,x_1,v)|{\rm d}\tilde\gamma{\rm d}t\nonumber\\
		&\leq \int_{0}^T\int_{\gamma_+\setminus\gamma_+^\vps}\int_{-\min\{t,t_{\textbf b}(t,x_1,v_1)\}}^0| h(t+s,X_1^f(t+s),v)|{\rm d}s{\rm d}\tilde\gamma{\rm d}t\nonumber\\
		&\quad+\mathop{\sup}_{\gamma_+\setminus\gamma_+^\vps}t_{\textbf b}(t,x_1,v_1)\int_{0}^T\int_{\gamma_+\setminus\gamma_+^\vps}\int_{t-\min\{t,t_{\textbf b}(t,x_1,v_1)\}}^t| N(\tau,X_1^f(\tau),v)|{\rm d}\tau{\rm d}\tilde\gamma{\rm d}t\nonumber\\
		&\lesssim \int_0^T\left\|(1+x_w(t))h\right\|_{L^1_{x_1,v}}{\rm d}t+\int_0^T\left\|(1+x_w(t))[\partial_t+\frac{v_1-v_w(t)x_1}{1+x_w(t)}\partial_{x_1}+\psi]h\right\|_{L^1_{x_1,v}}{\rm d}t,
	\end{align}
where we have used Lemma \ref{lem h} in the last inequality.
	On the other hand,
	\begin{equation*}
		t_{\textbf b}(t,x_1,v_1)>t\ \textrm{for all}\ (t,x_1,v)\in[0,\vps_1]\times\gamma_+\setminus\gamma_+^\vps.
	\end{equation*}
By using the change of variables \eqref{map5}, we get
	\begin{align}\label{t_b small}
		\int_0^{\vps_1}&\int_{\gamma_+\setminus\gamma_+^\vps}|h(t,x_1,v)|{\rm d}\tilde\gamma{\rm d}t\nonumber\\
		&=\int_0^{\vps_1}\int_{\gamma_+\setminus\gamma_+^\vps}|h_0(X_1^f(0),v)|{\rm d}\tilde\gamma{\rm d}t+\int_0^{\vps_1}\int_{\gamma_+\setminus\gamma_+^\vps}\int_0^{-t}|N(t+\tau,X_1^f(t+\tau),v)|{\rm d}\tau{\rm d}\tilde\gamma{\rm d}t\nonumber\\
		&\leq \left\|(1+x_{w0})h_0\right\|_{L^1_{x_1,v}}+\int_0^{\vps_1}\left\|(1+x_w(t))[\partial_t+\frac{v_1-v_w(t)x_1}{1+x_w(t)}\partial_{x_1}+\psi]h\right\|_{L^1_{x_1,v}}{\rm d}t.
	\end{align}
	Combining \eqref{t_b large} and \eqref{t_b small} together, we obtain \eqref{equ trace}. The proof of Lemma \ref{lem trace} is complete. 
\end{proof}

\medskip
\noindent {\bf Acknowledgment:}\,
The research of Renjun Duan was partially supported by the grant from the National Natural Science Foundation of China (Project No.~12425109). The research of Shuangqian Liu was supported by grants from the National Natural Science Foundation of China under contract 12325107. RJD would like to thank Prof.~Kazuo Aoki and Prof.~Tetsuro Tsuji for many stimulating helpful discussions and suggestions to the current problem.

\medskip
\noindent{\bf Data availability:} The manuscript contains no associated data.

\medskip
\noindent{\bf Conflict of Interest:} The authors declare that they have no conflict of interest.


\end{document}